\documentclass[11pt,letterpaper]{amsart}

\usepackage[T1]{fontenc}	
\usepackage{totpages}
\usepackage{graphicx}
\usepackage{amsmath, amsthm, amssymb}
\usepackage{upgreek}
\usepackage{mathrsfs}
\usepackage{pdfsync}
\usepackage{cite}
\usepackage[usenames,dvipsnames]{xcolor}
\usepackage[inline,shortlabels]{enumitem}
\usepackage[hyphens]{url}									
\usepackage{verbatim}								
\usepackage{dsfont}									
\usepackage{bbm}
\usepackage{bm}
\usepackage[all]{xy}

\usepackage{mathtools}
\usepackage{xparse}
\usepackage{microtype}

\usepackage{subcaption}
\usepackage{graphicx}
\usepackage[font=small]{caption}

\usepackage[pdftex,bookmarks,
						hidelinks,					
						colorlinks,				
						breaklinks,
						citecolor=OliveGreen,
						linkcolor=Maroon,
						]{hyperref} 

\usepackage{mathtools}

\RequirePackage{silence}
\allowdisplaybreaks

\usepackage{xparse}

\ExplSyntaxOn
\NewDocumentCommand{\makeabbrev}{mmm}
 {
  \yoruk_makeabbrev:nnn { #1 } { #2 } { #3 }
 }

\cs_new_protected:Npn \yoruk_makeabbrev:nnn #1 #2 #3
 {
  \clist_map_inline:nn { #3 }
   {
    \cs_new_protected:cpn { #2 } { #1 { ##1 } }
   }
 }
 \ExplSyntaxOff

\makeabbrev{\textbf}{tbf#1}{a,b,c,d,e,f,g,h,i,j,k,l,m,n,o,p,q,r,s,t,u,v,w,x,y,z,A,B,C,D,E,F,G,H,I,J,K,L,M,N,O,P,Q,R,S,T,U,V,W,X,Y,Z}

\makeabbrev{\textbf}{bf#1}{a,b,c,d,e,f,g,h,i,j,k,l,m,n,o,p,q,r,s,t,u,v,w,x,y,z,A,B,C,D,E,F,G,H,I,J,K,L,M,N,O,P,Q,R,S,T,U,V,W,X,Y,Z}

\makeabbrev{\textsf}{tsf#1}{a,b,c,d,e,f,g,h,i,j,k,l,m,n,o,p,q,r,s,t,u,v,w,x,y,z,A,B,C,D,E,F,G,H,I,J,K,L,M,N,O,P,Q,R,S,T,U,V,W,X,Y,Z}

\makeabbrev{\mathsf}{mss#1}{a,b,c,d,e,f,g,h,i,j,k,l,m,n,o,p,q,r,s,t,u,v,w,x,y,z,A,B,C,D,E,F,G,H,I,J,K,L,M,N,O,P,Q,R,S,T,U,V,W,X,Y,Z}

\makeabbrev{\mathfrak}{mf#1}{a,b,c,d,e,f,g,h,i,j,k,l,m,n,o,p,q,r,s,t,u,v,w,x,y,z,A,B,C,D,E,F,G,H,I,J,K,L,M,N,O,P,Q,R,S,T,U,V,W,X,Y,Z,
sl,gl
}

\makeabbrev{\mathrm}{mrm#1}{a,b,c,d,e,f,g,h,i,j,k,l,m,n,o,p,q,r,s,t,u,v,w,x,y,z,A,B,C,D,E,F,G,H,I,J,K,L,M,N,O,P,Q,R,S,T,U,V,W,X,Y,Z}

\makeabbrev{\mathbf}{mbf#1}{a,b,c,d,e,f,g,h,i,j,k,l,m,n,o,p,q,r,s,t,u,v,w,x,y,z,A,B,C,D,E,F,G,H,I,J,K,L,M,N,O,P,Q,R,S,T,U,V,W,X,Y,Z}

\makeabbrev{\mathcal}{mc#1}{A,B,C,D,E,F,G,H,I,J,K,L,M,N,O,P,Q,R,S,T,U,V,W,X,Y,Z}

\makeabbrev{\mathbb}{mbb#1}{A,B,C,D,E,F,G,H,I,J,K,L,M,N,O,P,Q,R,S,T,U,V,W,X,Y,Z}

\makeabbrev{\mathscr}{ms#1}{A,B,C,D,E,F,G,H,I,J,K,L,M,N,O,P,Q,R,S,T,U,V,W,X,Y,Z}

\makeabbrev{\mathrm}{#1}{
id,ran,rk,diag,stab,ann,conv,pr,ev,tr,End,Hom,sgn,im,op,can,fin,ext,red,tot,lex,Aut,Inn,unit,
rot,usc,lsc,Lip,lip,bSymLip,osc,AC,loc,coz,z,
supp,Opt,Adm,Cpl,Geo,GeoOpt,GeoAdm,GeoCpl,reg,res,
bd,co,Ric,Exp,dExp,dist,seg,Seg,cut,fcut,Cut,SDiff,Iso,Isom,diam,cl,Homeo,Diff,Der,vol,dvol,inj,relint, Graph, sub,
var,law,Var,Poi,Gam,pa,so,iso,fs,inv,pqi,mix,erg,form,
TestF,
ob,cod,inp,
}

\makeabbrev{\mathsf}{#1}{CD,BE,MCP,Ent,wMTW,MTW,Ch,RCD,EVI,Rad,dRad,SL,cSL,dSL,ScL,Irr,SC,wFe,VA,MetMeas,UMeas,CSMet,Met,USp,Meas,Mbl,alg,Alg}

\makeabbrev{\mathsc}{#1}{mmaf,cg}

\newcommand{\inter}{\mathrm{int}}

\newcommand{\boldupalpha}{{\boldsymbol\upalpha}}

\newcommand{\boldnu}{\boldsymbol\nu}
\newcommand{\boldmolli}{\boldsymbol\molli}
\newcommand{\boldsigma}{\boldsymbol\sigma}

\newcommand{\boldvarphi}{\boldsymbol\varphi}

\newcommand{\boldDelta}{\boldsymbol\Delta}
\newcommand{\boldUpsilon}{\boldsymbol\Upsilon}
\newcommand{\boldLambda}{\boldsymbol\Lambda}

\newcommand{\boldOmega}{\boldsymbol\Omega}

\newcommand{\bmssV}{\boldsymbol\mssV}
\newcommand{\bmssW}{\boldsymbol\mssW}
\newcommand{\bmssE}{\boldsymbol\mssE}
\newcommand{\bmssH}{\boldsymbol\mssH}
\newcommand{\bmssh}{\boldsymbol\mssh}
\newcommand{\bmssp}{\boldsymbol\mssp}
\newcommand{\bmssL}{\boldsymbol\mssL}

\newcommand{\bmssT}{\boldsymbol\mssT}

\newcommand{\mssBeta}{\mathsf{B}}
\newcommand{\bmssBeta}{\boldsymbol{\mathsf{B}}}
\newcommand{\mssGamma}{\mathsf{\Gamma}}
\newcommand{\bmssGamma}{\boldsymbol{\mathsf{\Gamma}}}

\newcommand{\Flat}{\textrm{flat}}

\newcommand{\sprod}{\Pi}
\newcommand{\ssum}{\Sigma}
\newcommand{\sotimes}{\otimes}
\newcommand{\shotimes}{\widehat\otimes}
\newcommand{\scap}{\cap}
\newcommand{\scup}{\cup}

\newcommand{\Leb}{\msL}

\newcommand{\mssdiv}{\mathsf{div}}
\renewcommand{\div}{\mathrm{div}}
\newcommand{\msfdiv}{\mathsf{div}}
\newcommand{\T}{\tau} 
\newcommand{\Bo}[1]{\msB_{#1}} 

\newcommand{\DF}[1]{\mcD_{#1}}

\newcommand{\Mb}{\msM_b}
\newcommand{\Mbp}{\msM_b^+}

\newcommand{\FleVio}{\mathsc{fv}}
\newcommand{\DawWat}{\mathsc{dw}}
\newcommand{\BatKan}{\mathsc{bk}}

\newcommand{\Lipu}{\Lip^1}

\newcommand{\martp}[2]{(\mathsc{mp})^{#1}_{#2}}
\newcommand{\hmartp}[2]{(\widehat{\mathsc{mp}})^{#1}_{#2}}

\renewcommand{\iint}{\int\!\!\!\!\!\int}
\renewcommand{\iiint}{\int\!\!\!\!\!\int\!\!\!\!\!\int}

\newcommand{\defeq}{\eqqcolon}
\renewcommand{\complement}{\mathrm{c}}
\newcommand{\acts}{\,\raisebox{\depth}{\scalebox{1}[-1]{$\circlearrowleft$}}\,}
\newcommand{\mathsc}[1]{\text{\textsc{#1}}}
\newcommand{\emparg}{{\,\cdot\,}}
\newcommand{\dint}[2][]{\;\sideset{^{\scriptstyle{#1}}\!\!\!\!}{_{#2}^{\scriptscriptstyle\oplus}}\int}
\newcommand{\forallae}[1]{{\textrm{\,for ${#1}$-a.e.~}}}
\newcommand{\as}[1]{\quad #1\text{-a.e.}}
\renewcommand{\Cap}{\mathrm{cap}}
\newcommand{\dom}[1]{\msD(#1)}

\newcommand{\dotloc}[1]{#1^\bullet_{\loc}}

\DeclareMathOperator{\eqdef}{\coloneqq}

\newcommand{\hCylP}[2]{\widehat{\mcA}^{#1}_{#2}}
\newcommand{\Cyl}[2]{\mfA^{#1}_{#2}}
\newcommand{\hFC}[2]{\widehat{\mcC}^{#1}_{#2}}
\newcommand{\FC}[2]{\mcC^{#1}_{#2}}

\newcommand{\slo}[1]{\abs{\mrmD#1}}

\newcommand{\otym}[1]{{\scriptscriptstyle\otimes #1}}
\newcommand{\tym}[1]{{\scriptscriptstyle\times #1}}

\newcommand{\OF}[1]{\msH_{#1}}

\newcommand{\boldnabla}{\boldsymbol\nabla}

\newcommand{\longrar}{\longrightarrow}
\newcommand{\rar}{\rightarrow}

\newcommand{\nlim}{\lim_{n}}								

\newcommand{\diff}{\mathop{}\!\mathrm{d}}						

\newcommand{\tabs}[1]{\big\lvert#1\big\rvert}	
\newcommand{\abs}[1]{\left\lvert#1\right\rvert}						
\newcommand{\norm}[1]{\left\lVert#1\right\rVert}					
\newcommand{\set}[1]{\left\{#1\right\}}							
\newcommand{\tset}[1]{\big\{#1\big\}}							
\newcommand{\paren}[1]{\left(#1\right)}							
\newcommand{\tparen}[1]{\big({#1}\big)}
\newcommand{\ttparen}[1]{({#1})}
\newcommand{\braket}[1]{\left[#1\right]}							
\newcommand{\tbraket}[1]{\big[#1\big]}

\newcommand{\quadvar}[1]{\left[#1\right]}							
\newcommand{\tquadvar}[1]{\big[#1\big]}

\newcommand{\tsharpb}[1]{\big\langle#1\big\rangle}

\newcommand{\class}[2][]{\left[#2\right]_{#1}}						
\newcommand{\tclass}[2][]{\big [#2\big]_{#1}}						
\newcommand{\ttclass}[2][]{[#2]_{#1}}						
\newcommand{\ceiling}[1]{\left\lceil#1\right\rceil}					

\newcommand{\IE}{W}
\newcommand{\sym}[1]{{\scriptscriptstyle{(#1)}}}

\newcommand{\rep}[1]{\tilde{#1}}								
\newcommand{\reptwo}[1]{\tilde{#1}}							
\newcommand{\ttscalar}[2]{\langle #1 \, |\, #2\rangle}	
\newcommand{\tscalar}[2]{\big\langle #1 \, \big |\, #2\big\rangle}			
\newcommand{\scalar}[2]{\left\langle #1 \,\middle |\, #2\right\rangle}		
\newcommand{\seq}[1]{\paren{#1}}								
\newcommand{\tseq}[1]{{\big(#1\big)}}
\newcommand{\ttseq}[1]{(#1)}
\newcommand{\Cb}{\mcC_b}									
\newcommand{\Cc}{\mcC_c}									
\newcommand{\Cz}{\mcC_0}									
\newcommand{\Bb}{\mcB_b}									

\newcommand{\EM}{\mathrm{em}}

\newcommand{\pfwd}{\sharp}
\DeclareMathOperator*{\esssup}{esssup}
\DeclareMathOperator*{\essinf}{essinf}

\DeclareMathOperator{\car}{\mathbf 1}

\DeclareMathOperator{\emp}{\varnothing}
\newcommand{\N}{{\mathbb N}}
\newcommand{\R}{{\mathbb R}}

\DeclareMathOperator{\Z}{{\mathbb Z}}

\newcommand{\restr}[1]{\big\lvert_{#1}}

\newcommand{\Id}{\mssI}

\newcommand{\iref}[1]{\ref{#1}}

\newcommand{\comma}{\,\mathrm{,}\;\,}
\newcommand{\semicolon}{\,\mathrm{;}\;\,}
\newcommand{\fstop}{\,\mathrm{.}}

\DeclareMathOperator{\zero}{{\mathbf 0}}

\newcommand{\Beta}{\mathrm{B}}

\newcommand{\card}[1]{\abs{#1}}

\newcommand{\trid}{\star}
\newcommand{\molli}{\varphi}

\usepackage{scrextend}						

\let\temp\phi
\let\phi\varphi
\let\varphi\temp

\let\temp\epsilon
\let\epsilon\varepsilon
\let\varepsilon\temp
\newcommand{\eps}{\epsilon}
\newcommand{\vareps}{\varepsilon}

\numberwithin{equation}{section}
\theoremstyle{plain}
\newtheorem{theorem}{Theorem}[section]
\newtheorem*{theorem*}{Theorem}

\newtheorem{proposition}[theorem]{Proposition}
\newtheorem{lemma}[theorem]{Lemma}
\newtheorem{corollary}[theorem]{Corollary}

\theoremstyle{definition}
\newtheorem{definition}[theorem]{Definition}
\newtheorem{notation}[theorem]{Notation}
\newtheorem*{defs*}{Definition}

\theoremstyle{remark}
\newtheorem{remark}[theorem]{Remark}
\newtheorem{example}[theorem]{Example}
\newtheorem{assumption}[theorem]{Assumption}
\newtheorem*{ass*}{Assumption}
\newtheorem*{question}{Question}
\newtheorem*{answer}{Answer}

\newcommand{\boldparagraph}[1]{\medskip\textbf{#1}.\quad}
\renewcommand{\paragraph}[1]{\medskip\emph{#1}.\quad}
\newcommand{\nparagraph}[1]{\medskip\emph{#1}\quad}

\begin{document}

\title[Massive Particle Systems, Wasserstein BM's, Dean--Kawasaki spdes]{Massive Particle Systems,\\Wasserstein Brownian Motions,\\ and the Dean--Kawasaki Equation}
\thanks{Research funded by: the Austrian Science Fund (FWF) project \href{https://doi.org/10.55776/ESP208}{{10.55776/ESP208}}; PRIN Department of Excellence MatMod@TOV (CUP: E83C23000330006).
Part of this work has been written while the author was attending the symposium \emph{Stochastic Analysis on Large Scale Interacting Systems} and the conference \emph{Stochastic Analysis} at the Research Institute for Mathematical Sciences (RIMS) of Kyoto University; the conference \emph{Mean field interactions with singular kernels and their approximations} at the Institute Henri Poincar\'e in Paris.
The author thanks the organizers of these events and he is very grateful to Hiroshi Kawabi, Seichiro Kusuoka, and Fran\c{c}ois Delarue for the respective invitations.
The author is grateful to Federico Cornalba, Benjamin Gess, Fenna M\"uller, and Max-Konstantin von~Renesse for very stimulating conversations about the Dean--Kawasaki equation.
}

\author[L.~Dello Schiavo]{Lorenzo Dello Schiavo}
\address{Dipartimento di Matematica -- Universit\`a degli Studi di Roma ``Tor Vergata''
\\
Via della Ricerca Scientifica 1\\
00133 Roma\\
Italy
}
\email{delloschiavo@mat.uniroma2.it}

\begin{abstract}
We develop a unifying theory for four different objects:
\begin{enumerate}[$(1)$]
\item\label{i:Abstract:1} infinite systems of interacting massive particles;
\item\label{i:Abstract:2} solutions to the Dean--Kawasaki equation with singular drift and space-time white noise;
\item\label{i:Abstract:3} Wasserstein diffusions with  a.s.\ purely atomic reversible random measures;
\item\label{i:Abstract:4} metric measure Brownian motions induced by Cheeger energies on $L^2$-Wasserstein spaces.
\end{enumerate}

For the objects in~\ref{i:Abstract:1}--\ref{i:Abstract:3} we prove existence and uniqueness of solutions, and several characterizations, on an arbitrary locally compact Polish ambient probability space~$(M,\nu)$ for a $\nu$-symmetric exponentially recurrent Feller driving noise.
In the case of the Dean--Kawasaki equation, this amounts to replacing the Laplace operator with some arbitrary diffusive Markov generator~$\mssL$ with ultracontractive semigroup.
In addition to a complete discussion of the free case, we consider singular interactions, including, e.g., mean-field repulsive isotropic pairwise interactions of Riesz and logarithmic type under the assumption of local integrability.

We further show that each $\nu$-symmetric Markov diffusion generator~$\mssL$ on~$M$ induces in a natural way a geometry on the space of probability measures over~$M$.
When~$M$ is a manifold and~$\mssL$ is a drifted Laplace--Beltrami operator, this geometry coincides with the geometry of $L^2$-optimal transportation.
The corresponding `geometric Brownian motion' coincides with the `metric measure Brownian motion' in~\ref{i:Abstract:4}.
\end{abstract}


\maketitle

%
%

\newpage

\setcounter{tocdepth}{2}
\tableofcontents

\newpage

\section*{Preface}
We construct and fully characterize a class of stochastic dynamics on the space of Borel probability measures over a general locally compact Polish space, arising as the empirical distribution of a free system of infinitely many independently diffusing particles of unequal mass. 

These \emph{free massive systems} serve as a physically motivated model for mesoscopic phenomena, interpolating between microscopic particle systems and macroscopic continuum limits, and are rigorously formulated using the machinery of Dirichlet forms and stochastic partial differential equations (\textsc{spde}s) on Kantorovich--Rubinstein (Wasserstein) spaces.
Our framework recovers and unifies a wide range of previously studied dynamics ---including the Dirichlet--Ferguson diffusion and Dean--Kawasaki equations with singular drift--- and extends them in several directions.

A main result in our work is the identification of the \emph{measure representation} of a free massive system as a reversible Markov process properly associated with an explicitly constructed quasi-regular Dirichlet form over a space of probability measures. 
We compute its generator and semigroup on a dense algebra of extended cylinder functions, and show that it coincides with the Cheeger energy associated to the Wasserstein geometry whenever the ambient space is a weighted Riemannian manifold.
In this setting, we further identify the process as the \emph{Wasserstein Brownian motion}, extend known Rademacher-type theorems, and obtain sharp Varadhan-type short-time asymptotics for its heat kernel.
Notably, we prove that the generator of the Dirichlet form is essentially self-adjoint, which yields uniqueness of solutions to the associated martingale problem even in the presence of infinite-dimensional gradient-type noise.

Furthermore, we rigorously interpret the associated \textsc{spde} as a \emph{Dean–Kawa\-saki equation} with singular drift, driven by vector-valued white noise, and provide a complete characterization of its solutions in terms of the underlying particle dynamics.
This resolves longstanding difficulties concerning well-posedness of such equations and collocates them within a rigorous, unified geometric and analytic framework. 

Finally, we also extend the theory to a broad class of \emph{interacting systems}, including sub-critical Riesz- and Dyson-type singular pairwise mean-field interactions, via Girsanov transforms of the underlying Dirichlet forms.
In fact, we allow for interaction energies defined only via local Sobolev-type regularity in the Dirichlet structure, which need be neither mean-field, nor pairwise, nor smooth.

Our work reveals deep connections between infinite-dimensional geometric analysis, stochastic particle systems, and the intrinsic Riemannian structure of Wasserstein spaces.
The resulting framework blends stochastic analysis, optimal transport, and Dirichlet-form theory on highly singular curved infinite-dimensional spaces, and provides a new foundational approach for the analysis of measure-valued dynamics under minimal geometric or analytic assumptions.

\section{Introduction}\label{s:Intro}
Particle systems are fundamental in probability theory, mathematical physics, statistics, and various other areas of mathematics.
In physical contexts, they model large ensembles of identical particles, such as electrons or photons, at the \emph{microscopic scale}.

In many applications, the admissible configurations consist of up to countably many \emph{indistinguishable} particles which cannot occupy the same position and remain locally finite.
The study of stochastic dynamics for these configurations subject to noise has been a very active area of research, in itself (see e.g.~\cite{AlbKonRoe98,AlbKonRoe98b,Osa96,LzDSSuz21} and references therein); as well as in connection with: random matrices (see e.g.\ the monograph~\cite{AkeBaiDiF15} and references therein), infinite \textsc{sde}s (e.g.~\cite{OsaTan20} and references therein), metric measure geometry~\cite{ErbHue15,LzDSSuz22a}, among others.

Here, we consider a distinct class of (\emph{marked}) particle systems that naturally describe physical systems at the \emph{mesoscopic scale}.
At this scale we postulate that:

\begin{itemize}
\item the system retains its \emph{particulate nature};
\item a \emph{macroscopic profile} emerges;
\item the \emph{ambient's shape} becomes relevant;
\item the system has \emph{finite total size}.
\end{itemize}

\subsection{Free massive systems}\label{ss:FreeMassiveSystems}
In the following,~$n$ is always a non-negative integer, while~$N$ denotes either a non-negative integer or the countable infinity.

We consider a stochastic particle system of~$N$ \emph{marked} points in an ambient space~$M$, identified with their position~$X^i$ and mark~$s_i>0$.
While we allow for an \emph{infinite} number of particles, we assume the system to have \emph{finite size}, ensuring that ---with no loss of generality, after normalization--- $\sum_i^Ns_i=1$.
The marks~$s_i$ serve as placeholders for a context-relevant scalar quantity, such as mass or charge.
For simplicity, we will refer to this quantity only as to the \emph{mass} of the particle, and call any particle system as above \emph{massive}.

We assume the sequence of masses~$s_i$ to be randomly distributed in the \emph{ordered} $N$-simplex according to some arbitrary law~$\pi$.
Additionally, each particle experiences a noise, with volatility inversely proportional to its mass:
\begin{equation}\label{eq:Free}
\diff X^i_t = \diff \mssW^i_{t/s_i}\comma \qquad i\leq N\fstop
\end{equation}
Here,~$\mssW^i_\bullet$ represents \emph{independent} instances of a same Markov process~$\mssW_\bullet$, possibly with jumps, such as a standard Wiener process or an $\alpha$-stable L\'evy process.

When the ambient space is `open', low-mass particles can rapidly escape every bounded region, leading to potential mass dissipation.
To prevent this, we require the total mass of a massive system to be conserved.
Precisely, we make the following assumption:
\begin{assumption}[Setting, cf.\ Ass.~\ref{ass:Setting}]\label{ass:IntroSetting}
The \emph{ambient space}~$(M,\T)$ is a locally compact second-countable Hausdorff topological space, equipped with an atomless Borel probability measure~$\nu$.
The \emph{driving noise}~$\mssW_\bullet$ is a $\nu$-reversible irreducible recurrent (thus conservative) Hunt process on~$M$ with bounded continuous transition functions and exponential convergence to equilibrium.
\end{assumption}

For example, the assumption holds when~$\mssW$ is the Brownian motion on any connected closed (i.e., compact boundaryless) smooth Riemannian manifold.

\begin{theorem}[Well-posedness, see~\S\ref{s:FIS}]\label{t:Intro1}
For any initial condition~$\seq{X^i_0}_{i\leq N}$ and any initial mass distribution~$\seq{s_i}_{i\leq N}$, the system~\eqref{eq:Free} has a unique global solution for all~$t\geq 0$. 
\end{theorem}

\begin{figure}[htb!]
\includegraphics[scale=.5, trim={0 70 0 100}, clip]{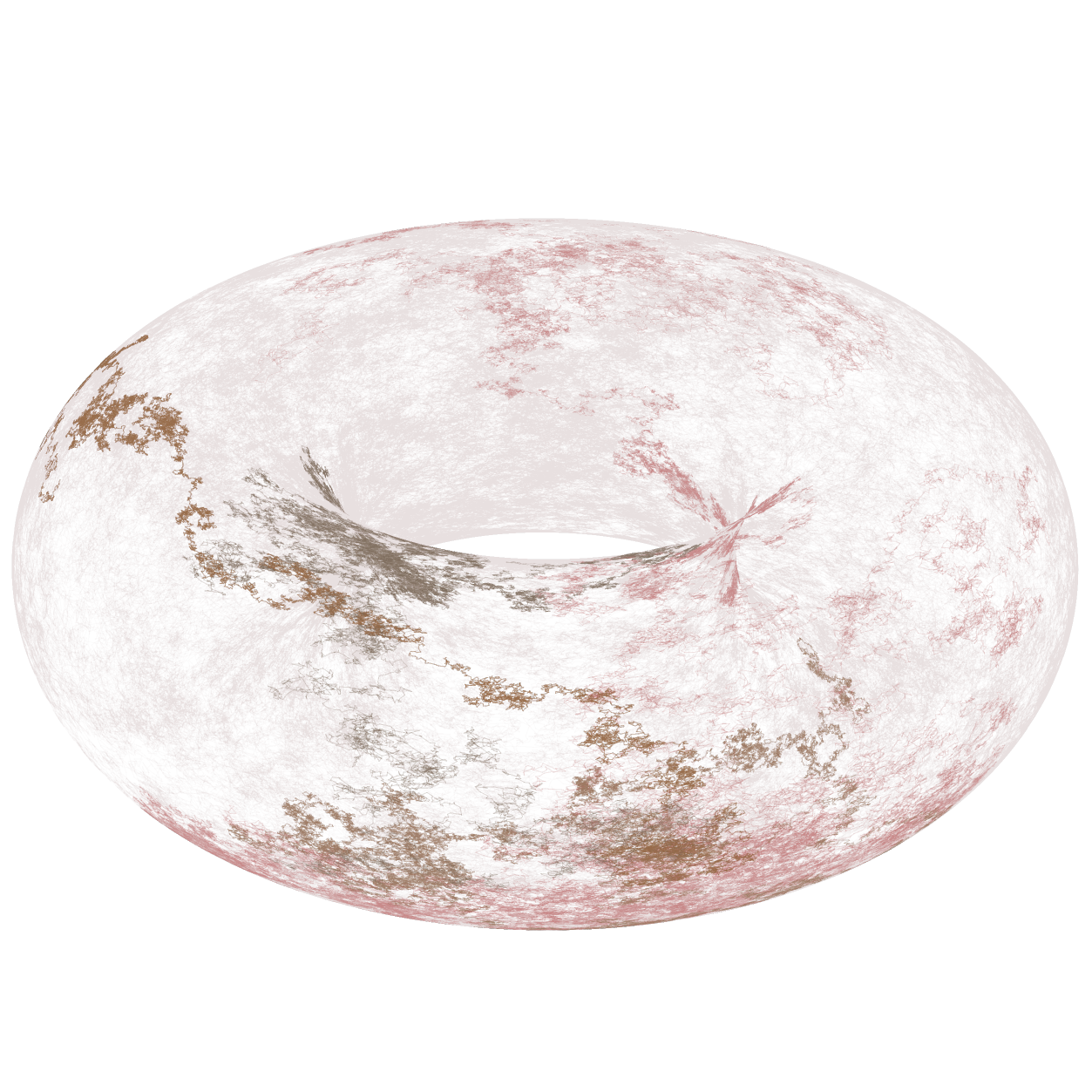}
\caption{The trajectory up to time~$1$ of a free massive system consisting of five independent massive Brownian particles on a two-dimensional torus.
The trajectory of each particle is represented by a different color, with saturation proportional to the mass carried by the particle.
As the lightest particle has already explored the whole space, the heavier particles remain more localized due to lower diffusivity.}
\end{figure}

\subsubsection{Main difficulties I: singularity of product densities}
While the case~$N<\infty$ is rather straightforward, the claim in Theorem~\ref{t:Intro1} in the case~$N=\infty$ becomes more nuanced.

On the one hand, the system's finite size plays a crucial role.
Even if the transition kernel of~$\mssW$ has a continuous bounded density with respect to the invariant measure~$\nu$, this property does not necessarily extend to the transition kernel~$\bmssh_\bullet$ of a solution to~\eqref{eq:Free} relative to the product measure~$\nu^\otym{\infty}$.
Indeed, if~$\sum_i^\infty s_i=\infty$, then~$\bmssh_\bullet$ may be singular with respect to~$\nu^\otym{\infty}$.
This phenomenon is evident even in simple cases such as ~$M=\mbbS^1$, as demonstrated by Ch.~Berg in~\cite{Ber76} in light of S.~Kakutani's  characterization~\cite{Kak48} of absolute continuity for infinite-product measures.
Notably, if~$\bmssh_\bullet$ is not absolutely continuous with respect to~$\nu^\otym{\infty}$, neither the $L^2$-theory, nor many results in the Feller theory apply.
In the first case this is due to the lack of a natural invariant measure; in the second case since the density of~$\bmssh_\bullet$ (even when it exists) may fail to be continuous.

On the other hand, given that the system has finite size, the sequence~$\seq{s_i}_{i}$ is summable, and thus vanishing.
Consequently, for~$N=\infty$, particles experience arbitrarily large noise strength~$s_i^{-1}$, leading to further complexities in the system’s stochastic behavior such as the absence of natural space-time scaling.

\subsubsection{Main results~I: a unified setting}
In proving Theorem~\ref{t:Intro1}, we greatly expand on results by T.~Gill, by A.~Bendikov and L.~Saloff-Coste, and by S.A.~Albeverio,~A.Yu.~Daletskii, and Yu.G.~Kondratiev. 
Their respective contributions concern infinite-product strongly continuous contraction semigroups, infinite-product semigroups of heat-kernel measures, and Dirichlet forms on infinite-product spaces.
Specifically (see Thm.~\ref{t:BendikovSaloffCoste} below), we achieve the following:

\begin{enumerate}[$(1)$, leftmargin=1.75em]
\item\label{i:BSCADK:1} \emph{Extending absolute-continuity results.} We generalize the proof by A.~Bend\-ikov and L.~Saloff-Coste~\cite{BenSaC97} of the absolute continuity of infinite-product heat-kernel measures with respect to infinite-product probability reference measures. Their original work was restricted to compact spaces; our extension applies to more general non-compact spaces.

\item\label{i:BSCADK:2} \emph{Generalizing Dirichlet-form constructions.} We extend the canonical Dirichlet-form construction on infinite products of compact Riemannian manifolds, originally developed by S.A.~Albeverio, A.Yu.~Daletskii, and Yu.G.~Kondratiev~\cite{AlbDalKon97,AlbDalKon00}, to non-local, non-smooth, and non-compact settings.

\item\label{i:BSCADK:3} \emph{Identifying correspondences between semigroups and forms.} We establish that the semigroup of infinite-product heat-kernel measures from~\ref{i:BSCADK:1} corresponds precisely to the Dirichlet form constructed in~\ref{i:BSCADK:2}, and that the latter is in turn isomorphic to an infinite-product Dirichlet form in the sense of~\cite[\S{V.2.2}]{BouHir91}.

\item \emph{Characterizing the semigroup kernel.} We further show that the infinite-product heat-kernel measure in~\ref{i:BSCADK:1} is the integral kernel of the infinite-tensor-product semigroup on a specific von Neumann-incomplete infinite tensor product of the corresponding $L^2$-spaces, as constructed by T.~Gill in~\cite{Gil78, Gil74}.
\end{enumerate}

For a visual overview of our main results, see Figure~\ref{fig:Diagram}.

\subsection{The measure representation}\label{ss:IntroMeasRep}
Each massive system can be associated with an \emph{empirical measure}, defined by the map
\begin{equation}\label{eq:IntroMeasureRep}
\EM \colon \tseq{(s_i,X^i_t)}_{i\leq N} \longmapsto \mu_t\eqdef \sum_i^N s_i\, \delta_{X^i_t} \comma \qquad t>0\fstop
\end{equation}

The process~$\mu_\bullet$ is then a stochastic process taking values in the space~$\msP$ of all Borel probability measures on the ambient space~$M$.
We emphasize that, generally,~$\mu_t$ does \emph{not} display the usual properties of empirical measures of particle systems. 
Typically, the number~$N$ of particles is infinite and, for each~$t>0$, the positions~$\tseq{X^i_t}_i$ form a \emph{dense} subset of the ambient space.
In particular,~$\mu_t$ cannot be modelled as a configuration in the sense of, e.g.,~\cite{AlbKonRoe98}, and the topological support of~$\mu_t$ satisfies~$\supp \mu_t= M$ for every~$t>0$.

To describe further properties of~$\mu_\bullet$, we make a second assumption:
\begin{assumption}[Non-intersecting trajectories]\label{ass:QPPIntro}
The trajectories of two particles never meet a.s.
\end{assumption}

We will later give a sufficient condition for this assumption to hold, expressed in terms of the $\mssW_\bullet$-capacity of the diagonal of~$M$; see Assumption~\ref{ass:QPP}.
For now, note that this holds, for example, when~$\mssW_\bullet$ is the (drifted) Brownian motion on any weighted Riemannian manifold of dimension~$d\geq 2$ ---but \emph{not} when~$d=1$.
For more discussion, see also~\S\ref{sss:IntroZeroRange}.

\subsubsection{Main results~II: form, generator, ergodic decomposition}
Since the processes~$\mssW_\bullet^i$ are independent and $\nu$-reversible, and since the marks~$s_i$ remain constant in time, the process~$\mu_\bullet$ has an invariant measure~$\mcQ_{\pi,\nu}$ on~$\msP$, defined as the image of the product measure~$\pi\otimes \nu^\otym{N}$ under~$\EM$.
When particle trajectories do not meet, $\mu_\bullet$~retains the Markov property from the corresponding massive system, in turn inherited from that of~$\mssW_\bullet$.
Accordingly, we may study~$\mu_\bullet$ via the corresponding Dirichlet form on~$L^2(\mcQ_{\pi,\nu})$.

\begin{theorem}[See~\S\ref{s:FISMeas}]\label{t:IntroMeasRep}
Under Assumptions~\ref{ass:IntroSetting} and~\ref{ass:QPPIntro}, the process~$\mu_\bullet$ is
\begin{itemize}[leftmargin=1.5em]
\item properly associated with a quasi-regular recurrent (in particular: conservative) symmetric Dirichlet form on~$L^2(\mcQ_{\pi,\nu})$;
\item the solution ---unique in law--- to the martingale problem for the form's generator;
\item a diffusion (i.e., with a.s.\ continuous sample paths) if and only if so is~$\mssW_\bullet$;
\item ergodic if and only if~$\pi$ is concentrated on a singleton in the $N$-simplex.
Otherwise, Borel $\mu_\bullet$-invariant sets (up to $\mcQ_{\pi,\nu}$-equivalence) are in one-to-one correspondence with Borel subsets of the $N$-simplex (up to $\pi$-equivalence).
\end{itemize}
Furthermore, the form, its generator, and the predictable quadratic variation for the martingale problem are all computed explicitly on an explicit core.
\end{theorem}

As part of Theorem~\ref{t:IntroMeasRep}, we identify the topology on~$\msP$ natural for the construction of the measure representation~$\mu_\bullet$.
This is the \emph{weak atomic topology}~$\T_\mrma$ introduced by S.N.~Ethier and~T.G.~Kurtz in~\cite{EthKur94}, a Polish topology on~$\msP$ finer than the usual narrow topology~$\T_\mrmn$.

The identification is relevant when~$\mssW_\bullet$ is additionally a diffusion, in which case~$\mu_\bullet$ has a.s.~$\T_\mrma$-continuous sample paths.
Also, it will be relevant in the discussion of a stochastic partial differential equation (\textsc{spde}) solved by~$\mu_\bullet$, namely the Dean--Kawasaki \textsc{spde} below.

\subsubsection{Wasserstein geometry}
Let~$(M,g)$ be a complete Riemannian manifold with intrinsic distance~$\mssd_g$. 
We denote by~$(\msP_2,W_2)$ the \emph{$L^2$-Kantorovich--Rubinstein} (also: \emph{Wasserstein}) \emph{space} over~$(M,\mssd_g)$, see~\eqref{eq:WassersteinSp}.
Since the foundational work of Y.~Brenier, R.J.~McCann, F.~Otto~\cite{Bre91,BenBre00,McC97,Ott01,McC01}, and others, suitable infinite-dimensional analogues on~$\msP_2$ of standard geometric objects have been introduced, including: F.~Otto's \emph{gradient}~$\boldnabla$ and \emph{tangent space}~$T_\eta\msP_2$ to~$\msP_2$ at a measure~$\eta$~\cite{Ott01}, J.~Lott's \emph{connection}~\cite{Lot07}, N.~Gigli's \emph{exponential map}~\cite{Gig11}, etc.
Indeed,~$\msP_2$ is often seen as carrying a formal `infinite-dimensional Riemannian structure', inherited from that of~$(M,g)$, with~$W_2$ as intrinsic distance.
Further merits and limitations of this analogy were eventually pointed out by Gigli in~\cite{Gig12}, including the potential lack of a natural Laplacian --- a gap which we later addressed here by identifying a generator via Dirichlet-form theory.

While \emph{differential}-geometric aspects of~$\msP_2$ are relatively well-studied, \emph{Riemannian} aspects remain less understood.
In particular, the quest is ongoing for a truly natural volume measure on~$\msP_2$ compatible with the Riemannian structure.
Some proposed candidates include:
M.-K.\ von Renesse and K.-T.\ Sturm's \emph{entropic measure}~\cite{vReStu09} on~$\msP_2(\mbbS^1)$,
a construction eventually extended by Sturm~\cite{Stu11,Stu24} on~$\msP_2$ over a closed Riemannian manifold;
the \emph{Dirichlet--Ferguson measure}~\cite{Fer73} on~$\msP_2$, proposed by the author in~\cite{LzDS17+};
P.~Ren and F.-Y.~Wang's \emph{Ornstein--Uhlenbeck} measure~\cite{RenWan24} on~$\msP_2(\R^d)$;
and many others, especially on~$\msP_2(\mbbS^1)$.

From a purely metric-geometric point of view, uniqueness results for such a measure ---if any--- ought not to be expected.
Indeed, M.~Fornasier, G.~Savar\'e, and G.E.~Sodini have recently shown in~\cite{ForSavSod22} that~$\msP_2$ is \emph{universally infinitesimally Hilbertian}: every reference measure~$\mcQ$ on~$\msP_2$ gives rise to an infinitesimally Hilbertian metric measure space in the sense of~\cite{Gig13}, i.e.\ the natural \emph{Cheeger energy} functional
\begin{equation}
\Ch_{W_2,\mcQ}(f)\eqdef \inf \set{\liminf_{n\to\infty} \int \slo{f_n}^2 \diff\mcQ : \begin{aligned} f_n \in \Lip(\msP_2,W_2)\comma \\ \lim_{n\to\infty}\norm{f_n-f}_{L^2(\mcQ)}=0\end{aligned}}
\end{equation}
induced by~$W_2$ and~$\mcQ$ is \emph{quadratic}.
Furthermore, even imposing strong relations between the geometric structure and the measure does not restrict the focus.
Indeed, as pointed out by the author in~\cite{LzDS20}, there exist mutually singular measures~$\mcQ$ with full support on~$\msP_2(M)$ all satisfying the \emph{Rademacher Theorem} ($\mcQ$-a.e.\ \emph{Fr\'echet} differentiability of $W_2$-Lipschitz functions).

Nonetheless, once a reference measure~$\mcQ$ is fixed, the geometric construct of~$\msP_2$ induces a canonical pre-Dirichlet energy functional
\begin{equation}\label{eq:IntroCheeger}
\mcE_\mcQ= \int \tscalar{(\boldnabla \emparg)_\eta}{(\boldnabla \emparg)_\eta}_{T_\eta\msP_2}^2 \diff\mcQ(\eta) \fstop
\end{equation}
Provided the latter is closable on a sufficiently large class of `smooth functions', its closure is a quasi-regular strongly local Dirichlet form, and it is further identical to~$\Ch_{W_2,\mcQ}$ by the abstract results in~\cite{ForSavSod22}.
By quasi-regularity,~$\mcE_\mcQ$ is properly associated with a Markov diffusion, namely the \emph{$\mcQ$-reversible Brownian motion} on~$\msP_2$.
Establishing the closability of~$\mcE_\mcQ$ is usually quite challenging and has been proved ---on a case by case basis--- for the entropic measure on~$\msP_2(\mbbS^1)$, giving rise to the \emph{Wasserstein diffusion}~\cite{vReStu09}; for the Dirichlet--Ferguson measure, giving rise to the \emph{Dirichlet--Ferguson diffusion}~\cite{LzDS17+}; for the Ornstein--Uhlenbeck process~\cite{RenWan24} on~$\msP_2(\R^d)$;
for a class of measures satisfying strict quasi-invariance assumptions with respect to the natural action on~$\msP_2$ of the diffeomorphism group of~$M$ in~\cite{LzDS19b}.

\subsubsection{Free massive systems as Wasserstein Brownian motions}\label{sss:WBM}
To see how free massive systems enter this picture, let us specialize Theorem~\ref{t:IntroMeasRep} to a case of particular interest.
All objects in the next assumption are supposed to be smooth.

\begin{assumption}[Assumption~\ref{ass:WRM}]\label{ass:IntroManifold}
The ambient space~$(M,g,\nu)$ is a complete weighted Riemannian manifold with Riemannian metric~$g$, of dimension~$d\geq 2$.
Precisely,~$\nu=\rho\vol_g$ for some nowhere-vanishing probability density~$\rho>0$, and~$\mssW$ is the $\nu$-reversible {Brownian motion with drift}, i.e.\ the stochastic process with generator the {drifted Laplace--Beltrami operator}~$\Delta_g+(\diff\log\rho)\circ\nabla_g$ on~$L^2(\nu)$.
Finally, assume that~$\rho$ is so that~$\mssW$ satisfies Assumption~\ref{ass:IntroSetting}. 
\end{assumption}

As an instance, let us note that Assumption~\ref{ass:IntroManifold} is satisfied, for any smooth~$\rho>0$, whenever the manifold~$(M,g)$ is closed.
(For sufficient conditions in the case of non-compact manifolds see the discussion after Assumption~\ref{ass:WRM}.)

In the next result, for a function~$f\colon M\to\R$ and for the generator~$\mssL$ of~$\mssW$, we write~$\mssL^z\restr{z=x}f(z)$ to indicate that~$\mssL$ is acting on~$f$ in the variable~$z$ at the point~$x\in M$.

\begin{theorem}[\S\ref{sss:Wasserstein}]\label{t:IntroCheeger}
In the setting of Assumption~\ref{ass:IntroManifold}, suppose further that~$\nu\in\msP_2$.
Then,
\begin{enumerate}[$(i)$]
\item the measure representation~$\mu_\bullet$ is properly associated with the form~$\mcE_\mcQ$ in~\eqref{eq:IntroCheeger} for~$\mcQ=\mcQ_{\pi,\nu}$;
\item its generator is the Friedrichs extension of the operator\footnote{Heuristically, this operator captures the infinitesimal motion of mass in~$\eta$ obtained by moving its atom at~$x$ according to~$\mssW$, weighted by local mass proportions.}
\begin{equation}\label{eq:GeneratorIntro}
(\mcL u)(\eta)\eqdef \int_M \frac{\mssL^z\restr{z=x} u(\eta+\eta\!\set{x}\delta_z-\eta\!\set{x}\delta_x)}{(\eta\!\set{x})^2} \diff\eta(x)
\end{equation}
on the algebra of extended cylinder functions~\eqref{eq:HCylinderFEps}.
\item the form~$\mcE_\mcQ$ coincides with the Cheeger energy~$\Ch_{W_2,\mcQ}$;
\item the Rademacher Theorem holds for~$\mcE_\mcQ$ on~$(\msP_2,W_2)$, that is: every $W_2$-Lipschitz function~$u\colon\msP_2\to\R$ is an element of the form domain, is Fr\'echet differentiable $\mcQ$-a.e., and satisfies
\[
\norm{\boldnabla u}_{T_\eta\msP_2} \leq \Lip_{W_2}[u] \fstop
\]
\item the semigroup~$\mcH_\bullet$ of~$\mcE_\mcQ$ satisfies the following one-sided integral Varadhan short-time asymptotic estimate: for every pair of Borel sets~$A_1,A_2\subset\msP_2$ with~$\mcQ A_1, \mcQ A_2>0$,
\[
-2\, \lim_{t\downarrow 0} t \log \scalar{\car_{A_1}}{\mcH_t \car_{A_2}}_{L^2(\mcQ)} \ \geq\ \mcQ\text{-}\essinf_{\eta_i\in A_i} W_2(\eta_1,\eta_2)^2 \comma
\]
while the opposite inequality fails as soon as~$\pi$ is not concentrated on a singleton.
\end{enumerate}
\end{theorem}

In other words,
\begin{quotation}
the measure representation~$\mu_\bullet$ of any free massive particle system is the Brownian motion for both the geometric and the metric measure structure of~$(\msP_2,W_2,\mcQ_{\pi,\nu})$.
\end{quotation}

Furthermore, Theorem~\ref{t:IntroCheeger} shows ---at once for the reversible measure of every free massive system--- that the form~\eqref{eq:IntroCheeger} is closable, and provides the much sought after explicit expression~\eqref{eq:GeneratorIntro} for the Laplacian of~$(\msP_2,W_2,\mcQ_{\pi,\nu})$, its generator.

In particular, 
\begin{itemize}
\item (see Ex.~\ref{ex:DFDiffusion}) choosing~$\pi= \Pi_\beta$ the \emph{Poisson--Dirichlet distribution} on the infinite simplex (e.g.~\cite{DonGri93}), the measure~$\mcQ_{\pi,\nu}$ is the \emph{Dirichlet--Ferguson measure}~\cite{Fer73} with intensity measure~$\nu$. 
When~$(M,g)$ is closed (i.e.\ compact and boundaryless) and~$\mssW_\bullet$ is its standard Brownian motion,~$\mu_\bullet$ is the \emph{Dirichlet--Ferguson diffusion}~\cite{LzDS17+}, i.e.\ the measure representation of a massive particle system of infinitely many Brownian particles.
That is, Theorem~\ref{t:IntroCheeger} extends the results in~\cite{LzDS17+} from the case of closed manifolds to that of complete weighted manifolds with bounded heat kernel densities;

\item (see Ex.~\ref{ex:OU}) choosing~$M=\R^d$ and~$\rho$ the standard Gaussian density, this proves the structure of the Laplacian of~$\msP_2(\R^d)$ conjectured by Ren and Wang in~\cite{RenWan24}, and provides an explicit class of cylinder functions on which this Laplacian is essentially $L^2(\mcQ_{\pi,\nu})$-selfadjoint.
\end{itemize}

Finally, Theorem~\ref{t:IntroCheeger} greatly expands the class of measures~$\mcQ$ satisfying the Rademacher theorem on~$(\msP_2,W_2)$, extending the results of~\cite{LzDS19b} to all measures of the form~$\mcQ_{\pi,\nu}$.
These measures are only \emph{partially} quasi-invariant (Dfn.~\ref{d:PQI}) with respect to the natural action on~$\msP_2$ of the group of diffeomorphisms of~$M$, (as opposed to: quasi-invariant with Radon--Nikod\'ym derivative satisfying~\cite[Ass.~2.7]{LzDS19b}) and do not generally satisfy the integration-by-parts formula~\cite[Eqn.~(2.7)]{LzDS19b} for~$\boldnabla$ on standard cylinder functions.

\subsubsection{Main difficulties~II: identification of the process}
Dirichlet forms as in~\eqref{eq:IntroCheeger} have been studied for a relatively wide range of reference measures~$\mcQ$.
While being the starting point of the theory, the \emph{closability} of these forms is \emph{not} the main technical hurdle, and has been achieved both by \emph{ad hoc} and by standard techniques for different choices of~$\mcQ$, or entirely bypassed by considering lower semicontinuous relaxations ---in particular: Cheeger energies--- which are automatically closed.
Analogously, \emph{regularity} and thus the \emph{existence} of a Markov process properly associated with the form, is a standard consequence of the Stone--Weierstra\ss\ Theorem ---at least when~$M$, therefore~$\msP$, is compact.

The main difficulty in the analysis is rather the \emph{identification} of the Hunt process in question, which would necessarily go through the identification of the form's generator and strongly continuous contraction semigroup.
Especially for the latter, no explicit expression is usually available, in the least when a Cheeger energy is considered.
As a consequence, the understanding of the process is usually very limited in all the aforementioned examples.
In sharp contrast with the literature, for the choice(s) of~$\mcQ$ considered here we are able to characterize the Markov process associated to the Dirichlet forms~\eqref{eq:IntroCheeger} \emph{to the fullest possible extent}.

\subsection{The Dean--Kawasaki equation}\label{ss:IntroDK}
Let~$\mbbT^d$ be the standard $d$-dimensional torus and denote by~$\msD'$ the space of distributions on~$\mbbT^d$.
For a potential~$F\colon\msD'\to\R$, we call \emph{Dean--Kawasaki equation with $F$-interaction on~$\mbbT^d$} the \textsc{spde}
\begin{equation}\label{eq:DK0}
\partial_t\rho = T\,\Delta\rho +\nabla\cdot \paren{\rho\, \nabla \tfrac{\delta F}{\delta\rho}(\rho)} + \nabla\cdot(\sqrt{T\rho}\, \xi) \fstop
\end{equation}
Here, $T>0$ is a parameter, $\xi$ is an $\R^d$-valued space-time white noise, and the variable~$\rho$ is a time-dependent $\msD'$-valued random field.
Finally, for any sufficiently regular~$F\colon \msD'\to\R$, for a distribution~$\mu$ on~$\mbbT^d$ the `\emph{extrinsic}' \emph{derivative} $\tfrac{\delta F}{\delta \mu}$ at~$\mu$ is defined as the function
\begin{equation}\label{eq:ExtrinsicDerivative}
\frac{\delta F}{\delta\mu}(\mu)\colon x\longmapsto \diff_s\restr{s=0} F(\mu+s\delta_x)\comma \qquad x\in\mbbT^d \comma
\end{equation}
i.e., the derivative along the increment of~$\mu$ by an added point mass. When $\mu$ is a measure,
\[
\nabla \frac{\delta F}{\delta\mu}(\mu) = (\boldnabla F)_\mu
\]
is the $L^2$-Kantorovich--Rubinstein gradient (also `intrinsic' derivative) of~$F$ at~$\mu$, cf.~e.g.~\cite[\S3.1]{LzDS17+} or~\cite{RenWan21}.

As it will be clear later on, there is no reason to confine oneself to the torus, and one may as well consider the equation on ---for instance--- the Euclidean space, or a Riemannian manifold.
In this introduction, we choose~$\mbbT^d$ as the ambient space in order to remark that the noise~$\xi$ takes values in a space of sections of the tangent bundle to the ambient space.
On the one hand, the construction of the noise~$\xi$ poses no issue, since~$\mbbT^d$ is parallelizable; on the other hand, it will be important to distinguish points (in~$\mbbT^d$) from tangent vectors (in~$\R^d$).

Equation~\eqref{eq:DK0} has been independently proposed by D.S.~Dean in~\cite[Eqn.~(19)]{Dea96} and by K.~Kawasaki in~\cite[Eqn.~(2.28)]{Kaw94}.
From a physical point of view, the equation describes the density function of a system of $N\gg 1$ particles subject to a diffusive Langevin dynamics at temperature~$T$, combining a deterministic pair-potential interaction~$F$ with a noise~$\xi$ describing the particles' thermal fluctuations.
The equation is an instance of a broader class of Ginzburg--Landau stochastic phase field models:
together with its variants, it has been used to describe super-cooled liquids, colloidal suspensions, the glass-liquid transition, some bacterial patterns, and other systems; see, e.g.,~\cite{FruHay00,KimKawJacvWi14,DelOllLopBlaHer16,MarTar99} and the recent reviews~\cite{VruLoeWitt20,Ill24}.

On~$\R^d$ (respectively, on $\mbbT^d$) equations similar to~\eqref{eq:DK0} ---with a non-linear viscous term~$\Delta\Phi(\rho)$ in place of~$\Delta\rho$--- model in the continuum the \emph{fluctuating hydrodynamics} of interacting particle systems on (respectively, periodic) lattices, e.g.\ the weakly asymmetric simple exclusion process,~\cite[\S4.2]{GiaLebPre99}; see the survey~\cite{BerDeSGabJon15}.
These equations are also used to describe the hydrodynamic large deviations of simple exclusion and zero-range particle processes, e.g.~\cite{QuaRezVar99,DirStaZim16,BenKipLan95,FehGes23}.

\subsubsection{Relations with Wasserstein geometry}
From a mathematical point of view, the interest in~\eqref{eq:DK0} partly arises from the structure of its noise in connection with the geometry of~$(\msP_2,W_2)$.
Indeed, let
\[
E(\rho)\eqdef \int \rho\log\rho \diff\Leb^d
\]
be the Boltzmann--Shannon entropy functional on~$\mbbT^d$.
In the free case~$F\equiv 0$, we can rewrite~\eqref{eq:DK0} as
\begin{equation}\label{eq:IntroDKfree}
\partial_t\rho = \nabla\cdot \paren{T \rho\, \nabla\frac{\delta E}{\delta \rho}(\rho)} + \nabla\cdot\tparen{\sqrt{T\rho\,}\, \xi} \fstop
\end{equation}
In this case,~$\rho$ is an intrinsic random perturbation of the gradient flow of~$E$ on~$\msP_2$ by a noise~$\xi$ distributed according to the energy dissipated by the system, i.e.\ by the natural isotropic noise arising from the Riemannian structure of~$\msP_2$.
As the $W_2$-gradient flow of~$E$ is a solution to the heat equation, it was suggested by Sturm and von~Renesse in~\cite{vReStu09} that~\eqref{eq:IntroDKfree} can be formally understood as a `stochastic heat equation with intrinsic noise' on the base manifold.
By `intrinsic noise' we mean a noise acting on the argument of the solution rather than (e.g., additively or multiplicatively) on its values.

This eventually led to the construction of stochastic processes on~$\msP_2$ formally solving other \textsc{spde}s with same noise term and different `drift term', including: Sturm--von Renesse's \emph{Wasserstein diffusion}~\cite{vReStu09} on~$\msP_2(\mbbS^1)$; V.V.~Konarovskyi and von Renesse's \emph{coalescing-fragmentating Wasserstein dynamics}~\cite{KonvRe17} on~$\msP_2(\R)$; the aforementioned Dirichlet--Ferguson diffusion on~$\msP_2(M)$ over a closed Riemannian manifold~$M$.

The connection between~\eqref{eq:IntroDKfree} and Wasserstein geometry was also made apparent by R.L.~Jack and J.~Zimmer in~\cite{JacZim14}, also cf.~\cite{EmbDirZimRei18}, linking the deterministic hydrodynamic evolution of the particle system with the steepest descent of the free energy.
Indeed, solutions to~\eqref{eq:IntroDKfree} are \emph{formally} related to the most probable paths of \emph{mesoscopic dissipative systems}, and are again linked to their fluctuating hydrodynamics,~\cite{JacZim14}.

\subsubsection{Well-posedness in approximate settings and rigidity of exact solutions}
Even in their simplest form~\eqref{eq:IntroDKfree}, equations of Dean--Kawasaki-type are ---in every dimension~$d\geq 1$--- `scaling supercritical' in the language of regularity structures~\cite{Hai14} and of paracontrolled calculus~\cite{GubImkPer15}, making these equations inaccessible to both frameworks, cf.~e.g.~\cite{FehGes24, DjuKrePer24}.

\paragraph{Colored-noise approximations}
Partially in order to overcome this issue, Dean--Kawasaki and related equations are often considered in the approximation of truncated noise~\cite{CorFis23,DjuKrePer24}, spatially correlated noise~\cite{FehGes24,MarMay22}, additional additive noise~\cite{MarMay22}, etc.
In particular: A.~Djurdjevac, H.~Kremp, and N.~Perkowski~\cite{DjuKrePer24} show well-posedness of measure-valued solutions with $L^2$~initial datum for a suitable truncation of the noise;
B.~Fehrmann and B.~Gess~\cite{FehGes24} prove well-posedness and construct a robust solution theory for a very large class of Dean--Kawasaki-type equations with non-linear viscosity term and a noise of Stratonovich-type colored in space and white in time;
F.~Cornalba and J.~Fischer~\cite{CorFis23,CorFis23b}, also cf.~\cite{CorFisIngRai23}, show that the description of hydrodynamics fluctuations via solutions to~\eqref{eq:IntroDKfree} is still particularly effective.


\paragraph{Rigidity}
Coming back to the original Dean--Kawasaki equation, Konarov\-skyi, T.~Lehmann, and von~Renesse, eventually showed in~\cite{KonLehvRe19} that the free Dean--Kawasaki equation~\eqref{eq:IntroDKfree} may be equivalently formulated as a martingale-type problem, which is then meaningful in the far more general abstract setting of standard Markov triples with essential selfadjointness~\cite[\S3.4.2, p.~170]{BakGenLed14} satisfying a Bakry–\'Emery Ricci-curvature lower bound~\cite[Eqn.~(2.3)]{KonLehvRe19} ---a setting similar to, but logically independent from, the one in Assumption~\ref{ass:IntroSetting}.

This martingale-type problem satisfies a striking rigidity result~\cite[Thm.~2.2]{KonLehvRe19}. (See~Thm.~\ref{t:KLR}.)
It admits $\msP$-valued solutions if and only if~$T=\tfrac{1}{n}$ and~$\rho_0=\tfrac{1}{n}\sum_i^n \delta_{x_i}$, and in this case the solution is unique and satisfies~$\rho_t=\tfrac{1}{n}\sum_i^n \delta_{X^i_{nt}}$ for all~$t>0$.
Analogous rigidity results were subsequently shown by Konarovskyi, Lehmann, and von~Renesse in~\cite{KonLehvRe19b} in the case of sufficiently smooth interactions; and more recently: by Konarovskyi and F.~M\"uller~\cite{KonMue23} for tempered-distributional solutions on~$\R^d$; by M\"uller, von~Renesse, and Zimmer~\cite{MuevReZim24} for Dean--Kawasaki models of Vlasov--Fokker--Planck type on~$\R^d$.
Importantly, the results in~\cite{MuevReZim24} also apply to certain \emph{non-symmetric} driving diffusions and as such are not recovered by our approach via symmetric Dirichlet forms.

\subsubsection{Dean--Kawasaki revisited}\label{sss:IntroDKRevisited}
Let us now rephrase the Dean--Kawasaki equation into a more general equation which will be the object of our analysis.
We shall focus on the free case, viz.\ ---in stochastic notation---
\begin{equation}\tag*{$(\mathsc{dk})$}\label{eq:DK}
\diff\mu_t = \nabla\cdot(\sqrt{\mu_t}\, \xi) + \alpha \Delta\mu_t \diff t \comma
\end{equation}
postponing the case of interacting systems to~\S\ref{sss:PairwiseInteraction} below.
Here,~$\alpha>0$ is a parameter, replacing the temperature~$T$, exactly as in~\cite{KonLehvRe19}.
(The \emph{pure-noise} case~$\alpha=0$ has recently been shown to be ill-posed; see~\cite{LzDSKon25}.)

For simplicity of notation, given any measure~$\mu$ and any~$x\in M$, in the following we write~$x\in\mu$ to indicate that~$\mu\set{x}>0$, i.e., that~$x$ is an atom for~$\mu$.
Fix a positive integer~$n$.
In light of Konarovskyi--Lehmann--von Renesse's rigidity, the only Equation~\ref{eq:DK} with non-trivial solutions for the initial datum~$\mu_0=\tfrac{1}{n}\sum_i^n \delta_{x_i}$ reads
\[
\diff\mu_t = \nabla\cdot(\sqrt{\mu_t}\, \xi) + n\, \Delta\mu_t \diff t\fstop
\]
Owing to rigidity that we only have solutions of the form~$\mu_t=\tfrac{1}{n}\sum_i^n \delta_{Y^i_t}$ for some $M$-valued stochastic processes~$Y^i_\bullet$, $i\leq n$, we may rewrite the drift coefficient as
\[
n\, \Delta\mu_t =n\, \Delta \frac{1}{n}\sum_i^n \delta_{Y^i_t} = \sum_i^n \Delta \delta_{Y^i_t} = \sum_{x\in \mu_t} \Delta \delta_x\comma
\]
so that~\ref{eq:DK} becomes
\begin{equation}\label{eq:Intro:TrueDK}
\diff\mu_t =  \nabla\cdot(\sqrt{\mu_t}\, \xi) + \sum_{x\in \mu_t} \Delta \delta_x \diff t \fstop
\end{equation}
This observation allows us to reinterpret the drift in a way that naturally extends to purely atomic measures with arbitrary (even \emph{infinite}) support, independent of~$n$.
Also, this is the Dean--Kawasaki equation with `singular drift' considered in~\cite{KonvRe17} on the real line.
In the notation of~\cite{KonvRe17}, it amounts to take~$\beta=0$ and~$i=0$ in the family of models considered there.
Whereas this choice of~$\beta$ required some justification in~\cite{KonvRe17}, it is in fact the only natural one in light of the rigidity result~\cite{KonLehvRe19}.\footnote{We note that~\cite{KonvRe17} in fact predates~\cite{KonLehvRe19} by several years.} 
It is clear that every solution to~\ref{eq:DK} is as well a solution to~\eqref{eq:Intro:TrueDK}.
Solutions to this new equation will be completely characterized in the next section.

In spite of being called a `drift', the second summand in the right-hand side of~\eqref{eq:Intro:TrueDK} is not `driving' solutions towards a certain region of~$\msP$.
Rather than driving motion in the classical sense, this term enforces a singular constraint on the dynamics, effectively projecting the evolution onto the space~$\msP^\pa$ of \emph{purely atomic} probability measures on~$M$ for all times.
In a sense it is thus rather a `boundary term', as detailed in~\cite[\S2.7]{LzDS17+}.

\subsubsection{Main results~III: Dean--Kawasaki and massive particle systems}
Let us now turn to a further characterization of the measure representation~$\mu_\bullet$ of a free massive system as the solution to some \textsc{spde} in the same form as~\eqref{eq:Intro:TrueDK}.
In the generality of Assumption~\ref{ass:IntroSetting}, assume further that~$\mssW$ is a diffusion.
Further let~$\mssL$ be the generator of~$\mssW$, and denote by~$\mssL'$ its distributional adjoint on measures on~$M$.

The theory of Hilbert tangent bundles to a Dirichlet space~\cite{IonRogTep12,HinRoeTep13,Ebe99} equips each local Dirichlet space with a bundle structure mimicking Riemannian geometry, thereby enabling the definition of vector fields and divergence operators.
Thus, even when~$M$ does not have any proper differential structure we can give rigorous meaning to a divergence operator~$\msfdiv$, and to a `vector-valued' white noise~$\xi$ on~$M$.

\begin{theorem}[See~\S\ref{s:DK}]\label{t:IntroDK}
In the setting of Assumptions~\ref{ass:IntroSetting} and~\ref{ass:QPPIntro}, suppose further that~$\mssW$ is a diffusion with generator~$\mssL$.
Then,~$\mu_\bullet$ is the unique analytically weak martingale solution to the \textsc{spde}
\begin{equation}\label{eq:IntroSPDE}
\diff \mu_t = \msfdiv(\sqrt{\mu_t }\, \xi)+\sum_{x\in \mu_t} \mssL'\delta_x \diff t\fstop
\end{equation}
Furthermore, if~$\mu_0=\tfrac{1}{n}\sum_{i=1}^n \delta_{x_i}$, then~$\mu_t= \tfrac{1}{n}\sum_i^n \delta_{X^i_{n t}}$ for all~$t>0$ and $\seq{x_i}_{i\leq n}\subset M$, and~$\mu_\bullet$ is the unique analytically weak martingale solution to the \emph{$\mssL$-driven Dean--Kawasaki equation} on~$M$
\[
\diff\mu_t = \msfdiv(\sqrt{\mu_t}\, \xi) + n\,\mssL' \mu_t \diff t\fstop
\]
\end{theorem}

In other words, the \textsc{spde}~\eqref{eq:IntroSPDE} is well-posed ---in the sense of a suitable martingale problem akin to~\cite{KonLehvRe19,KonLehvRe19b}--- for every diffusive recurrent Markov generator~$\mssL$ with spectral gap and ultracontractive semigroup. 
Its unique solution is the measure representation of the free massive system with driving noise~$\mssW$ the Markov diffusion process generated by~$\mssL$.

\paragraph{Particular cases}
Let us point out some particular cases of interest:
\begin{itemize}[leftmargin=1.5em]
\item when the massive system consists of~$n$ particles with identical mass, \eqref{eq:IntroSPDE} reduces to~\ref{eq:DK}, and its unique solution coincides with the one constructed by Konarovskyi, Lehmann, and von Renesse;
\item when~$M=\mbbS^1$ is the standard unit circle and~$\mssL$ is the one-dimensional Laplacian,~\eqref{eq:IntroSPDE} is the \emph{Dean--Kawasaki equation} `with singular drift' considered by Ko\-narov\-skyi and von Renesse in~\cite[Eqn.~(1.2)]{KonvRe18}.
In this \emph{one}-dimensional case, solutions are not unique by~\cite{KonvRe17}, and indeed Assumption~\ref{ass:QPPIntro} does not hold.

\item when $M=\mbbI$ is the standard unit interval and~$\mssL$ is the one-dimensional Laplacian with Neumann boundary conditions, a solution~$\mu_\bullet$ to~\eqref{eq:IntroSPDE} would provide a construction for solutions to the Dean--Kawasaki \textsc{spde} with purely reflecting boundary conditions, a problem recently considered with a thermodynamic-limit approach by P.C.~Bresloff in~\cite{Bre23}.
\end{itemize}

\begin{figure}[htb!]
\begin{subfigure}[t!]{.5\textwidth}
\centering
\includegraphics[scale=.35,trim={20 0 20 0},clip]{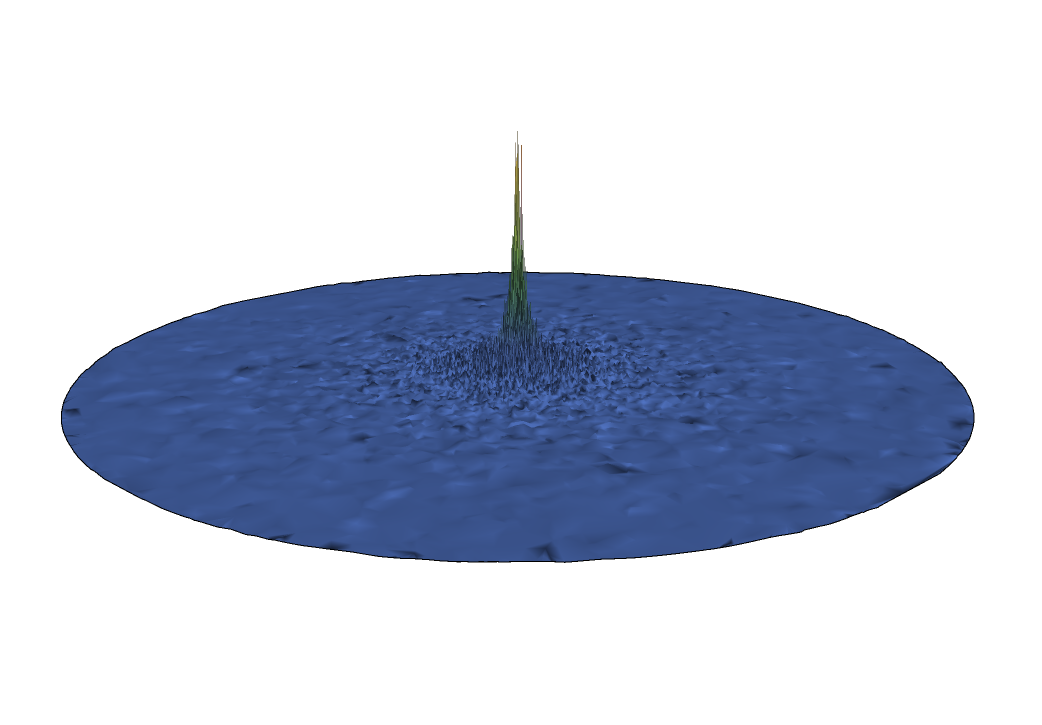}
\end{subfigure}
~~
\begin{subfigure}[t!]{.5\textwidth}
\centering
\includegraphics[scale=.35,trim={20 0 20 0},clip]{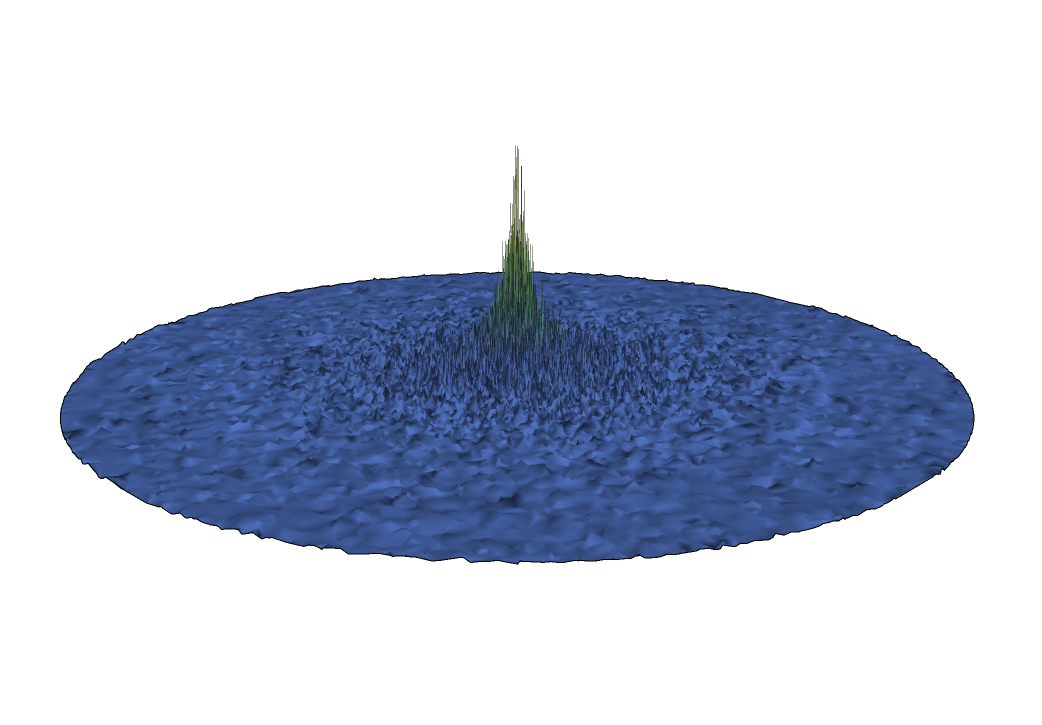}
\end{subfigure}

\begin{subfigure}[t!]{.5\textwidth}
\centering
\includegraphics[scale=.35,trim={20 0 20 0},clip]{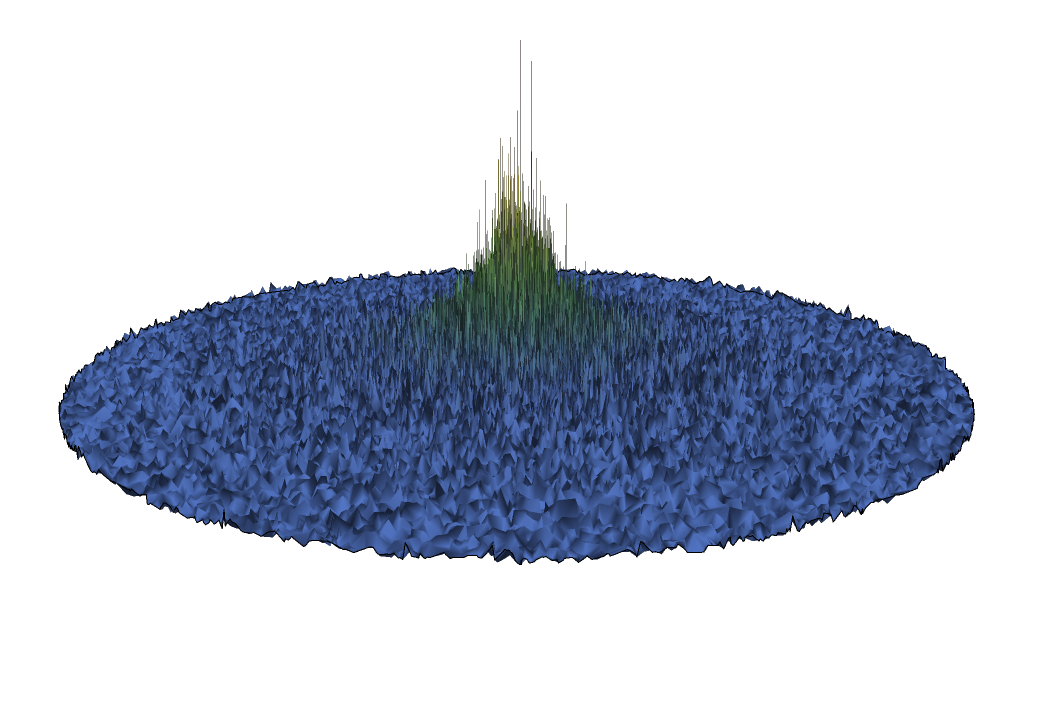}
\end{subfigure}
~~
\begin{subfigure}[t!]{.5\textwidth}
\centering
\includegraphics[scale=.35,trim={20 0 20 0},clip]{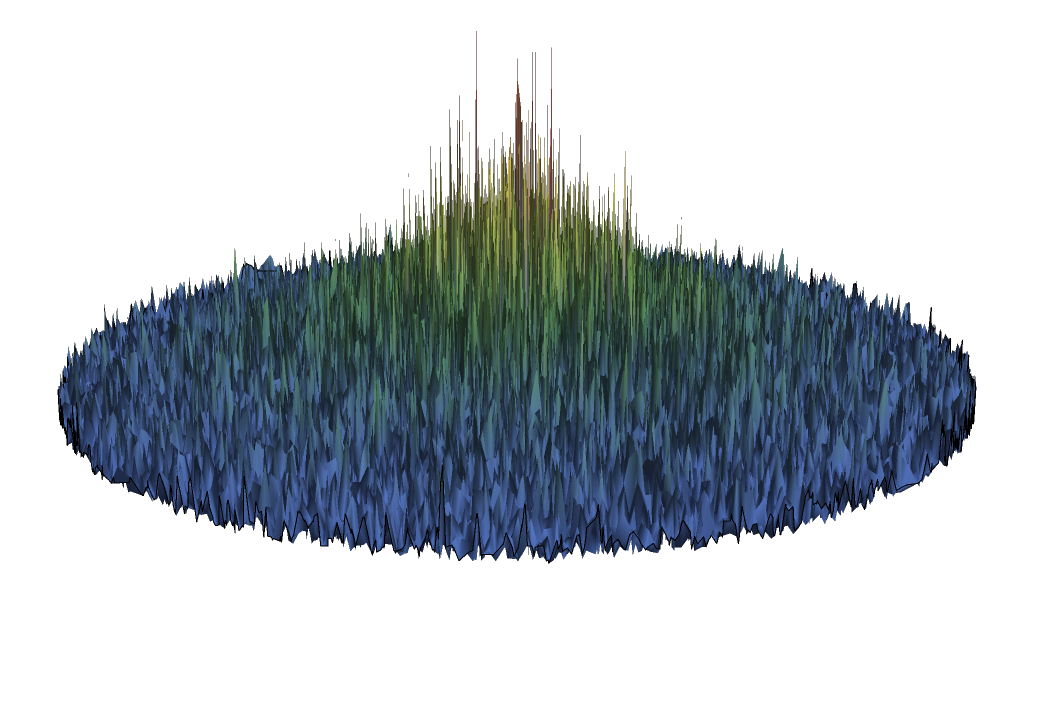}
\end{subfigure}

\begin{subfigure}[t!]{.5\textwidth}
\centering
\includegraphics[scale=.35,trim={20 0 20 0},clip]{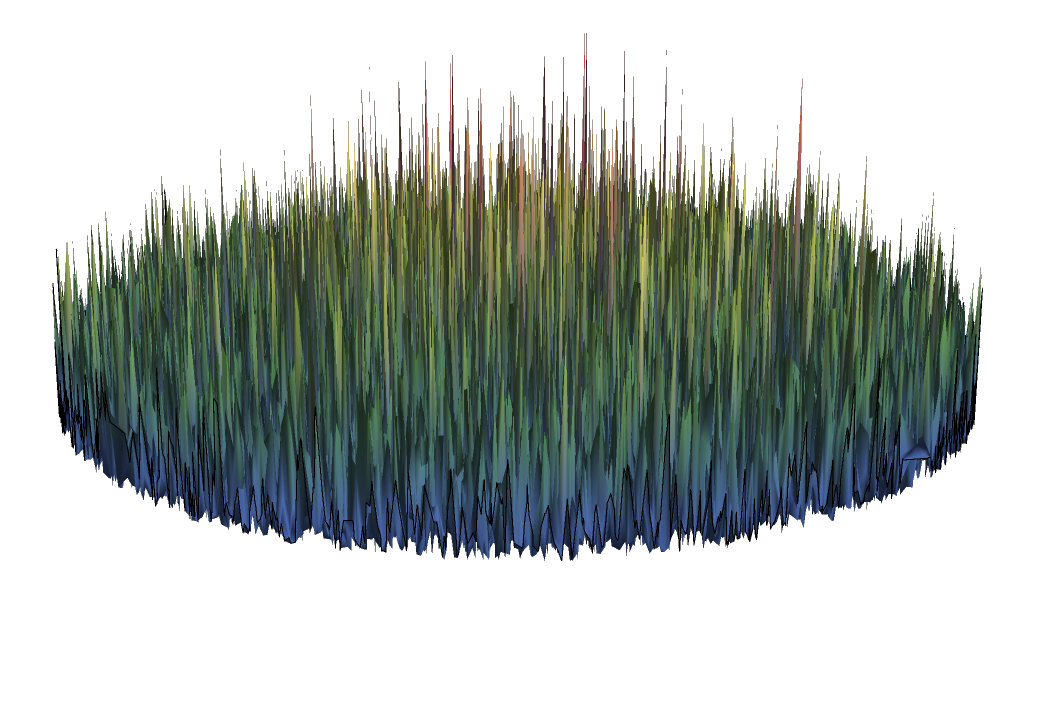}
\end{subfigure}
~~
\begin{subfigure}[t!]{.5\textwidth}
\centering
\includegraphics[scale=.35,trim={20 0 20 0},clip]{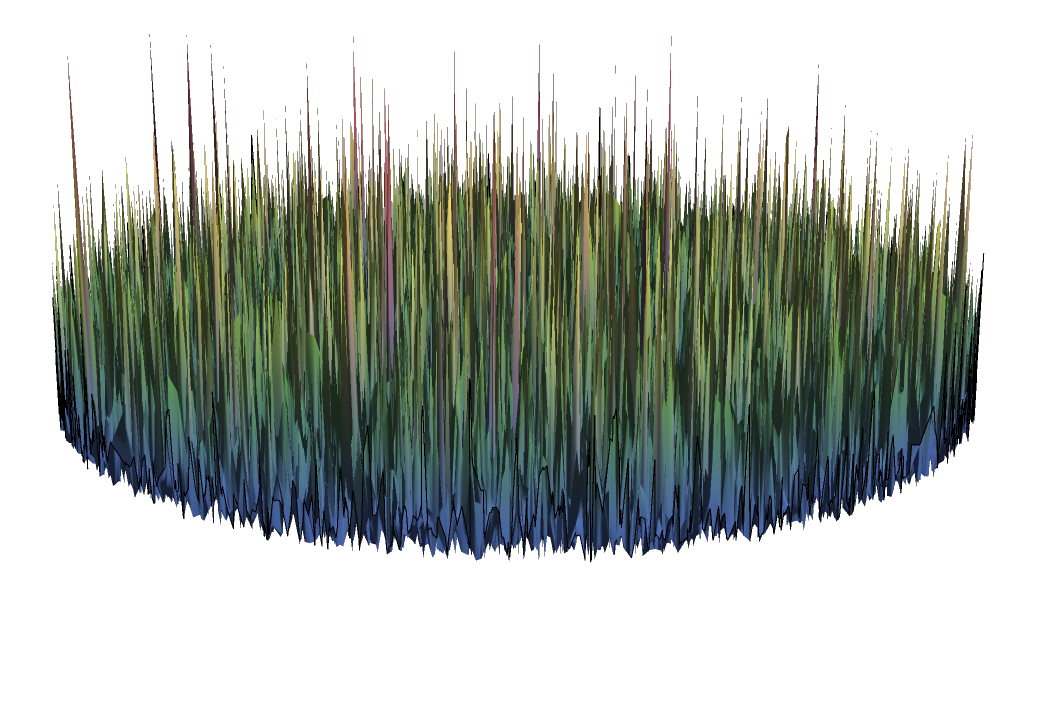}
\end{subfigure}

\caption{The diffusive nature of solutions to~\eqref{eq:IntroSPDE} is exemplified by the evolution of a concentrated initial datum.
Here, the measure representation~$\mu_\bullet$ as in~\eqref{eq:IntroMeasureRep} of a massive particle system driven by the Laplacian on the disk is shown at increasing times, with~$N=10^5$,  initial datum concentrated around~$X^i_0=\zero$, and masses~$\seq{s_i}_{i\leq N}$ sampled from a standard Dirichlet distribution over the $N$-simplex.}
\label{fig:2}
\end{figure}

\paragraph{Systems at the mesoscopic scale}
The free Dean--Kawasaki equation~\ref{eq:DK} has occasionally been deemed \guillemotleft \emph{of dubious mathematical meaning}\guillemotright, e.g.~\cite[p.~2]{DjuKrePer24} and~\cite[p.~2]{CorFis23}.
On the one hand, this is partially due to the singular gradient noise and to the square-root non-linearity; on the other hand, whereas solutions to~\ref{eq:DK} are supposed to model particle systems at the \emph{mesoscopic scale}, see e.g.~\cite{JacZim14,KimKawJacvWi14}, this is seemingly not the case in light of Konarovsky--Lehmann--von Renesse's rigidity, as solutions exactly describe free systems of~$n$ identical particles.

Theorem~\ref{t:IntroDK} reconciles at once both these aspects with the original intuition.
Firstly, solutions to~\eqref{eq:IntroSPDE} truly describe systems at the mesoscopic scale:
while solutions still retain a particulate ---i.e., microscopic--- nature, the singular-drift term allows for the emergence of a macroscopic mass profile, enabling the system to reflect observable mesoscopic structure; see Fig.~\ref{fig:2}.
Secondly, the \textsc{spde}~\eqref{eq:IntroSPDE} may be given a rigorous meaning in the space~$\msP^\pa$, interpreted in the dual pairing with test functions that are continuous in the weak atomic topology~$\T_\mrma$.
Indeed, if for a purely atomic~$\eta$ we define its \emph{squared measure}
\[
\eta^2\eqdef \sum_{x\in\eta} (\eta\!\set{x})^2 \, \delta_x \comma
\]
then~$\T_\mrma$ is the coarsest topology on~$\msP^\pa$ for which the map~$\eta\mapsto \eta^2$ is continuous, see~\cite[Lem.~2.2]{EthKur94}.
This completes the aforementioned identification of~$\T_\mrma$ as the natural topology in the study of (analytically weak) solutions to Dean--Kawasaki-type equations, as it is the coarsest one making dualization continuous on a point-separating class of continuous test functions.

\subsubsection{Main difficulties~III: uniqueness}
In light of the results in~\cite{KonvRe17} in the case when~$M=\R$, it is not entirely surprising that massive particle systems driven by Brownian motion are solutions to the Dean--Kawasaki equation with singular drift~\eqref{eq:Intro:TrueDK}.
Also, albeit technically more demanding, the extension to~\eqref{eq:IntroSPDE} ought to be expected, given that it matches the level of generality we chose for our description of massive particle systems.

Our main result about Dean--Kawasaki equations 
is rather \emph{uniqueness}.
We show that the generator on~$L^2(\mcQ)$ of the martingale problem identifying solutions to~\eqref{eq:IntroSPDE} is essentially self-adjoint.
Throughout the literature on Dirichlet forms, this is widely recognized as a most challenging task, especially on infinite-dimensional non-flat spaces; here, on~$\msP$.
Thanks to essential self-adjointness we conclude uniqueness of solutions \emph{in the strongest possible form} expected for Dean--Kawasaki-type equations, even in the presence of non-linearity, singularity, and for \emph{white} noise.

Also, we show that uniqueness is tightly linked to the geometric setting. On the one hand, the \emph{shape} of the ambient space must grant essential self-adjointness for the generator defined on a class of test functions; on the other hand, its \emph{dimension} plays a critical role, with uniqueness failing in $d=1$ due to collisions.

\subsubsection{On the level of generality}
Let us point out here that the level of generality we chose for our discussion, while seemingly ostentatious, is not whimsical.

In fact, allowing for the ambient space to have a distinct shape ---possibly: non-smooth and with a boundary--- is functional to the discussion of interacting systems below, and more generally it is necessary to the applications.
See e.g.: \cite{VenIllGraBen25},~detailing the dynamics of overdamped Brownian particles interacting through soft pairwise potentials on a \emph{comb-like structure};
\cite{Bre23},~considering a generalized Dean--Kawasaki equation for an interacting Brownian gas in a \emph{partially absorbing} one-dimensional {medium} (a half-line with partially or totally reflecting boundary);
\cite{ShuEijvdB07},~discussing droplet formation within multiphase flow in micro- and nano\emph{channels};
\cite{OshPopDie17},~concerned with lab-on-chip applications of the dynamics of colloids constrained by a \emph{spherical} interacting surface.

\subsection{Interacting massive systems}\label{ss:IntroInteraction}
Let us now proceed and consider \emph{interacting} particle systems. 
It is convenient and interesting to also discuss \emph{zero-range} interaction, which can arise even in free massive systems when Assumption~\ref{ass:QPPIntro} fails.

\subsubsection{Zero-range interaction}\label{sss:IntroZeroRange}
Fix~$\mbfs\eqdef \seq{s_i}_{i\leq N}$ in the ordered $N$-simplex.
We classify the behaviour of an interacting massive system in terms of the \emph{first collision time}
\begin{equation}\label{eq:Intro:FCT}
\tau_c=\tau^\mbfs_c\eqdef \sup\set{t>0: X^i_r\neq X^j_r \text{ for all $r<t$, for all~$i\neq j$}} \fstop
\end{equation}
We call the interaction
\begin{itemize}[leftmargin=1.5em]
\item \emph{paroxysmal} if~$\tau_c=0$ with positive probability;
\item \emph{strong} if~$0<\tau_c<\infty$ almost surely;
\item \emph{moderate}\footnote{In the free case, this amounts to no interaction at all. However, we choose the terminology to also include interacting systems with interaction different from zero-range, see below.} if~$\tau_c=\infty$ almost surely.
\end{itemize}

When~$\mssW$ is a diffusion, it is clear that the above distinction is most interesting in the case of infinitely many particles, as this is the only case when paroxysmal interaction may occur.
Thus, let us restrict our attention to the case~$N=\infty$.

If~$\tau_c<\infty$ with positive probability, the stochastic dynamics of the measure representation~$\mu_\bullet=\seq{\mu_t}_{t\leq \tau_c}$ will have multiple Markovian extensions after~$\tau_c$, two trivial and extremal ones being: the process killed upon reaching~$\tau_c$, and the free dynamics discussed above, in which collisions have no effect whatsoever.
All these extensions are in one-to-one correspondence with Markovian extensions of restrictions of the generator of~$\mcE_\mcQ$ to classes of functions describing specific `boundary conditions'.
Indeed, this non-uniqueness after~$\tau_c$ reflects a genuine ambiguity in the post-collision dynamics, and forces one to select among competing Markovian extensions, each corresponding to a different physical or probabilistic interpretation.

\paragraph{Modified massive Arratia flow and coalescing-fragmentating Wasserstein dynamics}
The necessity of considering the paroxysmal regime arises even in the simplest case when the driving noise~$\mssW$ is a Brownian motion on the unit interval or the real line.
This is exemplified by the following two (classes of) natural extensions:
\begin{enumerate*}[$(a)$]
\item Konarovsky's \emph{modified massive Arratia flow}~\cite{Kon17,KonvRe18}, i.e.\ the extension with \emph{sticky boundary conditions}, in which  colliding particles coalesce after their first collision, and behave as a single particle driven by an independent instance of the noise with volatility the inverse of the total mass of the collided particles; 
\item the aforementioned Konarovsky--von Renesse \emph{coalescing-fragmentating Wasserstein dynamics}~\cite{KonvRe17}, a class of extensions with non-trivial boundary conditions, in which colliding particles stick together for a random time and eventually split again in a random way.
\end{enumerate*}
In both cases, it is shown that when~$N=\infty$ the interaction is paroxysmal, i.e.\ particles collide immediately after~$t=0$ a.s..
In fact, in the case of the modified massive Arratia flow, the system is cofinitely degenerate, in the sense that the number of particles in the system is finite for every~$t>0$.

\paragraph{Zero-range interaction in the free case}
In the free case, moderate interaction (i.e.\ non-interaction) is considered in Theorem~\ref{t:Intro1}.
As for paroxysmal interaction, we have the following characterization, where ‘non-polarity of points’ means that singletons have positive capacity with respect to $\mssW$ ---intuitively, that particles can collide--- while ‘quantitative polarity’ ensures collisions are avoided.

\begin{proposition}[See~\S\S\ref{sss:Collisions} and~\ref{sss:Nests}]
Paroxysmal and moderate interaction of a free massive system only depend on the $\mssW$-capacity of points. In the setting of Assumption~\ref{ass:IntroSetting},
\begin{enumerate}[$(i)$, leftmargin=1.75em]
\item the \emph{quantitative polarity of points} Assumption~\eqref{eq:QPP} is sufficient to grant that only moderate interaction occurs;
\item the \emph{non-polarity of points} Assumption~\eqref{eq:QNPP} is sufficient to grant that paroxysmal interaction occurs a.s.\ for $\nu^\otym{\infty}$-distributed starting positions.
\end{enumerate}
In the setting of Assumption~\ref{ass:IntroManifold}, each of the above conditions is both necessary and sufficient for the corresponding conclusion.
In particular, in this setting~\eqref{eq:QPP} and~\eqref{eq:QNPP} form a dichotomy and, if~\eqref{eq:QPP} holds, then solutions to the free Dean--Kawasaki equation~\eqref{eq:IntroSPDE} are unique for each given starting point in~$\msP^\pa$.
\end{proposition}

As a consequence of the proposition, the modified massive \emph{drifted} Arratia flow ---if any--- exists only on one-dimensional manifolds.

\subsubsection{Main results~IV: positive-range singular interaction}\label{sss:PairwiseInteraction}
Let us now turn to the case of positive-range interaction.
In~\S\ref{ss:Girsanov}, we will give a broad treatment of general interactions, external-field potentials, and other modifications of massive particle systems for a general (diffusive) driving noise under Assumption~\ref{ass:IntroSetting}.
These modifications correspond to Girsanov transforms of the measure representation by a general weight~$\phi$ in the broad local space of the Dirichlet form in Theorem~\ref{t:IntroMeasRep}.
Unlike prior results requiring $\mathcal{C}^2$-smoothness in Wasserstein space (e.g., \cite{KonLehvRe19b}), our method accommodates singular and non-smooth weights~$\phi$, thus capturing a much broader class of interactions, possibly neither pairwise nor mean-field.

In order to make sense of~\eqref{eq:IntroSPDE} in the presence of an additional drift accounting for the interaction, similarly to the case of~$F$ in~\eqref{eq:DK0}, let us first note that the extrinsic derivative in~\eqref{eq:ExtrinsicDerivative} ---restricted to measures--- is defined independently of any structure on~$\mbbT^d$, and is thus well-defined for measures on an arbitrary ambient space~$M$.

Again via the theory of Hilbert tangent bundles to a Dirichlet space, an `exterior differential'~$\mssd$ is induced on functions on~$M$ with values in sections of a `tangent bundle' to~$M$, and we may replace the standard gradient~$\nabla$ in~\eqref{eq:DK0} with such~$\mssd$.

\begin{theorem}[See~\S\ref{sss:GeneralGirsanov}]\label{t:IntroGirsanov2}
In the setting of Assumptions~\ref{ass:IntroSetting} and~\ref{ass:QPPIntro}, suppose further that~$\mssW$ is a diffusion with generator~$\mssL$.
Assume~$0<\phi^2<\infty$ is in the domain of the Dirichlet form of the measure representation~$\mu_\bullet$, i.e.\ of the solution to the free Dean--Kawasaki equation with singular drift~\eqref{eq:IntroSPDE}.
Then, the Dean--Kawasaki-type equation
\begin{equation}\label{eq:IntroSPDEInteractions}
\diff\mu_t= \mssdiv(\sqrt{\mu_t}\, \xi)+\sum_{x\in \mu_t} \mssL'\delta_x \diff t + \mssdiv\paren{\mu_t\, \mssd\frac{\delta \phi}{\delta\mu_t}(\mu_t)} \diff t
\end{equation}
admits a unique (analytically weak) solution~$\mu_\bullet^\phi$, which is the Girsanov transform of~$\mu_\bullet$ by weight~$\phi$.
\end{theorem}

Besides the general setting, one main advantage of our Dirichlet-form-theory approach to Girsanov transforms is that we only need a \emph{Sobolev-type regularity} for the weight~$\phi$.
Thus, Theorem~\ref{t:IntroGirsanov2} improves on the analogous result by Konarovsky, Lehmann, and von Renesse~\cite{KonLehvRe19b} in the rigid case (i.e.\ for systems consisting of finitely many particles with equal mass on~$\R^d$) and with interaction potentials~$F$ on~$\msP$ which are \emph{$\mcC^2$-differentiable} in the Kantorovich--Rubinstein (i.e.\ intrinsic) sense.

\paragraph{Examples} Two simple examples to which our theory applies are as follows.
Consider the standard $d$-dimensional torus~$\mbbT^d\eqdef (\R/\Z)^d$ with the flat metric~$g=g_{\Flat}$ and the associated intrinsic distance~$\mssd=\mssd_{\Flat}$.
We identify~$\mbbT^d$ with the unit cube~$[0,1)^d$, inheriting from~$\R^d$ the additive group structure and the Lebesgue measure.
Up to this identification,
\[
\mssd=\mssd_{\Flat}(x,y)^2=  \sum_k^d \tparen{\abs{x_k-y_k}\wedge \abs{1-x_k+y_k}}^2 \fstop
\]

\begin{example}[Subcritical Riesz interaction on~$\mbbT^d$ for~$d\geq 3$]
Let~$\mssW$ be a Brownian motion on~$\mbbT^d$ and~$\xi$ be an $\R^d$-valued space-time white noise. Fix~$p\in [1,d-1)$. 
For each~$\seq{s_i}_i$ in the infinite ordered simplex the infinite system of \textsc{sde}s
\begin{equation*}
\diff X^i_t = \diff \mssW^i_{t/s_i} - \frac{\beta}{2} \sum_j^N s_i\, s_j \frac{\nabla^{g} \mssd(X^i_t,\emparg) (X^j_t)}{\mssd(X^i_t,X^j_t)^{p+1}} \diff t\comma \qquad i\in \N\fstop
\end{equation*}
admits a unique solution and its measure representation solves the \textsc{spde}
\begin{equation*}
\diff \mu_t = \div^{g}(\sqrt{\mu_t }\, \xi) +\sum_{x\in \mu_t} \Delta\delta_x \, \diff t + \frac{\beta}{2} \div^{g}\paren{\mu_t \int_{\mbbT^d} \frac{\nabla^{g} \mssd(y,\emparg)}{\mssd(y,\emparg)^{p+1}} \diff\mu_t(y)} \diff t\fstop
\end{equation*}
\end{example}

\begin{example}[Dyson Brownian motion on~$\mbbT^d$ for~$d\geq 2$]
Let~$\mssW$ be a Brownian motion on~$\mbbT^d$ and~$\xi$ be an $\R^d$-valued space-time white noise.
For each~$\seq{s_i}_i$ in the infinite ordered simplex the infinite system of \textsc{sde}s
\begin{equation*}
\diff X^i_t = \diff \mssW^i_{t/s_i} - \frac{\beta}{2} \sum_j^N s_i\, s_j \frac{\nabla^{g} \mssd(X^i_t,\emparg) (X^j_t)}{\mssd(X^i_t,X^j_t)}  \diff t\comma \qquad i\in \N\fstop
\end{equation*}
admits a unique solution and its measure representation solves the \textsc{spde}
\begin{equation*}
\diff \mu_t =  \div^{g}(\sqrt{\mu_t }\, \xi) +\sum_{x\in \mu_t} \Delta\delta_x  \, \diff t + \frac{\beta}{2} \div^{g}\paren{\mu_t \int_{\mbbT^d} \frac{\nabla^{g} \mssd(y,\emparg)}{\mssd(y,\emparg)}  \diff\mu_t(y)} \diff t\fstop
\end{equation*}
\end{example}

These examples demonstrate that classical interacting particle systems, such as Dyson Brownian motion or Riesz-type models, can be rigorously embedded into our general framework ---even under singular noise and infinite-particle settings.

\subsubsection{Main difficulties IV: conservativeness and Markov uniqueness}
In light of the measure representation for free systems, the existence of massive-particle solutions to~\eqref{eq:IntroSPDEInteractions} \emph{up to the first collision time} follows by a careful application to our setting of the general theory of Girsanov transforms for quasi-regular Dirichlet forms, developed in~\cite{Ebe96, Kuw98b, CheZha02,CheSun06}.

The main difficulty of our analysis rather lies in showing that Girsanov weights~$\phi$ which are merely locally of Sobolev type~$W^{1,4}$ give rise to interactions which are still moderate in the sense of the classification in~\S\ref{sss:IntroZeroRange} ---that is, no particle collision occurs.
A proof of this fact is non-trivial even in the case of \emph{repulsive} interactions.
Indeed ---especially on compact ambient spaces--- it is not unreasonable to expect that several particles with very small mass may aggregate into a single particle of larger mass, thus constituting a better competitor in an energy landscape driven by repulsive mean-field interactions.
This form of \emph{intermittency} is indeed observed in low-dimensional systems of overdamped Brownian particles; see e.g.~\cite[Fig.~3]{DelOllLopBlaHer16}.

In stochastic terms, moderate interactions correspond to conservativeness of the transformed process on the cozero set of~$\phi^{-1}$.
We prove the latter conservativeness thanks to a general result by K.~Kuwae,~\cite{Kuw98b}.
As it is often the case for Dirichlet forms on abstract state spaces, the assertions in~\cite{Kuw98b} depend on a good interplay between the ambient's original topology and the geometry induced on it by the form via some sub-intrinsic distance; cf.\ the notions of \emph{strict locality} and \emph{strong regularity} for Dirichlet forms.
In the present case, this good interplay is granted by the existence of an auxiliary topology coarser than the natural weak atomic topology~$\T_\mrma$ ---namely, the narrow topology~$\T_\mrmn$--- metrized by a sub-intrinsic distance ---the bounded-Lipschitz distance on~$\msP$ or the $L^2$-Kantorovich--Rubinstein distance on~$\msP_2$.

In other words, there is much more to massive particle systems than the solution theory to infinite systems of \textsc{sde}s on infinite product spaces:
While considerations about the narrow topology~$\T_\mrmn$ or the $L^2$-Kantorovich--Rubinstein distance~$W_2$ do not intervene in the construction of an interacting massive system, its identification with the corresponding measure representation as an element of~$\msP$ endowed with~$\T_\mrmn$ and~$W_2$ becomes an essential tool in assessing some of the system's fundamental stochastic properties, such as conservativeness or Varadhan's short-time asymptotics.

\subsection{Outlook}\label{ss:Outlook}
Let us collect here results in connection with other topics.

\subsubsection{Regularization by noise vs.\ turbulence}
At the intersection of control and game theories, stochastic analysis, and calculus of variations, the theory of mean-field models has recently seen a fast development; see, e.g., the monographs~\cite{CarDelLasLio19,CarDel18a,CarDel18b}.
In the mean-field regime, control and game problems may be rephrased as partial differential equations on spaces of probability measures; for an account of recent developments in this direction see e.g.~\cite{DelHam22,CecDel22,GanMaySwi21} and references therein.
Approaches to well-posedness of such equations include ---among others--- metric gradient flows~\cite{ConKraTon23,ConKraTon23b,ConKraTamTon24}, and regularization by noise, e.g.~\cite{ChoGan17,DelHam22,DelHam24}.

\paragraph{Common noise and smoothing}
In particular for this second approach, it is a relevant goal of the theory to construct semigroups acting on functionals on~$\msP$ and displaying good smoothing properties.
On the one hand, the strength of the smoothing effect depends on the reach of the noise: a finite-dimensional noise may induce Sobolev-to-continuously-differentiable regularization, as it is the case for the process generated by Y.T.~Chow and W.~Gangbo's \emph{partial Laplacian}~\cite{ChoGan17} on~$\msP_2(\R^d)$; while an infinite-dimensional noise may induce the stronger bounded-to-Lipschitz regularization, as recently shown by F.~Delarue and W.R.P.~Hammersley for the \emph{rearranged stochastic heat equation}~\cite{DelHam22,DelHam24} on~$\msP(\R)$.
On the other hand, however, the very presence of a smoothing effect is related to the structural properties of the noise, and usually arises in mean-field models in which players are subject to a \emph{common} noise.

\paragraph{Private noise and turbulence}
By contrast, the measure representation~$\mu_\bullet$ of a free massive system experiences only \emph{private} noise, in that each particle in the system is subject to its own independent instance of the driving noise~$\mssW$.
This begs the following question, already hinted at in~\cite[pp.~2f.]{DelHam22} and~\cite[p.~3]{DelHam24}.

\begin{question}[Delarue--Hammersley]
Does the semigroup~$\mcH_\bullet$ of~$\mu_\bullet$ enjoy any smoothing property?
\end{question}

The answer to this question is resoundingly negative.
In order to see this, let~$\T_\mathrm{conn}$ be any topology on~$\msP$ for which~$\msP$ is connected, e.g.\ the narrow topology, the vague topology, the $L^p$-Wasserstein topology for any~$p\geq 1$, or the norm topology induced by the total variation.
Then, there exist (plenty of) bounded Borel functions~$u\colon \msP\to\R$ such that

\begin{answer}[negative]
$\mcH_t u$ is \emph{not} $\T_\mathrm{conn}$-continuous for any~$t>0$.
\end{answer}

Indeed, it follows from Theorem~\ref{t:IntroMeasRep} that~$\mu_\bullet$ is generically not ergodic.
Thus, since~$\mcH_\bullet$ fixes the indicator functions of (non-trivial) $\mu_\bullet$-invariant sets,~$\mcH_t\car_A=\car_A$ is continuous for some topology on~$\msP$ if and only if~$A$ is simultaneously open and closed.
In particular,~$\mcH_\bullet$ is not bounded-to-Lipschitz regularizing, for any distance for which~$\msP$ is connected.
In other words, the absence of ergodicity ---encoded in the non-trivial invariant $\sigma$-algebra--- manifests analytically as a failure of even basic regularization properties of the semigroup $\mcH_\bullet$.

Furthermore, since not only particles experience private noise, but also the \emph{strength} of the noise is private,~$\mu_\bullet$ is not a stochastic flow; and in the absence of the flow property, any smoothing effect of the semigroup~$\mssH_\bullet$ of~$\mssW$ is not immediately reflected by the semigroup~$\mcH_\bullet$ of~$\mu_\bullet$, not even on invariant sets.
Precisely,~$\mu_\bullet$ is \emph{not} a \emph{stochastic flow} in the sense of Le~Jan--Raimond~\cite{LeJRai04}, since the family of semigroups~$\mssP^\sym{n}_\bullet$ describing the evolution of the largest~$n$ particles of~$\mu_\bullet$ does not satisfy the compatibility condition in~\cite[Dfn.~1.1]{LeJRai04}.
In particular, as a consequence of the characterization in~\cite{LeJRai04}, there exists no stochastic flow of random measurable mappings~$\phi_\bullet\eqdef\seq{\phi_{s,t}}_{0\leq s\leq t}$, with~$\phi_{s,t}\colon M\to M$, so that~$\mu_t={\phi_{0,t}}_\pfwd \mu_0$ for~$t>0$.
Intuitively, the `flow'~$\mu_\bullet$ is so `turbulent' that the ---arbitrarily fast--- particles in~$\mu_\bullet$ `tear apart the fabric of space in a non-measurable way'.
This is further reflected by the non-linear nature of the operator in the singular drift term of~\eqref{eq:IntroSPDE}, which does not depend on the masses of the atoms of~$\mu_\bullet$, but rather only on their locations.

\subsubsection{Possible extensions: jump noise, multiple species, momenta}
We conclude this Introduction by listing some further possible generalizations and extensions to be addressed elsewhere.

\paragraph{The Dean--Kawasaki equation with jump noise}
From a physical point of view,
\begin{quotation}
\guillemotleft [\emph{mesoscopic systems}] \emph{are macroscopic in extent yet microscopic in their reflection of quantum mechanical character.
Mesoscopic systems manifest a multitude of unusual quantum phenomena from localization to strong sample-to-sample fluctuations} [\emph{...}].\guillemotright
\cite[p.~302]{AltSim10}.
\end{quotation}

In mathematical terms, we expect these quantum phenomena ---a simple one being quantum tunnelling--- to be modelled by a noise~$\mssW$ with non-trivial jump part.
Due to our use of the machinery of Hilbert tangent bundles, we are not able to give a rigorous meaning to~\eqref{eq:IntroSPDE} when~$\mssL$ generates such a noise~$\mssW$.
Indeed, in the jump setting, the absence of a tangent structure on the underlying space prevents the direct application of Hilbert-module constructions for the divergence operators needed to rigorously formulate the \textsc{spde}.
However, by analogy with the diffusive case, we expect~\eqref{eq:IntroSPDE} to admit solutions given as the measure representation of a free particle system driven by some jump noise~$\mssW$ generated by~$\mssL$, exactly as in the diffusive case.
Such systems and their measure representations exist and are unique by Theorems~\ref{t:Intro1} and \ref{t:IntroMeasRep}.

\paragraph{Multiple species}
In this work we only consider single-species systems.
That is, all the particles are subject to (independent instances of) the \emph{same} driving noise.
While we focus here on single-species systems for clarity, our framework ---particularly the product Dirichlet-form construction--- naturally accommodates heterogeneous species with differing dynamics and interaction types, as it is readily derived from Theorem~\ref{t:BendikovSaloffCoste} below.

\paragraph{Momenta and other physically relevant quantities}
Our description of massive particle systems relies on the markings~$s_i$.
The interpretation of these markings as the mass of each particle is however merely expedient.
Indeed, we may as well replace the markings by any other set of spins~$q_i$ in some abstract spin space~$Q$ (e.g.~$\R$,~$\R^d$,~$\mbbS^{d-1}$, etc.), together with a function~$\kappa\colon Q\to \R_+$, and satisfying~$\sum_i \kappa(q_i)=1$.
These new markings~$q_i$ may then be used to better describe interactions in a multispecies system, where particles belonging to different species are still driven by the same noise with volatility~$\kappa(q_i)$, but their interaction takes into account the additional degrees of freedom encoded by the spins~$q_i$'s.
The simplest example is that of charged particles: $q_i$ is the (positive or negative) charge of a particle, driven by a noise with strength~$\kappa(q_i)=\abs{q_i}$ simply proportional to the charge's magnitude.

\bigskip

These directions suggest that the mathematical theory of massive systems at the mesoscopic scale is not only robust, but also deeply extensible ---bridging probability, geometry, and physics across several dimensions.

\begin{figure}[htb!]
\centerline{
\xymatrix@C=4em{
&& \txt{Wasserstein-type\\diffusions~\cite{LzDS17+,KonvRe17,KonvRe18}} \ar@{<=>}_{\txt{\tiny local\\ \tiny case}}[d]
\\
& \txt{form on\\marked space} \ar@{<=>}[r]^{\txt{\tiny quasi-homeo.}}_{\txt{\tiny Thm.~\ref{t:TransferP}}} & \txt{form on probab.\\ measures}\ar@/^1.5pc/[r]^{\txt{\tiny Girsanov transf.}}_{\txt{\tiny Thm.~\ref{t:Girsanov}}} \ar@{<=>}[dr] & \txt{interacting\\system} \ar@{=>}[d]
\\
\txt{$\infty$-product\\form~\cite{BouHir91}} \ar@{<=>}[r]^{\txt{\tiny Thm.~\ref{t:BendikovSaloffCoste}}} \ar@{<=>}[d]_{\acts} & \txt{form on\\$\infty$-product~\cite{AlbDalKon00}} \ar@{<=>}[d]^{\txt{\tiny Thm.~\ref{t:BendikovSaloffCoste}}}  \ar[u]^{\txt{\tiny direct\\ \tiny integration}}_{\txt{\tiny Thm.~\ref{t:DirInt}}} & \txt{Dean--Kawasaki\\\textsc{spde}~\cite{KonLehvRe19,KonLehvRe19b,KonMue23}} \ar@{<=>}[d]_{\acts} & \txt{martingale\\problem} \ar@{<=>}[l]_{\txt{\quad \tiny local case}}
\\
\txt{$\infty$-product\\semigroup~\cite{Gil74}} \ar@{<=>}[r]_{\txt{\tiny Cor.~\ref{c:Gill}}} & \txt{$\infty$-product\\kernel~\cite{BenSaC97}} \ar@{<=>}[r]^{\txt{\tiny Cor.~\ref{c:Wellposedness}}} & \txt{free massive\\system}
}
}
\caption{A diagram of the main results}\label{fig:Diagram}

\end{figure}

\newpage

\section{Free massive systems: particle representation}\label{s:FIS}
Throughout this work, we shall make extensive use of the theory of Dirichlet forms both in its analytic and its probabilistic aspects.
We adhere to the standard terminology and notation in the monographs~\cite{MaRoe92,FukOshTak11} and in~\cite{Kuw98,CheMaRoe94}.
For ease of exposition and for reference, all the relevant definitions and some standard facts are recalled in the Appendix \S\ref{app:DirichletForms}.

\subsection{Infinite-product spaces}
We will be interested in several constructions on products of up-to-countably many factors.
In order to discuss infinite products, it will be convenient to establish the following shorthand notation. 

\begin{notation}[Product notation]
Whenever~$\square$ is a product symbol (e.g.\ a tensor product, or a Cartesian product) of objects~$a_j,b_j$ indexed over a totally ordered index set~$J$, we write
\[
a_j \square_{k\neq j} b_j \eqdef \paren{\square_{k<j} b_j }\ \square\ a_j\ \square \paren{\square_{k>j} b_j} \fstop
\]
This shorthand is convenient in abbreviating longer expressions. No misunderstanding may arise on its meaning, since~$J$ is totally ordered and the product~$\square_{j\in J} a_j$ is only considered in the order of~$J$.
\end{notation}

\subsubsection{Main assumptions}
Let us state the main assumptions on each factor.

\begin{assumption}\label{ass:SettingK}
For each~$k\in\N$, assume that
\begin{itemize}[wide, labelindent=0pt]
\item $(M_k,\T_k)$ is a second-countable locally compact Hausdorff space;
\item $\nu_k$ is a Borel probability measure on~$M_k$;
\item $\tparen{\mssE_k,\dom{\mssE_k}}$ is a $\T_k$-\emph{regular}, \emph{conservative} (symmetric) Dirichlet form on~$L^2(\nu_k)$;
\item $\tparen{\mssL_k,\dom{\mssL_k}}$ is the non-negative self-adjoint generator of~$\tparen{\mssE_k,\dom{\mssE_k}}$ on~$L^2(\nu_k)$;
\item $\mssH_{k,\bullet}$ is the self-adjoint Markov semigroup on~$L^2(\nu_k)$ generated by $\tparen{\mssL_k,\dom{\mssL_k}}$;
\item $\msA_k\subset \Cz(\T_k)\cap\dom{\mssE_k}$ is an algebra of functions both $\Cz$-dense in~$\Cz(\T_k)$ and~$(\mssE_k)^{1/2}_1$-dense in~$\dom{\mssE_k}$;\footnote{The reader will note that~$\msA_k$ is \emph{not} a \emph{core} for~$\tparen{\mssE_k,\dom{\mssE_k}}$ in the standard sense, e.g.~\cite[p.~6]{FukOshTak11}, in that we do not require~$\msA_k\subset \Cc(\T_k)$. (Note that~\cite{FukOshTak11} indicates by the symbol~$\Cz$ what we indicate by~$\Cc$.)}
\item $\mssW_k$ is the $\nu_k$-symmetric \emph{Hunt} process properly associated with~$\tparen{\mssE_k,\dom{\mssE_k}}$.
\end{itemize}
\end{assumption}

In order to discuss infinite tensor products below, we need the following definition.

\begin{notation}[Unital extension]\label{n:UnitalExtension}
For any algebra~$\msA$, we denote by~$\msA^\mrmu\eqdef \msA\oplus\R\car$ its smallest unital extension, the appropriate unit~$\car$ ---usually, the constant function~$\equiv 1$--- being apparent from context.
\end{notation}

If~$M_k$ is non-compact, then~$\msA_k\subsetneq \msA_k^\mrmu$.
If~$M_k$ is compact, we can assume~$\msA_k^\mrmu=\msA_k$ without loss of generality.
We will comment on why this is possible in Remark~\ref{r:Unitality} below.

\subsubsection{Cylinder functions}
In order to define a Dirichlet form on the infinite product of the factors~$M_k$, we need a good class of `smooth' functions defined on it.
This will be defined as the natural image of a class of functions in the form of an infinite tensor product.

\paragraph{Infinite tensor products}
We will consider \emph{von Neumann incomplete infinite tensor products},~\cite{Neu39}.
We recall the main definitions of such tensor products after A.~Guichardet~\cite{Gui69a,Gui69b}.
Everywhere in this section,~$A=\set{a_1,\dotsc,a_{\card{A}}}$ is a \emph{finite} subset of~$\N$.

For each~$k\in\N$ let~$(H_k,\scalar{\emparg}{\emparg}_k)$ be a Hilbert space, and fix a unit vector~$t_k\in H_k$.
For each~$\mbft\eqdef\seq{t_k}_k$, we denote by~$\sotimes_k^{\mbft} E_k$ the linear span of all elementary tensors of the form
\[
\sotimes_k f_k\comma \qquad f_k\in E_k\comma \quad f_k= t_k \text{ for every } k\notin A\comma
\]
for some \emph{finite}~$A\subset \N$.
The scalar product defined as the linear extension from elementary tensors of
\[
\scalar{\sotimes_k f_k}{\sotimes_k g_k}= \sprod_k \scalar{f_k}{g_k}_k
\]
is a pre-Hilbert scalar product on~$\sotimes_k^\mbft H_k$, and we denote by~$\shotimes_k^\mbft H_k$ its Hilbert completion.

\medskip

Now, let~$\mbfM\eqdef \sprod_k M_k$ be endowed with the product topology~$\T_\mrmp$ and with the Borel probability measure~$\boldnu\eqdef\otimes_k \nu_k$.
Further note that~$\car_k$ is a unit vector in~$L^2(\nu_k)$ for each~$k$, and set~$\car\eqdef \seq{\car_1,\car_2,\dotsc}$.
For each finite~$A\subset \N$, let~$M_A\eqdef \sprod_{a\in A} M_a$ be endowed with the product $\sigma$-algebra and with the product measure~$\nu_A\eqdef \otimes_{a\in A} \nu_a$.
Further set
\[
\mbfx_A\eqdef \ttseq{x_{a_1},\dotsc, x_{a_{\card{A}}}}\comma \qquad \mbfx\in\mbfM\fstop
\]

Consider the infinite tensor product of Hilbert spaces~$\sotimes_k^{\car} L^2(\nu_k)$.
Elementary tensors in~$\sotimes_k^{\car}L^2(\nu_k)$ have the form
\begin{equation}\label{eq:ElemTensorAlgebra}
F= F_A= \sotimes_k f_k \comma \qquad f_k\in L^2(\nu_k) \comma \quad f_k\equiv \car_k \text{ for every } k\notin A\comma
\end{equation}
for some finite~$A\subset \N$.
The \emph{evaluation}~$\ev$ is the linear extension to~$\sotimes_k^{\car} L^2(\nu_k)$ of the map defined on elementary tensors by
\[
\ev \colon F_A \longmapsto F_A(\emparg_A)\colon \mbfx\mapsto \sprod_{a\in A} f_a(x_a) \comma
\]
where~$F_A$ is as in~\eqref{eq:ElemTensorAlgebra}.

\begin{lemma}[Guichardet,~{\cite[\S6, no. 6.4, Cor.~6, p.~30']{Gui69b}}]\label{l:Guichardet}
The map~$\ev$ has a unique (non-relabel\-ed) linear extension
\begin{equation}\label{eq:l:Guichardet:0}
\ev \colon \shotimes^{\car}_k L^2(\nu_k) \longrar L^2(\boldnu)
\end{equation}
and the latter is an isomorphism of Hilbert spaces.
\end{lemma}

\paragraph{Cylinder functions}
Similarly to the case of the infinite product of $L^2$-spaces discussed above, consider the infinite tensor product of (unital) algebras~$\sotimes_k^{\car} \msA^\mrmu_k$, see~\cite[\S3, p.~10]{Gui69b}.
Elementary tensors in~$\sotimes_k^{\car} \msA^\mrmu_k$ have the form~\eqref{eq:ElemTensorAlgebra} with~$f_k\in\msA_k^\mrmu$ for every~$k$.
It is readily verified that~$\sotimes_k^{\car} \msA^\mrmu_k$ is an algebra when endowed with (the linear extension of) the product
\[
\sotimes_k f_k \cdot \sotimes_k g_k \eqdef \sotimes_k f_kg_k \fstop
\]

\begin{definition}[Cylinder functions]\label{d:CylinderProduct}
We say that a function~$u\colon \mbfM\to\R$ is a \emph{cylinder function} if~$u=\ev F$ for some~$F\in \sotimes_k^{\car} \msA^\mrmu_k$.
A cylinder function~$u=\ev F$ is \emph{elementary} if~$F=F_A$ is elementary, in which case~$u=\sprod_{a\in A} f_a(\emparg_a)$.
We denote by~$\Cyl{}{}\eqdef \ev (\sotimes_k^{\car} \msA^\mrmu_k)$ the space of all cylinder functions.
\end{definition}

\begin{remark}
The evaluation~$\ev$ is an algebra isomorphism; in particular,~$\Cyl{}{}$ is an algebra and the representation of a cylinder function~$u=\ev F$ by~$F\in \sotimes_k^{\car} \msA^\mrmu_k$ is unique.
\end{remark}

\subsubsection{The generator and the Dirichlet form}
Define an operator~$(\bmssL,\Cyl{}{})$ as the linear extension to~$\Cyl{}{}$ of the operator defined on elementary cylinder functions by
\begin{equation}\label{eq:t:BendikovSaloffCoste:00}
\bmssL u \eqdef \ssum_k (\mssL_k f_k)(\emparg_k)\, \sprod_{i\neq k} f_i(\emparg_i)\comma \qquad u= \sprod_k f_k(\emparg_k) \fstop
\end{equation}

\begin{lemma}\label{l:GeneratorWellDefined}
The operator~$\bmssL$ is well-defined on~$\Cyl{}{}$ and
\begin{equation}\label{eq:t:BendikovSaloffCoste:000}
\bmssL u = \ssum_{a\in A} (\mssL_a f_a)(\emparg_a)\, \sprod_{\substack{i\in A\\ i\neq a}} f_i(\emparg_i) \comma \qquad u=\sprod_{a\in A} f_a(\emparg_a) \fstop
\end{equation}

\begin{proof}
Since~$\mssH_{k,\bullet}$ is conservative by Assumption~\ref{ass:SettingK}, we have~$\mssL_k \car_k\equiv 0$ for every~$k$.
It follows that the series in~\eqref{eq:t:BendikovSaloffCoste:00} reduces to the finite sum in~\eqref{eq:t:BendikovSaloffCoste:000}.
\end{proof}
\end{lemma}

We are now ready to state the main result of this section.

\begin{theorem}\label{t:BendikovSaloffCoste}
The quadratic form~$(\bmssE,\Cyl{}{})$ on~$L^2(\boldnu)$ defined by
\begin{equation}\label{eq:t:BendikovSaloffCoste:0}
\bmssE(u,v)\eqdef \scalar{u}{-\bmssL v}_{L^2(\boldnu)} \comma \qquad u,v\in \Cyl{}{}\comma
\end{equation}
is closable.
Its closure~$\tparen{\bmssE,\dom{\bmssE}}$ is a quadratic form on~$L^2(\boldnu)$ with generator the $L^2(\boldnu)$-Friedrichs extension~$\tparen{\bmssL,\dom{\bmssL}}$ of~$(\bmssL,\Cyl{}{})$.

Further define the following assumptions:
\begin{enumerate}[$(a)$]
\item\label{i:t:BendikovSaloffCoste:a} $\mssH_{k,\bullet}$ is irreducible and is represented by a kernel~$\mssh_{k,\bullet}$ absolutely continuous w.r.t.~$\nu_k$, and $\tparen{-\mssL_k,\dom{\mssL_k}}$ has spectral gap
\[
\lambda_k\eqdef \inf \set{\scalar{-\mssL_k f}{f}_{L^2(\nu_k)}: f\in\dom{\mssL_k}\comma \norm{f}_{L^2(\nu_k)}=1\comma \int_{M_k} f \diff\nu_k=0}>0 \fstop
\]
\item\label{i:t:BendikovSaloffCoste:d} there exist constants $c,C>0$ such that, for each~$k$,
\[
\nu_k^\otym{2}\text{-}\esssup_{x,y\in M_k} \mssh_{k,\eps_k}(x,y) \leq C \quad \text{for some~$\eps_k$ satisfying~$\eps_k\leq c/\lambda_k$}\semicolon
\]
\item\label{i:t:BendikovSaloffCoste:e} setting~$t_*\eqdef\inf\set{t>0: \sum_k e^{-2t\lambda_k}<\infty}$ it holds that~$t_*=0$;
\item\label{i:t:BendikovSaloffCoste:f} $\msA_k\subset\dom{\mssL_k}$ and~$\mssH_{k,t}\msA_k\subset \msA_k$ for every~$t>0$.
\end{enumerate}

If~\ref{i:t:BendikovSaloffCoste:a}--\ref{i:t:BendikovSaloffCoste:e} hold, then
\begin{enumerate}[$(i)$]
\item\label{i:t:BendikovSaloffCoste:1} the measure $\bmssh_t(\mbfx,\diff\mbfy)\eqdef \sprod_k \mssh_{k,t}(x_k,\diff y_k)$ is absolutely continuous w.r.t.~$\boldnu$ for every~$t>0$ and every~$\mbfx\in\mbfM$ if and only if~\ref{i:t:BendikovSaloffCoste:e} holds.
In this case,~$\bmssh_t$ has density~$\bmssh_t(\mbfx,\mbfy)\eqdef \sprod_k \mssh_{k,t}(x_k,y_k)$ in~$L^\infty(\boldnu)$.
\end{enumerate}

If~\ref{i:t:BendikovSaloffCoste:a}--\ref{i:t:BendikovSaloffCoste:f} hold, then, additionally,
\begin{enumerate}[$(i)$]\setcounter{enumi}{1}
\item\label{i:t:BendikovSaloffCoste:2} $(\bmssL,\Cyl{}{})$ is essentially self-adjoint on~$L^2(\boldnu)$;
\item\label{i:t:BendikovSaloffCoste:3} the semigroup~$\bmssT_\bullet$ of~$\tparen{\bmssL,\dom{\bmssL}}$ coincides with the Markov semigroup~$\bmssH_\bullet$ represented by~$\bmssh_\bullet$. In particular,~$(\bmssE,\dom{\bmssE})$ is a Dirichlet form;
\item\label{i:t:BendikovSaloffCoste:3.5} the semigroup~$\bmssT_\bullet$ is irreducible;
\item\label{i:t:BendikovSaloffCoste:4} $(\bmssE,\dom{\bmssE})$ is a $\T_\mrmp$-quasi-regular Dirichlet form, and it is $\T_\mrmp$-regular if and only if cofinitely many~$M_k$'s are compact.
\item\label{i:t:BendikovSaloffCoste:5} $(\bmssE,\dom{\bmssE})$ is properly associated with a Hunt process~$\bmssW$ identical to~$\seq{\mssW_k}_k$ in distribution;
\item\label{i:t:BendikovSaloffCoste:6} for every Borel probability measure~$\boldsigma$ on~$\mbfM$ with~$\boldsigma\ll\boldnu$, the process~$\bmssW$ is the ---unique up to $\boldnu$-equivalence--- $\boldnu$-substationary (see Dfn.~\ref{d:Substationary}) $\boldnu$-special standard process solving the martingale problem for~$(\bmssL,\Cyl{}{},\boldsigma)$. 
In particular, denoting by~$\mbfX_\bullet$ the trajectories of~$\bmssW$, and by~$\bmssGamma$ the carr\'e du champ of~$\bmssL$, viz.
\[
\bmssGamma(u,v)\eqdef \bmssL(uv)-u\bmssL v -v\bmssL u\comma \qquad u,v\in\Cyl{}{}\comma
\]
then, for every~$u\in\Cyl{}{}$ ($\subset \Cb(\T_\mrmp)$), the process
\begin{equation}\label{eq:t:BendikovSaloffCoste:1}
M^{\class{u}}_t\eqdef u(\mbfX_t)- u(\mbfX_0)-\int_0^t (\bmssL u)(\mbfX_s)\diff s
\end{equation}
is an adapted square-integrable $\boldsigma$-martingale with predictable quadratic variation
\begin{equation}\label{eq:t:BendikovSaloffCoste:2}
\tsharpb{M^{\class{u}}}_t = \int_0^t \bmssGamma(u)(\mbfX_s)\diff s \fstop
\end{equation}
\end{enumerate}
\end{theorem}

\begin{remark}
We note that, since we are only concerned with $\nu$-\emph{symmetric} objects (semigroups, generators, Dirichlet forms) on \emph{probability} spaces, the notion of ergodicity and that of irreducibility coincide, see e.g.~\cite[Cor.~2.33, Prop.~3.5]{BezCimRoe18}.
Thus, for simplicity, we replaced the assumption of ergodicity in~\cite{BenSaC97} with that of irreducibility.
\end{remark}

\begin{remark}\label{r:NecessityBendikovSaloffCoste}
Let us note that the assumptions in~\ref{i:t:BendikovSaloffCoste:a},~\ref{i:t:BendikovSaloffCoste:d}, and~\ref{i:t:BendikovSaloffCoste:e} are in fact necessary to the conclusion of~\ref{i:t:BendikovSaloffCoste:1}, as shown in~\cite[Thm.~3.1]{BenSaC97}.
Indeed, if, for example,~$t_*>0$ in~\ref{i:t:BendikovSaloffCoste:e}, then there exists~$\mbfx\in\mbfM$ and~$t<t_*$ such that~$\bmssh_t(\mbfx,\diff\emparg)$ is singular w.r.t.~$\boldnu$.
\end{remark}

\begin{remark}\label{r:BHTensor}
In light of the essential self-adjointness of~$(\bmssL,\Cyl{}{})$ in Theorem~\ref{t:BendikovSaloffCoste}\ref{i:t:BendikovSaloffCoste:2}, the Dirichlet form~$\tparen{\bmssE,\dom{\bmssE}}$ coincides with the standard tensor-product Dirichlet form~$\tparen{\sotimes_k \mssE_k,\dom{\sotimes_k \mssE_k}}$ constructed in~\cite[Thm.~V.2.2.1]{BouHir91}.
\end{remark}

\begin{remark}\label{r:Unitality}
As anticipated, when~$(M_k,\T_k)$ is \emph{compact} it is not restrictive to assume that~$\msA_k$ is additionally unital.
Indeed suppose~$\msA_k$ is \emph{not} unital, and let~$\msA_k^\mrmu$ be its smallest unital extension.
Since~$(M_k,\T_k)$ is compact, we still have~$\msA_k^\mrmu\subset \Cz(\T_k)$.
Since~$\nu_k$ is a finite measure, we have~$\car_k\in\dom{\mssL_k}\subset\dom{\mssE_k}$.
Furthermore, since~$\tparen{\mssE_k,\dom{\mssE_k}}$ is conservative by Assumption~\ref{ass:SettingK}, for every~$t>0$ we have~$\mssH_{k,t}\car_k=\car_k$, and therefore~$\mssH_{k,t}\car_k\in \msA_k^\mrmu$.
Thus, if~$\msA_k$ satisfies~\ref{i:t:BendikovSaloffCoste:f}, then~$\msA_k^\mrmu$ does so as well, and we may therefore replace~$\msA_k$ with~$\msA_k^\mrmu$.
\end{remark}

As already hinted to in Remark~\ref{r:Unitality}, the proof of the non-compact case in Theorem~\ref{t:BendikovSaloffCoste} relies on the assertion in the compact case.
In order to transfer results from the compact case to the non-compact one, we will make use of Alexandrov compactifications.

\begin{notation}[Alexandrov compactifications]\label{n:Alexandrov}
Let~$(X,\T)$ be a second-countable locally compact Hausdorff topological space.
We denote by~$(X^\upalpha,\T^\upalpha)$ the \emph{Alexandrov} (\emph{one-point}) \emph{compactification} of~$(X,\T)$ with point-at-infinity~$*$.
We note that~$(X^\upalpha,\T^\upalpha)$ is a compact Polish space, and that 
the compactification embedding~$\upalpha\colon X\to X^\upalpha$ is $\T/\T^\upalpha$-continuous, thus Borel/Borel measurable.
\end{notation}

\begin{proof}[Proof of Theorem~\ref{t:BendikovSaloffCoste}]
The closability follows by the representation in~\eqref{eq:t:BendikovSaloffCoste:0} and standard arguments,~\cite[Thm.~X.23]{ReeSim75}.
We divide the rest of the proof into several steps.

\boldparagraph{Compact case}
Assume first that~$M_k$ is compact for \emph{every}~$k$.
In light of Remark~\ref{r:Unitality} we and will assume with no loss of generality that~$\msA_k$ is additionally unital.
Since~$\msA_k$ is ---only in the compact case--- $\Cz$-dense in~$\Cz(\T_k)=\Cb(\T_k)$, it is point-separating.

Note that~$\mbfM$ is compact and~$\boldnu$ is a Radon measure will full support.

\paragraph{Proof of~\ref{i:t:BendikovSaloffCoste:1}--\ref{i:t:BendikovSaloffCoste:3.5}} Assuming~\ref{i:t:BendikovSaloffCoste:a}--\ref{i:t:BendikovSaloffCoste:e}, assertion~\ref{i:t:BendikovSaloffCoste:1} is~\cite[Thm.~3.1]{BenSaC97}.
The boundedness of the heat-kernel density is~\cite[Lem.~4.3(2)]{BenSaC97}.

As a consequence of~\ref{i:t:BendikovSaloffCoste:1},~$\bmssH_\bullet$ is a Markovian semigroup on~$L^2(\boldnu)$.
Denote by $\tparen{\tilde\bmssL,\dom{\tilde\bmssL}}$ its generator.
Then, for every~$u=F_A(\emparg_A)\in\Cyl{}{}$ with~$F_A=\otimes_{a\in A} f_a$ and~$f_a\in\msA_a$, for~$\boldnu$-a.e.~$\mbfx\in\mbfM$,
\begin{align*}
(\tilde\bmssL u)(\mbfx)=&\ \lim_{t\downarrow 0} \tfrac{1}{t} (\bmssH_t u-u)(\mbfx)
\\
=&\ \lim_{t\downarrow 0} \frac{1}{t} \paren{\int_\mbfM \prod_k^\infty \mssh_{k,t}(x_k,y_k) F_A(\mbfy_A)\diff\boldnu(\mbfy_A)-F_A(\mbfx_A)}
\\
=&\ \lim_{t\downarrow 0} \frac{1}{t} \paren{\int_{M_A} \prod_{a\in A} \mssh_{a,t}(x_k,y_k) F_A(\mbfy_A)\diff \nu_A(\mbfy_A)-F_A(\mbfx_A)}
\\
=&\ \lim_{t\downarrow 0} \frac{1}{t} \prod_{a\in A} \paren{\int_{M_a}\mssh_{a,t}(x_k,y_k) f_a(y_a)\diff \nu_a(y_a)-f_a(x_a)}
\\
=&\ (\bmssL u) (\mbfx)\fstop
\end{align*}
That is,~$(\tilde\bmssL,\Cyl{}{})=(\bmssL,\Cyl{}{})$.
In order to prove~\ref{i:t:BendikovSaloffCoste:2} it suffices to show that~$(\tilde\bmssL,\Cyl{}{})$ is essentially self-adjoint.
Indeed, in this case, since~$(\tilde\bmssL,\Cyl{}{})=(\bmssL,\Cyl{}{})$, the Friedrichs extension~$\tparen{\bmssL,\dom{\bmssL}}$ is a self-adjoint extension of~$(\tilde\bmssL,\Cyl{}{})$ and thus $\tparen{\tilde\bmssL,\dom{\tilde\bmssL}}=\tparen{\bmssL,\dom{\bmssL}}$ by uniqueness of the self-adjoint extension of~$(\tilde\bmssL,\Cyl{}{})$.
Assuming~\ref{i:t:BendikovSaloffCoste:f}, it is readily verified that~$\bmssH_t \Cyl{}{}\subset\Cyl{}{}$ for every~$t>0$ by definition of~$\Cyl{}{}$ and the assumption.
Since~$\Cyl{}{}$ is dense in~$L^2(\boldnu)$, the operator~$(\tilde\bmssL,\Cyl{}{})$ is essentially self-adjoint by~\cite[Thm.~X.49]{ReeSim75}, which shows~\ref{i:t:BendikovSaloffCoste:2}.

Assertion~\ref{i:t:BendikovSaloffCoste:3} is a consequence of the equality of the corresponding generators, which in turn follows from~$(\tilde\bmssL,\Cyl{}{})=(\bmssL,\Cyl{}{})$ above together with~\ref{i:t:BendikovSaloffCoste:2}.
Since~$\bmssH_\bullet$ is Markovian,~$\tparen{\bmssE,\dom{\bmssE}}$ is a Dirichlet form.

As already claimed in~\cite[p.~172]{BenSaC97},~$\bmssH_\bullet$ is irreducible (ergodic) by irreducibility~\ref{i:t:BendikovSaloffCoste:a} of each~$\mssH_{k,\bullet}$.
Assertion~\ref{i:t:BendikovSaloffCoste:3.5} is a consequence of this fact together with the identification~$\bmssT_\bullet=\bmssH_\bullet$ in~\ref{i:t:BendikovSaloffCoste:3}.

\paragraph{Regularity}
Now, since~$\msA_k\subset\Cz(\T_k)=\mcC(\T_k)$ for every~$k$, the family~$\Cyl{}{}$ is a unital point-separating subalgebra of~$\mcC(\T_\mrmp)$.
By Stone--Weierstra\ss\ Theorem,~$\Cyl{}{}$ is uniformly dense in~$\mcC(\T_\mrmp)\subset L^2(\boldnu)$, and thus dense in~$L^2(\boldnu)$ by the Radon property of~$\boldnu$.
Thus,~$(\bmssE,\Cyl{}{})$ is densely defined on~$L^2(\boldnu)$ and~$\Cyl{}{}$ is a form-core for~$\tparen{\bmssE,\dom{\bmssE}}$.
Furthermore, since~$\Cyl{}{}$ is uniformly dense in~$\Cz(\T_\mrmp)=\mcC(\T_\mrmp)$, it is a core for~$\tparen{\bmssE,\dom{\bmssE}}$ in the usual sense.

Since~$\mbfM$ is compact (in particular: locally compact), since~$\boldnu$ is a Radon measure will full support, and since~$\Cyl{}{}$ is a core for~$\tparen{\bmssE,\dom{\bmssE}}$, the latter form is a regular Dirichlet form, and the existence of~$\bmssW$ follows by the standard theory of Dirichlet forms.

\boldparagraph{General case}
If~$M_k$ is not compact for some~$k$ we replace it with its Alexandrov compactification, see Notation~\ref{n:Alexandrov}.
For each~$f_k\colon M_k\to\R$ let~$f^\upalpha_k$ be defined by~$f^\upalpha_k(\upalpha_k(x))\eqdef f_k(x)$ for~$x\in M_k$ and~$f^\upalpha_k(*_k)=0$.
It is readily verified that the $\upalpha$-image form~$\tparen{\mssE^\upalpha_k,\dom{\mssE^\upalpha_k}}$ on~$L^2(\nu^\upalpha_k)$ as in~\S\ref{sss:ImageObjects} with~$\upalpha$ in place of~$j$ is a conservative Dirichlet form on~$L^2(\nu^\upalpha_k)$ with semigroup the $\upalpha$-image~$\mssH^\upalpha_{k,\bullet}$ of~$\mssH_{k,\bullet}$ and generator the $\upalpha$-image~$\tparen{\mssL^\upalpha_k,\dom{\mssL^\upalpha_k}}$ of~$\tparen{\mssL_k,\dom{\mssL_k}}$.

Now, set
\begin{equation}\label{eq:Compactification}
\mbfM^\boldupalpha\eqdef \sprod_k M^\upalpha_k \ (\footnote{This is different from the Alexandrov compactification of~$\mbfM$.})\comma \qquad \boldnu^\boldupalpha\eqdef \otimes_k \nu^\upalpha_k\comma \qquad \text{and} \qquad \boldupalpha\eqdef \sprod_k \upalpha_k\colon \mbfM\to\mbfM^\boldupalpha \fstop
\end{equation}
It is readily verified that~$\boldupalpha$ is injective, $\T_\mrmp/\T_\mrmp^\boldupalpha$-continuous (thus Borel) and that~$\boldupalpha_\pfwd \boldnu=\boldnu^\boldupalpha$.
Since~$\boldupalpha$ is injective,~$\boldupalpha_*\colon L^2(\boldnu)\to L^2(\boldnu^\boldupalpha)$ is an isomorphism of $L^2$-spaces.
Define~$\Cyl{\boldupalpha}{}$,~$\tparen{\bmssL^\boldupalpha,\Cyl{\boldupalpha}{}}$, $\tparen{\bmssE^\boldupalpha,\dom{\bmssE^\boldupalpha}}$, and~$\bmssH^\boldupalpha_\bullet$ as in~\S\ref{sss:ImageObjects} with~$\boldupalpha$ in place of~$j$ and note that this definitions are consistent with~\eqref{eq:t:BendikovSaloffCoste:00} and~\eqref{eq:t:BendikovSaloffCoste:0} when replacing~$\mssL_k$ with~$\mssL^\upalpha_k$.

\paragraph{Closability and Markovianity}
Applying the first part of the proof to the spaces~$M^\upalpha_k$, we conclude that~$\tparen{\bmssE^\boldupalpha,\dom{\bmssE^\boldupalpha}}$ is a Dirichlet form on~$L^2(\boldnu^\boldupalpha)$ with generator the Friedrichs extension~$\tparen{\bmssL^\boldupalpha,\dom{\bmssL^\boldupalpha}}$ of~$(\bmssL^\boldupalpha,\Cyl{\boldupalpha}{})$.
Since~$\boldupalpha$ is an isomorphism of $L^2$-spaces, this concludes that~$(\bmssE,\Cyl{}{})$ is closable, and that its closure~$\tparen{\bmssE,\dom{\bmssE}}$ is a Dirichlet form on~$L^2(\boldnu)$ with generator the $L^2(\boldnu)$-Friedrichs extension~$\tparen{\bmssL,\dom{\bmssL}}$ of~$(\bmssL,\Cyl{}{})$ satisfying~$\boldupalpha_*\tparen{\bmssL,\dom{\bmssL}}=\tparen{\bmssL^\boldupalpha,\dom{\bmssL^\boldupalpha}}$.
In particular,~$\Cyl{}{}$ is dense in~$L^2(\boldnu)$.

\paragraph{Proof of~\ref{i:t:BendikovSaloffCoste:1}} 
We denote by~\ref{i:t:BendikovSaloffCoste:a}$_\upalpha$--\ref{i:t:BendikovSaloffCoste:e}$_\upalpha$ the assertions in~\ref{i:t:BendikovSaloffCoste:a}--\ref{i:t:BendikovSaloffCoste:e} \emph{mutatis mutandis}, resp.\ by~\ref{i:t:BendikovSaloffCoste:1}$_\boldupalpha$--\ref{i:t:BendikovSaloffCoste:3}$_\boldupalpha$ the assertions~\ref{i:t:BendikovSaloffCoste:1}--\ref{i:t:BendikovSaloffCoste:3} \emph{mutatis mutandis}.
Since~$\upalpha_k^*\colon L^2(\nu^\upalpha_k)\to L^2(\nu)$ is an isomorphism of $L^2$-spaces,~\ref{i:t:BendikovSaloffCoste:a}--\ref{i:t:BendikovSaloffCoste:e} hold if and only so do~\ref{i:t:BendikovSaloffCoste:a}$_\upalpha$--\ref{i:t:BendikovSaloffCoste:e}$_\upalpha$.
Since~$\boldupalpha$ is $\T_\mrmp/\T^\boldupalpha_\mrmp$-continuous and since~$\boldupalpha_\pfwd\boldnu=\boldnu^\boldupalpha$,~\ref{i:t:BendikovSaloffCoste:1} holds if and only so does~\ref{i:t:BendikovSaloffCoste:1}$_\boldupalpha$.
Thus, applying the proof in the compact case shows~\ref{i:t:BendikovSaloffCoste:1}.

\paragraph{Proof of~\ref{i:t:BendikovSaloffCoste:2}--\ref{i:t:BendikovSaloffCoste:3}}
Assume~\ref{i:t:BendikovSaloffCoste:f} and set~$\msA^\upalpha_k\eqdef \set{f^\upalpha_k: f_k\in\msA_k}$.
Since~$\msA_k\subset \Cz(\T_k)$ we have~$\msA^\upalpha_k\subset \Cb(\T^\upalpha_k)$. 
Since~$\msA^\upalpha_k$ is not unital, we replace it by its unital extension~$\msA^{\upalpha\,\mrmu}_k$.
As in Remark~\ref{r:Unitality}, it is clear that~$\msA^{\upalpha\, \mrmu}_k$ satisfies~\ref{i:t:BendikovSaloffCoste:f} with~$\tparen{\mssL^\upalpha_k,\dom{\mssL^\upalpha_k}}$ in place of~$\tparen{\mssL_k,\dom{\mssL_k}}$ and~$\mssH^\upalpha_{k,\bullet}$ in place of~$\mssH_{k,\bullet}$.
Since~$\boldupalpha$ is an isomorphism of $L^2$-spaces,~\ref{i:t:BendikovSaloffCoste:2}--\ref{i:t:BendikovSaloffCoste:3} hold if and only so do~\ref{i:t:BendikovSaloffCoste:2}$_\boldupalpha$--\ref{i:t:BendikovSaloffCoste:3}$_\boldupalpha$.
Thus, applying the proof in the compact case shows~\ref{i:t:BendikovSaloffCoste:2}--\ref{i:t:BendikovSaloffCoste:3}.

\paragraph{Regularity vs.\ quasi-regularity}
Whereas the transfer to compactifications preserves analytical properties, it does not preserve topological properties, so that the regularity of~$\tparen{\bmssE^\boldupalpha,\dom{\bmssE^\boldupalpha}}$ does not imply that of~$\tparen{\bmssE,\dom{\bmssE}}$.
Indeed, an infinite product of locally compact (Hausdorff) spaces is locally compact if and only if cofinitely many factors are compact, so that~$\tparen{\bmssE,\dom{\bmssE}}$ may be a regular form only in the latter case.
We proved the regularity in the case that \emph{all} the factors~$M_k$ are compact.
The case when cofinitely many~$M_k$'s are compact requires only simple adjustments which are left to the reader.
Thus, we turn to show the \emph{quasi}-regularity when infinitely many factors are non-compact.

\paragraph{Proof of quasi-regularity}
It follows from the compact case that the form $\tparen{\bmssE^\boldupalpha,\dom{\bmssE^\boldupalpha}}$ is regular.
Therefore, it suffices to show that the forms $\tparen{\bmssE^\boldupalpha,\dom{\bmssE^\boldupalpha}}$ and~$\tparen{\bmssE,\dom{\bmssE}}$ are quasi-homeomorphic (Dfn.~\ref{d:QuasiHomeo}).
To this end, note that~$\boldupalpha$ is invertible on its image and that its inverse~$\boldupalpha^{-1}$ is $\T^\boldupalpha_\mrmp/\T_\mrmp$-continuous (hence Borel) since it is the product of the $\T^\upalpha_k/\T_k$-continuous maps~$\upalpha_k^{-1}$, each defined on~$\upalpha_k(M_k)$.
Thus,~$\boldupalpha$ is a homeomorphism onto its image.
Since~$\boldnu^\boldupalpha \mbfM^\boldupalpha_*=0$ we have that~$(\boldupalpha^{-1})_\pfwd\boldnu^\boldupalpha=\boldnu$, that~$\boldupalpha^*=(\boldupalpha^{-1})_*\colon L^2(\boldnu^\boldupalpha)\to L^2(\boldnu)$ is the isomorphism of $L^2$-spaces inverting~$\boldupalpha_*$, and that it intertwines~$\tparen{\bmssE^\boldupalpha,\dom{\bmssE^\boldupalpha}}$ and~$\tparen{\bmssE,\dom{\bmssE}}$.

Finally, it is straightforward that every homeomorphism intertwining two Dirichlet forms is indeed a quasi-homeomorphism, which concludes the proof of quasi-regularity.

\bigskip

For the sake of further discussion, let us however give as well a constructive proof of the quasi-regularity, by verifying \ref{i:d:QuasiReg:1}--\ref{i:d:QuasiReg:3} in Definition~\ref{d:QuasiReg}.
Since~$\mssH_{k,\bullet}$ is conservative, it is a tedious yet straightforward verification that $\upalpha_k\colon M_k\to M_k^\upalpha$ is a quasi-homeomorphism for each~$k$.
Thus,~$*_k$ is $\mssE^\upalpha_k$-polar and therefore~$\mssW^\upalpha_{k,\bullet}$-polar.
In light of~\ref{i:t:BendikovSaloffCoste:5} for the compact case, we conclude that
\begin{equation}\label{eq:PolarMAlpha}
\mbfM^\boldupalpha_* \eqdef \mbfM^\boldupalpha\setminus \boldupalpha(\mbfM)= \scup_j \tparen{ \set{*_j} \times \sprod_{k\neq j} M^\upalpha_k}
\end{equation}
is $\bmssW^\boldupalpha_\bullet$-polar, and therefore $\bmssE^\boldupalpha$-polar (in particular, $\boldnu^\boldupalpha$-negligible).

Since~$\mbfM^\boldupalpha_*$ is $\bmssE^\boldupalpha$-polar, there exists a $\T^\boldupalpha_\mrmp$-compact $\bmssE^\boldupalpha$-nest~$\seq{\mbfF^\boldupalpha_n}_n$ with~$\mbfF^\boldupalpha_n\subset \boldupalpha(\mbfM)$ for every~$n$.
Define~$\mbfF_n\eqdef \boldupalpha^{-1}(\mbfF^\boldupalpha_n)$.
As noted above,~$\boldupalpha^{-1}$ is $\T^\boldupalpha_\mrmp/\T_\mrmp$-continuous, thus~$\mbfF_n$ is $\T_\mrmp$-compact for every~$n$.
Since the forms~$\tparen{\bmssE^\boldupalpha,\dom{\bmssE^\boldupalpha}}$ and~$\tparen{\bmssE,\dom{\bmssE}}$ are intertwined,~$\seq{\mbfF_n}_n$ is a $\T_\mrmp$-compact $\bmssE$-nest, which shows \ref{i:d:QuasiReg:1}.
As noted above, $\Cyl{}{}$~is a subalgebra of~$\Cb(\T_\mrmp)$, which, together with the $\bmssE^{1/2}_1$-density of~$\Cyl{}{}$ in~$\dom{\bmssE}$, verifies~\ref{i:d:QuasiReg:2}.
Finally, recall that~$\msA_k$ is $\Cz$-dense in~$\Cz(\T_k)$ for every~$k$.
Since~$(M_k,\T_k)$ is second countable,~$\Cz(\T_k)$ is a separable Banach space, thus~$\msA_k$ admits a countable subset~$D_k$ $\Cz$-dense in~$\Cz$.
It follows that~$D_k$ is point-separating on~$M_k$.
Thus,~$\cup_j \tparen{\set{D_j}\times \sprod_{k\neq j} D_k}$ is a countable set point-separating on~$\mbfM$, which verifies~\ref{i:d:QuasiReg:3}.

\boldparagraph{Proof of~\ref{i:t:BendikovSaloffCoste:5}--\ref{i:t:BendikovSaloffCoste:6}}
By~\ref{i:t:BendikovSaloffCoste:4} and~\cite[Thm.~IV.3.5]{MaRoe92} there exists a $\boldnu$-special standard process~$\bmssW$ properly associated with~$\tparen{\bmssE,\dom{\bmssE}}$.
Since each~$\mssH_{k,\bullet}$ is conservative,~$\bmssH_\bullet$ too is conservative, thus~$\bmssW$ is a Hunt process.
The identification with~$\seq{\mssW_k}_k$ is a consequence of~\ref{i:t:BendikovSaloffCoste:3}.
This concludes the proof of~\ref{i:t:BendikovSaloffCoste:5}.

Since~$\boldnu \mbfM=1$, since~$\boldnu$ has full $\T_\mrmp$-support, and since every~$u\in\Cyl{}{}$ is $\T_\mrmp$-continuous, Assertion~\ref{i:t:BendikovSaloffCoste:6} is a consequence of~\ref{i:t:BendikovSaloffCoste:2} by~\cite[Thm.~3.5]{AlbRoe95}.
Since~$\Cyl{}{}\subset\dom{\bmssL}$ is an algebra, the computation of the predictable quadratic variation is standard.
\end{proof}

\begin{corollary}\label{c:BSCFeller}
In the setting of Theorem~\ref{t:BendikovSaloffCoste} and under assumptions~\ref{i:t:BendikovSaloffCoste:a}--\ref{i:t:BendikovSaloffCoste:f} there, suppose further that
\begin{enumerate}[$(a)$]\setcounter{enumi}{4}
\item $(t,x,y)\mapsto\mssh_{k,t}(x,y)$ is continuous on~$\R^+\times M_k^\tym{2}$ for every~$k$, and~$\mssh_{k,t}\in\Cb(\T_k^\tym{2})$ for every~$t>0$ and every~$k$.
\end{enumerate}

Then,
\begin{enumerate}[$(i)$]\setcounter{enumi}{7}
\item\label{i:c:BSCFeller:1} $(t,\mbfx,\mbfy)\mapsto \bmssh_t(\mbfx,\mbfy)$ is continuous on~$\R^+\times \mbfM^\tym{2}$ and~$\bmssh_t\in\Cb(\T_\mrmp^\tym{2})$ for every~$t>0$.
\end{enumerate}

In particular, there exists a unique $\boldnu$-version of~$\bmssW$ with space-time continuous transition kernel.
\begin{proof}
\ref{i:c:BSCFeller:1} is~\cite[Lem.~4.3]{BenSaC97}.
The uniqueness of~$\bmssW$ follows since~$\boldnu$ is Radon and has full $\T_\mrmp$-support.
\end{proof}
\end{corollary}

As a further characterization of the semigroup~$\bmssT_\bullet$, we identify it with the infinite tensor product of the semigroups~$\mssH_{k,\bullet}$.
Indeed, let~$\ev$ be as in~\eqref{eq:l:Guichardet:0}.

\begin{corollary}\label{c:Gill}
In the setting of Theorem~\ref{t:BendikovSaloffCoste} and under assumptions~\ref{i:t:BendikovSaloffCoste:a}--\ref{i:t:BendikovSaloffCoste:f} there,
\[
\bmssT_\bullet = \ev \circ (\shotimes_k \mssH_{k,\bullet}) \circ \ev^{-1} \fstop
\]
\begin{proof}
Let~$(\bmssL^\otym{},\dom{\bmssL^\otym{}})$ be the generator of~$\shotimes_k \mssH_{k,\bullet}$ on~$\shotimes_k^{\car} L^2(\nu_k)$. By~\cite[Thm.~4.10]{Gil74} we have $\sotimes_k^{\car}\dom{\mssL_k}\subset \dom{\bmssL^\otym{}}$ and
\[
\bmssL^\otym{} (\sotimes_k f_k) = \ssum_k (\mssL_k f_k) \sotimes_{i\neq k} f_i\comma \qquad \sotimes_k f_k\in \sotimes_k^{\car}\dom{\mssL_k} \fstop
\]
As a consequence,~$\ev^{-1}\circ \bmssL^\otym{} \circ \ev =\tilde\bmssL$ on~$\ev\tparen{ \sotimes_k^{\car} \dom{\mssL_k}} \supset \Cyl{}{}$.
Since, as shown in the proof of Theorem~\ref{t:BendikovSaloffCoste}\ref{i:t:BendikovSaloffCoste:2}, the operator~$(\tilde\bmssL,\Cyl{}{})$ is essentially self-adjoint on~$L^2(\boldnu)$, the operator~$\tparen{\bmssL^\otym{},\sotimes_k^{\car}\dom{\mssL_k}}$ is essentially self-adjoint on~$\shotimes_k^{\car} L^2(\nu_k)$, and~$(\tilde\bmssL,\dom{\tilde\bmssL})\circ\ev = \ev \circ (\bmssL^\otym{},\dom{\bmssL^\otym{}})$.
It follows that the corresponding semigroups too are intertwined by~$\ev$, i.e.~$\bmssH_\bullet \circ \ev= \ev \circ (\shotimes_k \mssH_{k,\bullet})$, and the conclusion follows by the identification~$\bmssH_\bullet=\bmssT_\bullet$ in Theorem~\ref{t:BendikovSaloffCoste}\ref{i:t:BendikovSaloffCoste:3}.
\end{proof}
\end{corollary}

\subsubsection{The strongly local case}
We define a bilinear form~$(\bmssGamma,\Cyl{}{})$ as the bilinear extension to~$\Cyl{}{}$ of the bilinear form on elementary cylinder functions defined by polarization of the quadratic form
\[
\bmssGamma(u)= \ssum_k \mssGamma_k(f_k)(\emparg_k)\, \sprod_{i\neq k} f_i(\emparg_i)\comma \qquad u=\sprod_k f_k(\emparg_k)\fstop
\]

The next lemma is the analogue for carr\'e du champ operators of Lemma~\ref{l:GeneratorWellDefined}.

\begin{lemma}
The carr\'e du champ operator~$\bmssGamma$ is well-defined on~$\Cyl{}{}$ and
\[
\bmssGamma(u)= \ssum_{a\in A} \mssGamma_a(f_a)(\emparg_a)\, \sprod_{\substack{i\in A\\ i\neq a}} f_i(\emparg_i)\comma \qquad u=\sprod_{a\in A}f_a(\emparg_a)\fstop
\]
\end{lemma}

\begin{corollary}\label{c:BSCStronglyLocal}
In the setting of Theorem~\ref{t:BendikovSaloffCoste} and under assumptions~\ref{i:t:BendikovSaloffCoste:a}--\ref{i:t:BendikovSaloffCoste:f} there, further define the following assumptions:
\begin{enumerate}[$(a)$]\setcounter{enumi}{5}
\item\label{i:c:BSCStronglyLocal:A} $\tparen{\mssE_k,\dom{\mssE_k}}$ is strongly local for every~$k$;
\item\label{i:c:BSCStronglyLocal:B} $\tparen{\mssE_k,\dom{\mssE_k}}$ admits carr\'e du champ operator~$\tparen{\mssGamma_k,\dom{\mssGamma_k}}$ for every~$k$.
\end{enumerate}
Then, in addition to all the conclusions of Theorem~\ref{t:BendikovSaloffCoste},
\begin{enumerate}[$(i)$]
\item\label{i:c:BSCStronglyLocal:1} if~\ref{i:c:BSCStronglyLocal:A} holds, then~$\paren{\bmssE,\dom{\bmssE}}$ is strongly local;
\item\label{i:c:BSCStronglyLocal:2} if~\ref{i:c:BSCStronglyLocal:B} holds, then~$\paren{\bmssE,\dom{\bmssE}}$ admits carr\'e du champ operator~$\tparen{\bmssGamma,\dom{\bmssGamma}}$ the closure of~$(\bmssGamma,\Cyl{}{})$;

\item\label{i:c:BSCStronglyLocal:3} if~\ref{i:c:BSCStronglyLocal:A} and~\ref{i:c:BSCStronglyLocal:B} hold, then the martingale~$M^{\class{u}}_\bullet$ in~\eqref{eq:t:BendikovSaloffCoste:1} has quadratic variation~$\tquadvar{M^{\class{u}}}_\bullet=\tsharpb{M^{\class{u}}}_\bullet$ given by the right-hand side of~\eqref{eq:t:BendikovSaloffCoste:2} for every~$u\in\Cyl{}{}$.
\end{enumerate}
\begin{proof}
\ref{i:c:BSCStronglyLocal:1} Strong locality (in the generalized sense) is invariant under quasi-homeo\-morphisms by (the proof of)~\cite[Thm.~5.1]{Kuw98}.
Thus, applying the quasi-homeo\-morphism~$\boldupalpha$ constructed in the proof of Theorem~\ref{t:BendikovSaloffCoste}, it suffices to prove the assertion in the case when~$M_k$ is compact for every~$k$.
In this case, the form~$\tparen{\bmssE,\dom{\bmssE}}$ is additionally regular, and \emph{strong locality} in the sense of~\cite[p.~6, $(\msE.7)$]{FukOshTak11} coincides with \emph{locality} in the sense of~\cite[Dfn.~I.5.1.2]{BouHir91}, as noted, e.g., in~\cite[p.~76]{tElRobSikZhu06}.
Thus, the assertion holds by~\cite[Prop.~V.2.2.2]{BouHir91} in light of Remark~\ref{r:BHTensor}.

\ref{i:c:BSCStronglyLocal:2} holds by~\cite[Prop.~V.2.2.2]{BouHir91} in light of Remark~\ref{r:BHTensor}.

\ref{i:c:BSCStronglyLocal:3} is standard, cf.\ e.g.~\cite[Thm.~A.3.17 and Eqn.~(A.3.17)]{FukOshTak11}.
\end{proof}
\end{corollary}

\subsection{Conformal rescaling}\label{ss:Rescaling}
We now apply Theorem~\ref{t:BendikovSaloffCoste} to a particular case of our interest.

Throughout this section, let~$(M,\T)$ be a second-countable locally compact Hausdorff topological space, and~$\nu$ be an \emph{atomless} Borel probability measure on~$M$ with full $\T$-support.
Further let~$\mssH_{\bullet}$ be a self-adjoint Markov semigroup on~$L^2(\nu)$, and denote by~$\tparen{\mssL,\dom{\mssL}}$ its generator and by~$\tparen{\mssE,\dom{\mssE}}$ the corresponding Dirichlet form.
Since~$\mssH_\bullet$ is self-adjoint on~$L^2(\nu)$, so is~$\tparen{\mssL,\dom{\mssL}}$.

\begin{assumption}[Setting]\label{ass:Setting}
We assume that
\begin{enumerate}[$(a)$]
\item\label{i:ass:Setting:a} $\mssH_\bullet$ is \emph{conservative} and \emph{irreducible}; 
\item\label{i:ass:Setting:b} $\mssH_\bullet$ is \emph{ultracontractive}, i.e.\ for some constant~$c(t)>0$
\[
\norm{\mssH_t}_{L^1\to L^\infty} \leq c(t)\comma \qquad t>0 \semicolon
\]
\item\label{i:ass:Setting:c} there exists a subalgebra~$\msA$ of~$\Cz(\T)$ satisfying
\begin{enumerate}[$(c_1)$]
\item\label{i:ass:Setting:c1} $\msA\cap\Cc(\T)$ is a core for~$\tparen{\mssE,\dom{\mssE}}$;
\item\label{i:ass:Setting:c2} $\msA\subset\dom{\mssL}$ and~$\msA$ is $\mssL$-invariant, i.e.~$\mssL\msA\subset\msA$;

\item\label{i:ass:Setting:c3} $\msA$ is $\mssH_\bullet$-invariant, i.e.
\[
\mssH_t\msA\subset\msA\comma\qquad t>0\fstop
\]
\end{enumerate}
\end{enumerate}
\end{assumption}

\begin{remark}[On condition~\ref{i:ass:Setting:b}]
Our definition of ultracontractivity is the one by E.B.~Davies, e.g.~\cite[p.~59]{Dav89}.
We do not require any asymptotics of~$c(t)$, neither as~$t\to 0$, nor as~$t\to\infty$.
In fact, we have~$\lim_{t\to\infty} c(t)=1$ (rather than~$0$) since~$\nu$ is a probability measure.
\end{remark}

For further comments about the necessity of ultracontractivity, see Remark~\ref{r:NecessityUC} below.

\begin{remark}[On condition~\ref{i:ass:Setting:c2}]
Out of the above assumptions,~\ref{i:ass:Setting:c2} is the only one not justified by the heuristic arguments presented in the Introduction. 
Indeed, this technical assumption allows us to identify a specific $\nu$-representative of each function~$\mssL f$, with~$f\in\msA$, namely the $\T$-continuous representative, since~$\mssL\msA\subset\msA\subset\Cz(\T)$.
As such, it appears to be unnecessarily strong, and we expect that it could be removed, if however at great technical cost.
We note that, in the setting of manifolds, when~$\mssL$ is a second-order linear elliptic operator in divergence form, the condition~\ref{i:ass:Setting:c2} forces the coefficients of the operator to be continuous and with $L^2$-integrable derivatives.
\end{remark}

\paragraph{Various properties}
Let us list here some consequences of Assumption~\ref{ass:Setting} which we will make use of in the following.
Since~$\mssH_\bullet$ is ultracontractive and~$\nu$ is finite, $\mssH_\bullet$~is Hilbert--Schmidt, and both~$\mssH_\bullet$ and~$\tparen{\mssL,\dom{\mssL}}$ have purely discrete spectrum, e.g.~\cite[Thm.~2.1.4]{Dav89} and~\cite[Cor.~6.4.1]{BakGenLed14}.
Thus, since~$\mssH_\bullet$ is conservative,
\[
\text{$\tparen{-\mssL,\dom{\mssL}}$ has spectral gap~$\lambda>0$.}
\]
Furthermore, again since~$\mssH_\bullet$ is ultracontractive, it is represented by a kernel density~$\mssh_\bullet\in L^\infty(\nu^\otym{2})$.
In particular, there exists~$C>0$ such that
\begin{equation}\label{eq:HKupperbound}
\nu^\otym{2}\text{-}\esssup \mssh_1 \leq C \fstop
\end{equation}

Since~$(M,\T)$ is a second-countable locally compact Hausdorff space, it is Polish.
Thus, since~$\nu$ is a finite Borel measure, it is a Radon measure. 
Since~$\msA\cap\Cc(\T)$ is a core for~$\tparen{\mssE,\dom{\mssE}}$, it is $L^2$-dense in~$L^2(\nu)$.
Since~$\msA$ is $L^2$-dense in $L^2(\nu)$, since~$\msA\subset \dom{\mssL}$, and since~$\msA$ is~$\mssH_\bullet$-invariant,~$(\mssL,\msA)$ is essentially self-adjoint on~$L^2(\nu)$ by~\cite[Thm.~X.49]{ReeSim75}.

Since~$\msA\cap \Cc(\T)$ is a core for~$\tparen{\mssE,\dom{\mssE}}$, the latter is a regular Dirichlet form.
Thus, there exists a (unique up to $\nu$-equivalence) special standard process
\begin{equation}\label{eq:BaseProcess}
\mssW\eqdef \seq{\Omega,\msF,\msF_\bullet,X_\bullet, \seq{P_x}_{x\in M},\zeta}
\end{equation}
properly associated with~$\tparen{\mssE,\dom{\mssE}}$.
Since~$\mssH_\bullet$ is conservative,~$\zeta\equiv \infty$ a.s.\ and~$\mssW$ is a Hunt process. (We will henceforth omit~$\zeta$ from the notation.)

\subsubsection{The mark space}
Let~$\mbfI\eqdef \mbbI^\tym{\infty}$ and~$\boldDelta\eqdef\set{\mbfs\in \mbfI : \norm{\mbfs}_{\ell^1}=1}$ and set
\begin{equation}\label{eq:DeltaTau}
\begin{aligned}
\mbfT\eqdef&\set{\mbfs\eqdef \seq{s_i}_i \in \boldDelta : s_i \geq s_{i+1} \geq 0 \text{~~for all~}i} \comma
\\
\mbfT_*\eqdef&\set{\mbfs\eqdef \seq{s_i}_i \in \boldDelta : s_i \geq s_{i+1} > 0 \text{~~for all~}i} \comma
\\
\mbfT_\circ\eqdef&\set{\mbfs\eqdef \seq{s_i}_i \in \boldDelta : s_i > s_{i+1} > 0 \text{~~for all~}i} \fstop
\end{aligned}
\end{equation}
We endow~$\mbfI$ with the product topology and~$\mbfT_\circ\subset\mbfT_*\subset \mbfT\subset  \boldDelta\subset \mbfI$ with the corresponding trace topologies.

Intuitively, we regard~$\mbfT$ as the space of marks for point sequences in~$\mbfM$.

\subsubsection{The product space}
Analogously to the previous section, let~$\mbfM\eqdef M^\tym{\infty}$ and set
\begin{equation}\label{eq:DiagonalBoundary}
\mbfM_\circ\eqdef \set{\mbfx\eqdef \seq{x_i}_i \in \mbfM : x_i\neq x_j \text{~~for all~} i\neq j} \fstop
\end{equation}
We endow~$\mbfM$ with the product topology~$\T_\mrmp$.
Since~$\nu$ is an atomless Borel probability with full $\T$-support,~$\boldnu\eqdef \nu^\otym{\infty}$, is an atomless probability with full $\T_\mrmp$-support concentrated on~$\mbfM_\circ$.

In order to apply Theorem~\ref{t:BendikovSaloffCoste}, set~$M_k\eqdef M$, $\nu_k\eqdef\nu$, and~$\msA_k\eqdef \msA$.
For each~$\mbfs\in\mbfT$, further set~$\mssL^\mbfs_k\eqdef s_k^{-1}\mssL$.
Conventionally, $1/\infty\eqdef 0$ and~$0\,\mssL\eqdef \zero$ is the null operator on~$L^2(\nu)$.
Respectively denote by~$\tparen{\mssE^\mbfs_k,\dom{\mssE^\mbfs_k}}$ and~$\mssH^\mbfs_{k,\bullet}$ the associated Dirichlet form and semigroup.
We respectively denote by~$\tparen{\bmssE^\mbfs,\dom{\bmssE^\mbfs}}$, by~$\tparen{\bmssL^\mbfs,\dom{\bmssL^\mbfs}}$, and by~$\bmssH^\mbfs_\bullet$, the Dirichlet form, the generator, and the semigroup appearing in Theorem~\ref{t:BendikovSaloffCoste} for the above choice.

The following statements hold true for every~$\mbfs\in\mbfT$. 
For simplicity, we present the proofs for~$\mbfs\in\mbfT_*$.
If~$\mbfs\in\mbfT\setminus\mbfT_*$, then~$s_i=0$ eventually in~$i$, and the statements reduce to the case of a finite product in light of the above conventions.

\begin{proposition}\label{p:BendikovSaloffCosteRescaling}
For each~$\mbfs\in\mbfT$, all the conclusions in Theorem~\ref{t:BendikovSaloffCoste} hold.
Furthermore, the semigroup~$\bmssH^\mbfs_\bullet$ is ultracontractive and~$\tparen{\bmssL^\mbfs,\dom{\bmssL^\mbfs}}$ has spectral gap~$\lambda/s_1$.
\end{proposition}

\begin{proof}
It is readily verified that~$\mssh_{k,t}=\mssh_{t/s_k}$ for every~$t>0$.
Thus, the assumption in Theorem~\ref{t:BendikovSaloffCoste}\ref{i:t:BendikovSaloffCoste:d} is verified choosing~$c\eqdef 1$,~$\eps_k\eqdef s_k/\lambda$, and~$C$ as in~\eqref{eq:HKupperbound}.

Since~$\mbfs\in\mbfT_*$, there exists~$\lim_k k s_k=0$ by the Abel--Olivier--Pringsheim criterion.
In particular,~$\liminf_k k s_k<\infty$ and thus~$\limsup_k e^{-2\lambda t/(k s_k)}<1$ since~$\lambda t>0$ for every~$t>0$.
Thus,~$\sum_k^\infty e^{-2(\lambda/s_k) t}<\infty$ for every~$t>0$ by the root test.
Since~$\lambda_k=\lambda/s_k$, this verifies the assumption in Theorem~\ref{t:BendikovSaloffCoste}\ref{i:t:BendikovSaloffCoste:e}.
The remaining assumptions are readily verified.

The ultracontractivity of~$\bmssH^\mbfs_\bullet$ follows from the boundedness of the heat kernel~$\bmssh^\mbfs_\bullet$ in Theorem~\ref{t:BendikovSaloffCoste}\ref{i:t:BendikovSaloffCoste:1}.
As for the spectral gap, by Corollary~\ref{c:Gill} it suffices to compute the spectral gap of~$\bmssL^\otym{}$.
For the latter we have~$\sigma(\bmssL^\otym{})=\sum_k^\infty \sigma(\mssL_k) = \sum_k^\infty s_k^{-1}\sigma(\mssL)$ so that the smallest positive eigenvalue of~$-\bmssL$ (i.e.\ of~$-\bmssL^\otym{}$) is~$\lambda/s_1$.
\end{proof}

\begin{remark}[On the necessity of ultracontractivity]\label{r:NecessityUC}
The \emph{ultracontractivity} Assumption~\ref{ass:Setting}\ref{i:ass:Setting:b} is almost necessary for Proposition~\ref{p:BendikovSaloffCosteRescaling} to hold.
Indeed, ultracontractivity is equivalent to the heat-kernel estimate
\[
\nu^\otym{2}\text{-}\esssup \mssh_t \leq c(t) \comma \qquad t>0\comma
\]
for some constant~$c(t)>0$; see, e.g.,~\cite[Lem.~2.1.2, p.~59]{Dav89}.
This shows that the heat-kernel estimate~\eqref{eq:HKupperbound} is in turn equivalent to~$\mssh_\bullet$ being \emph{eventually} ultracontractive, or ---equivalently--- such that, for some~$t_0>0$,
\[
\nu^\otym{2}\text{-}\esssup \mssh_t \leq c(t) \comma \qquad t> t_0 \fstop
\]
In the setting of Proposition~\ref{p:BendikovSaloffCosteRescaling}, the bound~\eqref{eq:HKupperbound} translates into assumption~\ref{i:t:BendikovSaloffCoste:d} in Theorem~\ref{t:BendikovSaloffCoste}, the necessity of which was noted in Remark~\ref{r:NecessityBendikovSaloffCoste}.
\end{remark}

For each~$\mbfs\in\mbfT$, Proposition~\ref{p:BendikovSaloffCosteRescaling} guarantees the existence of a  (unique up to $\boldnu$-equivalence) Hunt process
\begin{equation}\label{eq:BSCRescalingProcess}
\bmssW^\mbfs\eqdef\seq{\Omega^\mbfs,\msF^\mbfs,\msF^\mbfs_\bullet,\mbfX^\mbfs_\bullet, \seq{P^\mbfs_\mbfx}_{\mbfx\in \mbfM}} \comma
\end{equation}
properly associated with the Dirichlet form~$\tparen{\bmssE^\mbfs,\dom{\bmssE^\mbfs}}$.
We write~$X^\mbfs_{i,\bullet}$ for the $i^\text{th}$ coordinate of~$\mbfX^\mbfs_\bullet$.

\begin{remark}[Standard stochastic basis for~$\bmssW^\mbfs$]
In light of Proposition~\ref{p:BendikovSaloffCosteRescaling} and Theorem~\ref{t:BendikovSaloffCoste}\ref{i:t:BendikovSaloffCoste:5}, the process~$\bmssW^\mbfs$ may be represented on the probability space
\begin{equation}\label{eq:r:StandardStochB:1}
\Omega^\mbfs=\boldOmega\eqdef \Omega^\tym{\infty}\comma
\end{equation}
in which case the path probabilities~$P^\mbfs_\mbfx$ may in fact be chosen independently of~$\mbfs$ and satisfying~$P^\mbfs_\mbfx =P_\mbfx\eqdef \sotimes_k P_{x_k}$.
In the following, we shall always assume that~$\bmssW^\mbfs$ is given on this stochastic basis.
\end{remark}

\subsubsection{Kernel continuity}
Let us now present a necessary and sufficient assumption for the infinite-product kernel~$\bmssh_\bullet$ to be continuous, translating into the present setting the results obtained in Corollary~\ref{c:BSCFeller} for a sequence of spaces~$\seq{M_k}_k$.

\begin{assumption}\label{ass:ContinuityH}
$(t,x,y)\mapsto\mssh_t(x,y)$ is continuous on~$\R^+\times M^\tym{2}$ and~$\mssh_t\in\Cb(\T^\tym{2})$ for every~$t>0$.
\end{assumption}

\begin{corollary}\label{c:Wellposedness}
Suppose that Assumptions~\ref{ass:Setting} and~\ref{ass:ContinuityH} are satisfied.
Then, for each~$\mbfs\in\mbfT$, the system
\begin{equation}\label{eq:c:Wellposedness:0}
\diff X^{\mbfs}_{i,t} =\diff X_{t/s_i}\comma \quad X^\mbfs_{i,0}=x_i\comma
\end{equation}
is well-posed for \emph{every}~$\seq{x_i}_i\subset M$.
(Conventionally,~$X^\mbfs_{i,t}=X^\mbfs_{i,0}$ for all~$t>0$ whenever~$s_i=0$.)
\end{corollary}

\begin{proof}
The validity of~\eqref{eq:c:Wellposedness:0} for~$\bmssE^\mbfs$-q.e.~$\mbfx\in\mbfM$ is an immediate consequence of the identification in Theorem~\ref{t:BendikovSaloffCoste}\ref{i:t:BendikovSaloffCoste:5}.
The uniqueness up to $\boldnu$-equivalence holds by Theorem~\ref{t:BendikovSaloffCoste}\ref{i:t:BendikovSaloffCoste:6}.
The extension to every starting point holds as a consequence of Corollary~\ref{c:BSCFeller}.
\end{proof}

\subsubsection{Collisions}\label{sss:Collisions}
In this section we give a sufficient (and practically necessary) condition under which the first collision time~$\T^\mbfs_c$ in~\eqref{eq:Intro:FCT} is a.s.\ vanishing.

Let~$\Cap$ be the Choquet 1-capacity associated to~$\tparen{\mssE,\dom{\mssE}}$.
Since~$\tparen{\mssE,\dom{\mssE}}$ is $\T$-regular, it is not difficult to show that the map~$x\mapsto \Cap(\set{x})$ is $\nu$-measurable, and we set
\begin{equation}\label{eq:NonPolarityConstant}
\kappa_\mssE\eqdef \int \Cap(\set{x}) \diff\nu(x) \quad \in [0,1]\fstop
\end{equation}

\begin{assumption}[Non-polarity of points]
Under Assumption~\ref{ass:Setting}, further suppose that
\begin{equation}\tag{\mathsc{npp}}\label{eq:QNPP}
\kappa_\mssE>0 \fstop
\end{equation}
\end{assumption}

It is clear that~\eqref{eq:QNPP} is equivalent to the existence of some set~$A\subset X$ of positive $\nu$-measure such that each~$x\in A$ has positive capacity.

\begin{proposition}\label{p:Collisions}
Assume~\eqref{eq:QNPP}. Then,~$\tau^\mbfs_c\equiv 0$ $P^\mbfs_{\boldnu}$-a.s.\ for every~$\mbfs\in\mbfT_*$.

\begin{proof}
For simplicity of notation, throughout the proof let~$a_i\eqdef s_i^{-1}$ for every~$i\in\N_1$. 
It is clear from~\eqref{eq:Intro:FCT} that, for every~$i,j$ with~$i\neq j$,
\begin{align*}
\tau^\mbfs_c\leq \tau^{s_i,s_j}_c\eqdef&\ \sup\set{t>0: X^\mbfs_{i,r}\neq X^\mbfs_{j,r} \text{ for all } r\in(0, t)}
\\
=&\ \inf\set{t>0 : (X^\mbfs_{i,t}, X^\mbfs_{j,t}) \in \Delta M } \fstop
\end{align*}
In light of Proposition~\ref{p:BendikovSaloffCosteRescaling} and Theorem~\ref{t:BendikovSaloffCoste}\ref{i:t:BendikovSaloffCoste:5}, we further have~$X^\mbfs_{i,t} \overset{d}{=} X_{a_i t}$ for every~$i\in \N_1$, and thus~$\tau^{s_i,s_j}_c \overset{d}{=} \tau^{a_i,a_j}_{\Delta M}$ is the hitting time of the diagonal for the process~$\mssW^{a_i,a_j}$ properly associated with the regular product Dirichlet form~$\tparen{\mssE^{a,b},\dom{\mssE^{a,b}}}$ in~\S\ref{sss:AppCapacityRescaling} with~$a=a_i$ and~$b=a_j$.

Thus, it suffices to show that, for every~$\mbfs\in\mbfT_*$
\[
\limsup_{\substack{i,j\to\infty\\ i\neq j}} \tau^{a_i,a_j}_{\Delta M} = 0 \as{P_{\nu^\otym{2}}} \fstop
\]

We show the following (strictly) stronger statement.

\nparagraph{Claim: for every~$\mbfs\in\mbfT_*$,
\begin{equation}\label{eq:p:Collisions:0}
\lim_{\substack{i,j\to\infty\\ i\neq j}} E_{\nu^\otym{2}}\braket{\exp\tbraket{-\tau^{a_i,a_j}_{\Delta M} }} = 1 \fstop
\end{equation}
}
Since the product form~$\tparen{\mssE^{a_i,a_j},\dom{\mssE^{a_i,a_j}}}$ is conservative,
\begin{equation}\label{eq:p:Collisions:1}
E_{\nu^\otym{2}}\braket{\exp\tbraket{-\tau^{a_i,a_j}_{\Delta M}}}=\Cap_{a_i,a_j}(\Delta M)\fstop
\end{equation}
By Lemma~\ref{l:ProductPointHitting} with~$a=a_i$ and~$b=a_j$ we further have, for~$\kappa_\mssE$ as in~\eqref{eq:NonPolarityConstant},
\begin{equation}\label{eq:p:Collisions:2}
\Cap_{a_i,a_j}(\Delta M) \geq \kappa_\mssE^{1/(a_i\wedge a_j)} \fstop
\end{equation}
Combining~\eqref{eq:p:Collisions:1} and~\eqref{eq:p:Collisions:2},
\[
\lim_{\substack{i,j\to\infty\\ i\neq j}} E_{\nu^\otym{2}}\braket{\exp\tbraket{-\tau^{a_i,a_j}_{\Delta M} }} \geq \lim_{\substack{i,j\to\infty\\ i\neq j}} \kappa_\mssE^{s_i\vee s_j} \fstop
\]
Since~$\mbfs\in\mbfT_*$, we can choose~$s_i\vee s_j$ arbitrarily close to~$0$, and~\eqref{eq:p:Collisions:0} follows.
\end{proof}
\end{proposition}

Proposition~\ref{p:Collisions} has the following immediate consequence.

\begin{corollary}\label{c:Collisions}
Assume~\eqref{eq:QNPP} holds. Then, for every~$\mbfs\in \mbfT_*$,
\begin{enumerate}[$(i)$]
\item the killed process~$\mssW^\mbfs_{\tau^\mbfs_c}$ is the null process;
\item the set~$\mbfM_\circ$ in~\eqref{eq:DiagonalBoundary} is $\bmssE^\mbfs$-polar.
\end{enumerate}
\end{corollary}

In Proposition~\ref{p:AndiPolar} below we will show an almost converse to Proposition~\ref{p:Collisions}: if all points in~$M$ are $\mssE$-polar (in a uniform way), then~$\tau^\mbfs_c\equiv \infty$ a.s.
See Remark~\ref{r:NoCollisions} for further comments.

\subsection{The extended space}\label{sss.ExtendedSpace}
So far we have considered massive particle systems with deterministically assigned masses. Here, we discuss the theory of massive particle systems with randomized masses.
Set
\[
\widehat\mbfM\eqdef\mbfT\times\mbfM\qquad \text{and} \qquad \widehat\mbfM_\circ\eqdef \mbfT_\circ\times\mbfM_\circ \fstop
\]
We endow~$\widehat\mbfM$ with the product topology~$\T_\mrmp$, and each of its subspaces with the corresponding trace topology.
Further let~$\msP(\mbfT)$ be the set of all Borel probability measures on~$\mbfT$.
One important example of such a measure is the \emph{Poisson--Dirichlet distribution}.
We recall its construction after P.~Donnelly and G.~Grimmet~\cite{DonGri93}.

\begin{example}[The Poisson--Dirichlet measure on~$\mbfT$]\label{e:PoissonDirichlet}
Let~$\beta>0$, denote by~$\Beta_\beta$ the Beta distribution of parameters~$1$ and~$\beta$, viz.
\[
\diff\mssBeta_\beta(r)\eqdef \beta(1-r)^{\beta-1}\diff r \comma
\]
and by~$\bmssBeta_\beta\eqdef \sotimes^\infty \Beta_\beta$ the infinite product measure of~$\Beta_\beta$ on~$\mbfI$.
For~$\mbfr\in\mbfI$ let~$\boldUpsilon(\mbfr)$ be the vector of its entries in non-increasing order, and denote by~$\boldUpsilon\colon \mbfI\to\mbfT$ the reordering map.
Further let~$\boldLambda\colon \mbfI\to\boldDelta$ be defined by
\begin{gather*}
\Lambda_1(r_1)\eqdef r_1\comma \qquad \Lambda_i(r_1,\dotsc, r_i)\eqdef r_i \prod_k^{i-1} (1-r_k)\comma
\\
\boldLambda\colon \mbfr \longmapsto \seq{\Lambda_1(r_1), \Lambda_2(r_1,r_2), \dotsc} \fstop
\end{gather*}
The \emph{Poisson--Dirichlet measure~$\Pi_\beta$ with parameter~$\beta$} on~$\mbfT$, concentrated on~$\mbfT_\circ$ is
\[
\Pi_\beta\eqdef (\boldUpsilon\circ\boldLambda)_\pfwd \bmssBeta_\beta \fstop
\]
\end{example}

For a fixed~$\pi\in\msP(\mbfT)$ we define the Borel probability measure~$\widehat\boldnu_\pi\eqdef \pi\otimes\boldnu$ on~$\widehat\mbfM$.
We denote by~$\mbfu,\mbfv$, etc., functions in~$\Cb(\mbfT)\otimes \Cyl{}{}$ of the form
\[
\mbfu\eqdef \sum_i^m \phi_i\otimes u_i\comma \mbfv_j\eqdef \sum_j^n \psi_j\otimes v_j\comma \qquad \phi_i,\psi_j\in\Cb(\mbfT)\comma u_i,v_j\in \Cyl{}{}\fstop
\]

The next result fully describes $\pi$-mixtures of the marked-particle systems constructed in Proposition~\ref{p:BendikovSaloffCosteRescaling}.

\begin{theorem}\label{t:DirInt}
Fix~$\pi\in\msP(\mbfT)$, and let~$\mbbV_\pi$ be any dense linear subspace of~$L^2(\pi)$.
The quadratic form~$\tparen{\widehat\bmssE_\pi,\mbbV_\pi\otimes \Cyl{}{}}$ defined by
\begin{align*}
\widehat\bmssE_\pi(\mbfu,\mbfv)\eqdef \sum_{i,j}^{m,n} \int_{\mbfT} \phi_i(\mbfs)\, \psi_j(\mbfs)\, \bmssE^\mbfs(u_i,v_j)\diff\pi(\mbfs)\comma \qquad \mbfu,\mbfv\in \mbbV_\pi \otimes \Cyl{}{}\comma 
\end{align*}
is closable on~$L^2(\widehat\boldnu_\pi)$ and its closure~$\tparen{\widehat\bmssE_\pi,\dom{\widehat\bmssE_\pi}}$ is a Dirichlet form independent of~$\mbbV_\pi$.

Furthermore:
\begin{enumerate}[$(i)$]
\item\label{i:t:DirInt:1} the $L^2(\widehat\boldnu_\pi)$-semigroup of~$\tparen{\widehat\bmssE_\pi,\dom{\widehat\bmssE_\pi}}$ is the unique $L^2(\widehat\boldnu_\pi)$-bounded linear extension~$\widehat\bmssH^\pi_\bullet$ of
\[
(\widehat\bmssH_\bullet\mbfu)(\mbfs,\mbfx) \eqdef \tparen{(\Id\otimes \bmssH^\mbfs_\bullet) \mbfu}(\mbfs,\mbfx) = \sum_i^m \phi_i(\mbfs)\, (\bmssH^\mbfs_\bullet u_i)(\mbfx) \quad \text{on}\quad \mbbV_\pi \otimes \Cyl{}{} \semicolon
\]
\item\label{i:t:DirInt:2} the generator~$\tparen{\widehat\bmssL_\pi,\dom{\widehat\bmssL_\pi}}$ of~$\tparen{\widehat\bmssE_\pi,\dom{\widehat\bmssE_\pi}}$ is the unique $L^2(\widehat\boldnu_\pi)$-self-adjoint (i.e.\ the $L^2(\widehat\boldnu_\pi)$-Friedrichs) extension of the operator~$(\widehat\bmssL,\mbbV_\pi\otimes\Cyl{}{})$ defined by
\begin{equation}\label{eq:t:DirInt:0.1}
(\widehat\bmssL \mbfu)(\mbfs,\mbfx)\eqdef \tparen{(\Id\otimes \bmssL^\mbfs)\mbfu}(\mbfs,\mbfx) = \sum_i^m \phi_i(\mbfs)\, (\bmssL^\mbfs u_i)(\mbfx) \quad \text{on}\quad \mbbV_\pi \otimes \Cyl{}{} \fstop
\end{equation}

\item\label{i:t:DirInt:3.5} $\tparen{\widehat\bmssE_\pi,\dom{\widehat\bmssE_\pi}}$ admits carr\'e du champ operator~$\tparen{\widehat\bmssGamma_\pi,\dom{\widehat\bmssGamma_\pi}}$, the closure of the bilinear form~$(\widehat\bmssGamma,\Cb(\mbfT)\otimes\Cyl{}{})$ defined by
\begin{equation}\label{eq:t:DirInt:0.2}
\widehat\bmssGamma(\mbfu,\mbfv)\eqdef \widehat\bmssL(\mbfu\, \mbfv)-\mbfu\, \widehat\bmssL \mbfv -\mbfv\, \widehat\bmssL \mbfu\comma \qquad \mbfu,\mbfv\in\Cb(\mbfT)\otimes\Cyl{}{}\comma
\end{equation}

\item\label{i:t:DirInt:3} $\tparen{\widehat\bmssE_\pi,\dom{\widehat\bmssE_\pi}}$ is $\widehat\T_\mrmp$-quasi-regular and properly associated with a Hunt process~$\widehat\bmssW_\pi$ with state space~$\widehat\mbfM$;

\item\label{i:t:DirInt:5} a $\widehat\boldnu_\pi$-measurable set~$\mbfA\subset \widehat\mbfM$ is $\widehat\bmssE_\pi$-invariant if and only if~$\mbfA\in \class[\widehat\boldnu_\pi]{\mbfC\times\mbfM}$ for some Borel~$\mbfC\subset \mbfT$;

\item\label{i:t:DirInt:6} Assume each~$\bmssW^\mbfs$,~$\mbfs\in\mbfT$, is defined on a same measurable space~$(\Omega,\msF)$. Then,~$\widehat\bmssW_\pi$ can be realized as the stochastic process on~$\widehat\Omega\eqdef \boldOmega\times \mbfT$ with~$\boldOmega$ as in~\eqref{eq:r:StandardStochB:1} and with trajectories and transition probabilities respectively given by
\begin{equation}\label{eq:i:t:DirInt:6:0}
\widehat\mbfX^\pi_\bullet(\omega,\mbfs)\eqdef \tparen{\mbfs,\mbfX^\mbfs_\bullet(\omega)} \quad\text{and} \quad \widehat P^\pi_{(\mbfs,\mbfx)}\eqdef P^\mbfs_\mbfx\comma \qquad \mbfx\in\mbfM\comma \mbfs\in\mbfT\comma\omega\in\Omega \semicolon
\end{equation}

\item\label{i:t:DirInt:7}
the process $\widehat\bmssW_\pi$ is the ---unique up to $\widehat\boldnu_\pi$-equivalence--- $\widehat\boldnu_\pi$-sub-stationary $\widehat\boldnu_\pi$-special standard process solving the martingale problem for $(\widehat\bmssL_\pi,\Cb(\mbfT)\otimes\Cyl{}{})$. 
In particular, for every~$\mbfu\in \Cb(\mbfT)\otimes\Cyl{}{}$ ($\subset \Cb(\widehat\T_\mrmp)$), the process
\begin{equation}\label{eq:t:DirInt:1}
M^{\class{\mbfu}}_t\eqdef \mbfu(\widehat\mbfX^\pi_t)- \mbfu(\widehat\mbfX^\pi_0)-\int_0^t (\widehat\bmssL \mbfu)(\widehat\mbfX^\pi_s)\diff s
\end{equation}
is an $\widehat\msF^\pi_\bullet$-adapted square-integrable 
martingale with predictable quadratic variation
\begin{equation}\label{eq:t:DirInt:2}
\tsharpb{M^{\class{\mbfu}}}_t = \int_0^t \widehat\bmssGamma(\mbfu)(\widehat\mbfX^\pi_s)\diff s \fstop
\end{equation}
\end{enumerate}
\end{theorem}

\begin{remark}\label{r:DirInt}
In the following it will be occasionally convenient to regard the form $\tparen{\widehat\bmssE_\pi,\dom{\widehat\bmssE_\pi}}$ as the \emph{direct integral} over~$(\mbfT,\pi)$ of the measurable field of Dirichlet forms~$\mbfs\mapsto \tparen{\bmssE^\mbfs,\dom{\bmssE^\mbfs}}$, of the type constructed in~\cite{LzDS20}. 
We refer the reader to~\cite{LzDS20, LzDSWir21} for the definitions and a complete account of direct integrals of (Dirichlet) forms.
It then follows from~\cite[Prop.s~2.13,~2.16]{LzDS20} that~$\widehat\bmssH^\pi_\bullet$ is the direct integral of the measurable field of semigroup operators~$\mbfs\mapsto \bmssH^\mbfs_\bullet$, and that~$\tparen{\widehat\bmssL_\pi,\dom{\widehat\bmssL_\pi}}$ is the direct integral of the measurable field of generators~$\mbfs\mapsto\tparen{\bmssL^\mbfs,\dom{\bmssL^\mbfs}}$, all three of them regarded as decomposable direct-integral constructions on the space~$L^2(\widehat\boldnu_\pi)$ realized as the direct integral of the measurable field of Hilbert spaces~$\mbfs\mapsto L^2(\delta_\mbfs\otimes\boldnu)$ induced by the (separated) product disintegration of~$\widehat\boldnu_\pi=\pi\otimes\boldnu$.
\end{remark}

\begin{proof}[Proof of Theorem~\ref{t:DirInt}]
Let~$s\colon \widehat\mbfM\to \mbfT$ denote the projection on the $\mbfs$-coordinate, and note that
\begin{enumerate*}[$(a)$]
\item\label{i:proof:t:DirInt:1} $(\mbfT,\pi)$ is a separable countably generated probability space in the sense of e.g.~\cite[Dfn.~2.1]{LzDSWir21};
\item the product disintegration~$\widehat\boldnu_\pi=\int_\mbfT (\delta_\mbfs\otimes\boldnu) \diff\pi(\mbfs)$ is strongly consistent with~$s$, and thus
\item\label{i:proof:t:DirInt:3} the product disintegration $s$-\emph{separated} in the sense of~\cite[Dfn.~2.19]{LzDS20}.
\end{enumerate*}
Then, the closability and~\ref{i:t:DirInt:1} follow from~\cite[Prop.s~2.13, 2.29]{LzDS20}.

\medskip

Differentiating~\ref{i:t:DirInt:1} at~$t=0$ shows~\eqref{eq:t:DirInt:0.1} on~$\mbbV_\pi \otimes \Cyl{}{}$.
Since, as shown in the proof of Theorem~\ref{t:BendikovSaloffCoste}\ref{i:t:BendikovSaloffCoste:3},~$\bmssH^\mbfs_t\Cyl{}{}\subset\Cyl{}{}$ for every~$\mbfs\in\mbfT$ and~$t>0$, it is readily verified that~$\widehat\bmssH^\pi_t (\mbbV_\pi\otimes \Cyl{}{})\subset \mbbV_\pi\otimes\Cyl{}{}$ for every~$t>0$.
Thus, it follows from~\cite[Thm.~X.49]{ReeSim75} that~$\tparen{\widehat\bmssL,\mbbV_\pi \otimes \Cyl{}{}}$ is essentially self-adjoint and thus its unique (Friedrichs) self-adjoint extension on~$L^2(\widehat\boldnu_\pi)$ generates~$\widehat\bmssH^\pi_\bullet$, which proves~\ref{i:t:DirInt:2}.

\paragraph{Independence from~$\mbbV_\pi$}
Let~$\mbbV_\pi$ and~$\mbbV_\pi'$ be dense subspaces of~$L^2(\pi)$ and note that the operator~$\tparen{\widehat\bmssL, L^2(\pi)\otimes \Cyl{}{}}$ extends both~$\tparen{\widehat\bmssL, \mbbV_\pi \otimes \Cyl{}{}}$ and $\tparen{\widehat\bmssL, \mbbV_\pi' \otimes \Cyl{}{}}$.
The conclusion follows by essential self-adjointness of the latter operators, which was shown above.
It follows that we may freely choose~$\mbbV_\pi$.

In the rest of the proof we choose~$\mbbV_\pi\eqdef \Cb(\mbfT)$.

\paragraph{Proof of~\ref{i:t:DirInt:3.5}}
In light of Assumption~\ref{ass:Setting}\ref{i:ass:Setting:c}, the form~$\tparen{\mssE,\dom{\mssE}}$ admits carr\'e du champ the closure~$\tparen{\mssGamma,\dom{\mssGamma}}$ of~$(\mssGamma,\msA)$, by e.g.~\cite[Thm.~I.4.2.1~$(R_1')$, p.~18]{BouHir91}.
Thus~$\tparen{\bmssE^\mbfs,\dom{\bmssE^\mbfs}}$ admits carr\'e du champ for every~$\mbfs\in\mbfT$ in light of Corollary~\ref{c:BSCStronglyLocal}\ref{i:c:BSCStronglyLocal:2}.
The expression~\eqref{eq:t:DirInt:0.2} of the carr\'e du champ follows from the direct-integral representation of carr\'e du champ operators in~\cite[Lem.~3.8]{LzDS20}.

\paragraph{Proof of~\ref{i:t:DirInt:3}--\ref{i:t:DirInt:6}}
Firstly, note that~\ref{i:t:DirInt:3}--\ref{i:t:DirInt:6} are invariant under quasi-homeomorphism of Dirichlet forms ---\emph{mutatis mutandis}.
Secondly, it is readily seen that the quasi-homeomorphism $\boldupalpha\colon (\mbfM,\T_\mrmp,\boldnu) \rar (\mbfM^\boldupalpha,\T_\mrmp^\boldupalpha,\boldnu^\boldupalpha)$ induces, for every~$\pi\in\msP(\mbfT)$, a quasi-homeomorphism
\[
\widehat\boldupalpha\eqdef \Id\otimes \boldupalpha\colon (\widehat\mbfM,\widehat\T_\mrmp,\widehat\boldnu_\pi)\to\tparen{\widehat{\mbfM^\boldupalpha},\widehat{\T_\mrmp^\boldupalpha},\widehat{\boldnu_\pi^\boldupalpha}}\eqdef \tparen{\mbfT\times\mbfM^\boldupalpha,\T_\mrmp \times \T_\mrmp^\boldupalpha,\pi\otimes\boldnu^\boldupalpha}\fstop
\]
Thus, it suffices to show assertions~\ref{i:t:DirInt:3}--\ref{i:t:DirInt:6} under the additional assumption that $(M,\T)$ is compact, in which case~$(\widehat\mbfM,\T_\mrmp)$ too is compact second-countable Hausdorff.
Since~$\widehat\boldnu_\pi$ is a finite Borel measure, it is thus Radon, and~$\Cb(\mbfT)\otimes\Cyl{}{}$ is dense in~$L^2(\widehat\boldnu_\pi)$.
Therefore, in this case the form~$\tparen{\widehat\bmssE_\pi,\dom{\widehat\bmssE_\pi}}$ is regular with core~$\Cb(\mbfT)\otimes\Cyl{}{}$, which shows~\ref{i:t:DirInt:3}.

Since~$\bmssH^\mbfs_\bullet$ is irreducible by Theorem~\ref{t:BendikovSaloffCoste}\ref{i:t:BendikovSaloffCoste:3.5}, and in light of~\ref{i:proof:t:DirInt:1}--\ref{i:proof:t:DirInt:3} above, $\mbfs\mapsto \tparen{\bmssE^\mbfs,\dom{\bmssE^\mbfs}}$ is an ergodic decomposition of~$\tparen{\widehat\bmssE_\pi,\dom{\widehat\bmssE_\pi}}$ indexed by~$(\mbfT,\pi)$ in the sense of~\cite[Dfn.~4.1]{LzDSWir21}, $\widehat\boldnu_\pi$-essentially unique by~\cite[Thm.~4.4]{LzDSWir21}, and thus identical with the ergodic decomposition constructed in~\cite[Thm.~3.4]{LzDS20}.
It then follows from the proof of~\cite[Thm.~3.4 Step~6]{LzDS20} that a $\widehat\boldnu_\pi$-measurable set~$\mbfA\subset \widehat\mbfM$ is invariant if and only if~$\mbfA\in \class[\widehat\boldnu_\pi]{s^{-1}(\mbfC)}$ for some Borel $\mbfC\subset \mbfT$, which concludes the proof of~\ref{i:t:DirInt:5} by definition of~$s$.
Finally, $\tparen{\widehat\bmssE_\pi,\dom{\widehat\bmssE_\pi}}$ is properly associated with a Hunt process~$\widehat\bmssW_\pi$, and the representation in assertion~\ref{i:t:DirInt:6} is a straightforward ---albeit tedious--- consequence of~\ref{i:t:DirInt:1}.

\ref{i:t:DirInt:7} A proof of~\ref{i:t:DirInt:7} is similar to the proof of Theorem~\ref{t:BendikovSaloffCoste}\ref{i:t:BendikovSaloffCoste:6}.
However, it is worth noting that the assumption of~\cite[Thm.~3.5]{AlbRoe95} in~\cite[p.~517, (A)]{AlbRoe95} is not fully satisfied, since~$\pi$ does not need to have full $\T_\mrmp$-support on~$\mbfT$ and thus~$\widehat\boldnu_\pi$ does not have full $\widehat\T_\mrmp$-support on~$\widehat\mbfM$.
One can check that this fact does not affect the proof of~\cite[Thm.~3.5]{AlbRoe95} (i.e.\ that of~\cite[Thm.~3.4(ii)]{AlbRoe95}), by replacing~$\mbfT\times\mbfM$ with~$\supp\pi\times\mbfM$.
This is possible since~$(\supp\pi\times\mbfM)^\complement= (\supp\pi)^\complement\times\mbfM$ is $\widehat\bmssE_\pi$-polar and may therefore be removed from the space.
\end{proof}

As a matter of fact, the process~$\widehat\bmssW_\pi$ in~Theorem~\ref{t:DirInt}\ref{i:t:DirInt:3} depends on~$\pi$ only via its natural (augmented) filtration, say~$\widehat\msF^\pi_\bullet$, defined as in~\cite[Dfn.~IV.1.8, Eqn.~(1.7), p.~90]{MaRoe92}.
Indeed both its trajectories and its transition probabilities in~\eqref{eq:i:t:DirInt:6:0} are defined for \emph{every}~$\mbfs\in\mbfT$ and \emph{independently of~$\pi$}.
Furthermore, the filtration~$\widehat\msF^\pi_\bullet$ is the (minimal) augmentation of~$\msF_\bullet\otimes\Bo{\supp[\pi]}$, so that it only depends on the equivalence class of~$\pi$, and it does not in fact depend on~$\pi$ among all~$\pi\in\msP(\mbfT)$ with full $\T_\mrmp$-support.
We may therefore define a single Hunt process
\[
\widehat\bmssW\eqdef \tparen{\widehat\Omega,\widehat\msF,
\widehat\mbfX_\bullet, \ttseq{\widehat P_{\mbfs,\mbfx}}_{(\mbfs,\mbfx)\in\widehat\mbfM}}\comma 
\]
with~$\widehat\Omega$ as in~\eqref{eq:r:StandardStochB:1} and further satisfying~\eqref{eq:i:t:DirInt:6:0}.
For every~$\pi\in\msP(\mbfT)$, the process so constructed is a $\widehat\boldnu_\pi$-version of~$\widehat\bmssW_\pi$ by definition; it can be started at~$(\mbfs,\mbfx)$ for every~$\mbfs\in\mbfT$ and $\bmssE^\mbfs$-quasi-every~$\mbfx\in\mbfM$.

In principle, since for every~$\mbfs,\mbfs'\in\mbfT_\circ$ the process~$\bmssW^\mbfs$ is a coordinate-wise conformal rescaling of~$\bmssW^{\mbfs'}$, one can hope for the $\bmssE^\mbfs$-polar set~$N^\mbfs$ of inadmissible starting points~$\mbfx$ for~$\bmssW^\mbfs$ to be independent of~$\mbfx$.
If this is the case, one can further characterize the $\widehat\bmssW$-exceptional set of inadmissible starting points for~$\widehat\bmssW$ as~$\mbfT\times N$ with~$N=N^\mbfs$.

In the general case when~$N^\mbfs$ is merely known to be $\bmssE^\mbfs$-polar, the independence on~$\mbfs$ is related to the $\widehat\bmssE^\pi$-polarity of the set~$\widehat N\subset \widehat\mbfM$ defined to have~$N^\mbfs$,~$\mbfs\in\mbfT$, as its sections, which we will not investigate here.
The situation becomes however much simpler in the case when we can choose~$N^\mbfs$ to be empty for every~$\mbfs\in\mbfT$, as it is the case under Assumption~\ref{ass:ContinuityH} by Corollary~\ref{c:Wellposedness}.
In that case, we have the following result.

\begin{corollary}\label{c:WellposednessExtended}
In the setting of Theorem~\ref{t:DirInt}, suppose further that Assumption~\ref{ass:ContinuityH} is satisfied.
Then, there exists a unique Hunt process
\begin{equation}\label{eq:c:WellposednessExtended:0}
\widehat\bmssW\eqdef \tparen{\widehat\Omega,\widehat\msF,
\widehat\mbfX_\bullet, \ttseq{\widehat P_{\mbfs,\mbfx}}_{(\mbfs,\mbfx)\in\widehat\mbfM}}\comma 
\end{equation}
with~$\widehat\mbfX_\bullet$ and~$\widehat P_{\mbfs,\mbfx}$ as in~\eqref{eq:i:t:DirInt:6:0}, with transition semigroup~$\widehat\bmssp_\bullet$ satisfying
\begin{enumerate}[$(i)$]
\item $(t,\mbfx,\mbfy)\mapsto \widehat\bmssp_t\tparen{(\mbfs,\mbfx),(\mbfs,\mbfy)}$ is continuous on $\R^+\times \mbfM^\tym{2}$ for every~$\mbfs\in\mbfT$;
\item $\widehat\bmssp_t\tparen{(\mbfs,\emparg),(\mbfs,\emparg)} \in \Cb(\T_\mrmp^\tym{2})$ for every~$t>0$ and for every~$\mbfs\in\mbfT$;
\item $\widehat\bmssp_t\tparen{(\mbfs,\emparg),(\mbfs',\emparg)}\equiv 0$ everywhere on~$\widehat\mbfM^\tym{2}$ for every~$\mbfs,\mbfs'\in\mbfT$ with~$\mbfs\neq\mbfs'$;
\item $\int \widehat\bmssp_t \mbfu\tparen{(\mbfs,\emparg),(\mbfs,\mbfy)}\diff\widehat\boldnu_\pi(\mbfs,\mbfy)$ is a $\widehat\boldnu_\pi$-version of~$\widehat\bmssH^\pi_t \mbfu$ for every~$\pi\in\msP(\mbfT)$. In particular, $\widehat\boldnu_\pi$ is $\widehat\bmssp_\bullet$-super-median for every~$\pi\in\msP(\mbfT)$;
\item\label{i:c:WellposednessExtended:5} for every~$\mbfu\in \Cb(\mbfT)\otimes\Cyl{}{}$, for every~$(\mbfs,\mbfx)\in\widehat\mbfM$, the process
\begin{equation}\label{eq:c:WellposednessExtended:00}
M^{\class{\mbfu}}_t\eqdef \mbfu(\widehat\mbfX_t)- \mbfu(\widehat\mbfX_0)-\int_0^t (\widehat\bmssL \mbfu)(\widehat\mbfX_s)\diff s
\end{equation}
is an adapted square-integrable $\widehat P_{(\mbfs,\mbfx)}$-martingale with predictable quadratic variation
\begin{equation*}
\tsharpb{M^{\class{\mbfu}}}_t = \int_0^t \widehat\bmssGamma(\mbfu)(\widehat\mbfX_s)\diff s \fstop
\end{equation*}
\end{enumerate}
\end{corollary}

\section{Free massive systems: measure representation}\label{s:FISMeas}
In this section we recast the results in~\S\ref{sss.ExtendedSpace} on the space of probability measures over~$M$.
We start from the simple observation that every point~$(\mbfs,\mbfx)$ in~$\widehat\mbfM$ may be regarded as the probability measure~$\sum_i s_i \delta_{x_i}$ on~$M$.

\subsection{Spaces of probability measures and cylinder functions}
Let~$\msP$ be the space of all Borel probability measures on~$M$, and~$\msP^\pa$ be the subset of all purely atomic measures on~$M$.

\subsubsection{The weak atomic topology and the transfer map}
On~$\msP$ and on its subsets we shall consider two topologies, namely
\begin{itemize}[wide]
\item the \emph{narrow topology~$\T_\mrmn$}, i.e.\ the topology induced by duality w.r.t.~$\Cb(\T)$, and 
\item the \emph{weak atomic topology}~$\T_\mrma$ introduced by S.N.~Ethier and T.G.~Kurtz in~\cite{EthKur94}.
\end{itemize}

We recall some basic facts about~$\T_\mrma$ after~\cite[\S2]{EthKur94} and~\cite[\S4]{LzDS17+}.
Let~$\mssd_1\colon M^\tym{2}\to [0,1]$ be a bounded distance on~$M$ completely metrizing~$\T$, and let~$\rho_\mrmn$ be the \emph{Prokhorov distance} on~$\msP$ induced by~$\mssd_1$, see e.g.~\cite[Chap.~6, p.~72ff.]{Bil99}.
The weak atomic topology~$\T_\mrma$ is the topology induced by the distance
\begin{align}\label{eq:WATrho}
\rho_\mrma(\mu,\nu)\eqdef&\ \rho_\mrmn(\mu,\nu)+ \sup_{\eps \in (0,1]} \abs{\iint_{M^\tym{2}} \cos\paren{\frac{2\,\mssd_1}{\pi\eps}}\diff \mu^\otym{2}-\iint_{M^\tym{2}} \cos\paren{\frac{2\,\mssd_1}{\pi\eps}}\diff \nu^\otym{2}} \fstop
\end{align}

\begin{proposition}[Properties of~$\T_\mrma$]\label{p:PropWeakAtomicTop}
The following assertions hold true:
\begin{enumerate}[$(i)$]
\item $\rho_\mrma$ is a complete and separable distance on~$\msP$, consequence of~\cite[Lem.~2.3]{EthKur94};
\item $\T_\mrma$ is strictly finer than~$\T_\mrmn$, see~\cite[Lem.~2.2]{EthKur94};
\item\label{i:p:PropWeakAtomicTop:3} the Borel $\sigma$-algebras of~$(\msP,\T_\mrma)$ and~$(\msP,\T_\mrmn)$ coincide, see~\cite[p.~5]{EthKur94}.
\end{enumerate}
\end{proposition}

Further properties of~$\T_\mrma$ are collected in~\cite[Lem.s~2.1--2.5]{EthKur94} and in~\cite[Prop.~4.6]{LzDS17+}.
We note that the terminology of `\emph{Borel}' object on~$\msP$ is unambiguous in light of Proposition~\ref{p:PropWeakAtomicTop}\ref{i:p:PropWeakAtomicTop:3}.

\paragraph{Transfer map}
Define a map~$\EM\colon \widehat\mbfM \longrar \msP^\pa$ by
\begin{equation}\label{eq:TransferMap}
\begin{aligned}
\EM\colon (\mbfs,\mbfx)&\longmapsto \sum_{k=1}^\infty s_k\delta_{x_k} \comma
\end{aligned}
\end{equation}
and denote by~$\msP^\pa_\iso\eqdef\EM(\widehat\mbfM_\circ)$ the space of purely atomic measures with \emph{infinite strictly ordered masses}.

The following fact is shown in~\cite[Prop.~4.9(ii)]{LzDS17+}. Its proof depends neither on the manifold structure of~$M$ nor on the compactness of~$M$ which are standing assumptions in~\cite{LzDS17+}.
\begin{lemma}
The map~$\EM\colon (\widehat\mbfM_\circ,\widehat\T_\mrmp)\to (\msP^\pa_\iso,\T_\mrma)$ is a homeomorphism.
\end{lemma}

\begin{corollary}\label{c:PpaisoPolish}
$(\msP^\pa_\iso,\T_\mrma)$ is a Polish space.
\begin{proof}
Since~$(M,\T)$ is Polish by assumption (cf.~\S\ref{ss:Rescaling}), its countable product~$(\mbfM,\T_\mrmp)$ is Polish as well.
The space~$(\mbfT,\T_\mrmp)$ too is Polish since it is a closed subset of the infinite cube~$\mbfI$.
In light of the above lemma, it suffices to show that~$(\widehat\mbfM_\circ,\widehat\T_\mrmp)$ is a Polish space.
In turn, since the product of two Polish spaces is Polish, the conclusion follows if we show that~$(\mbfM_\circ,\T_\mrmp)$ and~$(\mbfT_\circ,\T_\mrmp)$ are both Polish.
Since they are clearly~$G_\delta$ sets in~$\mbfM$ and~$\mbfT$ respectively, the fact that they are Polish follows from the latter spaces being Polish, by Alexandrov's Theorem, e.g.~\cite[Thm.~3.11, p.~17]{Kec95}.
\end{proof}
\end{corollary}

\subsubsection{Cylinder functions}
We aim to move the Dirichlet form~$(\widehat\bmssE_\pi,\dom{\widehat\bmssE_\pi})$ on~$\widehat\mbfM$ to its image via~$\EM$ on~$\msP$, defined as in~\S\ref{sss:ImageObjects}.
In order to explicitly construct such a form, we define several algebras of cylinder functions on~$\msP$, which will play the role of different cores for Dirichlet forms and generators on~$\msP$.

\begin{definition}[Testing]
For each~$\molli\in\Bb(\mbbI)$ and~$f\in\Bb(\T)$, we define a functional on~$\msP$ by
\begin{align}\label{eq:FundAugmentedCyl}
(\molli\otimes f)^\trid(\mu)\eqdef \int \molli(\mu_x)\, f(x) \,\diff\mu(x) \comma
\end{align}
and ---for notational simplicity, with slight abuse of notation--- we let
\[
\boldmolli\otimes \mbff\eqdef \seq{\molli_1\otimes f_1,\dotsc, \molli_k\otimes f_k}\comma \qquad k\in \N_1\comma \molli\in\Bb(\mbbI)\comma f_i\in \Bb(\T)\comma
\]
and we denote by
\[
(\boldmolli\otimes \mbff)^\trid \eqdef \tseq{(\molli_1\otimes f_1)^\trid,\dotsc, (\molli_k\otimes f_k)^\trid}\colon \msP\longrar\R^k
\]
the functional obtained by taking~$\square^\trid$ componentwise.
Finally, we simply write
\begin{equation}\label{eq:SimpleTrid}
f^\trid \eqdef\ (\car\otimes f)^\trid \colon \mu\longmapsto \int f \diff\mu \comma \qquad f\in \Bb(\T)\fstop
\end{equation}
\end{definition}

Let us recall that~$(\molli\otimes f)^\trid$ is generally \emph{non}-linear, while~$f^\trid$ is always linear.

\begin{definition}
Let~$\msR$ be any subalgebra of~$\Bb(\mbbI)$, and set
\begin{equation}\label{eq:MassAlgebra}
\msR_\eps\eqdef \set{f\in\msR: \supp f \subset (\eps, 1]}\comma \qquad \eps\in\mbbI \fstop
\end{equation}

We say that~$\msR_0$ \emph{separates points from open sets} (\emph{in~$(0,1]$}) if for every~$s\in (0,1]$ and every open~$U\subset \mbbI \setminus\set{s}$ there exists a function~$\molli_s\in \msR_0$ such that~$\molli_s(s)\neq 0$ and~$\molli_s\equiv 0$ on~$U$.
\end{definition}

Note that~$\msR_0=\scup_{\eps>0} \msR_\eps$ and, in general,~$\msR_0\subsetneq \msR$. Furthermore, suppose that~$\msR_0$ separates points from open sets.
Since~$\msR_0$ is an algebra, we may fix the value of~$\molli_s(s)$ as we please.
Everywhere in this work, we assume the following.

\begin{assumption}\label{ass:SubalgebraR}
$\msR$ is a subalgebra of~$\Bb(\mbbI)$ and~$\msR_0$ separates points from open sets.
\end{assumption}

In fact, all constructions in the following will \emph{not} depend on the choice of~$\msR$ (provided that Assumption~\ref{ass:SubalgebraR} be satisfied).
However, it will be from time to time convenient to assume some largeness or smallness for~$\msR$.

\medskip

Now, let~$\R[t_1,\dotsc, t_k]$ be the space of all $k$-variate polynomials on~$\R$.

\begin{definition}[Cylinder functions]\label{d:CylinderF}
Set
\begin{subequations}
\begin{align}
\hCylP{}{}\eqdef& \set{u\colon \msP\to\R : \begin{gathered} u=F\circ {(\boldmolli\otimes \mbff)^\trid}\comma F\in \R[t_1,\dotsc, t_k] \comma
\\
k\in \N_0\comma \seq{\molli_i}_{i\leq k}\subset \msR\comma \seq{f_i}_{i\leq k}\subset\msA \end{gathered}}\comma
\intertext{and, for every~$\eps\in\mbbI$,}
\label{eq:HCylinderFEps}
\hCylP{}{\eps}\eqdef& \set{u\colon \msP\to\R : \begin{gathered} u=F\circ {(\boldmolli\otimes \mbff)^\trid}\comma F\in \R[t_1,\dotsc, t_k] \comma
\\
k\in \N_0\comma \seq{\molli_i}_{i\leq k}\subset \msR_\eps\comma \seq{f_i}_{i\leq k}\subset\msA \end{gathered}}\fstop
\end{align}
\end{subequations}
For each~$u\in\hCylP{}{0}$ further define the \emph{vanishing threshold}~$\eps_u$ of~$u$ as
\begin{equation}\label{eq:VanishingThreshold}
\eps_u\eqdef \inf\set{\eps>0: u\in\hCylP{}{\eps}}>0 \fstop
\end{equation}
\end{definition}

Let us collect here some properties of these classes of cylinder functions on~$\msP$.
\begin{proposition}[Properties of cylinder functions]\label{p:PropertiesCylinder}
The following assertions hold:
\begin{enumerate}[$(i)$]
\item\label{i:p:PropertiesCylinder:1} $\hCylP{}{}$,~$\hCylP{}{\eps}$ are algebras w.r.t.\ the pointwise multiplication for every~$\eps\in\mbbI$;
\item\label{i:p:PropertiesCylinder:2} $\eps\mapsto \hCylP{}{\eps}$ is strictly decreasing and left-continuous, i.e.
\[
\hCylP{}{\delta}\subsetneq \hCylP{}{\eps}\comma \quad \delta>\eps\comma \qquad \text{and} \quad \hCylP{}{\eps}=\scup_{\delta>\eps}\hCylP{}{\delta} \semicolon
\]
\item\label{i:p:PropertiesCylinder:3} $\hCylP{}{0}$ generates the Borel $\sigma$-algebra on~$\msP^\pa$;
\item\label{i:p:PropertiesCylinder:4} every function in~$\hCylP{}{}$ is $\T_\mrma$-continuous (in particular: Borel-measurable);
\item\label{i:p:PropertiesCylinder:5} every non-constant function in~$\hCylP{}{0}$ is $\T_\mrmn$-discontinuous at every~$\mu\in\msP^\pa$;
\item\label{i:p:PropertiesCylinder:6} $\hCylP{}{0}$ is dense in~$L^2(\mcQ)$ for \emph{every} Borel probability measure~$\mcQ$ on~$\msP^\pa$.
\end{enumerate}
\end{proposition}

\begin{proof}
Proofs of~\ref{i:p:PropertiesCylinder:1}-\ref{i:p:PropertiesCylinder:5} may easily be deduced from those of the corresponding statements in~\cite[Rmk.~5.3, Lem.~5.4]{LzDS17+}, which make use of neither the manifold structure of~$M$ nor its compactness.
Assertion~\ref{i:p:PropertiesCylinder:6} follows from~\cite[Prop.~A.3]{LzDS17+}.
\end{proof}

Furthermore,~$\EM$ intertwines cylinder functions in the sense of the next Lemma.
Indeed, in the following let~$\msS$ stand for either~$\Bb$ or~$\Cb$.
\begin{lemma}\label{l:TransferCyl}
Assume that~$\msR\subset \msS(\mbbI)$. Then,
\[
\EM^*u\eqdef u\circ \EM \in \msS(\mbfT)\otimes\Cyl{}{} \quad \text{for every} \quad u\in \hCylP{}{0} \fstop
\]
\begin{proof}
Since~$\msS(\mbfT)\otimes\Cyl{}{}$ is an algebra, it suffices to show that~$\EM^*u \in \msS(\mbfT)\otimes\Cyl{}{}$ for every~$u\in \hCylP{}{0}$ of the form~$u=(\molli\otimes f)^\trid$.
For every such~$u$, and since~$u\in \hCylP{}{0}$, there exists~$\eps=\eps_u>0$ such that~$\molli\in\msR_\eps$.
Setting~$n\eqdef \ceiling{1/\eps}\in \N$, we then have
\begin{equation}\label{eq:FundFuncTrid}
(\molli\otimes f)^\trid\tparen{\EM(\mbfs,\mbfx)}= \sum_{k=1}^\infty s_k\molli(s_k) f(x_k)= \sum_{k=1}^n s_k\molli(s_k) f(x_k)
\fstop
\end{equation}
Since~$\msR\subset \msS(\mbbI)$, the function~$\mbfs\mapsto s_k\molli(s_k)$ is an element of~$\msS(\mbfT)$ for every~$k$, by continuity (Borel measurability) of the coordinate projections.
Furthermore, since~$f\in\msA$, the function~$\mbfx\mapsto f(x_k)$ is an element of~$\Cyl{}{}$ for every~$k$, for the choice~$A=\set{k}$.
Thus,~$(\mbfs,\mbfx)\mapsto s_k\molli(s_k) f(x_k)$ is an element of~$\msS(\mbfT)\otimes\Cyl{}{}$, hence so is~$u\circ \EM$ being a finite a sum of functions therein.
\end{proof}
\end{lemma}

\subsection{Analytic theory}
Here, we show how to transfer the form~$\tparen{\widehat\bmssE_\pi,\dom{\widehat\bmssE_\pi}}$ to a Dirichlet form on~$\msP$.
We start by computing a suitable generator.

\subsubsection{Generators on measures}
As a consequence of Lemma~\ref{l:TransferCyl}, we may compute~$\widehat\bmssL$ on functions in~$\EM^*\hCylP{}{0}$.
In particular, for elementary functions as in~\eqref{eq:FundFuncTrid} we have
\begin{align}\label{eq:ExtendedLonElementaryCyl}
\widehat\bmssL \tparen{\EM^* (\molli\otimes f)^\trid}(\mbfs,\mbfx) =\sum_{k=0}^\infty s_k \molli(s_k)\, (\mssL_k f)(x_k) = \sum_{k=0}^\infty \molli(s_k)\, (\mssL f)(x_k) \fstop
\end{align}

\begin{proposition}\label{p:PDirIntESA}
Fix~$\pi\in\msP(\mbfT_\circ)$.
Then, the operator~$\tparen{\widehat\bmssL, \EM^*\hCylP{}{0}}$ is essentially self-adjoint on~$L^2(\widehat\boldnu_\pi)$.
Furthermore, its closure~$\tparen{\widehat\bmssL_\pi,\dom{\widehat\bmssL_\pi}}$ is independent of~$\msR$.
\begin{proof}
We divide the proof into three steps.

\paragraph{Fiber-wise essential self-adjointness}
Fix~$\mbfs\in\mbfT_\circ$.
We claim that the operator~$\tparen{\Id\otimes \bmssL^\mbfs, \EM^*\hCylP{}{0}}$ is essentially self-adjoint on~$L^2(\delta_\mbfs\otimes \boldnu)$ for every~$\mbfs\in\mbfT$.

We note that~$\Cyl{}{}$ embeds as~$L^2(\delta_\mbfs)\otimes \Cyl{}{}=\R\otimes \Cyl{}{}\cong \Cyl{}{}$ into~$L^2(\delta_\mbfs\otimes\boldnu)$, and we shall therefore, whenever necessary, regard~$\Cyl{}{}$ as a subspace of~$L^2(\delta_\mbfs\otimes\boldnu)$ without further mention to the embedding.
Since~$(\bmssL^\mbfs,\Cyl{}{})$ is essentially self-adjoint on~$L^2(\boldnu)$ for every~$\mbfs\in\mbfT$, the operator~$(\Id\otimes\bmssL^\mbfs,\Cyl{}{})$ is essentially self-adjoint on~$L^2(\delta_\mbfs\otimes\boldnu)$ for every~$\mbfs\in\mbfT$ by~\cite[Thm.~VIII.33]{ReeSim80a}.

Then, it suffices to show that~$\EM^*\hCylP{}{0}$ contains~$\Cyl{}{}$.
To this end, we argue as follows. 
Since~$\msR_0$ is an algebra and separates points from open sets, and since~$\mbfs\in\mbfT_\circ$ (as opposed to~$\mbfT$), for each~$k$ there exists~$\molli_k\in \msR_0$ with~$\molli_k(s_k)=s_k^{-1}$ and~$\molli_k\equiv 0$ on~$[0,s_{k+1}]\cup [s_{k-1},1]$ (where, conventionally,~$s_0\eqdef 1$).
Then, for each~$f\in\Cyl{}{}$, the function~$(\molli_k\otimes f)^\trid$ satisfies~$(\molli_k\otimes f)^\trid\in \EM^*\hCylP{}{0}$.
Furthermore, in light of~\eqref{eq:FundFuncTrid}, we have~$(\molli_k\otimes f)^\trid \tparen{\EM(\mbfs,\mbfx)}= f(x_k)$ for every~$k$.
Since~$f$ was arbitrary in~$\msA$, since every element of~$\Cyl{}{}$ has the form~$\sprod_{a\in A} f_a(\emparg_a)$ for some finite~$A\subset\N$, and since~$\EM^*\hCylP{}{0}$ is an algebra, this shows that~$\EM^*\hCylP{}{0}\supset \Cyl{}{}$, which concludes the proof of the claim.

\paragraph{Essential self-adjointness of the direct integral}
In light of Remark~\ref{r:DirInt}, the operator~$\tparen{\widehat\bmssL, \EM^*\hCylP{}{0}}$ is the direct integral of the measurable field of closable operators~$\mbfs\mapsto \tparen{\Id\otimes \bmssL^\mbfs, \EM^*\hCylP{}{0}}$.
The conclusion then follows from Lemma~\ref{l:DirIntESA} and the first step.

\paragraph{Independence from~$\msR$}
It is clear from~\eqref{eq:ExtendedLonElementaryCyl} that, independently of the choice of~$\msR$, the domain~$\dom{\widehat\bmssL_\pi}$ contains the algebra~$\hCylP{}{0}$ defined with~$\msR=\Bb(\mbbI)$.
The conclusion then follows from the essential self-adjointness of $\tparen{\widehat\bmssL, \EM^*\hCylP{}{0}}$, similarly to the proof of the independence from~$\mbbV_\pi$ in the proof of Theorem~\ref{t:DirInt}.
\end{proof}
\end{proposition}

\subsubsection{Measure representation}
For notational simplicity, let us set
\begin{align}\label{eq:DefGeneratorCylP}
\mcQ_\pi\eqdef \EM_\pfwd \widehat\boldnu_\pi\qquad \text{and} \qquad \widehat\mcL\eqdef \EM_*\circ \widehat\bmssL \circ \EM^* \fstop
\end{align}

\begin{remark}\label{r:SupportQPi}
In light of Proposition~\ref{p:PropWeakAtomicTop}\ref{i:p:PropWeakAtomicTop:3},~$\mcQ_\pi$ is a Borel measure for both the weak atomic topology~$\T_\mrma$ and the narrow topology~$\T_\mrmn$.
The support properties of~$\mcQ_\pi$, however, sensibly depend on the chosen topology.
Indeed, on the one hand it is not difficult to show that~$\T_\mrmn$-$\supp\mcQ_\pi=\msP$ whenever~$\pi\mbfT_\circ>0$;
on the other hand, $\T_\mrma$-$\supp\mcQ_\pi = \EM(\T_\mrmp\text{-}\supp\, \pi \times \mbfM)\subset\msP^\pa$ for every~$\pi\in\msP(\mbfT)$.
In particular, choosing~$\pi=\delta_\mbfs$ with~$\mbfs\in\mbfT_\circ$ shows that~$\T_\mrma$-$\supp\mcQ_\pi$ can be very small even if~$\T_\mrmn$-$\supp\mcQ_\pi=\msP$.
\end{remark}

We recall one important example of a measure~$\mcQ_\pi$.
\begin{example}[The Dirichlet--Ferguson measure]\label{ex:DirichletFerguson}
Fix~$\beta>0$, and let~$\Pi_\beta$ be the Poisson--Dirichlet measure with parameter~$\beta$ constructed in Example~\ref{e:PoissonDirichlet}.
Then, $\mcQ_\pi$~is the \emph{Dirichlet--Ferguson measure with intensity~$\beta\nu$} in~\cite{Fer73}, cf.\ e.g.~\cite[App.~2]{Kin75}.
\end{example}

\begin{lemma}\label{l:SimpleTridLp}
Fix~$p\in [1,\infty)$ and~$\pi\in\msP(\mbfT)$.
Let~$f\in L^p(\nu)$. Further let~$\rep{f}$ be a any $\nu$-representative of~$f$ and let~$\reptwo{f}^\trid\colon \msP \to \R\cup\set{\pm\infty}$ be defined as in~\eqref{eq:SimpleTrid}.
Then,
\begin{enumerate}[$(i)$]
\item $\rep{f}^\trid$ is well-defined and independent of the choice of the representative;
\item its class~$f^\trid\eqdef \tclass[\mcQ_\pi]{\rep{f}^\trid}$ satisfies~$f^\trid\in L^p(\mcQ_\pi)$;
\item\label{i:l:SimpleTridLp:3} the map~$\square^\trid\colon L^p(\nu)\to L^p(\mcQ_\pi)$ is Lipschitz non-expansive.
\end{enumerate}

\begin{proof}
Assume first that~$f\geq 0$ and take~$\rep{f}\geq 0$ (everywhere on~$M$). Then $\rep{f}^\trid\colon\msP\to [0,\infty]$ is well-defined (possibly infinite), and, by Jensen's inequality,
\begin{equation}\label{eq:l:SimpleTridLp:1}
\int (\rep{f}^\trid)^p \diff\mcQ_\pi \leq \int (\rep{f}^p)^\trid \diff\mcQ_\pi = \int \sum_i^Ns_i \int \rep{f}(x_i)^p\diff\nu(x_i) \diff\pi(\mbfs) = \norm{f}_{L^p(\nu)}^p\fstop
\end{equation}
This proves all the assertions for~$f\in L^p(\nu)^+$. The conclusion follows by applying~\eqref{eq:l:SimpleTridLp:1} to the positive and negative part of~$\rep{f}^\trid$.
\end{proof}
\end{lemma}

\begin{corollary}
The random measure~$\mcQ_\pi$ has intensity measure~$\nu$, viz., for every Borel~$A\subset M$,
\[
\int \eta A \diff\mcQ_\pi(\eta) = \nu A \fstop
\]
\begin{proof}
Apply Lemma~\ref{l:SimpleTridLp}\ref{i:l:SimpleTridLp:3} with~$p=1$ and~$\reptwo{f}=\car_A$. The inequality in~\eqref{eq:l:SimpleTridLp:1} is in fact an equality since~$p=1$.
\end{proof}
\end{corollary}

We are now ready to present the main definitions of objects on~$\msP$.

\begin{theorem}\label{t:TransferP}
Fix~$\pi\in\msP(\mbfT_\circ)$.
Then,
\begin{enumerate}[$(i)$, leftmargin=2.5em]
\item\label{i:t:TransferP:1} the operator~$\tparen{\widehat\mcL, \hCylP{}{0}}$ is essentially self-adjoint on~$L^2(\mcQ_\pi)$, and its closure $\tparen{\widehat\mcL_\pi,\dom{\widehat\mcL_\pi}}$ is a Markov generator;

\item\label{i:t:TransferP:2} the Dirichlet form~$\tparen{\widehat\mcE_\pi,\dom{\widehat\mcE_\pi}}$ generated by~$\tparen{\widehat\mcL_\pi,\dom{\widehat\mcL_\pi}}$ is the image form of~$\tparen{\widehat\bmssE_\pi,\dom{\widehat\bmssE_\pi}}$ via~$\EM$, and admits carr\'e du champ operator~$\tparen{\widehat\mcG_\pi,\dom{\widehat\mcG_\pi}}$, the closure of the bilinear form~$(\widehat\mcG,\hCylP{}{0})$ defined by
\[
\widehat\mcG(u,v)\eqdef \widehat\mcL(uv)-u \widehat\mcL v -v \widehat\mcL u\comma \qquad u,v\in\hCylP{}{0}\comma
\]

\item\label{i:t:TransferP:3} a $\mcQ_\pi$-measurable set~$\mbfA\subset\msP$ is $\widehat\mcE_\pi$-invariant if and only if~$\mbfA\in \class[\mcQ_\pi]{\EM(\mbfC\times\mbfM)}$ for some Borel~$\mbfC\subset\mbfT_\circ$;
\end{enumerate}
\end{theorem}

\begin{proof}
\ref{i:t:TransferP:1} The essential self-adjointness of~$\tparen{\widehat\mcL,\hCylP{}{0}}$ holds as a consequence of Proposition~\ref{p:PDirIntESA}.

\ref{i:t:TransferP:2} Since~$\EM^*\colon L^2(\mcQ_\pi)\to L^2(\widehat\boldnu_\pi)$ is a unitary operator intertwining $\tparen{\widehat\mcL_\pi,\dom{\widehat\mcL_\pi}}$ and~$\tparen{\widehat\bmssL_\pi,\dom{\widehat\bmssL_\pi}}$, it also intertwines the corresponding quadratic forms $\tparen{\widehat\mcE_\pi,\dom{\widehat\mcE_\pi}}$ and $\tparen{\widehat\bmssE_\pi,\dom{\widehat\bmssE_\pi}}$.
Since~$\EM^*$ is a additionally a Riesz isomorphism, it preserves Markovianity, thus~$\tparen{\widehat\mcE_\pi,\dom{\widehat\mcE_\pi}}$ is a Dirichlet form since so is~$\tparen{\widehat\bmssE_\pi,\dom{\widehat\bmssE_\pi}}$.
The existence and expression for the carr\'e du champ of~$\tparen{\widehat\mcE_\pi,\dom{\widehat\mcE_\pi}}$ are immediate from the definition of carr\'e du champ operator and the existence and expression for the carr\'e du champ of~$\tparen{\widehat\bmssE_\pi,\dom{\widehat\bmssE_\pi}}$.

\ref{i:t:TransferP:3} The characterization of $\widehat\mcE_\pi$-invariant sets follows by the intertwining property of~$\EM$ from the analogous characterization of $\widehat\bmssE_\pi$-invariant sets in Theorem~\ref{t:DirInt}\ref{i:t:DirInt:5}.
\end{proof}

\subsection{Potential theory and the quasi-homeomorphism}
Theorem~\ref{t:TransferP} establishes that the transfer by~$\EM$ preserves the analytical properties of $\tparen{\widehat\bmssE_\pi,\dom{\widehat\bmssE_\pi}}$.
The preservation of stochastic properties, and in particular the association with $\tparen{\widehat\mcE_\pi,\dom{\widehat\mcE_\pi}}$ of a stochastic process, is more subtle and requires a thorough discussion of the potential theory for~$\tparen{\widehat\bmssE_\pi,\dom{\widehat\bmssE_\pi}}$.

\subsubsection{Nests and polar sets}\label{sss:Nests}
Let us firstly introduce one other main assumption on the reference form~$\tparen{\mssE,\dom{\mssE}}$, of which we will make use in the following.
Recall that we denote by~$\Cap$ be the first-order capacity associated to~$\tparen{\mssE,\dom{\mssE}}$;
we further let~$\Cap_{1,1}$ be the first-order capacity associated to the product Dirichlet form~$\tparen{\mssE^{1,1},\dom{\mssE^{1,1}}}\eqdef \tparen{\mssE,\dom{\mssE}}^\otym{2}$.
(See~\S\ref{app:Capacity}.)
Finally,~$\Delta M$ is the diagonal in~$M^\tym{2}$.

\begin{assumption}[Quantitative polarity of points]\label{ass:QPP}
Assumption~\ref{ass:Setting} holds, and additionally 
\begin{equation}\tag{\textsc{qpp}}\label{eq:QPP}
\Cap_{1,1}(\Delta M)=0\fstop
\end{equation}
\end{assumption}

Let us collect some observations about Assumption~\ref{ass:QPP}.

\begin{remark}\label{r:QPP}

In turn,~\eqref{eq:QPP} is implied by mildly quantitative assertions on the capacity of small sets, relating relative $\mssE$-capacities, $\nu$-volumes, and covering number.
To better see this, suppose further that~$(M,\T)$ is compact, and that~$\mssd$ is a distance metrizing~$\T$.
Further choose~$\seq{B_r^i}_{i\leq c_r}$ to be a minimal $\T$-open covering of~$\Delta M$ consisting of $\mssd$-ball of radius~$r>0$.
(Since~$M$ is compact, the minimal number of such balls covering~$\Delta M$, i.e.\ the \emph{covering number}~$c_r$, is finite.)
Finally, for each~$i\leq c_r$, let~$e^{r,2r,i}$ be the equilibrium $\mssE$-potential of the pair~$(B_r^i,B_{2r}^i)$.
Then, by Lemma~\ref{l:Capacity1},
\begin{align*}
\Cap_{1,1}(\Delta M) &\leq \sum_i^{c_r} \Cap_{1,1}\tparen{{B_r^i}^\tym{2}} \leq \sum_i^{c_r} \Cap_{1,1}\tparen{{B_r^i}^\tym{2},{B_{2r}^i}^\tym{2}}
\\
&\leq 2 \sum_i^{c_r} \Cap(B_r^i, B_{2r}^i)\norm{e^{r,2r,i}}_{L^2(\nu)}^2 \leq c_r \, \Cap(B_r, B_{2r})\, \nu B_{2r}\fstop
\end{align*}
Thus, we see that~\eqref{eq:QPP} is implied by
\[
\limsup_{r\to 0} \Cap(B_r, B_{2r})=0\qquad \text{and} \qquad c_r \, \nu B_{2r} \in O(1) \quad \text{as} \quad r\to 0 \fstop
\]
%
\end{remark}

\bigskip

Let~$\pr_n\colon \mbfM\to M^\tym{n}$ be the projection on the first $n$ coordinates.
As a consequence of~Proposition~\ref{p:BendikovSaloffCosteRescaling} and Theorem~\ref{t:BendikovSaloffCoste}\ref{i:t:BendikovSaloffCoste:5}, the process~$\mssW^{n,\mbfs}\eqdef {\pr_n}_*\bmssW^\mbfs$ is a Hunt process with state space~$M^\tym{n}$ and satisfies
\[
X^{n,\mbfs}_t = \pr_n \circ \mbfX^\mbfs_t=\seq{X_{t/s_1},\dotsc, X_{t/s_n}}\comma \qquad P^{n,\mbfs}_{\pr_n(\mbfx)}= {\pr_n}_\pfwd P_\mbfx = \sotimes_k^n P_{x_k} \fstop
\]
As for~$\bmssW$, the path probabilities~$P^{n,\mbfs}_{\pr_n(\mbfx)}$ may be constructed on the probability space~$\Omega^\tym{n}$, in which case they do not depend on~$\mbfs$.
We will therefore drop the superscript~$\mbfs$ from the notation, thus writing simply~$\sotimes_k^n P_{x_k}$ in place of~$P^{n,\mbfs}_{\pr_n(\mbfx)}$.

\bigskip

In the next lemma, let~$M^\tym{n}_\circ\eqdef \pr_n (\mbfM_\circ)\subset M^\tym{n}$.
We let~$\tparen{\mssE^{n,\mbfs},\dom{\mssE^{n,\mbfs}}}$ be the regular Dirichlet form properly associated to~$\mssW^{n,\mbfs}$, and by~$\Cap_{1,\dotsc, n}$ its Choquet capacity.

\begin{lemma}\label{l:ExceptionalFiniteDim}
Assume~\eqref{eq:QPP} holds.
Then,~$(M^\tym{n}_\circ)^\complement$ is $\mssW^{n,\mbfs}$-exceptional for every~$n\in\N$ and every~$\mbfs\in\mbfT$.
\begin{proof}
By e.g.~\cite[Thm.~IV.5.29(i)]{MaRoe92}, the assertion is equivalent to:~$(M^\tym{n}_\circ)^\complement$ is $\mssE^{n,\mbfs}$-polar, which we now prove.
For each~$i\neq j$ set $M^n_{i,j}\eqdef \set{\mbfx\in M^\tym{n} : x_i=x_j}$.
By the countable sub-additivity of Choquet capacities, and since~$(M^\tym{n}_\circ)^\complement=\bigcup_{i,j:i\neq j} M^n_{i,j}$, it suffices to show that~$M^n_{i,j}$ is $\mssE^{n,\mbfs}$-polar for every~$i,j$ with~$i\neq j$.
We show that~$M^n_{1,2}$ is $\mssE^{n,\mbfs}$-polar.
A proof for~$M^n_{i,j}$ is analogous yet notationally more involved.
For simplicity of notation, put~${}_2\mbfs\eqdef \seq{s_3, s_4,\dotsc}$.
Denote by~$\Cap_{2,\mbfs}$ the Choquet capacity of the Dirichlet form~$\tparen{\mssE^{2,\mbfs},\dom{\mssE^{2,\mbfs}}}$ on~$L^2(\nu^\otym{2})$ and by~$\Cap_{n-2,{}_2\mbfs}$ the Choquet capacity of the form~$\tparen{\mssE^{n-2,{}_2\mbfs},\dom{\mssE^{n-2,{}_2\mbfs}}}$ on~$L^2(\nu^{\otym{n-2}})$.

Let~$e_{2,\mbfs}$ be the equilibrium $\mssE^{2,\mbfs}$-potential for~$(\Delta M, M^\tym{2})$.
By Lemma~\ref{l:Capacity1} applied to~$\mssE^1=\mssE^{2,\mbfs}$ and~$\mssE^2=\mssE^{n-2,{}_2\mbfs}$,
\begin{align*}
\Cap_{n,\mbfs}(M^n_{1,2})=&\ \Cap_{n,\mbfs}(\Delta M \times M^\tym{n-2})
\\
\leq&\ \Cap_{2,\mbfs} (\Delta M) \norm{\car}_{L^2(\nu^\otym{n-2})}^2 + \Cap_{n-2,{}_2\mbfs}(M^\tym{n-2})\norm{e_{2,\mbfs}}_{L^2(\nu^\otym{2})}^2
\\
\leq&\ 2\, \Cap_{2,\mbfs}(\Delta M) \fstop
\end{align*}
The latter term vanishes choosing~$a\eqdef s_1^{-1}\geq 1$ and~$b\eqdef s_2^{-1}\geq 1$ in Lemma~\ref{l:Capacity2}, which proves the assertion.
\end{proof}
\end{lemma}

\begin{proposition}\label{p:AndiPolar}
Assume~\eqref{eq:QPP} holds.
Then,~$\mbfM_\circ^\complement$ is $\bmssW^\mbfs$-exceptional for all~$\mbfs\in\mbfT$.

\begin{proof}
Up to resorting to the compactification~$\mbfM^\boldupalpha$ of~$\mbfM$ and to the quasi-homeo\-morphism~$\boldupalpha$ constructed~\eqref{eq:Compactification} in the proof of Theorem~\ref{t:BendikovSaloffCoste}, we may and will assume with no loss of generality that~$(M,\T)$ (hence~$(\mbfM,\T_\mrmp)$) is compact.

For~$\mbfB\subset\mbfM$, define the \emph{first touching time}~\cite[\S{IV.5}, Eqn.~(5.14)]{MaRoe92} of~$\mbfB$ for~$\bmssW^\mbfs$ as
\begin{equation}\label{eq:p:AndiPolar:1}
\tau^\mbfs_\mbfB\eqdef \inf\set{t\geq 0 : \overline{\mbfX^\mbfs_{[0,t]}}\cap \mbfB\neq \emp}\comma
\end{equation}
where, for simplicity of notation,~$\overline{\mbfX^\mbfs_{[0,t]}}$ is the $\T_\mrmp$-closure of~$\set{\mbfX^\mbfs_r(\omega): r\in [0,t]}$.

In the rest of the proof, we choose~$\mbfB=\mbfM_\circ^\complement$ and omit it from the notation.
Let~$P^\mbfs_{\boldnu}$ be defined according to~\eqref{eq:MarkovProbab} with~$P^\mbfs_\emparg$ as in~\eqref{eq:BSCRescalingProcess}.
Since~$\mbfM_\circ^\complement$ is measurable and~$\bmssW^\mbfs$ has infinite lifetime by Proposition~\ref{p:BendikovSaloffCosteRescaling} (cf.\ Thm.~\ref{t:BendikovSaloffCoste}\ref{i:t:BendikovSaloffCoste:5}), it suffices to show that
\begin{equation}\label{eq:p:AndiPolar:0}
P^\mbfs_{\boldnu}\set{\tau^\mbfs<\infty}=0\fstop
\end{equation}
To this end, note that~$\mbfM_\circ^\complement=\scap_n \pr_n^{-1}\tparen{(M^\tym{n}_\circ)^\complement}$ and ${\pr_n}_* \bmssW^\mbfs = \mssW^{n,\mbfs}$.
Since $(\mbfM,\T_\mrmp)$ is a compact space,~$\pr_n$ is a closed map by e.g.~\cite[Cor.~3.1.11]{Eng89}, and therefore~$\pr_n(\overline{\mbfX^\mbfs_{[0,t]}})=\overline{X^{n,\mbfs}_{[0,t]}}$.
Thus, letting~$\tau^{n,\mbfs}$ be the first touching time of~$(M^\tym{n}_\circ)^\complement$ for~$\mssW^{n,\mbfs}$,
\begin{equation}\label{eq:p:AndiPolar:1.5}
\set{\tau^{n,\mbfs}<\infty} \subset \set{\tau^{n+1,\mbfs}<\infty}\subset \set{\tau^\mbfs<\infty}= \lim_n \set{\tau^{n,\mbfs}<\infty} \fstop
\end{equation}
Furthermore, since~$(M^\tym{n}_\circ)^\complement$ is $\mssW^{n,\mbfs}$-exceptional by Lemma~\ref{l:ExceptionalFiniteDim},
\begin{equation}\label{eq:p:AndiPolar:2}
P_{\boldnu} \set{\tau^{n,\mbfs}<\infty}= P_{\nu^\otym{n}}\set{\tau^{n,\mbfs}<\infty}=0 \fstop
\end{equation}
Thus, finally, combining~\eqref{eq:p:AndiPolar:1.5} and~\eqref{eq:p:AndiPolar:2} yields~\eqref{eq:p:AndiPolar:0} by the Borel--Cantelli Lem\-ma.
\end{proof}
\end{proposition}

\begin{remark}[No collisions]\label{r:NoCollisions}
Let us note that the first collision time~\eqref{eq:Intro:FCT} for~$\bmssW^\mbfs$ is always larger than the first touching time~\eqref{eq:p:AndiPolar:1} of~$\mbfM_\circ$ for~$\bmssW^\mbfs$.
As a consequence, it follows from Lemma~\ref{p:AndiPolar} that~\eqref{eq:QPP} implies~$\tau_c\equiv +\infty$ a.s.; in particular, no collisions occur.
This (almost) complements the results in~\S\ref{sss:Collisions} in light of Remark~\ref{r:QPP}.
\end{remark}

\paragraph{Off-diagonal nests}
In order to further our analysis of~$\bmssW^\mbfs$, we give an explicit nest for the properly associated Dirichlet form.
This nest will in fact be independent of~$\mbfs$.

Since~$(M,\T)$ is second-countable locally compact Hausdorff, it is Polish, and in particular metrizable.
Let~$\mssd$ be any distance metrizing~$\T$.
Further let~$\vareps\colon \R^+\to\R^+$ be any Borel-measurable function satisfying
\begin{equation}\label{eq:NestPhi}
\vareps\quad \text{is non-increasing} \quad \text{and}\quad \lim_{t\to\infty} \vareps(t)=0 \fstop
\end{equation}
Finally set
\begin{equation}\label{eq:NestUij}
\mbfU_{i,j,\delta}\eqdef \set{\mbfx\in\mbfM : \mssd(x_i,x_j)< \delta}\comma \qquad i,j\in\N\comma \delta>0\comma
\end{equation}
and
\begin{equation}\label{eq:Nest}
\mbfU_n\eqdef \bigcup_{\substack{i,j\\i<j}} \mbfU_{i,j,\vareps(ijn)}\comma 
\quad \mbfF_n\eqdef\mbfU_n^\complement \comma
\qquad n\in \N\fstop
\end{equation}

\begin{lemma}\label{l:Nest}
Fix~$\mbfs\in\mbfT$.
Assume~\eqref{eq:QPP} holds.
Then,~$\seq{\mbfF_n}_n$ is an $\bmssE^\mbfs$-nest.

\begin{proof}
Let~$\mbfM^\boldupalpha$ and~$\boldupalpha\colon \mbfM\to \mbfM^\boldupalpha$ be defined as in the proof of Theorem~\ref{t:BendikovSaloffCoste}.
As in the proof of that theorem,~$\boldupalpha$ is a homeomorphism onto its image and a quasi-homeomorphism of Dirichlet forms intertwining~$\tparen{\bmssE^\mbfs,\dom{\bmssE^\mbfs}}$ and~$\tparen{\bmssE^{\mbfs,\boldupalpha},\dom{\bmssE^{\mbfs,\boldupalpha}}}$.
As a consequence, letting~$\Cap_\mbfs$ be the first-order capacity associated to~$\tparen{\bmssE^\mbfs,\dom{\bmssE^\mbfs}}$, and analogously for~$\Cap_\mbfs^\boldupalpha$, we have~$\Cap_\mbfs(\mbfA)=\Cap_\mbfs^\boldupalpha(\boldupalpha(\mbfA))$ for every capacitable~$\mbfA\subset\mbfM$.

Now, note that~$\mbfU_{i,j,\delta}$ is $\T_\mrmp$-open, hence so is~$\mbfU_n$, so that~$\mbfF_n^\complement$ is $\T_\mrmp$-closed.
Furthermore, it is readily seen that~$\mbfU_n\downarrow_n \mbfM_\circ^\complement$, hence that~$\mbfF_n \uparrow_n \mbfM_\circ$.
Thus, it remains to show that~$\limsup_n \Cap_\mbfs(\mbfF_n^\complement)=0$.

To this end, note that~$\overline{\boldupalpha(\mbfU_n)}$ is $\T^\boldupalpha_\mrmp$-compact in~$\mbfM^\boldupalpha$.
Therefore,
\begin{align*}
\limsup_n \Cap_\mbfs(\mbfF_n^\complement)=&\ \limsup_n \Cap^\boldupalpha_\mbfs\tparen{\boldupalpha(\mbfF_n^\complement)} \leq \limsup_n \Cap^\boldupalpha_\mbfs\tparen{\overline{\boldupalpha(\mbfF_n^\complement)}}
\\
=&\ \inf_n \Cap^\boldupalpha_\mbfs\tparen{\overline{\boldupalpha(\mbfU_n)}}= \Cap^\boldupalpha_\mbfs\tparen{\scap_n \overline{\boldupalpha(\mbfU_n)}}
\end{align*}
where the last equality holds by properties of Choquet capacities and since $\overline{\boldupalpha(\mbfU_n)}$ is $\T^\boldupalpha_\mrmp$-compact.
Then,
\begin{align*}
\limsup_n \Cap_\mbfs(\mbfF_n^\complement)\leq&\ \Cap^\boldupalpha_\mbfs\tparen{\scap_n \overline{\boldupalpha(\mbfU_n)}}
\leq \Cap^\boldupalpha_\mbfs\tparen{\scap_n \boldupalpha(\overline{\mbfU_n}) \cup \mbfM^\boldupalpha_*}
\\
\leq&\ \Cap^\boldupalpha_\mbfs\tparen{\scap_n \boldupalpha(\mbfU_n)} + \Cap^\boldupalpha_\mbfs\tparen{\mbfM^\boldupalpha_*}
\\
=&\ \Cap_\mbfs\tparen{\scap_n \mbfU_n} + 0
= \Cap_\mbfs\tparen{\mbfM_\circ^\complement}=0\comma
\end{align*}
where~$\mbfM^\boldupalpha_*$ is defined as in~\eqref{eq:PolarMAlpha} and $\bmssE^{\mbfs,\boldupalpha}$-polar by the arguments there, and the last equality holds by Lemma~\ref{p:AndiPolar}.
\end{proof}
\end{lemma}

\begin{remark}\label{r:Nest}
Let us note that the sets~$\mbfF_n$ \emph{depend on~$\mssd$} and on~$\vareps$, yet they form an $\bmssE^\mbfs$-nest for every~$\mssd$ (metrizing~$\T$), for every~$\vareps$ as in~\eqref{eq:NestPhi}, and for every~$\mbfs\in\mbfT$.
\end{remark}

\begin{proposition}\label{p:AndiHatPolar}
Let~$\mbfF_n$ be as in~\eqref{eq:Nest} and assume that~\eqref{eq:QPP} holds. Then,
\begin{enumerate}[$(i)$]
\item\label{i:p:AndiHatPolar:1} setting~$\widehat\mbfF_n\eqdef \mbfT\times\mbfF_n$ defines an $\widehat\bmssE_\pi$-nest~$\tseq{\widehat\mbfF_n}_n$;

\item\label{i:p:AndiHatPolar:2} the set~$\widehat\mbfM_\circ$ is~$\widehat\bmssW_\pi$-coexceptional.
\end{enumerate}

\begin{proof}
The family~$\seq{\mbfF_n}_n$ as in~\eqref{eq:Nest} is an $\bmssE^\mbfs$-nest for every~$\mbfs\in\mbfT$ by Lemma~\ref{l:Nest} (also cf.~Rmk.~\ref{r:Nest}) and~\ref{i:p:AndiHatPolar:1} follows in a straightforward way.
\ref{i:p:AndiHatPolar:2} is a consequence of~\ref{i:p:AndiHatPolar:1} by~\cite[Thm.~IV.5.29(i)]{MaRoe92}.
\end{proof}
\end{proposition}

\subsubsection{Quasi-homeomorphism}
We are now ready to transfer the stochastic properties of~$\tparen{\widehat\bmssE_\pi,\dom{\widehat\bmssE_\pi}}$ to~$\tparen{\widehat\mcE_\pi,\dom{\widehat\mcE_\pi}}$.

\begin{theorem}\label{t:TransferQuasiHomeo}
Fix~$\pi\in\msP(\mbfT_\circ)$ and assume~\eqref{eq:QPP} holds.
Then,
\begin{enumerate}[$(i)$, leftmargin=2em]
\item\label{i:t:TransferP:4} $\EM$ is a quasi-homeomorphism intertwining~$\tparen{\widehat\mcE_\pi,\dom{\widehat\mcE_\pi}}$ and~$\tparen{\widehat\bmssE_\pi,\dom{\widehat\bmssE_\pi}}$;

\item\label{i:t:TransferP:5} $\tparen{\widehat\mcE_\pi,\dom{\widehat\mcE_\pi}}$ is a $\T_\mrma$-quasi-regular Dirichlet form properly associated with the $\mcQ_\pi$-symmetric Hunt process with state space~$\msP^\pa_\iso$
\begin{equation}\label{eq:t:TransferP:0}
\widehat\mcW_\pi\eqdef \EM_*\widehat\bmssW_\pi=\seq{\widehat\Omega_\circ,\widehat\msF^\pi,\widehat\msF^\pi_\bullet,\widehat\mcX^\pi_\bullet, \tseq{\widehat\mcP^\pi_\eta}_{\eta\in\msP^\pa_\iso}} \semicolon
\end{equation}
with
\[
\widehat\mcX^\pi_\bullet= \EM\circ \widehat\mbfX^\pi_\bullet \qquad \text{and} \qquad \widehat\mcP^\pi_\eta= \widehat P^\pi_{\EM^{-1}(\eta)}\comma \quad \eta\in\msP^\pa_\iso \semicolon
\]

\item\label{i:t:TransferP:6} 
$\widehat\mcW_\pi$ is the ---unique up to~$\mcQ_\pi$-equivalence--- $\mcQ_\pi$-sub-stationary $\mcQ_\pi$-special standard process process solving to the martingale problem for~$\tparen{\widehat\mcL,\hCylP{}{0}}$.
In particular, for every~$u\in\hCylP{}{0}$ ($\subset \Cb(\T_\mrma)$), the process
\begin{equation}\label{eq:t:Transfer:1}
M^{\class{u}}_t\eqdef u(\widehat\mcX^\pi_t)- u(\widehat\mcX^\pi_0)-\int_0^t (\widehat\mcL u)(\widehat\mcX^\pi_s)\diff s
\end{equation}
is an adapted square-integrable martingale with predictable quadratic variation
\begin{equation}\label{eq:t:Transfer:2}
\tsharpb{M^{\class{u}}}_t = \int_0^t \widehat\mcG(u)(\widehat\mcX^\pi_s)\diff s \fstop
\end{equation}
\end{enumerate}

\begin{proof}
\ref{i:t:TransferP:4} By Proposition~\ref{p:AndiHatPolar}\ref{i:p:AndiHatPolar:2} the set~$\widehat\mbfM_\circ$ is $\widehat\bmssE_\pi$-polar.
Since~$\EM\colon \widehat\mbfM_\circ\to\msP^\pa_\iso$ is a $\widehat\T_\mrmp/\T_\mrma$-homeomorphism intertwining~$\tparen{\widehat\bmssE_\pi,\dom{\widehat\bmssE_\pi}}$ and~$\tparen{\widehat\mcE_\pi,\dom{\widehat\mcE_\pi}}$, it is a quasi-homeomorphism.

\ref{i:t:TransferP:5} The $\T_\mrma$-quasi-regularity of~$\tparen{\widehat\mcE_\pi,\dom{\widehat\mcE_\pi}}$ follows from~\ref{i:t:TransferP:4} and the $\widehat\T_\mrmp$-quasi-regularity of $\tparen{\widehat\bmssE_\pi,\dom{\widehat\bmssE_\pi}}$ shown in Theorem~\ref{t:DirInt}\ref{i:t:DirInt:3}.

\ref{i:t:TransferP:6} A proof follows similarly to the proof of Theorem~\ref{t:DirInt}\ref{i:t:DirInt:7}.
Again~$\mcQ_\pi$ does not need to have full $\T_\mrma$-support on~$\msP^\pa_\iso$, but one can restrict the arguments to~$\supp\mcQ_\pi$ since~$(\supp\mcQ_\pi)^\complement$ is $\widehat\bmssE_\pi$-polar.
\end{proof}
\end{theorem}

\begin{example}[Crystallization and thermalization]
Fix~$\beta>0$ and let~$\pi=\Pi_\beta$ be the Poisson--Dirichlet distribution in Example~\ref{e:PoissonDirichlet}, in which case~$\mcQ_\pi$ is the Dirichlet--Ferguson measure~$\DF{\beta\nu}$ with intensity~$\beta\nu$ in Example~\ref{ex:DirichletFerguson}.
In the context of Theorem~\ref{t:TransferQuasiHomeo}, the parameter~$\beta$ may be understood as an \emph{inverse temperature}, cf.~\cite[Rmk.~3.16]{LzDS19a}.
Whenever~$\mu_0$ is randomly distributed according to~$\mcQ=\mcQ_\beta\ll\DF{\beta\nu}$, its can be shown as a consequence of~\cite[Cor.~3.14]{LzDS19a} that the law of~$\widehat\mcW_\pi$ on the path space crystallizes as~$\beta\to\infty$ to the invariant measure~$\delta_\nu$ and thermalizes as~$\beta\to\infty$ to the law $\delta_\pfwd(\law\mssW)$ of the Dirac measure centered at a single $\mssW$-driven particle.
\end{example}

\subsection{The strongly local case}\label{ss:StronglyLocalCase}
Let us now address the case when~$\tparen{\mssE,\dom{\mssE}}$ is additionally a local form.
We shall see that the local property has far-reaching implications, including the existence of an \emph{$\mssL$-induced geometry on~$\msP$}.

\medskip

In the same setting of \S\ref{ss:Rescaling}, we further make the following assumption.

\begin{assumption}[Strongly local case]\label{ass:SettingLocal}
Suppose Assumption~\ref{ass:Setting} is satisfied.
We further assume that the form~$\tparen{\mssE,\dom{\mssE}}$ is strongly local.
\end{assumption}

Let us first collect some standard consequences of this assumption.
As a consequence of Corollary~\ref{c:BSCStronglyLocal}, and since strong locality is invariant under quasi-homeomorphism (see e.g.~\cite[Thm.~5.1]{Kuw98}), we have the following.

\begin{corollary}[Strongly local case]
Fix~$\pi\in\msP(\mbfT_\circ)$.
If Assumption~\ref{ass:SettingLocal} 
holds, then
\begin{enumerate}[$(i)$, leftmargin=2em]
\item in the setting of Theorem~\ref{t:DirInt},
\begin{enumerate}[$({i}_1)$]
\item the Dirichlet form~$\tparen{\widehat\bmssE_\pi,\dom{\widehat\bmssE_\pi}}$ on~$L^2(\widehat\boldnu_\pi)$ is additionally strongly local;
 \item the process~$\widehat\bmssW_\pi$ in~\eqref{eq:i:t:DirInt:6:0} is a conservative $\T_\mrmp$-diffusion on~$\widehat\mbfM_\circ$;
 \item the martingale~$M^{\class{\mbfu}}_\bullet$ in~\eqref{eq:t:DirInt:1} has quadratic variation~$\quadvar{M^{\class{\mbfu}}}_\bullet=\tsharpb{M^{\class{\mbfu}}}_\bullet$ given by the right-hand side of~\eqref{eq:t:DirInt:2} for every~$\mbfu\in\Cb(\mbfT)\otimes\Cyl{}{}$;
 \end{enumerate}

\item in the setting of Theorem~\ref{t:TransferQuasiHomeo},
\begin{enumerate}[$({ii}_1)$]
\item the Dirichlet form~$\tparen{\widehat\mcE_\pi,\dom{\widehat\mcE_\pi}}$ on~$L^2(\mcQ_\pi)$ is additionally strongly local;
\item the process~$\widehat\mcW_\pi$ in~\eqref{eq:t:TransferP:0} is a conservative~$\T_\mrma$-diffusion on~$\msP^\pa$;
\item the martingale~$M^{\class{u}}_\bullet$ in~\eqref{eq:t:Transfer:1} has quadratic variation~$\quadvar{M^{\class{u}}}_\bullet=\tsharpb{M^{\class{u}}}_\bullet$ given by the right-hand side of~\eqref{eq:t:Transfer:2} for every~$u\in\hCylP{}{0}$.
\end{enumerate}
\end{enumerate}
\end{corollary}

\subsubsection{Intrinsic representation of the generator and the carr\'e du champ}
We start with a simple computation.

\begin{lemma}[Diffusion property and representation]\label{l:DiffusionP}
Let~$u\eqdef F\circ \hat\mbff^\trid$ and $v\eqdef G\circ \hat\mbfg^\trid$ be in~$\hFC{}{0}$. Then,~$u,v\in\dom{\widehat\mcL_\pi}$ ($\subset\dom{\widehat\mcG_\pi}$) and
\begin{gather}
\label{eq:l:DiffusionP:0.1}
\widehat\mcG_\pi(u,v)= \sum_{i,j}^{k,h} (\partial_i F) \circ \hat\mbff^\trid \cdot (\partial_j G) \circ \hat\mbfg^\trid \cdot \mssGamma(\hat f_i, \hat g_j)^\trid \comma
\\
\label{eq:l:DiffusionP:0.2}
\widehat\mcL_\pi u = \sum_i^k (\partial_i F)\circ \hat\mbff^\trid \cdot \tparen{(\tfrac{1}{\id_\mbbI}\otimes\mssL) \hat f_i}^\trid + \sum_{i,j}^{k,k} (\partial_{ij}^2F)\circ \hat\mbff^\trid \cdot \mssGamma(\hat f_i,\hat f_j)^\trid\comma
\end{gather}
are well-defined (i.e., independent of the choice of the representatives~$F$,~$\hat\mbff$ for~$u$ and~$G$,~$\hat\mbfg$ for~$v$) and independent of~$\pi\in\msP(\mbfT)$.

Furthermore, we have the intrinsic representations
\begin{align}
\label{eq:l:DiffusionP:0.3}
\widehat\mcG(u,v)_\eta =& \int_M \mssGamma^z\restr{z=x} \tparen{u(\eta+\eta_x\delta_z-\eta_x\delta_x), u(\eta+\eta_x\delta_z-\eta_x\delta_x)} \diff\eta(x) \comma
\\
\label{eq:l:DiffusionP:0.4}
\widehat\mcL(u)_\eta=& \int_M \frac{\mssL^z\restr{z=x}u(\eta+\eta_x\delta_z-\eta_x\delta_x)}{{\eta_x}^2} \diff\eta(x) \fstop
\end{align}

\begin{proof}
We prove the assertions for the generator. The argument for the carr\'e du champ operator is completely analogous and therefore it is omitted.
The well-posedness and the independence from~$\pi$ follow at once from the intrinsic representation~\eqref{eq:l:DiffusionP:0.4}, so it suffices to show the latter.
To this end, choose freely~$\hat f,\hat f_1,\hat f_2\in\msR_0\otimes \msA$. Since~$\hat f^\trid\eta$ is a \emph{finite} sum, simply by linearity of~$\mssL$ (bilinearity of~$\mssGamma$), we may compute
\begin{align}
\nonumber
\mssL^z\restr{z=x}&\ \tparen{\hat f^\trid(\eta+\eta_x\delta_z-\eta_x\delta_x)} =
\\
\nonumber
=&\ \mssL^z\restr{z=x} \braket{\int \hat f(\eta_y,y) \diff\eta(y)+ \eta_x\hat f_i(\eta_x,z)-\eta_x\hat f_i(\eta_x,x)}
\\
\label{eq:l:DiffusionP:4}
=&\ \eta_x (\mssL \hat f)(\eta_x,x)\comma
\end{align}
and
\begin{align}
\nonumber
\mssGamma^z\restr{z=x}&\ \paren{\hat f_1^\trid (\eta+\eta_x\delta_z-\eta_x\delta_x),\hat f_2^\trid (\eta+\eta_x\delta_z-\eta_x\delta_x)}=
\\
\nonumber
=&\ \mssGamma^z\restr{z=x}\Big[\hat f_1^\trid\eta +\eta_x\hat f_i(\eta_x,z)-\eta_x\hat f_1(\eta_x,x), 
\hat f_2^\trid\eta +\eta_x\hat f_2(\eta_x,z)-\eta_x\hat f_2(\eta_x,x)\Big]
\\
\nonumber
=&\ \mssGamma^z\restr{z=x}\braket{\eta_x\hat f_1(\eta_x,z), \eta_x\hat f_2(\eta_x,z)}
\\
\label{eq:l:DiffusionP:5}
=&\ {\eta_x}^2\, \mssGamma(\hat f_1,\hat f_2)(\eta_x,x) \fstop
\end{align}

Now, by the diffusion property for~$(\mssL,\msA)$,
\begin{align}
\nonumber
\mssL^z\restr{z=x}&\ u\tparen{\eta+\eta_x\delta_z-\eta_x\delta_x}=
\\
\label{eq:l:DiffusionP:3}
=&\ \sum_i^k (\partial_i F)(\hat\mbff^\trid \eta)\, \mssL^z\restr{z=x} \tparen{\hat f_i^\trid(\eta+\eta_x\delta_z-\eta_x\delta_x)}
\\
\nonumber
&+ \sum_{i,j}^{k,k} (\partial^2_{ij} F)(\hat\mbff^\trid \eta)\, \mssGamma^z\restr{z=x} \Big(\hat f_i^\trid (\eta+\eta_x\delta_z-\eta_x\delta_x),
\hat f_j^\trid (\eta+\eta_x\delta_z-\eta_x\delta_x)\Big)\fstop
\end{align}

Substituting~\eqref{eq:l:DiffusionP:4} and~\eqref{eq:l:DiffusionP:5} into~\eqref{eq:l:DiffusionP:3}, dividing by~${\eta_x}^2$ and integrating w.r.t.~$\eta$ we obtain the right-hand side of~\eqref{eq:l:DiffusionP:0.2}, and thus the conclusion by~\eqref{eq:l:DiffusionP:0.2} and the diffusion property for~$\widehat\mcL_\pi$.
\end{proof}
\end{lemma}

\subsubsection{Standard cylinder functions}\label{sss:StandardCylFormDom}
While the smallness of~$\hCylP{}{0}$ better serves to the purpose of showing essential self-adjointness, it will be helpful to compute the generator~$\widehat\mcL$ and the square field~$\widehat\mcG$ on a larger class of cylinder functions.

Furthermore, even if we have already shown that the weak atomic topology~$\T_\mrma$ is a natural topology in relation to the measure representation of massive systems, we also aim at a better understanding of the properties of the form~$\tparen{\widehat\mcE_\pi,\dom{\widehat\mcE_\pi}}$ w.r.t.\ the more standard narrow topology~$\T_\mrmn$.
As noted in Proposition~\ref{p:PropertiesCylinder}\ref{i:p:PropertiesCylinder:5}, functions in~$\hCylP{}{0}$ are not $\T_\mrmn$-continuous, and therefore they are not suitable for the above purpose.
It will be interesting to see that~$\dom{\widehat\mcE_\pi}$ contains `sufficiently many' $\T_\mrmn$-continuous functions.

\begin{definition}[Standard cylinder functions]\label{d:StandardFC}
Set
\begin{gather}\label{eq:d:StandardFC:0}
\FC{\msA}{}\eqdef \set{u\colon \msP\to\R : \begin{gathered} u=F\circ \mbff^\trid\comma F\in \Cb^\infty(\R^k) \comma
\\
k\in \N_0\comma f_i \in \msA\comma i\leq k \end{gathered}}\fstop
\end{gather}

Further let~$\hFC{\msA}{}$ and~$\hFC{\msA}{\eps}$ with~$\eps\in\mbbI$ be defined as in Definition~\ref{d:CylinderF} by replacing the space of polynomials~$\R[t_1,\dotsc, t_k]$ there with the space~$\Cb^\infty(\R^k)$ of smooth functions with continuous and bounded derivatives of all orders.
\end{definition}

\begin{remark}\label{r:ExtensionPropertiesCA}
It is not difficult to show that all the assertions in Proposition~\ref{p:PropertiesCylinder} remain valid when~$\hCylP{}{*}$ is replaced by~$\hFC{\msA}{*}$ with~$*=\emp$ or~$\eps\in \mbbI$.
In fact, they were proved in~\cite[Rmk.~5.3, Lem.~5.4]{LzDS17+} for a larger still class of functions.
Furthermore, since functions in~$\msA$ are $\T$-continuous, we have~$\FC{\msA}{}\subset \Cb(\T_\mrmn)$.
\end{remark}

Lemma~\ref{l:DiffusionP} has the following important consequence.

\begin{proposition}[Closability of standard cylinder functions]\label{p:StandardCyl}
Fix~$\pi\in\msP(\mbfT_\circ)$.
Then,
\begin{enumerate}[$(i)$]
\item\label{i:p:StandardCyl:1} the space of cylinder functions~$\FC{\msA}{}$ is a Markovian subspace of~$\dom{\widehat\mcE_\pi}$;
\item\label{i:p:StandardCyl:2} the quadratic form~$\tparen{\widehat\mcE_\pi,\FC{\msA}{}}$ is closable on~$L^2(\mcQ_\pi)$.
Its closure \linebreak $\tparen{\mcE_\pi,\dom{\mcE_\pi}}$ is a conservative strongly local Dirichlet form;
\item\label{i:p:StandardCyl:3} the form~$\tparen{\mcE_\pi,\dom{\mcE_\pi}}$ is $\T_\mrmn$-regular if and only if~$(M,\T)$ is compact;
\item\label{i:p:StandardCyl:4} if additionally~\eqref{eq:QPP} holds, then~$\tparen{\mcE_\pi,\dom{\mcE_\pi}}=\tparen{\widehat\mcE_\pi,\dom{\widehat\mcE_\pi}}$ and the identity map $\id\colon (\msP,\T_\mrma)\to(\msP,\T_\mrmn)$ is a quasi-homeomorph\-ism.
In particular, the form is additionally $\T_\mrmn$-quasi-regular.
\end{enumerate}

\begin{proof}
\ref{i:p:StandardCyl:1} Since~$\msR_0$ separates points from open sets, by Stone--Weierstra\ss\ Theorem it is dense in~$\Cz((0,1])$.
Thus, we can find a sequence~$\seq{\molli_n}_n\subset\msR_0$ satisfying~$\nlim \molli_n=1$ locally uniformly on~$(0,1]$.
Fix~$u=F\circ\mbff^\trid\in\FC{\msA}{}$ with~$\mbff=\seq{f_1,\dotsc, f_k}$, set~$\hat\mbff_n\eqdef \seq{\molli_n \otimes f_1, \dotsc, \molli_n \otimes f_k}$, and~$u_n\eqdef F\circ \hat\mbff_n$.
It follows from~\eqref{eq:l:DiffusionP:0.1} that the function~$\mu\mapsto\widehat\mcG(u_n)_\mu$ is uniformly bounded in~$n$ and converges pointwise on~$\msP$ to a function~$\mu\mapsto \widehat\mcG_\pi(u)_\mu$ defined by the right-hand side of~\eqref{eq:l:DiffusionP:0.1} with~$\mbff$ in place of~$\hat\mbff$.
Thus,~$\widehat\mcG_\pi(u)\in L^1(\mcQ_\pi)$ by Dominated Convergence, and therefore~$u\in \dom{\widehat\mcE}$.
This shows that~$\FC{\msA}{}\subset \dom{\widehat\mcE_\pi}$.
As for its Markovianity, it suffices to notice that~$\FC{\msA}{}$ is closed under post-composition with smooth functions.

Since~$\car\in \FC{\msA}{}$, \ref{i:p:StandardCyl:2} is a standard consequence of~\ref{i:p:StandardCyl:1}, see e.g.~\cite[\S{V.5.1}]{BouHir91}.

\ref{i:p:StandardCyl:3}
It is readily seen that~$\FC{\msA}{}$ is a unital point-separating sub-algebra of~$\Cb(\T_\mrmn)$, which proves~\ref{i:d:QuasiReg:2} since~$\FC{\msA}{}$ is dense in~$\dom{\mcE_\pi}$ (by definition of the latter).
If~$(M,\T)$ is compact~$(\msP,\T_\mrmn)$ too is compact.
Thus~$\FC{\msA}{}$ is uniformly dense in~$\Cc(\T_\mrmn)=\Cb(\T_\mrmn)$ by the above reasoning and Stone--Weierstra\ss\ Theorem.
Again since~$\FC{\msA}{}$ is dense in~$\dom{\mcE_\pi}$, it is a core for~$\tparen{\mcE_\pi,\dom{\mcE_\pi}}$ which shows that the latter form is $\T_\mrmn$-regular.
If otherwise~$(M,\T)$ is not compact, then~$(\msP,\T_\mrmn)$ is not locally compact, thus~$\tparen{\mcE_\pi,\dom{\mcE_\pi}}$  is not $\T_\mrmn$-regular.

\ref{i:p:StandardCyl:4} In order to show the equality it suffices to show that~$\hCylP{}{0}\subset\dom{\mcE_\pi}$.
This requires the capacity estimates in Appendix~\ref{app:Capacity} and is eventually shown in Lemma~\ref{l:CoincidenceDomains}.

Since~$\tparen{\widehat\mcE_\pi,\dom{\widehat\mcE_\pi}}$ is $\T_\mrma$-quasi-regular, it admits a $\T_\mrma$-compact $\widehat\mcE_\pi$-nest~$\seq{\mcK_n}_n$.
Since~$\T_\mrmn$ is coarser than~$\T_\mrma$, the identity~$\id\colon (\msP,\T_\mrma)\to(\msP,\T_\mrmn)$ is continuous.
In particular, it is continuous and bijective on the compact Hausdorff space~$(\mcK_n,\T_\mrma)$ with values in the Hausdorff space~$(\mcK_n,\T_\mrmn)$.
It is then a standard topological fact that~$\id\colon (\mcK_n,\T_\mrma) \to (\mcK_n,\T_\mrmn)$ is a homeomorphism, and that~$\mcK_n$ is $\T_\mrmn$-compact, hence in particular $\T_\mrmn$-closed.
Since $\seq{\mcK_n}_n$ is an $\widehat\mcE_\pi$-nest, it follows that~$\id\colon (\msP,\T_\mrma)\to(\msP,\T_\mrmn)$ is a quasi-homeomorphism. 

The $\T_\mrmn$-quasi-regularity then follows from the $\T_\mrma$-quasi-regularity, proved in Theorem~\ref{t:TransferQuasiHomeo}.
\end{proof}
\end{proposition}

In order to discuss some examples, let us make the following observation.

\begin{remark}\label{r:IndependenceUltraC}
Fix~$\pi\in\msP(\mbfT)$. Suppose that the form~$\tparen{\widehat\mcE_\pi,\hCylP{}{0}}$ defined by
\[
\widehat\mcE_\pi(u,v)\eqdef \int \widehat\mcG(u,v)\diff\mcQ_\pi\comma \qquad u,v\in \hCylP{}{0}\comma
\]
with~$\widehat\mcG$ as in~\eqref{eq:l:DiffusionP:0.3} is closable on~$L^2(\mcQ_\pi)$, and that its closure~$\tparen{\widehat\mcE_\pi, \dom{\widehat\mcE_\pi}}$ is a strongly local Dirichlet form.
In this case, the proof of Proposition~\ref{p:StandardCyl}\ref{i:p:StandardCyl:1}-\ref{i:p:StandardCyl:3}, and the proof of the identification~$\tparen{\mcE_\pi,\dom{\mcE_\pi}}=\tparen{\widehat\mcE_\pi,\dom{\widehat\mcE_\pi}}$ in~\ref{i:p:StandardCyl:4} hold without any further assumption.
That is, the proof is independent of Assumption~\ref{ass:Setting}, and in particular it is \emph{not} necessary that~$\mssH_\bullet$ be ultracontractive.
\end{remark}

\subsection{Examples}
Let us now introduce some classes of examples.

\subsubsection{Diffusions over manifolds and the Wasserstein gradient}\label{sss:Wasserstein}
Let~$(M,g)$ be a smooth complete connected boundaryless Riemannian manifold with Laplace--Bel\-trami operator~$\Delta_g$ and Riemannian volume measure~$\vol_g$.
When~$M$ is closed (i.e.\ additionally compact),~$\vol_g$ is a finite measure, and we may choose~$\nu$ to be its normalization.
We may thus apply the results in the previous section to construct the generator~$(\widehat\mcL_g,\hCylP{}{0})$ obtained by choosing~$(\mssL,\msA)= (\Delta_g,\mcC^\infty_0)$ in~\eqref{eq:l:DiffusionP:0.4}, and the corresponding carr\'e du champ operator~$(\widehat\mcG_g,\hCylP{}{0})$ obtained by choosing~$\mssGamma=\mssGamma_g\eqdef \abs{\nabla_g\emparg}^2$ in~\eqref{eq:l:DiffusionP:0.3}.

\begin{example}[The Dirichlet--Ferguson diffusion on~$\msP(M)$]\label{ex:DFDiffusion}
It is the main result in~\cite{LzDS17+} that the assertions in Theorems~\ref{t:TransferP} and~\ref{t:TransferQuasiHomeo} hold for the particular case when~$M$ is a closed Riemannian manifold of dimension~$d\geq 2$,~$(\mssL,\msA)=(\Delta_g,\Cb^\infty)$ is the Laplace--Beltrami operator, and~$\pi=\Pi_\beta$, with~$\beta>0$, is the Poisson--Dirichlet distribution in Example~\ref{e:PoissonDirichlet}, in which case~$\mcQ_\pi$ is the Dirichlet--Ferguson measure~$\DF{\beta\vol_g}$ with intensity~$\beta\vol_g$ in Example~\ref{ex:DirichletFerguson}.
Then,~$\widehat\mcW_{\Pi_\beta}$ is the \emph{Dirichlet--Ferguson diffusion on~$\msP(M)$} in~\cite{LzDS17+}, i.e.\ the free massive system of Wiener processes on~$M$ with $\Pi_\beta$-distributed masses.
We note that the closability of the form~$\tparen{\mcE_{\Pi_\beta},\dom{\mcE_{\Pi_\beta}}}$ in~\cite{LzDS17+} is shown by resorting to a Mecke-type identity~\cite{LzDSLyt17} for~$\DF{\beta\vol_g}$, which completely characterizes~$\DF{\beta\vol_g}$, and therefore does not hold for~$\pi\neq \Pi_\beta$.
Furthermore, the presentation in~\cite{LzDS17+} relies in a crucial way on the smooth structure of~$(M,g)$, and in particular on a natural action on~$\msP$ of the group of diffeomorphisms of~$M$.
\end{example}

When~$M$ is non-compact, we may take~\eqref{eq:l:DiffusionP:0.3} and~\eqref{eq:l:DiffusionP:0.4} as definitions of~$\widehat\mcG_g$ and~$\widehat\mcL_g$ respectively, thus extending the construction of these operators to the case of non-compact base spaces as well.
In this case however, the results in the previous sections do not apply, since $\Delta_g$ and~$\nu=\vol_g$ do not satisfy Assumption~\ref{ass:Setting}.

Nevertheless, let~$\rho\in \mcC^\infty\cap L^1_+$ be a probability density strictly positive everywhere on~$M$, set~$\nu\eqdef\rho\vol_g\in\msP$, and consider the drift Laplacian
\begin{gather*}
\mssL^\rho= \Delta_g-\diff(\log\rho)\,\nabla_g\comma
\end{gather*}
with carr\'e du champ operator~$\mssGamma_g$ as above.
Then, in light of~\eqref{eq:l:DiffusionP:0.3} and~\eqref{eq:l:DiffusionP:0.4}, we have
\[
\widehat\mcL_g^\rho = \widehat\mcL_g + \widehat\mcG_g\paren{\paren{\tfrac{1}{\id_\mbbI}\otimes\nabla_g\log\rho}^\trid, \emparg} \fstop
\]

In order to apply the results of the previous sections to this setting, we need to verify Assumption~\ref{ass:Setting}.

\begin{assumption}[Weighted Riemannian manifolds]\label{ass:WRM}
Let~$(M,g,\nu)$ be a connected weighted Riemannian manifold with weight~$\rho\in \mcC^\infty\cap L^1_+$ a smooth strictly positive probability density, and choose~$(\mssL,\msA)$ to be the operator~$(\mssL^\rho,\Cc^\infty)$.
Further assume that
\begin{enumerate}[$(a)$]
\item the Riemannian manifold~$(M,g)$ is (metrically) complete\footnote{As a consequence of the Hopf--Rinow Theorem (e.g.~\cite[p.~295]{Gri09}), this is equivalent to~\cite[Dfn.~11.1, p.~295]{Gri09}.};
\item the weighted Riemannian manifold~$(M,g,\nu)$ is stochastically complete;
\item the weight~$\rho$ is such that the heat kernel of~$(M,g,\nu)$ is ultracontractive.
\end{enumerate}
\end{assumption}

We refer to~\cite[\S11.4]{Gri09} for a thorough account of stochastic completeness and for sufficient conditions in terms of the volume growth of~$\nu$; to~\cite[\S14]{Gri09} for a thorough account of ultracontractivity on weighted manifolds and for sufficient conditions in terms of, e.g., Nash and Faber--Krahn inequalities; to~\cite[Cor.~2.2.8]{Dav89} for sufficient conditions for ultracontractivity in terms of the logarithmic Sobolev inequality with excess; to e.g.~\cite[\S3]{GuiZeg02} for sufficient conditions for ultracontractivity in terms of the Sobolev inequality.
Since~$\rho$ is smooth and everywhere positive, it is clear that~$\mssL^\rho\, \Cc^\infty\subset \Cc^\infty$.

It follows from ultracontractivity and the smoothness of the heat kernel that Assumption~\ref{ass:Setting}\ref{i:ass:Setting:c3} is satisfied.
Under the assumption of completeness, the operator~$(\mssL^\rho,\Cc^\infty)$ is essentially self-adjoint by~\cite[Thm.~11.5]{Gri09}\footnote{Note that~\cite{Gri09} indicates by the symbol~$C^\infty_0$ what we indicate by~$\Cc^\infty$.}.

\bigskip

Let us now show how the measure-representation of a massive particle system related to the geometry of $L^2$-optimal transportation.
We refer the reader to~\cite{LzDS17+,LzDS20,LzDSSod24} for a discussion of the interplay between this geometry and the action on~$\msP_2$ of diffeomorphisms of~$M$, which we briefly recall below.

\paragraph{Wasserstein spaces} Let~$\mssd_g$ be the intrinsic (Riemannian) distance induced on~$M$ by~$g$.
We denote by
\begin{subequations}\label{eq:WassersteinSp}
\begin{equation}
\msP_2\eqdef \set{\eta\in\msP: \int \mssd_g^2(o,\emparg)\diff\eta<\infty}
\end{equation}
the \emph{$L^2$-Wasserstein space} over~$M$, i.e.\ the set of all Borel probability measures on~$M$ with finite second $\mssd_g$-moment,
for some (hence any) reference point~$o$ in~$M$.
For~$i=1,2$, a \emph{coupling} of~$\eta_i\in\msP_2$ is a Borel probability measure~$\pi$ on~$M^\tym{2}$ with marginals~$\eta_i$.
We endow~$\msP_2$ with its natural \emph{$L^2$-Kantorovich--Rubinstein distance}
\begin{equation}
W_2(\eta_1,\eta_2)\eqdef \inf_\pi \sqrt{\pi( \mssd_g^2)}\comma\qquad \eta_1,\eta_2\in\msP_2\comma
\end{equation}
where the infimum runs over all couplings~$\pi$ of~$\eta_1,\eta_2$.
\end{subequations}

\paragraph{Diffeomorphisms}
Further let~$\Diff^\infty_c(M)$ be the group of smooth, orient\-ation-preserving and compactly non-identical diffeomorphisms of~$M$.
The group acts naturally and continuously on~$\msP_2$ by push-forward of measures, viz.
\begin{equation}\label{eq:DiffeoAction}
.\colon \Diff^\infty_c(M)\times \msP_2 \longrightarrow \msP_2\comma \qquad  .\colon(\psi,\eta)\longmapsto \psi_\pfwd \eta \fstop
\end{equation}

\paragraph{Group actions}
Let~$(\Omega,\msF,\mbbP)$ be a probability space, and~$G$ be a measurable group acting measurably on~$\Omega$ with action denoted by~$.\colon (h,\omega)\mapsto h.\omega$.

\begin{definition}[Partial quasi-invariance~{\cite[Dfn.~9]{KonLytVer15}}]\label{d:PQI}
We say that~$\mbbP$ is \emph{partially $G$-quasi-invariant} if there exists a $\mbbP$-negligible set~$N\subset \Omega$ and a filtration~$\msF_\bullet\eqdef\seq{\msF_t}_{t\in T}$ of~$\Omega$, indexed by a totally ordered set~$T$, such that
\begin{enumerate}[$(a)$]
\item the trace $\sigma$-algebra of~$\msF$ on~$\Omega\setminus N$ coincides with the $\sigma$-algebra generated by~$\msF_\bullet$;
\item for each~$g\in G$ and each~$s\in T$ there exists~$t\in T$ so that~$g.\msF_s\subset \msF_t$;
\item for each~$g\in G$ and~$s\in T$ there exists an $\msF_s$-measurable function~$R_s^g\colon \Omega\to [0,\infty]$ so that
\[
\int_\Omega u \diff h._\pfwd\mbbP = \int_\Omega R_s^g \diff \mbbP
\]
for each $\msF_s$-measurable semi-bounded~$u\colon \Omega\to [-\infty,\infty]$.
\end{enumerate}
\end{definition}

In the next proposition, let~$\mbbI$ be the unit interval with reverse natural order.
For each~$\eps>0$, further let~$\sigma(\hCylP{}{\eps})$ be the $\sigma$-algebra on~$\msP_2$ generated by the algebra of functions~$\hCylP{}{\eps}$ in~\eqref{eq:HCylinderFEps}.
It follows from Proposition~\ref{p:PropertiesCylinder}\ref{i:p:PropertiesCylinder:2} and~\ref{i:p:PropertiesCylinder:3} that~$\tseq{\sigma(\hCylP{}{\eps})}_{\eps\in\mbbI}$ is a filtration of the Borel $\sigma$-algebra of~$(\msP^\pa,\T_\mrmn)$.

\begin{theorem}\label{t:Wasserstein}
Let~$(M,g,\nu)$ be a weighted Riemannian manifold as in Assumption~\ref{ass:WRM}.
Assume further that~$\nu\in \msP_2$ and fix~$\pi\in\msP(\mbfT_\circ)$.
Then,
\begin{enumerate}[$(i)$]
\item\label{i:t:Wasserstein:1} $\mcQ_\pi$ is concentrated on~$\msP_2$, i.e.~$\mcQ_\pi(\msP\setminus\msP_2)=0$;
\item\label{i:t:Wasserstein:1.5} $\mcQ_\pi$ is partially quasi-invariant under the action~\eqref{eq:DiffeoAction} w.r.t.\ the filtration~$\tseq{\sigma(\hCylP{}{\eps})}_{\eps\in\mbbI}$, with Radon--Nikod\'ym derivative
\begin{equation}\label{eq:t:Wasserstein:0}
R^\psi_\eps(\eta)= \prod_{x\in \mu : \eta\set{x}>\eps} \frac{\diff\psi_\pfwd\nu}{\diff \nu}(x) \comma \qquad \eps\in\mbbI\comma \quad \eta\in\msP_2\comma \quad \psi\in \Diff^\infty_c(M)\fstop
\end{equation}
\end{enumerate}
Furthermore,
\begin{enumerate}[$(i)$, resume]
\item\label{i:t:Wasserstein:2} the form~$\tparen{\widehat\mcE_\pi,\FC{\msA}{}}$ in Proposition~\ref{p:StandardCyl} satisfies assertions~\ref{i:p:StandardCyl:1}-\ref{i:p:StandardCyl:3} in Proposition~\ref{p:StandardCyl};
\item\label{i:t:Wasserstein:3} its closure~$\tparen{\mcE_\pi,\dom{\mcE_\pi}}$ coincides with the Cheeger energy~$\Ch_{W_2,\mcQ_\pi}$ of the metric measure space~$(\msP_2,W_2,\mcQ_\pi)$;
\item\label{i:t:Wasserstein:4} $\tparen{\mcE_\pi,\dom{\mcE_\pi}}$ is quasi-regular for the $W_2$-topology;
\item\label{i:t:Wasserstein:5} $\tparen{\mcE_\pi,\dom{\mcE_\pi}}$ satisfies the Rademacher Theorem on~$(\msP_2,W_2)$. That is, every Lipschitz function~$u\colon \msP_2\to\R$ belongs to~$\dom{\mcE_\mcQ}$, is Fr\'echet differentiable $\mcQ$-a.e.\ and its Otto gradient~$\boldnabla u\colon \msP_2\to T\msP_2$ satisfies
\[
\norm{\boldnabla u}_{T_\eta\msP_2}\leq \Lip_{W_2}[u] \fstop
\]
\item\label{i:t:Wasserstein:6} the semigroup~$\mcH_\bullet$ associated with~$\tparen{\mcE_\pi,\dom{\mcE_\pi}}$ satisfies the following one-sided integral Varadhan short-time asymptotic estimate: for every pair of Borel subsets~$A_1,A_2\subset\msP_2$ with~$\mcQ_\pi A_1, \mcQ_\pi A_2>0$,
\begin{equation}\label{eq:t:Wasserstein:1}
-2\, \lim_{t\downarrow 0} t \log \scalar{\car_{A_1}}{\mcH_t \car_{A_2}}_{L^2(\mcQ_\pi)} \ \geq\ \mcQ_\pi\text{-}\essinf_{\eta_i\in A_i} W_2(\eta_1,\eta_2)^2 \comma
\end{equation}

\end{enumerate}
\begin{proof}
\ref{i:t:Wasserstein:1} is a simple consequence of choosing~$f=\mssd_g(o,\emparg)$ in Lemma~\ref{l:SimpleTridLp}.

By assumption, the density~$\nu=\rho\vol_g$ is continuous and everywhere positive.
Thus, since each~$\psi\in \Diff^\infty_c(M)$ is compactly non-identical, the Radon--Nikod\'ym derivative~$\frac{\diff\psi_\pfwd\nu}{\diff\nu}$ is continuous, and uniformly bounded away from~$0$ and infinity on~$M$.
This shows that the product in~\eqref{eq:t:Wasserstein:0} is well-defined, finite, and non-zero for every~$\eps>0$, every~$\eta\in\msP_2$, and every~$\psi\in\Diff^\infty_c(M)$.
As a consequence,~$\psi_\pfwd\nu$ too is an element of~$\msP_2$.
Finally, for each~$\mbfx\eqdef\seq{x_1,\dotsc, x_N}\in M^\tym{N}$ (including in the case~$N=\infty$) and every~$\psi\in\Diff^\infty_c(M)$, set~$\psi^\diamond(\mbfx)\eqdef \seq{\psi(x_1),\dotsc, \psi(x_N)}$.
Note that the action of~$\Diff^\infty_c(M)$ on~$\msP_2$ commutes with the measure representation~$\EM$ in~\eqref{eq:TransferMap}, and we have
\[
\psi_\pfwd \EM(\mbfs,\mbfx)= \EM(\mbfs, \psi^\diamond(\mbfx)) \qquad \text{and}\qquad {\psi_\pfwd}_\pfwd\mcQ_{\pi,\nu}= \mcQ_{\psi_\pfwd\nu,\pi}\fstop
\]
Then,~\ref{i:t:Wasserstein:1.5} holds similarly to the proof of the same assertion for the Dirichlet--Ferguson measure in~\cite[Prop.~5.20]{LzDS17+}.

\ref{i:t:Wasserstein:2} is shown in Proposition~\ref{p:StandardCyl}. 
\ref{i:t:Wasserstein:3} is a consequence of~\cite[Rmk.~6.3]{ForSavSod22}.
\ref{i:t:Wasserstein:4} is a consequence of the identification with~$\Ch_{W_2,\mcQ_\pi}$ and either: the quasi-regularity of Cheeger energies shown in~\cite[Thm.~4.1]{Sav14}\footnote{The assumption on synthetic Ricci-curvature lower bound in~\cite[Thm.~4.1]{Sav14} is not necessary and indeed not used in the proof.}; or a consequence of quasi-regularity statements for general Dirichlet forms~\cite[Prop.~3.21]{LzDSSuz21}, or~\cite[Thm.~3.4]{RoeSch95}. (We omit the details.)

\ref{i:t:Wasserstein:5}
The proof follows as in~\cite[Prop.~6.19]{LzDS17+}.
It suffices to note that the auxiliary results~\cite[Lem.~A.21 and Prop.~A.22]{LzDS17+} do not in fact depend on properties of the Dirichlet--Ferguson measure and thus hold with~$\DF{\mssm}$ (in the notation of~\cite{LzDS17+}) replaced by~$\mcQ_\pi$ for any~$\pi\in\msP(\mbbT_\circ)$.

\ref{i:t:Wasserstein:6} follows from~\cite[Lem.~4.16]{LzDSSuz20} and Hino--Ram\'irez' integral Varadhan short-time asymptotics for general Dirichlet spaces,~\cite[Thm.~1.1]{HinRam03}, or~\cite[Thm.~5.2]{AriHin05}.
\end{proof}
\end{theorem}

\begin{remark}
Assume~$\pi\neq \delta_\mbfs$ for every~$\mbfs\in\mbfT$. Then, the converse inequality to~\eqref{eq:t:Wasserstein:1} does not hold.
Indeed, in this case~$\mcH_\bullet$ is not ergodic as a consequence of Theorem~\ref{t:TransferP}\ref{i:t:TransferP:3}.
Thus, there exists a non-trivial $\mcH_\bullet$-invariant set~$A\subset \msP_2$ satisfying~$\mcQ_\pi A, \mcQ_\pi A^\complement>0$, and we have~$\scalar{\car_{A}}{\mcH_t\car_{A^\complement}}_{L^2(\mcQ_\pi)}=0$ for every~$t>0$, so that the left-hand side of~\eqref{eq:t:Wasserstein:1} is infinite for~$A_1=A$ and~$A_2=A^\complement$.
\end{remark}

\subsubsection{Breaking local conservation of mass and gaining ergodicity}\label{sss:Bhattacharya}
Massive particle systems are heavily constrained by their intrinsic pointwise conservation of mass:
as soon as Assumption~\ref{ass:QPP} is satisfied, the mass of each particle is constant for all times.
Equivalently, there is no dynamics in the mass component, which is the very reason of the ergodic decomposition in Theorem~\ref{t:TransferP}\ref{i:t:TransferP:3}.

While intrinsic to (the measure representation of) massive particle systems, the lack of ergodicity is undesirable.
However, it would not be difficult to combine the dynamics of positions in Theorem~\ref{t:BendikovSaloffCoste} with natural dynamics on the space of masses, resulting in non-trivial \emph{irreducibile} dynamics on~$\msP$.

Indeed, the Dirichlet form $\tparen{\mcE_\pi,\dom{\mcE_\pi}}$ in Theorem~\ref{t:Wasserstein} can be combined with other known Dirichlet forms on~$\msP$.
Two examples are the form~$\tparen{\mcE^\FleVio,\dom{\mcE^\FleVio}}$ associated to the \emph{Fleming--Viot process with parent-independent mutation} introduced in~\cite{FleVio79}, and the form $\tparen{\mcE^\DawWat,\dom{\mcE^\DawWat}}$ of the \emph{Dawson--Watanabe super-process}, see, e.g., the monograph~\cite{Per02}.
Since these examples are very similar in spirit, we only discuss the first one, for which a thorough description via Dirichlet form is available, see~\cite{OveRoeSch95}.
In order to use the results in~\cite{OveRoeSch95}, we restrict our attention to case when~$\pi=\Pi_\beta$ is the Poisson--Dirichlet distribution, and~$\mcQ_\pi=\DF{\nu}$ is the Dirichlet--Ferguson measure with intensity~$\nu$ (see Ex.~\ref{ex:DFDiffusion}).

\begin{example}\label{ex:DF+FV}
Let~$(M,g,\nu)$ be a weighted Riemannian manifold as in Assumption~\ref{ass:WRM}.
Then, the form
\[
\tparen{\mcE_{\DF{\nu}} + \mcE^\FleVio,\FC{\msA}{}}
\]
is closable on~$L^2(\DF{\nu})$.
Its closure~$\tparen{\mcE^\BatKan,\dom{\mcE^\BatKan}}$ is a conservative strongly local Dirichlet form on~$L^2(\DF{\nu})$, quasi-regular for the narrow topology~$\T_\mrmn$.
The properly associated Markov process is a $\T_\mrmn$-diffusion~$\mu_\bullet^\BatKan$ with state space~$\msP^\pa$.

\begin{proof}
We note that~$\FC{\msA}{}\subset \dom{\mcE^\FleVio}$, where~$\dom{\mcE^\FleVio}$ is defined in~\cite[\S2]{OveRoeSch95}, since~$\FC{\msA}{}$ is contained in the algebra of cylinder functions in~\cite[Eqn.~(2.7)]{OveRoeSch95}.
It follows that both~$\mcE_{\DF{\nu}}$ and~$\mcE^{\FleVio}$ are closable on~$\FC{\msA}{}$, so that their sum is closable by standard arguments, e.g.~\cite[Prop.~I.3.7]{MaRoe92}.
Conservativeness holds since~$\car\in L^2(\DF{\nu})$ and~$\mcE^\BatKan(\car)=0$.
Strong locality follows from the diffusion property of both~$\tparen{\mcE_{\DF{\nu}},\dom{\mcE_{\DF{\nu}}}}$ and~$\tparen{\mcE^\FleVio,\dom{\mcE^\FleVio}}$, since~$\FC{\msA}{}$ is a Markovian form core by definition.

Since~$\FC{\msA}{}$ is a form core for~$\tparen{\mcE^\BatKan,\dom{\mcE^\BatKan}}$, the latter satisfies Definition~\ref{d:QuasiReg}\ref{i:d:QuasiReg:2} and~\ref{i:d:QuasiReg:3} for~$\T_\mrmn$ since~$\FC{\msA}{}\subset \Cb(\T_\mrmn)$ by Remark~\ref{r:ExtensionPropertiesCA}.
Further recall that the form~$\tparen{\mcE_{\DF{\nu}},\dom{\mcE_{\DF{\nu}}}}$ is $\T_\mrmn$-quasi-regular by Proposition~\ref{p:StandardCyl}\ref{i:p:StandardCyl:4} and note that the form~$\tparen{\mcE^\FleVio,\dom{\mcE^\FleVio}}$ too is $\T_\mrmn$-quasi-regular by~\cite[Thm.\ 3.5(i)]{OveRoeSch95}.
The existence of a $\T_\mrmn$-compact $\mcE^\BatKan$-nest as in Definition~\ref{d:QuasiReg}\ref{i:d:QuasiReg:1} follows from the existence of analogous nests for~$\tparen{\mcE_{\DF{\nu}},\dom{\mcE_{\DF{\nu}}}}$ and~$\tparen{\mcE^\FleVio,\dom{\mcE^\FleVio}}$.
It follows that~$\tparen{\mcE^\BatKan,\dom{\mcE^\BatKan}}$ is $\T_\mrmn$-quasi-regular.
\end{proof}

The importance of the form~$\tparen{\mcE^\BatKan,\dom{\mcE^\BatKan}}$ arises from the following considerations.
Recall that, under Assumption~\eqref{eq:QPP}, the process~$\mu_\bullet$ properly associated with~$\tparen{\mcE_{\DF{\nu}},\dom{\mcE_{\DF{\nu}}}}$ is the measure representation of a massive particle system.
In particular, $\mu_\bullet$ is a.s.\ a purely atomic measure, the masses of its atoms are constant in time, and the locations of its atoms move diffusively in~$M$ according to the driving noise~$\mssW$.
On the contrary, the Fleming--Viot process~$\mu^\FleVio_\bullet$ properly associated with~$\tparen{\mcE^\FleVio,\dom{\mcE^\FleVio}}$ is a.s.\ a purely atomic measure, the locations of its atoms are fixed, while the masses of its atoms move diffusively in~$[0,1]$.
As a consequence, the presence of the `vertical component'~$\mcE^\FleVio$ in~$\mcE^\BatKan$ breaks the shape of $\mcE_{\DF{\nu}}$-invariant sets as classified in Theorem~\ref{t:DirInt}\ref{i:t:DirInt:5}.
This strongly suggests that~$\tparen{\mcE^\BatKan,\dom{\mcE^\BatKan}}$ has no non-trivial invariant sets, and thus that ---in sharp contrast with both~$\tparen{\mcE_{\DF{\nu}},\dom{\mcE_{\DF{\nu}}}}$ and~$\tparen{\mcE^\FleVio,\dom{\mcE^\FleVio}}$--- it is \emph{irreducible} (ergodic).

Secondly, similarly to the case of~$\tparen{\mcE_{\DF{\nu}},\dom{\mcE_{\DF{\nu}}}}$, it is possible to identify the geometry on~$\msP$ associated with~$\tparen{\mcE^\BatKan,\dom{\mcE^\BatKan}}$ as the \emph{Bhattacharya--Kantorovich geometry}, defined analogously to the \emph{Hellinger--Kantorovich} (also: \emph{Wasserstein--Fisher--Rao}) \emph{geometry}~\cite{LieMieSav17, LieMieSav22,ChiPeySchVia16,KonMonVor16}.
Indeed, it follows from the definition that~$\tparen{\mcE^\BatKan,\dom{\mcE^\BatKan}}$ has carré du champ operator given by the sum of the carr\'e du champ operators of its two components.
The carré du champ of~$\tparen{\mcE_{\DF{\nu}},\dom{\mcE_{\DF{\nu}}}}$ is precisely Otto's Riemannian metric on~$\msP$ computed at $L^2$-Wasserstein gradients, and thus induces the $L^2$-Kantorovich--Rubinstein distance.
The carré du champ of~$\tparen{\mcE^\FleVio,\dom{\mcE^\FleVio}}$ is identified in~\cite[Eqn.~(2.13)]{OveRoeSch95}; the induced intrinsic distance is identified in~\cite[Thm.~1.1]{Sch97} as the \emph{Bhattacharya distance}
\[
\delta_{\mathsc{B}}(\mu_1,\mu_2) \eqdef \arccos \int \sqrt{\frac{\diff \mu_1}{\diff\sigma} \frac{\diff\mu_2}{\diff\sigma}}\diff\sigma \comma \qquad \text{for any $\sigma$ with $\mu_1,\mu_2\ll\sigma$}\fstop
\]
By analogy with the case of the Hellinger--Kantorovich geometry, thoroughly discussed in~\cite{LzDSSod24}, and by the characterization~\cite{DePSodTam25} of the Hellinger--Kantorovich distance as a metric inf-convolution, these facts suggest that the intrinsic distance induced by the carr\'e du champ of~$\tparen{\mcE^\BatKan,\FC{\msA}{}}$ is the metric inf-convolution of the Bhattacharya distance with the $L^2$-Kantorovich--Rubinstein distance.
\end{example}

\subsubsection{The geometry on probability measures induced by a Markov diffusion}\label{sss:Lifting}
Throughout the previous sections, we linked together objects (forms, semigroups, generators) on infinite-product spaces with objects on spaces of probability measures.
The main purpose of this link is to provide a robust way to identify the Markov process~$\mu_\bullet$ properly associated with the Dirichlet form~$\tparen{\widehat\mcE_\pi,\dom{\widehat\mcE_\pi}}$.

On the one hand, this identification is much stronger than the abstract one provided by the general theory of Dirichlet forms: it is explicit, and allows us to construct the process from \emph{every} starting point in~$\msP^\pa$ (as opposed to $\widehat\mcE_\pi$-\emph{quasi-every} starting point, as in the general theory).

On the other hand, the identification comes at the expense of generality: firstly, because we need the driving noise to be ergodic, recurrent and ---foremost--- ultracontractive and converging exponentially fast to equilibrium, and secondly because we need the invariant measure~$\mcQ_\pi$ to be of the specific form~$\mcQ_\pi\eqdef \EM_\pfwd\widehat\boldnu_\pi$ in~\eqref{eq:DefGeneratorCylP} for some \emph{probability} measure~$\nu$, which forces~$\mcQ_\pi$ to be itself a probability measure.

\medskip

However, if we only focus on Dirichlet-form theory, then we may formulate a far more general version of Theorem~\ref{t:TransferQuasiHomeo}, lifting practically any Markov diffusion operator to a generator on~$\msP$.

\begin{theorem}\label{t:GeneralForm}
Let~$(M,\T)$ be a metrizable Lusin topological space endowed with a $\sigma$-finite Borel measure~$\sigma$, and~$\msP$ be the space of all Borel probability measures on~$M$ endowed with its usual narrow topology.
Further let
\begin{enumerate*}[$(a)$, leftmargin=2em]
\item\label{i:t:GeneralForm:A} $\msA\subset L^2(\sigma)\cap \Cb(\T)$ be an algebra of continuous functions, dense in~$L^2(\sigma)$;
\item\label{i:t:GeneralForm:B} $\tparen{\mssL,\dom{\mssL}}$ be a self-adjoint Markov diffusion operator on~$L^2(\sigma)$ such that~$\mssL\msA\subset\msA$;
\item\label{i:t:GeneralForm:C} $\hFC{\msA}{0}$ be the algebra of cylinder functions on~$\msP$ defined as in Definition~\ref{d:StandardFC};
\item\label{i:t:GeneralForm:D} $\mcS$ be a $\sigma$-finite Borel measure on~$\msP$ such that~$\tclass[\mcS]{\hFC{\msA}{0}}\subset L^2(\mcS)$.
\end{enumerate*}

Then,
\begin{enumerate}[$(i)$, leftmargin=2em]
\item\label{i:t:GeneralForm:1} the operator~$\tparen{\widehat\mcL, \hFC{\msA}{0}}$ defined by the right-hand side of~\eqref{eq:l:DiffusionP:0.4} takes values in~$L^2(\mcS)$.
\end{enumerate}

Further assume that the measure~$\mcS$ is such that
\begin{enumerate}[$(a)$, leftmargin=2em]\setcounter{enumi}{4}
\item\label{i:t:GeneralForm:E} $\hFC{\msA}{0}$ is dense in~$L^2(\mcS)$
\item\label{i:t:GeneralForm:F} the operator~$\tparen{\widehat\mcL, \hFC{\msA}{0}}$ is $L^2(\mcS)$-symmetric.
\end{enumerate}

Then, the following assertions hold:
\begin{enumerate}[$(i)$, leftmargin=2em]\setcounter{enumi}{1}
\item\label{i:t:GeneralForm:2} the quadratic form~$\tparen{\widehat\mcE,\hFC{\msA}{0}}$ defined by~$\widehat\mcE(u,v)\eqdef \ttscalar{u}{-\widehat\mcL v}_{L^2(\mcS)}$ is closable on~$L^2(\mcS)$ and its closure~$\tparen{\widehat\mcE,\dom{\widehat\mcE}}$ is a Dirichlet form with generator the Friedrichs extension of~$\tparen{\widehat\mcE,\hFC{\msA}{0}}$;

\item\label{i:t:GeneralForm:3} the Dirichlet form~$\tparen{\widehat\mcE,\dom{\widehat\mcE}}$ admits carr\'e du champ operator
\[
\widehat\mcG(u,v)_\eta= \int \mssGamma^z\restr{z=x} \tparen{u(\eta+\eta_x\delta_z-\eta_x\delta_x), u(\eta+\eta_x\delta_z-\eta_x\delta_x)} \diff\eta(x) \comma\quad u,v\in \hFC{\msA}{0}\fstop
\]
\end{enumerate}
\end{theorem}

Let us first comment about the statement of the theorem.
The reader will note that conditions~\ref{i:t:GeneralForm:A}-\ref{i:t:GeneralForm:C} merely describe the setting: the only true assumptions (on~$\mcS$) are~\ref{i:t:GeneralForm:D},~\ref{i:t:GeneralForm:E}, and ---foremost---~\ref{i:t:GeneralForm:F}.
The first part of assumption~\ref{i:t:GeneralForm:D} requires $\mcS$-classes of Borel functions in~$\hFC{\msA}{0}$ to be $\mcS$-square-integrable.
This is trivially satisfied as soon as~$\mcS$ is a finite measure, since functions in~$\hFC{\msA}{0}$ are uniformly bounded.
The second part of assumption~\ref{i:t:GeneralForm:D} requires the algebra of all such classes to be sufficiently large.
This requirement is non-trivial and poses great limitations on sets of measures on which~$\mcS$ is concentrated.
For instance, if~$\mcS$ is concentrated on atomless measures, then~$\tclass[\mcS]{\hFC{\msA}{0}}$ reduces to the zero function.
However, if~$\mcS$ is a finite measure concentrated on~$\msP^\pa$, then~\ref{i:t:GeneralForm:D} holds in light of Proposition~\ref{p:PropertiesCylinder}\ref{i:p:PropertiesCylinder:6}.

Assumption~\ref{i:t:GeneralForm:F} is certainly the most delicate, and the hardest to verify in practice, as it relates~$\mcS$ with~$\sigma$.
We note, however, that (since~$\mssL$ is $L^2(\sigma)$-symmetric by assumption~\ref{i:t:GeneralForm:B}) we can always choose~$\sigma=\nu$ and~$\mcS=\mcQ_{\pi,\nu}$, irrespectively of all other assumptions on~$\mssL$, hence in particular irrespectively of ergodicity, conservativeness, spectral gap, ultracontractivity, etc.
This will be relevant in light of Example~\ref{ex:OU} below.

\begin{proof}[Proof of Theorem~\ref{t:GeneralForm}]
Since~$\mssL\msA\subset\msA$ by~\ref{i:t:GeneralForm:B}, we have~$\mssGamma (\msA^\tym{2})\subset \msA$ as well.

\paragraph{Proof of~\ref{i:t:GeneralForm:1}}
By the diffusion property for~$(\mssL,\msA)$, it follows from Lemma~\ref{l:DiffusionP} that we have the representation~\eqref{eq:l:DiffusionP:0.2}.
Thus, it follows from this representation that~$\widehat\mcL\, \hFC{\msA}{0}\subset \hFC{\msA}{0}\subset L^2(\mcS)$, where the last inclusion holds by~\ref{i:t:GeneralForm:D}.

\paragraph{Proof of~\ref{i:t:GeneralForm:2} and~\ref{i:t:GeneralForm:3}}
Since~$\tparen{\widehat\mcL, \hFC{\msA}{0}}$ is densely defined by~\ref{i:t:GeneralForm:D}, the closability assertion and the qualification of the generator as a Friedrichs extension are a consequence of~\ref{i:t:GeneralForm:1} and standard arguments,~\cite[Thm.~X.23]{ReeSim75}.
Since~$(\mssGamma,\msA)$ is the carr\'e du champ of~$(\mssL,\msA)$, assertion~\ref{i:t:GeneralForm:3} follows from~\ref{i:t:GeneralForm:2} by a direct computation.
\end{proof}

\begin{example}[The Ornstein--Uhlenbeck process on~$\msP_2(\R^d)$]\label{ex:OU}
Recall the notation for weighted Riemannian manifolds in Assumption~\ref{ass:WRM}, and choose~$M=\R^d$, $d\geq 2$, and~$\rho$ the standard Gaussian density, so that~$\nu\eqdef\rho\Leb^d=\gamma_d$ is the standard normal distribution on~$\R^d$, an element of~$\msP_2(\R^d)$.
In this case,
\[
\mssL^\textsc{ou}\eqdef\mssL^\rho=\Delta-x\cdot\nabla\comma
\]
and the corresponding Markov process~$\mssW$ is the standard $d$-dimensional Ornstein--Uhlenbeck process.
Since~$d\geq 2$, the process~$\mssW$ satisfies~\eqref{eq:QPP} by Proposition~\ref{p:VerificationPolarityManifolds}.
However, note that the corresponding heat semigroup~$\mssH_\bullet$ is \emph{not} ultracontractive, see e.g.~\cite[Thm.~4]{Nel73} or~\cite[p.~13]{LunMetPal20} on~$L^1(\nu)$.

\medskip

Now, the operator~$\tparen{\widehat\mcL, \hFC{\msA}{0}}$ defined by the right-hand side of~\eqref{eq:l:DiffusionP:0.4} is
\[
\widehat\mcL^\textsc{ou} = \widehat\mcL_\mathrm{eucl} + \widehat\mcG_\mathrm{eucl}\paren{\paren{\tfrac{1}{\id_\mbbI}\otimes\id_{\R^d}}^\trid, \emparg} \fstop
\]
For every~$\pi\in\msP(\mbfT)$, we may apply Theorem~\ref{t:GeneralForm} with~$\mcS=\mcQ_{\pi,\nu}$ and~$\nu=\gamma_d$ as above, to conclude that the form~$\tparen{\widehat\mcE^\textsc{ou}_\pi,\FC{\msA}{0}}$ is closable. 
As usual, we denote its closure by~$\tparen{\widehat\mcE^\textsc{ou}_\pi,\widehat\mcE^\textsc{ou}_\pi}$.

Therefore, by the identification in Proposition~\ref{p:StandardCyl}\ref{i:p:StandardCyl:4} (which does not require the ultracontractivity of~$\mssH^\textsc{ou}_\bullet$, see Rmk.~\ref{r:IndependenceUltraC}), the Dirichlet form $\tparen{\widehat\mcE^\textsc{ou}_\pi,\dom{\widehat\mcE^\textsc{ou}_\pi}}$ coincides with~$\tparen{\mcE^\textsc{ou}_\pi,\dom{\mcE^\textsc{ou}_\pi}}$ in Proposition~\ref{p:StandardCyl}\ref{i:p:StandardCyl:2}.
In turn, applying Theorem~\ref{t:Wasserstein}\ref{i:t:Wasserstein:3} and~\ref{i:t:Wasserstein:4} (which again do not require the ultracontractivity of~$\mssH^\textsc{ou}_\bullet$), the form~$\tparen{\mcE^\textsc{ou}_\pi,\dom{\mcE^\textsc{ou}_\pi}}$ is $\T_\mrmn$-quasi-regular.
Thus, there exists a properly associated Markov $\T_\mrmn$-diffusion~$\widehat\mcW_\pi^\textsc{ou}$, which we call the \emph{Ornstein--Uhlenbeck process on~$\msP_2(\R^d)$}.

Since~$\mssH^\textsc{ou}_\bullet$ is not ultracontractive, we may not directly apply the theory of massive particle systems developed in \S\ref{s:FIS} to identify~$\widehat\mcW_\pi^\textsc{ou}$.
However, since~$\mssH^\textsc{ou}_\bullet$ is represented by a continuous (though unbounded) heat kernel, it is nonetheless reasonable to expect that~$\widehat\mcW_\pi^\textsc{ou}$ is a free massive system of Ornstein--Uhlenbeck processes on~$\R^d$ with $\pi$-distributed masses.

\medskip

We remark that other Dirichlet forms/processes on~$\msP_2(\R^d)$ have also been called Ornstein--Uhlenbeck forms/processes on~$\msP$,~e.g.~\cite{RenWan24}.
These constructions do not coincide, since the chosen invariant measure in~\cite{RenWan24} is concentrated on measures with densities, and thus it is singular with respect to~$\mcQ_\pi$ for every choice of~$\pi$.
\end{example}

\section{Free massive systems: spde theory}\label{s:DK}
In this section we give a meaning to the Dean--Kawasaki-type \textsc{spde}~\eqref{eq:IntroSPDE} and prove that it is well-posed, i.e.\ that existence and uniqueness of solutions holds, in a suitable sense.
We further identify solutions to~\eqref{eq:IntroSPDE} as the empirical measures of free massive particle systems described in the previous section.

\subsection{Equivalence of martingale problems}
We start with a brief review of~\cite{KonLehvRe19}, in which solutions to~\eqref{eq:DK0} are \emph{defined} as solutions to a suitable martingale problem.
We then adapt the definition of such martingale problem to our setting, and show that solutions to this new martingale problem coincide with the measure representation of free massive particle systems.

\subsubsection{Konarovskyi–Lehmann–von Renesse martingale problem}\label{sss:KLvR}
Rigorous well-pos\-ed\-ness and triviality results for~\ref{eq:DK} were recently achieved by V.~Konarov\-sky, T.~Lehmann, and M.-K.~von Renesse in~\cite{KonLehvRe19, KonLehvRe19b}.
In fact, they classified solutions to~\ref{eq:DK} in the strongest possible way and in the very general setting of standard Markov triples satisfying Ricci-curvature lower bounds.

Let us recall the concept of solutions to~\ref{eq:DK} in~\cite{KonLehvRe19}, expressed in terms of a corresponding martingale problem.
We call this the \emph{shadow martingale problem} in order to distinguish it from the `true' martingale problem on~$\msP$ in Theorem~\ref{t:TransferQuasiHomeo}\ref{i:t:TransferP:6}, as we will explain later on.
Precisely, let~$M$ be a Polish space,~$\nu$ be a Radon measure on~$M$ (not necessarily finite), and~$\msP$ denote the space of all Borel probability measures on~$M$. 
Further let~$(\mssL,\msA)$ be an essentially self-adjoint diffusion operator on~$L^2(\nu)$ with carr\'e du champ
\[
\mssGamma(f,g)\eqdef \mssL (fg)- f\,\mssL g - g\,\mssL f \comma \qquad f,g\in\msA\fstop
\]
Finally, let~$\tparen{\Omega,\msF,\msF_\bullet,\mbbP}$ be a stochastic basis satisfying the usual conditions.

\begin{definition}[Shadow martingale problem,~{\cite[Dfn.~2.1]{KonLehvRe19}, cf.~\cite[Dfn.~1]{KonLehvRe19b}}]\label{d:ProjMP}
A $\msP$-valued process with $\mbbP$-a.s.\ $\T_\mrmn$-continuous trajectories~$\rho_\bullet$ is a \emph{distributional solution} to the \emph{shadow martin\-gale problem~$\martp{\alpha}{\mssL}$ with parameter~$\alpha\in\R$} if
\begin{enumerate}[$(i)$]
\item\label{i:d:ProjMP:1} for each $f\in\msA$
\[
M^{\class{f}}_t\eqdef \rho_t f-\rho_0 f -\alpha \int_0^t \rho_s(\mssL f) \diff s
\]
is an $\msF_\bullet$-adapted $\mbbP$-martingale,

\item\label{i:d:ProjMP:2} with quadratic variation
\[
\quadvar{M^{\class{f}}}_t=\int_0^t \rho_s \mssGamma(f) \diff s \fstop
\]
\end{enumerate}
\end{definition}

It is one main intuition of~\cite{KonLehvRe19} that, choosing~$(\mssL,\msA)$ as the standard Laplacian~$(\Delta,\mcC^\infty)$, then~$\rho_\bullet$ is a solution to~$\martp{\alpha}{\Delta}$ if and only if it is a distributional solution to~\ref{eq:DK}.
As in~\cite{KonLehvRe19}, one may thus consider the analogous martingale problem~$\martp{\alpha}{\mssL}$ for a more general operator~$(\mssL,\msA)$ as above, and call a solution~$\rho_\bullet$ to~$\martp{\alpha}{\mssL}$ a $\msP$-valued distributional solution to the \emph{$\mssL$-driven Dean--Kawasaki equation with~$\IE$-interaction}.
At this stage however, no \textsc{spde} (not even at the formal level) is defined for~$\rho_\bullet$.

\paragraph{Ill-posedness vs.\ triviality}
We note straight away that~$\martp{\alpha}{\mssL}$ is \emph{not} a martingale problem in the usual sense.
In particular it is not the martingale problem for any transpose~$\mssL'$ of~$\mssL$ on `distributions', nor it is the martingale problem for the $L^2$-adjoint operator~$\mssL^*$ of~$\mssL$, even when~$\rho_\bullet\ll \nu$ for all times.
For this reason, and due to the presence of the parameter~$\alpha$ (which is not part of the generator), the specification of the quadratic variation in~\ref{i:d:ProjMP:2} is necessary and cannot be omitted.

As it will turn out (see Theorem~\ref{t:DK} below), the shadow martingale problem~$\martp{\alpha}{\mssL}$ is rather the `projection' to~$L^2(\nu)$ of a true martingale problem for some self-adjoint operator~$\tparen{\mcL,\dom{\mcL}}$ induced by~$\mssL$ on the $L^2$-space of some measure~$\mcQ$ on~$\msP$.

We recall the main result in~\cite{KonLehvRe19}.
We refer the reader to~\cite[\S3.4.2, p.~170]{BakGenLed14} for the definition of \emph{standard Markov triple}, and to~\cite[Eqn.~(2.3)]{KonLehvRe19} for the one of \emph{Bakry--\'Emery Ricci-curvature lower bound} for a standard Markov triple.
 
\begin{theorem}[Ill-posedness vs.\ triviality~{\cite[Thm.~2.2]{KonLehvRe19}}]\label{t:KLR}
Assume \linebreak $(M,\nu,\msA,\mssL,\mssGamma)$ is a standard Markov triple satisfying a Bakry--\'Emery Ricci-curvature lower bound.
Then, distributional solutions~$\rho_\bullet$ to~$\martp{\alpha}{\mssL}$ exist if and only if~$\alpha=n$ is a non-negative integer and the (possibly random) initial value~$\rho_0$ consists of exactly~$n$ atoms of \emph{equal} mass $\mbbP$-a.s., viz.~$\rho_0=\frac{1}{n}\sum_i^n \delta_{X^i_0}$.

In this case, the distributional solution is unique in law and it is given by the empirical measure~$\rho_t=\frac{1}{n}\sum_i^n \delta_{X^i_{nt}}$ of~$n$ independent diffusion processes, each with generator~$\tparen{\mssL,\dom{\mssL}}$ and (random) starting point~$X^i_0$.
\end{theorem}

When~$M=\R^d$ and~$\mssL$ is the standard Laplacian, it was recently shown by V.V.~Konarovskyi and F.~M\"uller in~\cite{KonMue23} that the very same `ill-posedness vs.\ triviality' holds for~$\martp{\alpha}{\Delta}$ holds even in the larger class of~$\mcS'(\R^d)^+$-valued solutions.

\subsubsection{Extending the generator and a new shadow martingale problem}
In the next Section~\ref{ss:DKNoise} we will give a rigorous construction for the divergence~$\mssdiv$, the exterior differential~$\mssd$, and the noise~$\xi$ appearing in~\eqref{eq:IntroSPDE}, in the setting of Assumption~\ref{ass:SettingLocal}.
In itself, the construction of these objects is not sufficient to give a meaning to~\eqref{eq:IntroSPDE}, due to the presence of the square root of the solution in the noise term.

Following the Konarovsky--Lehmann--von Renesse approach, we will therefore \emph{define} solutions of~\eqref{eq:IntroSPDE} to be solutions to a shadow martingale problem defined similar to~$\martp{\alpha}{\mssL}$ but with the different class of test functions~$\msR_0\otimes\msA$ in place of~$\msA$.
The identification of formal solutions to~\eqref{eq:IntroSPDE} with solutions to this new martingale problem is justified by a heuristic computation identifying the quadratic variation of both these types of solutions, which we will present in this setting at the end of \S\ref{sss:WhiteNoiseFormalDerivation} below.

\paragraph{Dualities and dual generators}
Let~$\msA$ be as in Assumption~\ref{ass:Setting}\ref{i:ass:Setting:c}, and recall that~$\mssL\msA\subset \msA$ by Assumption~\ref{ass:Setting}\ref{i:ass:Setting:c2}.

\begin{definition}
We denote by~$\mbbA_2$ the space~$\msA$ endowed with the $\mssL$-graph $\Cz$-norm
\[
\norm{f}_{\mbbA_2}\eqdef \norm{\mssL f}_0 + \norm{f}_0\comma \qquad f\in\msA\comma
\]
and by~$(\mbbA_2',\norm{\emparg}_{\mbbA_2'})$ the Banach dual of~$\mbbA_2$.
\end{definition}

Further let~$\Mb(\T)$ be the space of all bounded Borel signed measures on~$M$, endowed with the total-variation norm.
Since~$(M,\T)$ is locally compact Hausdorff,~$\Mb(\T)$ is (identified with) the Banach dual of~$\Cz(\T)$ by Riesz--Markov--Kaku\-tani's Theorem. Thus, every~$\mu\in\Mb(\T)$ may be regarded as a continuous linear functional on~$\mbbA_2$, i.e.\ an element of~$\mbbA_2'$.
Since~$\mbbA_2$ is $\Cz$-dense in~$\Cz(\T)$ by Assumption~\ref{ass:Setting}\ref{i:ass:Setting:c1}, the above identification is in fact an embedding~$\Mb(\T) \hookrightarrow \mbbA_2'$, i.e.\ it is continuous and injective.
Finally, since~$\mssL\colon \mbbA_2\to \mbbA_2$ is a bounded operator, it has a well-defined (distributional) transpose~$\mssL'\colon \mbbA_2'\to\mbbA_2'$.
In particular,~$\mssL'$ is well-defined on~$\Mb(\T)$.

Now, let~$\rho_\bullet$ be the solution to~$\martp{\alpha}{\mssL}$ in Theorem~\ref{t:KLR}.
We observe that
\begin{equation}\label{eq:Observation}
\rho_t (n\, \mssL f)= n\sum_{x\in \rho_t} n^{-1}\delta_x(\mssL f) = \paren{\sum_{x\in \rho_t} \mssL'\delta_x} f \comma \qquad f\in\msA\comma
\end{equation}
where, for any~$\rho\in \Mb(\T)$, we write~$x\in \rho$ to indicate that~$\rho\set{x}>0$, i.e.\ that~$\rho$ has an atom at~$x$.
We may then define an operator~$\tparen{\mcL',\Mb(\T)}$ by
\begin{equation}\label{eq:ExtensionGeneratorCyl}
\mcL' \colon \rho \longmapsto \sum_{x\in \rho} \mssL'\delta_x \fstop
\end{equation}

On the one hand, it is clear that~$\mcL'\rho$ does not operate on~$\msA$, since~$(\mcL'\rho) f$ may be infinite on~$(\rho, f)\in \Mbp(\T)\times\msA^+$ and thus~$\mcL'\rho$ may be ill-defined.
In particular, this is always the case when the set of atoms of~$\rho$ is $\T$-dense in~$M$ and~$f$ is not $\mssL$-harmonic.

On the other hand however, letting~$\msR_0$ be as in~\eqref{eq:MassAlgebra}, it is readily seen that~$\mcL'\rho$ operates on~$\msR_0\otimes \msA$ as the linear extension of the operator defined on elementary tensors by
\[
\mcL'\rho \colon \molli\otimes f \longmapsto \sum_{x\in \rho} \molli(\rho_x)\, (\mssL f)(x) \comma \qquad \molli\in\msR_0\comma f\in\msA\comma
\]
and indeed, recalling~\eqref{eq:ExtendedLonElementaryCyl} and~\eqref{eq:DefGeneratorCylP}, we always have the well-defined expression
\begin{equation}\label{eq:DualityGenerator}
(\mcL'\rho) (\molli\otimes f) = \widehat\mcL \tparen{(\molli\otimes f)^\trid}_\rho \comma \qquad \molli\in\msR_0\comma f\in\msA \fstop
\end{equation}

\paragraph{Extended shadow martingale problem}
Let us now introduce another shadow martingale problem.
Let $\tparen{\Omega,\msF,\msF_\bullet,\mbbP}$ be a stochastic basis satisfying the usual conditions.

\begin{definition}[Shadow martingale problem for the extended generator]\label{d:ShadowMP}
A $\msP$-valued process~$(\rho_\bullet,\mbbP)$ with $\mbbP$-a.s.\ $\T_\mrmn$-continuous trajectories is a \emph{distributional solution} to the \emph{shadow martingale problem~$\hmartp{}{\mssL}$} if
\begin{equation}\label{eq:d:ShadowMP:0}
\widehat M^{\class{\molli\otimes f}}_t\eqdef\ (\molli\otimes f)^\trid(\rho_t) - (\molli\otimes f)^\trid(\rho_0) -\int_0^t (\mcL'\rho_s)(\molli\otimes f)\diff s
\end{equation}
is an $\msF_\bullet$-adapted $\mbbP$-martingale for each $\molli\in\msR_0$ and each~$f\in\msA$.
\end{definition}

For comparison with~\cite{KonLehvRe19,KonLehvRe19b}, let us note that ---in contrast to the Definition~\ref{d:ProjMP} of~$\martp{\alpha}{\mssL}$--- here the quadratic variation of~$\widehat M^{\class{\molli\otimes f}}_\bullet$ is not given.
Nonetheless, the above shadow martingale problem~$\hmartp{}{\mssL}$ encompasses~$\martp{\alpha}{\mssL}$, in the following sense.

\begin{lemma}\label{l:Extension}
Assume~$(M,\nu,\msA,\mssL,\mssGamma)$ is a standard Markov triple satisfying a Bakry--\'Emery Ricci-curvature lower bound.
Fix a $\msP$-valued process with $\mbbP$-a.s.\ $\T_\mrmn$-cont\-inuous trajectories~$\rho_\bullet$ and assume that~$\rho_0$ consists of~$n$ atoms of \emph{equal} mass $\mbbP$-a.s..
Then,~$\rho_\bullet$ is a distributional solution to~$\hmartp{}{\mssL}$ if (and only if) it is a distributional solution to~$\martp{n}{\mssL}$.
\end{lemma}

\begin{remark}
The point of choosing~$\msR_0\otimes\msA$ as a space of test functions is that to be able to identify the mass of the solution~$\mu_t$ at a given point~$x\in M$.
Indeed, the integration~$f^\trid$ for functions~$f$ in~$\msA$ is not enough to identify even the number of atoms of~$\mu_t$, which must then be specified in the martingale problem~$\martp{\alpha}{\mssL}$ by choosing~$\alpha=n$.
This specification then \emph{forces} solutions to retain for all times both the same number~$n$ of points and the same equal mass at each point.
On the contrary, the integration~$\hat f^\trid$ for functions in~$\msR_0\otimes\msA$ is sufficient to pinpoint the mass of~$\mu_t$ at any given~$x\in M$, so that specifying \emph{a priori} the number of points is not necessary.
\end{remark}

In the slightly different setting of Assumption~\ref{ass:SettingLocal}, the converse implication in Lemma~\ref{l:Extension} will be a consequence of the identification of~$\rho_\bullet$ in Theorem~\ref{t:DK} below.
Here, we only show the forward implication, which merely serves the purpose to justify our interest in the shadow martingale problem~$\hmartp{}{\mssL}$ from the point of view of the study of~\ref{eq:DK}.

\begin{proof}[Proof of Lemma~\ref{l:Extension}]
Let~$\rho_\bullet$ be as in the statement and choose~$\molli\in\msR_0$ with $\molli(n^{-1})=1$. 
Combining~\eqref{eq:Observation} and~\eqref{eq:ExtensionGeneratorCyl}, for every~$f\in\msA$,
\[
\widehat M^{\class{\molli\otimes f}}_t = \rho_t f - \rho_0 f - \int_0^t \rho_s(n \mssL f)\diff s \defeq M^{\class{f}}_t\comma
\]
which shows that~$\widehat M^{\class{\molli\otimes f}}_\bullet$ is a martingale in light of Definition~\ref{d:ProjMP}\ref{i:d:ProjMP:1}.
\end{proof}

\subsubsection{Equivalence of martingale problems and well-posedness}\label{sss:IndentificationMPFormal}
Here, we show how solutions to the martingale problem~\eqref{eq:d:ShadowMP:0} coincide with the measure representation of massive particle systems.
In light of the characterization of the measure representation in Theorem~\ref{t:TransferQuasiHomeo}, this identification shows that the martingale problem~\eqref{eq:d:ShadowMP:0} is well-posed.

Similarly to the extensions of~$\mssGamma$ and~$\mssL$ to~$\msR_0\otimes\msA$ considered in~\S\ref{ss:StronglyLocalCase}, let us respectively denote again by~$\mssd$ and~$\mssdiv$ the extensions of~$\mssd$ in~\eqref{eq:Differential} to~$\msR_0\otimes \mbbA_1$ and of~$\mssdiv$ in~\eqref{eq:PreDivergence} to~$\msR_0\otimes \OF{}$, defined as the linear extensions of
\[
\mssd(\molli\otimes f) \eqdef \molli\otimes \mssd f \qquad \text{and} \qquad \mssdiv(\molli\otimes v)\eqdef \molli \otimes \mssdiv v \fstop
\]

\begin{theorem}\label{t:DK}
Fix~$\pi\in\msP(\mbfT_\circ)$.
Let~$(\rho_\bullet,\mbbP)$ be any \emph{time-homogeneous} $\msP$-valued stochastic process adapted to its natural (augmented) filtration, with initial distribution~$\mbbP\circ\rho_0^{-1}$ equivalent to~$\mcQ_\pi$.
Then, the following are equivalent:
\begin{enumerate}[$(i)$]
\item\label{i:t:DK:1} $\rho_\bullet$ is a solution to the martingale problem for~$\tparen{\widehat\mcL_\pi,\dom{\widehat\mcL_\pi}}$, cf.~\eqref{eq:t:Transfer:1};
\item\label{i:t:DK:2} $\rho_\bullet$ is a solution to the martingale problem for~$(\widehat\mcL,\hCylP{}{0})$, see~\eqref{eq:t:Transfer:1};
\item\label{i:t:DK:3} $\rho_\bullet$ is a distributional solution to the shadow martingale problem~\eqref{eq:d:ShadowMP:0}.
\end{enumerate}
\begin{proof}
\ref{i:t:DK:1}$\implies$\ref{i:t:DK:2}$\implies$\ref{i:t:DK:3} holds by definition.

\ref{i:t:DK:2}$\implies$\ref{i:t:DK:1} holds since~$(\widehat\mcL,\hCylP{}{0})$ is essentially self-adjoint by Theorem~\ref{t:TransferP}\ref{i:t:TransferP:1}.

\ref{i:t:DK:3}$\implies$\ref{i:t:DK:2} is a standard consequence of the diffusion property for $\tparen{\widehat\mcL_\pi,\dom{\widehat\mcL_\pi}}$ together with the It\^o formula for continuous $\R^k$-valued semi-martingales.
Indeed let~$u= F\circ \hat\mbff^\trid$ with~$F\in\mcC^\infty(\R^k)$ and~$\hat f_i\eqdef \molli_i\otimes f_i\in \msR_0\otimes\msA$. 
Since~$\widehat M^{\ttclass{\hat f_i}}_\bullet$ as in~\eqref{eq:d:ShadowMP:0} is a continuous martingale,~$\hat f_i^\trid(\rho_t)$ is a continuous It\^o process satisfying
\begin{equation}\label{eq:t:DK:1}
\begin{gathered}
\diff \hat f_i^\trid(\rho_t) = (\mcL'\rho_t)(\hat f_i) \diff t= \tparen{(\tfrac{1}{\id_\mbbI}\otimes\mssL) \hat f_i}^\trid (\rho_t) \diff t 
\\
\diff \quadvar{\hat f_i^\trid(\rho_\bullet) ,\hat f_j^\trid(\rho_\bullet)}_t = \widehat\mcG(\hat f_i,\hat f_j)_{\rho_t} \diff t \comma
\end{gathered}
\end{equation}
where we used~\eqref{eq:DualityGenerator},~\eqref{eq:l:DiffusionP:0.4}, and~\eqref{eq:l:DiffusionP:4}.
Applying the It\^o formula to~$F\in\mcC^\infty(\R^k)$,
\begin{align*}
\diff u(\rho_t) =&\ \diff F\tparen{\hat f_1^\trid(\rho_t),\dotsc, \hat f_k^\trid(\rho_t)}
\\
=& \sum_i^k (\partial_i F) \tparen{\hat f_1^\trid(\rho_t),\dotsc, \hat f_k^\trid(\rho_t)} \diff\hat f_i^\trid(\rho_t)
\\
&+
\frac{1}{2}\sum_{i,j}^k (\partial_{ij}^2F) \tparen{\hat f_1^\trid(\rho_t),\dotsc, \hat f_k^\trid(\rho_t)} \diff \quadvar{\hat f_i^\trid(\rho_\bullet) ,\hat f_j^\trid(\rho_\bullet) }_t 
\\
=&\ (\widehat\mcL u)_{\rho_t} \diff t \comma
\end{align*}
where the last equality follows combining~\eqref{eq:t:DK:1} and~\eqref{eq:l:DiffusionP:0.2}.
\end{proof}
\end{theorem}

Let us now combine Theorem~\ref{t:DK} with Theorem~\ref{t:TransferQuasiHomeo}\ref{i:t:TransferP:6}.

\begin{corollary}
Fix~$\pi\in\msP(\mbfT_\circ)$ and assume~\eqref{eq:QPP} holds.
Let~$(\rho_\bullet,\mbbP)$ be any time-homogeneous $\msP$-valued $\mcQ_\pi$-sub-stationary stochastic process adapted to its natural (augmented) filtration with initial distribution~$\mbbP\circ\rho_0^{-1}$ equivalent to~$\mcQ_\pi$.
If any of the conditions~\ref{i:t:DK:1}-\ref{i:t:DK:3} in Theorem~\ref{t:DK} is satisfied, then~$(\rho_\bullet,\mbbP)$ is a version of~$\widehat\mcW_\pi$ in~\eqref{eq:t:TransferP:0} up to $\mcQ_\pi$-equivalence.

In particular, for every Borel probability measure~$\mcP$ on~$\msP$ equivalent to~$\mcQ_\pi$, the martingale problems for~$\tparen{\widehat\mcL_\pi,\dom{\widehat\mcL_\pi},\mcP}$ and~$\tparen{\widehat\mcL,\hCylP{}{0},\mcP}$, the shadow martingale problem~\eqref{eq:d:ShadowMP:0} with initial distribution~$\mcP$, and thus the Dean--Kawasaki-type \emph{\textsc{spde}}~\eqref{eq:IntroSPDE} with initial distribution~$\mcP$, are all well-posed in the class of processes~$(\rho_\bullet,\mbbP)$ as above, in the sense of weak existence for~$\widehat\mcE^\pi$-q.e.\ starting point and uniqueness up to $\mcQ_\pi$-equivalence.
\end{corollary}

\subsection{The Dean--Kawaski SPDE on the state space of a Dirichlet form}\label{ss:DKNoise}
As anticipated, we now define the exterior differential~$\mssd$, the divergence~$\mssdiv$, and the white noise~$\xi$ appearing in~\eqref{eq:IntroSPDE}.

\subsubsection{Measurable co-tangent bundle of a Dirichlet space}
To define~$\xi$, we need a notion of `tangent bundle' associated to~$M$.
Several competing constructions of the same ---up to some natural dualities--- object were achieved
by A.~Eberle in~\cite[\S3.D.2, pp.~149ff.]{Ebe99};
by M.~Hino in~\cite[\S3]{Hin13};
by M.~Ionescu, L.G.~Rogers, and A.~Teplyaev in~\cite[\S2.1]{IonRogTep12} (also cf.\ M.~Hinz, M.~R\"ockner, and A.~Teplyaev~\cite[\S2]{HinRoeTep13});
by F.~Baudoin and D.J.~Kelleher~\cite[\S2]{BauKel18};
similar ideas were also introduced
in the setting of non-commutative Dirichlet forms on~$C^*$-algebras, by F.~Cipriani and J.-L.~Savageot in~\cite{CipSav03};
in the setting of metric measure spaces, by N.~Gigli in~\cite{Gig18}.

As usual, we work under Assumption~\ref{ass:SettingLocal}.
Throughout, let~$\msS\eqdef\msA\otimes\Cb(\T)$.

We recall the construction of the space of $L^2$-integrable differential $1$-forms after~\cite[\S2]{HinRoeTep13}.
For simplicity of exposition, we replace the algebra of bounded Borel functions in their construction with~$\Cb(\T)$.
This is immaterial in light of~\cite[Rmk.~2.1(ii)]{HinRoeTep13}.
In fact, we might choose~$\msA^\mrmu$ in place of both~$\msA$ and~$\Cb(\T)$ (which would considerably simplify the notation).
However, we refrain from so doing, since it is relevant for the understanding of the construction that functions in the right factor \emph{need not be differentiable} (nor, in fact, continuous).

\paragraph{$L^2$-integrable differential $1$-forms}
Define a scalar product on~$\msS$ as the bilinear extension to~$\msS$ of
\[
\scalar{f_1\otimes g_1}{f_2\otimes g_2}_{\OF{}}\eqdef \int g_1\, g_2\, \mssGamma(f_1,f_2)\diff\nu \comma \qquad f_i\otimes g_i\in\msS\comma \quad i=1,2 \fstop
\]
Further denote by~$\norm{\emparg}_{\OF{}}$ the associated \emph{semi}-norm on~$\msS$ and note that it is non-negative definite, e.g.~\cite[Rmk.~2.1(i)]{HinRoeTep13}.
Finally, let~$\OF{}$ be the completion of the quotient~$\msS / \ker\norm{\emparg}_{\OF{}}$ w.r.t.\ the ---non-relabeled--- quotient norm induced by~$\norm{\emparg}_{\OF{}}$.
The space~$\OF{}$ is a Hilbert space playing the role of the space of $L^2$-integrable differential $1$-forms.

We note that~$\OF{}$ is a left~$\msA$-module, resp.\ a right $\Cb(\T)$-module, for the left, resp.\ right, multiplication defined by the linear extension of the actions on elementary tensors
\begin{gather*}
f_2 (f_1\otimes g_1) \eqdef (f_2 f_1) \otimes g_1 - f_2\otimes (f_1g_1)\comma \quad \text{resp.} \quad (f_1\otimes g_1) g_2 \eqdef f_1 \otimes (g_1g_2) \comma
\\
f_i\in\msA\comma g_i\in \Cb(\T)\comma\quad i=1,2\comma
\end{gather*}
and it is readily verified that the left, resp.\ right, action is continuous on~$L^\infty(\nu)\otimes \OF{}$, resp.~$\OF{}\otimes L^\infty(\nu)$.
Since~$\tparen{\mssE,\dom{\mssE}}$ is (strongly) local, the left and right multiplication coincide for every~$f_2\in\msA$ by the Leibniz rule, and thus coincide on~$\Cb(\T)$ by approximation.

\paragraph{Direct-integral representation}
Define a scalar product on~$\msS$ as the linear extension to~$\msS$ of the product on elementary tensors
\[
\mssGamma_{\OF{},x}(f_1\otimes g_1,f_2\otimes g_2)\eqdef \tparen{g_1 g_2\, \mssGamma(f_1,f_2)}(x) \comma \qquad f_i\otimes g_i\in\msS\comma\quad x\in M\comma
\]
and denote by~$\class[x]{\emparg}$ the class of an element in~$\msS/\ker\mssGamma_{\OF{},x}$.
Further let~$\OF{x}$ the Hilbert completion of the quotient~$\msS/\ker\mssGamma_{\OF{},x}$ w.r.t.\ the non-relabeled quotient scalar product, defined as the linear extension of
\[
\scalar{\class[x]{f_1\otimes g_1}}{\class[x]{f_2\otimes g_2}}_{\OF{x}}\eqdef \mssGamma_{\OF{},x}(f_1\otimes g_1,f_2\otimes g_2) \fstop
\]

\begin{proposition}[Direct-integral representation~{\cite[Lem.~2.4, Thm.~2.1]{HinRoeTep13}}]
The collection~$x\mapsto \OF{x}$ is a measurable field of Hilbert spaces with underlying space of measurable fields~$\msS$, and~$\OF{}$ admits the $\nu$-essentially unique direct-integral representation, indexed by~$(M,\nu)$,
\begin{equation}\label{eq:DirIntForms}
\OF{}\, \cong \dint[\msS]{M} \OF{\emparg} \diff\nu \fstop
\end{equation}
\end{proposition}

\begin{remark}
\begin{enumerate*}[$(a)$]
\item In the case when~$M=(M,g,\nu)$ is a weighted Riemannian manifold, the spaces~$\OF{x}$ coincide $\nu$-a.e.\ with the usual cotangent spaces~$T^*_xM$, while~$\OF{}$ coincides with the space of $L^2(\nu)$-integrable differential $1$-forms.
\item In the case when~$M=(M,\mssd,\nu)$ is an infinitesimally Hilbertian metric measure space, $\OF{}$~coincides with the cotangent module~$L^2(T^*M)$ in~\cite[Dfn.~2.2.1]{Gig18}.
\end{enumerate*}
\end{remark}

\paragraph{Exterior differential and divergence}
Define the \emph{exterior differential} $\mssd\colon \msA\to\OF{}$ of~$(\mssE,\dom{\mssE})$ as
\begin{equation}\label{eq:Differential}
\mssd\colon f\longmapsto f\otimes \car\comma \qquad f\in\msA \fstop
\end{equation}

Let us collect some facts about the exterior differential, consequence of the strong locality of~$\tparen{\mssE,\dom{\mssE}}$ and shown in~\cite[Cor.s~2.3,~2.4, Rmk.~2.6(ii), p.~4386]{HinRoeTep13}.

\begin{proposition}[Exterior differential]
The operator~$\mssd$
\begin{enumerate}[$(i)$]
\item is $\mssE$-bounded, viz.
\[
\mssE(f)\leq \norm{\mssd f}_{\OF{}}^2 \leq 2\, \mssE(f)\comma \qquad f\in\dom{\mssE} \semicolon
\]
\item extends to a non-relabeled derivation on~$\dom{\mssE}$, viz.
\[
\mssd(f_1f_2)= (\mssd f_1)f_2 + f_1 (\mssd f_2) \comma \qquad f_1,f_2\in\dom{\mssE} \semicolon
\]
\item is decomposable over the direct-integral representation~\eqref{eq:DirIntForms} and represented by a $\nu$-essentially unique measurable field~$x\mapsto \mssd_x$ of derivations
\[
\mssd_x(f_1f_2)= f_1(x)\, \mssd_x f_2+ f_2(x)\, \mssd_x f_1 \comma \qquad f_1,f_2\in\dom{\mssE} \fstop
\]
\end{enumerate}
\end{proposition}

\begin{definition}
We denote by~$\mbbA_1$ the space~$\msA^\mrmu$ endowed with the norm
\[
\norm{f}_{\mbbA_1} \eqdef \mssE^{1/2}(f)+ \norm{f}_0\comma \qquad f\in\msA^\mrmu\fstop
\]
We further denote by~$(\mbbA_1',\norm{\emparg}_{\mbbA_1'})$ the Banach dual of~$\mbbA_1$.
\end{definition}

For every fixed~$f,g\in\mbbA_1$, consider the map
\begin{equation}\label{eq:PreDivergence}
h\longmapsto - \scalar{g\, \mssd f}{\mssd h}_{\OF{}} = -\int g\, \mssGamma(f,h)\diff\nu \comma \qquad h\in\dom{\mssE}\fstop
\end{equation}
 By Cauchy--Schwarz inequality on~$\OF{}$,
 \[
 \abs{ \scalar{g\, \mssd f}{\mssd h}_{\OF{}}} \leq \sqrt{2} \norm{g\,\mssd f}_{\OF{}} \, \mssE(h)^{1/2}\comma
 \]
 which shows that~\eqref{eq:PreDivergence} defines an element~$\mssdiv(g\,\mssd f)$ of~$\mbbA_1'$ satisfying
 \[
 \norm{\mssdiv(g\, \diff f)}_{\mbbA_1'} \leq \sqrt{2} \norm{g\, \diff f}_{\OF{}} \fstop
 \]
 Whereas we refer to~$\mssdiv(g\, \mssd f)$ as to the \emph{divergence} of the vector field~$g\, \mssd f$, we stress that this is a not a pointwise defined object, but rather a linear functional on~$\mbbA_1$, so that~$(\mssdiv v)f$ is a scalar quantity, not a function.
 
 \begin{lemma}[e.g.,~{\cite[Lem.~3.1]{HinRoeTep13}}]
 The operator~$\mssdiv$ has a unique bounded linear extension~$\mssdiv\colon \OF{}\to \mbbA_1'$ with~$\norm{\mssdiv v}_{\mbbA_1'}\leq \sqrt{2}\norm{v}_{\OF{}}$ for~$v\in\OF{}$ and satisfying the adjunction relation
 \begin{equation}\label{d:DivergenceAdjoint:0}
(\mssdiv v)f = -\scalar{v}{\mssd f}_{\OF{}}\comma\qquad f\in\mbbA_1\comma v\in\OF{} \fstop
 \end{equation}
\end{lemma}

\subsubsection{Vector valued space-time white noise and formal derivation}\label{sss:WhiteNoiseFormalDerivation}
Let us now introduce the $\OF{}$-valued space-time white noise~$\xi$.
In order to avoid lengthy discussions on nuclear riggings of~$\OF{}$, 
we use the algebraic-dual construction by It\^o, see e.g.~\cite[Ex.~1.27]{Jan97}.
Let~$\Omega$ be the \emph{algebraic} dual of~$\OF{}$, endowed with the $\sigma$-algebra~$\msF$ generated by all evaluation maps~$\ev_\eta\colon \omega\mapsto \omega\eta$ index by~$\eta\in\OF{}$.
Then, there exists a unique probability measure~$\mbfP$ on~$(\Omega,\msF)$ such that, for every~$\eta\in\OF{}$, the $\R$-valued random variable~$\ev_\eta$ is centered Gaussian with variance~$\norm{\eta}_{\OF{}}^2$ under~$\mbfP$.

Let~$\mbfP_0$ be the law of the standard white noise on~$\mcS'(\R)$.

\begin{definition}[White noise]
The \emph{$\OF{}$-valued space-time white noise measure} is the probability measure~$\mbfP\otimes \mbfP_0$.
We denote by~$\xi_\bullet\eqdef\seq{\xi_t}_t$ any $\mbfP\otimes \mbfP_0$-distributed random field and call it the \emph{$\OF{}$-valued space-time white noise}.
\end{definition}

As usual,~$\xi_\bullet$ has a series representation in terms of orthonormal bases.
Let~$\seq{\eta^n}_n$ be any measurable field of orthonormal bases for the direct-integral representation~\eqref{eq:DirIntForms} of~$\OF{}$, that is,~$\eta^n \eqdef \seq{x\mapsto \eta^n_x}$ is a measurable field in~$\OF{}$ and~$\seq{\eta^n_x}_n$ is an orthonormal basis for~$\OF{x}$ for $\nu$-a.e.~$x\in M$.
Further let~$\gamma_\bullet\eqdef \seq{\gamma_t}_t$ be a white noise on~$L^2(\R)$. Then, we have the representation ---to be understood in the usual weak sense---
\begin{equation}\label{eq:WhiteNoise}
\xi_t(x) = \sum_n \gamma_t\, \eta^n_x \comma \qquad x\in M \fstop
\end{equation}

\paragraph{Formal derivation}
We conclude this section by verifying via a heuristic computation, that formal solutions to~\eqref{eq:IntroSPDE} are identical with solutions to the shadow martingale problem~\eqref{eq:d:ShadowMP:0}.
The computation is a formal one only in that one cannot give a rigorous meaning to the square root of a solution in such a way that the following algebraic rule holds:
\begin{equation}\label{eq:SquareRootFalse}
(\sqrt{\rho_t}f)^2=\rho_t(f^2)\comma \qquad f\in\msA \fstop
\end{equation}

Suppose that~$(\rho_\bullet,\mbbP)$ is a formal solution to~\eqref{eq:IntroSPDE} adapted to its natural (augmented) filtration.
It suffices to verify that the corresponding solutions have identical quadratic variation, as we now show.

On the one hand, in light of the equivalence~\ref{i:t:DK:3}$\iff$\ref{i:t:DK:2} in Theorem~\ref{t:DK}, it follows from~\eqref{eq:t:Transfer:2} that the quadratic variation of~$\widehat M_t^{\ttclass{\hat f}}=M_t^{\ttclass{\hat f^\trid}}$ for a solution~$(\rho_\bullet,\mbbP)$ of the shadow martingale problem~\eqref{eq:d:ShadowMP:0} is~$\widehat\mcG(\hat f^\trid)(\rho_t)= \rho_t \mssGamma(\hat f)$.
On the other hand, the quadratic variation of a solution~$\rho_\bullet$ to~\eqref{eq:IntroSPDE} tested on~$\hat f$ is
\begin{align*}
\quadvar{\mathsf{div}\tparen{\sqrt{\rho}\, \xi} \hat f}_t =&\ \quadvar{- \sqrt{\rho}\, \tscalar{\xi}{\mssd\hat f}_{\OF{}}}_t 
\\
=&\ \sum_{i,j} \quadvar{-\sqrt{\rho}\, \tscalar{\gamma_\bullet \eta^i}{ \mssd \hat f}_{\OF{}}\, , -\sqrt{\rho}\, \tscalar{\gamma_\bullet \eta^j}{\mssd \hat f}_{\OF{}}}_t
\\
=&\ \sum_i \paren{\sqrt{\rho_t}\, \tscalar{\mssd \hat f }{\eta^i}_{\OF{}}}^2 
=\rho_t \tscalar{\mssd \hat f}{\mssd \hat f}_{\OF{}}
=\rho_t \mssGamma(\hat f)\comma
\end{align*}
where we used, respectively: that the divergence is the $\OF{}$-adjoint to~$\mssd$, see~\eqref{d:DivergenceAdjoint:0}; the series representation~\eqref{eq:WhiteNoise} for the noise~$\xi$ and the fact that~$\xi$ is white in time; the $\OF{}$-orthogonality of~$\seq{\eta^i}_i$; \eqref{eq:SquareRootFalse} and Parseval's identity for~$\OF{}$\ ; and finally~\eqref{eq:PreDivergence} with~$g\equiv \car$ and~$h=f$.

\section{Interacting massive systems}\label{s:Interactions}
In this section, we consider massive systems of pairwise interacting particles.
%
Since we are interested in everywhere dense massive systems, the case of \emph{attractive} interaction is very delicate, as paroxysmal interaction ---in the form of coagulation phenomena--- may occur at each positive time.
An explicit description only by means of Dirichlet-form theory appears therefore out of reach.
Thus, we mostly confine ourselves to the case of \emph{repulsive} interactions, to which our techniques are well-suited.

\subsection{Girsanov transforms in the measure representation}\label{ss:Girsanov}
Many examples of physical interactions depend on the distance between the ineracting particles.
It is therefore natural to assume the existence of a distance function on~$M$ and use it to describe these interactions.

For a distance~$\mssd$, we denote by~$\Lip_b[\mssd]$ the space of all bounded real-valued $\mssd$-Lipschitz functions on the domain of~$\mssd$, and by~$\Lipu_b[\mssd]$ the subset of all functions in~$\Lip_b[\mssd]$ with $\mssd$-Lipschitz constant less than or equal to~$1$.

Throughout this section we work under the following assumption, refining Assumption~\ref{ass:Setting}.

\begin{assumption}\label{ass:Girsanov}
Suppose that Assumption~\ref{ass:SettingLocal} and~\eqref{eq:QPP} hold.
Further let~$\mssd$ be a distance on~$M$ such that
\begin{enumerate}[$(a)$]
\item $\mssd$ completely metrizes~$(M,\T)$;

\item it holds that 
\begin{equation}\label{eq:RademacherBaseSp}
\mssd(\emparg,x) \in\dom{\mssGamma} \quad \text{and} \quad \mssGamma\tparen{\mssd(\emparg,x)}\in L^\infty(\nu) \comma  \qquad x\in M\semicolon
\end{equation}

\item\label{i:ass:Girsanov:3} $\msA$ consists of $\mssd$-Lipschitz functions, i.e.~$\msA\subset \Lip_b[\mssd]$.
\end{enumerate}
\end{assumption}

Let us anticipate here that ---as thoroughly discussed in Example~\ref{e:RiemannianGirsanov} below--- Assumption~\ref{ass:Girsanov} is satisfied
on every weighted Riemannian manifold~$(M,g,\nu)$ as in Assumption~\ref{ass:WRM}.

\subsubsection{Girsanov transforms for general functions}\label{sss:GeneralGirsanov}
In this section we study the Girsanov transform of the form~$\tparen{\widehat\mcE_\pi,\dom{\widehat\mcE_\pi}}$ by a weight~$\phi$ in the \emph{broad local space} $\dotloc{\dom{\widehat\mcE_\pi}}$ of~$\tparen{\widehat\mcE_\pi,\dom{\widehat\mcE_\pi}}$, see e.g.~\cite[\S4]{Kuw98}.
Provided we verify the relative assumptions, our results will follow from the general theory of Girsanov transforms for Dirichlet forms fully developed in~\cite{Fit97,Ebe96,Kuw98b,CheZha02,CheSun06}.
In principle, one could equivalently study Girsanov transforms for~$\tparen{\widehat\bmssE_\pi,\dom{\widehat\bmssE_\pi}}$ rather than for its measure representation~$\tparen{\widehat\mcE_\pi,\dom{\widehat\mcE_\pi}}$. 
However, having the coarser narrow topology~$\T_\mrmn$ on~$\msP$ at our disposal greatly simplifies some arguments.

Recall that, under Assumption~\ref{ass:Girsanov},~$\tparen{\widehat\mcE_\pi,\dom{\widehat\mcE_\pi}}$ is a $\T_\mrma$-quasi-regular strongly local Dirichlet form.
Furthermore, by Proposition~\ref{p:StandardCyl}\ref{i:p:StandardCyl:4}, all $\widehat\mcE_\pi$-quasi-notions given w.r.t.~$\T_\mrma$ are identical with the corresponding $\widehat\mcE_\pi$-quasi-notions given w.r.t.~$\T_\mrmn$, so that there is no ambiguity in the definitions and notations for $\widehat\mcE_\pi$-quasi-continuous representatives, broad local spaces, etc.
However, in order to avoid confusion, we from now on assume all $\widehat\mcE_\pi$-quasi-notions to be given w.r.t.~$\T_\mrmn$.

Now, fix~$\phi\in\dotloc{\dom{\widehat\mcE_\pi}}$ with~$\phi\geq 0$.
By e.g.~\cite[Lem.~4.1]{Kuw98},~$\phi$ has an $\widehat\mcE_\pi$-quasi-continuous $\mcQ_\pi$-version~$\widetilde\phi$ and we denote by~$\msP_\phi$ a fixed finely open Borel $\widehat\mcE_\pi$-q.e.\ version of~$\tset{0<\widetilde\phi<\infty}$.
We set
\[
\mcQ_\pi^\phi\eqdef \phi^2\mcQ_\pi\comma
\]
and we note that~$\mcQ_\pi^\phi$ is a $\sigma$-finite Borel measure on~$\msP_\phi$.

We further let~$\tau_\phi\eqdef\inf\tset{t\geq 0: \widehat\mcX^\pi_t\not\in \msP_\phi}$ be the first-exit time of~$\widehat\mcW
_\pi$ in~\eqref{eq:t:TransferP:0} from~$\msP_\phi$.

As in~\cite[\S3]{Kuw98b}, we consider the super-martingale multiplicative functional
\begin{equation*}
L^{[\phi]}_t\eqdef \exp\paren{M^{[\log \phi]}_t -\tfrac{1}{2}\quadvar{M^{[\log \phi]}}_t}\car_{\set{t<\tau_\phi}} \comma
\end{equation*}
and we denote by
\[
\widehat\mcW_\pi^\phi \eqdef \seq{\widehat\Omega_\circ,\widehat\msF^\pi, \widehat\mcX^{\pi,\phi}_\bullet, \seq{\widehat \mcP^{\pi,\phi}_\eta}_{\eta\in\msP_\phi}, \zeta^\phi}\comma
\]
the part process on~$\msP_\phi$ of the process~$\widehat\mcW_\pi$ transformed by~$L^{[\phi]}_t$.
In light of the discussion in~\cite[p.~741]{Kuw98b}, we may and will always regard~$\widehat\mcW_\pi^\phi$ as a $\mcQ_\pi^\phi$-symmetric Borel right process ---for the moment--- w.r.t.~$\T_\mrmn$.

\begin{theorem}\label{t:Girsanov}
Fix~$\pi\in\msP(\mbfT_\circ)$ and suppose that Assumption~\ref{ass:Girsanov} is satisfied.
Then, the quadratic form~$\tparen{\widehat\mcE_\pi^\phi,\dom{\widehat\mcE_\pi^\phi}}$ defined by
\begin{equation}\label{eq:t:Girsanov:0}
\begin{gathered}
\dom{\widehat\mcE_\pi^\phi}\eqdef \set{u\in \dotloc{\dom{\widehat\mcE_\pi}} \cap L^2(\mcQ_\pi^\phi) : \int_\msP \widehat\mcG_\pi(u)\, \diff\mcQ_\pi^\phi<\infty }\comma
\\
\widehat\mcE_\pi^\phi(u,v)\eqdef \int_\msP \widehat\mcG_\pi(u,v)\, \diff\mcQ_\pi^\phi \comma \qquad u,v\in \dom{\widehat\mcE_\pi^\phi} \comma 
\end{gathered}
\end{equation}
is a $\T_\mrma$- and $\T_\mrmn$-quasi-regular local Dirichlet form on~$L^2(\mcQ_\pi^\phi)$.

If, additionally,
\begin{equation}\label{eq:t:Girsanov:00}
\phi\in L^2(\mcQ_\pi) \qquad \text{and} \qquad\int \widehat\mcG_\pi(\phi) \diff \mcQ_\pi <\infty \fstop
\end{equation}
Then, 
\begin{enumerate}[$(i)$]
\item\label{i:t:Girsanov:0} the form~$\tparen{\widehat\mcE_\pi^\phi,\dom{\widehat\mcE_\pi^\phi}}$ is additionally strongly local and recurrent;
\item\label{i:t:Girsanov:1} for the generator~$\tparen{\widehat\mcL_\pi^\phi,\dom{\widehat\mcL_\pi^\phi}}$ of~$\tparen{\widehat\mcE_\pi^\phi,\dom{\widehat\mcE_\pi^\phi}}$ we have~$\hCylP{}{0}\subset \dom{\widehat\mcL_\pi^\phi}$ and
\[
\widehat\mcL_\pi^\phi u= \widehat\mcL^\phi u\eqdef \widehat\mcL u+\phi^{-1}\,\widehat\mcG_\pi(\phi, u) \quad \text{on}\quad \hCylP{}{0} \fstop
\]

\item\label{i:t:Girsanov:2} $\tparen{\widehat\mcE_\pi^\phi,\dom{\widehat\mcE_\pi^\phi}}$ is properly associated with the $\mcQ_\pi^\phi$-symmetric process~$\widehat\mcW_\pi^\phi$, which is thus recurrent (in particular: conservative) and a right Borel process for both~$\T_\mrma$ and~$\T_\mrmn$;

\item\label{i:t:Girsanov:4} a $\mcQ_\pi^\phi$-measurable set~$\mbfA\subset \msP$ is $\widehat\mcE_\pi^\phi$-invariant if and only if it is $\widehat\mcE_\pi$-invariant;

\item\label{i:t:Girsanov:5} 
$\widehat\mcW^\phi_\pi$ is a solution to the martingale problem for~$\ttparen{\widehat\mcL^\phi,\hCylP{}{0}}$. 
In particular, for every~$u\in\hCylP{}{0}$ ($\subset \Cb(\T_\mrma)$), the process
\begin{equation}\label{eq:t:Girsanov:1}
M^{\class{u}}_t\eqdef u(\widehat\mcX^{\pi,\phi}_t)- u(\widehat\mcX^{\pi,\phi}_0)-\int_0^t (\widehat\mcL^\phi u)(\widehat\mcX^{\pi,\phi}_s)\diff s
\end{equation}
is an adapted square-integrable martingale with quadratic variation
\begin{equation}\label{eq:t:Girsanov:1.1}
\tquadvar{M^{\class{u}}}_t = \int_0^t \widehat\mcG(u)(\widehat\mcX^{\pi,\phi}_s)\diff s \fstop
\end{equation}
\end{enumerate}

Finally, suppose, in addition to~\eqref{eq:t:Girsanov:00}, that
\begin{equation}\label{eq:t:Girsanov:000}
\phi^2\in L^2(\mcQ_\pi) \qquad \text{and}\qquad \int \phi^2\, \widehat\mcG_\pi(\phi)\diff\mcQ_\pi <\infty\fstop
\end{equation}
Then, additionally,
\begin{enumerate}[$(i)$, resume]
\item\label{i:t:Girsanov:6}
$\hCylP{}{0}$ is a form core for~$\tparen{\widehat\mcE_\pi^\phi,\dom{\widehat\mcE_\pi^\phi}}$.

\item\label{i:t:Girsanov:7} 
$\widehat \mcL^\phi \hCylP{}{0}\subset L^1(\mcQ_\pi^\phi)$ and there exists a unique strongly continuous semigroup on~$L^1(\mcQ_\pi^\phi)$ the generator of which extends~$\tparen{\widehat\mcL^\phi,\hCylP{}{0}}$.
In particular, the self-adjoint operator $\tparen{\widehat\mcL_\pi^\phi,\dom{\widehat\mcL_\pi^\phi}}$ is the unique Markovian extension of $\tparen{\widehat\mcL_\pi^\phi,\hCylP{}{0}}$, and~$\tparen{\widehat\mcE_\pi^\phi,\dom{\widehat\mcE_\pi^\phi}}$ is the unique Dirichlet form extending~$\tparen{\widehat\mcE_\pi^\phi,\hCylP{}{0}}$;

\item\label{i:t:Girsanov:8} 
$\widehat\mcW_\pi^\phi$ is the unique ---up to~$\mcQ_\pi^\phi$-equivalence--- $\mcQ_\pi^\phi$-sub-stationary $\mcQ_\pi^\phi$-special standard process solving the martingale problem~\eqref{eq:t:Girsanov:1}.
\end{enumerate}
\end{theorem}

\begin{proof}
It follows from the standard theory, see e.g.~\cite[\S3, p.~741]{Kuw98b}, that~\eqref{eq:t:Girsanov:0} is a $\check\T$-quasi-regular local Dirichlet form for every topology~$\check\T$ on~$\msP_\phi$ such that~$\check\T$ is the trace topology on~$\msP_\phi$ of a topology~$\bar\T$ on~$\msP$ for which $\tparen{\widehat\mcE_\pi,\dom{\widehat\mcE_\pi}}$ is $\bar\T$-quasi-regular and local.
This applies in particular to both~$\T_\mrma$ and~$\T_\mrmn$.

\paragraph{Proof of~\ref{i:t:Girsanov:0}}
Since~$\phi\in L^2(\mcQ_\pi)$ by~\eqref{eq:t:Girsanov:00}, the measure~$\mcQ_\pi^\phi$ is finite, and~$\car\in L^2(\mcQ_\pi^\phi)$.
By strong locality of~$\tparen{\widehat\mcE_\pi,\dom{\widehat\mcE_\pi}}$ we have~$\widehat\mcG_\pi(\car)\equiv 0$, thus~$\car\in\dom{\widehat\mcE_\pi^\phi}$ and~$\widehat\mcE_\pi^\phi(\car)=0$, which shows that~$\tparen{\widehat\mcE_\pi^\phi,\dom{\widehat\mcE_\pi^\phi}}$ is recurrent, and therefore strongly local since it is local.

\paragraph{Proof of~\ref{i:t:Girsanov:1}}
For every~$u\in\hCylP{}{0}=\hCylP{}{0}\cap L^\infty(\mcQ_\pi)$ we have~$\widehat\mcL u\in \hCylP{}{0}$ and~$\widehat\mcG(u)\in \hCylP{}{0}$. 
As a consequence,~$\widehat\mcL u \in L^\infty(\mcQ_\pi^\phi)\subset L^1(\mcQ_\pi^\phi)\cap L^2(\mcQ_\pi^\phi)$. 
Furthermore, combining~\eqref{eq:t:Girsanov:00} with~\cite[Cor.~3.1]{Kuw98b} shows that~$\phi\in\dom{\widehat\mcE_\pi}$.
Thus,~$\phi_n\eqdef \phi \wedge n$ satisfies~$\phi_n\in\dom{\widehat\mcE_\pi}_b$ for every~$n$ and therefore~$\phi_n^2\in\dom{\widehat\mcE_\pi}_b$ for every~$n$ since~$\dom{\widehat\mcE_\pi}_b$ is an algebra.
Applying the carr\'e du champ identity twice,
\begin{equation}\label{eq:t:Girsanov:2}
\begin{aligned}
2 \int \widehat\mcG_\pi(u,v)\, \phi_n^2\diff\mcQ_\pi &= \widehat\mcE_\pi(u,\phi_n^2 v) + \widehat\mcE_\pi(u\phi_n^2, v) - \widehat\mcE_\pi(uv,\phi_n^2)
\\
2 \int \widehat\mcG_\pi(\phi_n^2,v)\, u\diff\mcQ_\pi &= \widehat\mcE_\pi(\phi_n^2, uv) + \widehat\mcE_\pi(u\phi_n^2, v) - \widehat\mcE_\pi(\phi_n^2 v,u)
\end{aligned}\comma \qquad u,v\in\hCylP{}{0}\fstop
\end{equation}
Using the symmetry of~$\tparen{\widehat\mcE_\pi,\dom{\widehat\mcE_\pi}}$ and summing up the two equalities in~\eqref{eq:t:Girsanov:2},
\begin{align}\label{eq:t:Girsanov:3}
\int \widehat\mcG_\pi(u,v)\, \phi_n^2\diff\mcQ_\pi = -\int u \tparen{\widehat\mcL_\pi v+ \phi_n^{-1}\widehat\mcG_\pi(\phi_n,v)} \, \phi_n^2\diff \mcQ_\pi \fstop
\end{align}
By definition of~$\phi_n$, we have~$0\leq \phi_n \leq \phi \in L^1(\mcQ_\pi)$.
By definition of~$\phi_n$ and by locality of~$\widehat\mcG_\pi$, we further have, for every~$v\in\hCylP{}{0}$,
\begin{align*}
0\leq \tabs{\widehat\mcG_\pi(\phi_n,v)} \phi_n \leq \tabs{\widehat\mcG_\pi(\phi,v)}\phi \leq \widehat\mcG_\pi(\phi)^{1/2}\, \widehat\mcG_\pi(v)^{1/2} \phi \in L^1(\mcQ_\pi)\comma
\end{align*}
since~$\widehat\mcG_\pi(v)$ is uniformly bounded.
Thus, by Dominated Convergence, letting~$n$ to infinity in~\eqref{eq:t:Girsanov:3},
\begin{align*}
\widehat\mcE_\pi^\phi(u,v) = -\tscalar{u}{\widehat\mcL_\pi v+ \phi^{-1}\widehat\mcG_\pi(\phi,v)}_{L^2(\mcQ_\pi^\phi)}\comma \qquad u,v\in\hCylP{}{0} \comma
\end{align*}
which concludes the assertion by arbitrariness of~$u\in\hCylP{}{0}$ and density of~$\hCylP{}{0}$ in $L^2(\mcQ_\pi^\phi)$.

\paragraph{Proof of~\ref{i:t:Girsanov:2}}
The assertion follows from~\cite[Thm.~3.1]{Kuw98b} as soon as we show that~$\widehat\mcW_\pi^\phi$ is conservative.
This follows from~\cite[Thm.~5.1]{Kuw98b} provided we verify~$(\mbfA)$,~$(\mbfB)$, and~$(\mbfC)$ and~$(\mbfE)$ there. $(\mbfC)$ and~$(\mbfE)$  follow from~\eqref{eq:t:Girsanov:00}.
$(\mbfB)$ is trivial since~$\tparen{\widehat\mcE_\pi,\dom{\widehat\mcE_\pi}}$ is a strongly local Dirichlet form on a finite-measure space.

In order to show~$(\mbfA)$, we argue as follows.
For the distance~$\mssd$ in Assumption~\ref{ass:Girsanov}, let
\[
\rho_{\mathrm{bL}}(\eta_1,\eta_2) \eqdef \sup_{f\in \Lip_b[\mssd]} \abs{f^\trid \eta_1-f^\trid\eta_2}\comma\qquad \eta_1,\eta_2 \in\msP\comma
\]
be the corresponding \emph{bounded-Lipschitz distance} metrizing~$(\msP,\T_\mrmn)$.
Further let~$\FC{\msA}{}$ be defined as in~\eqref{eq:d:StandardFC:0}.
In light of Assumption~\ref{ass:Girsanov}\ref{i:ass:Girsanov:3}, it is not difficult to show that~$\FC{\msA}{}\subset \Lip[\rho_{\mathrm{bL}}]$.

Now, let~$\msC\eqdef \FC{\msA}{}\cap \Lipu_b[\rho_{\mathrm{bL}}]$ in~\cite[Eqn.~(7)]{Kuw98b}. For every~$\eta_1,\eta_2\in\msP$,
\[
\rho(\eta_1,\eta_2)\eqdef \sup\set{u(\eta_1)-u(\eta_2) : u\in\FC{\msA}{}\cap \Lipu_b[\rho_{\mathrm{bL}}]\comma \widehat\mcG_\pi(u)\leq 1} \fstop
\]
Since the supremum of $1$-Lipschitz functions is $1$-Lipschitz,~$\rho$ is in particular $\T_\mrmn^\tym{2}$-continuous, thus, the topology generated by~$\rho$ is coarser than~$\T_\mrmn$, and coincides with~$\tau_\mrmn$ if and only if~$\rho$ separates points.
Since~$\FC{\msA}{}$ separates points in~$\msP$, and since it is a unital algebra, for every~$\eta_1,\eta_2\in\msP$ we can find~$u\in \FC{\msA}{}$ such that~$u(\eta_1)>0$ and~$u(\eta_2)=0$.
Since~$\widehat\mcG(\FC{\msA}{})\subset \FC{\msA}{}\subset \Lip_b[\rho_{\mathrm{bL}}]$, the function~$u'\eqdef \tparen{\Lip_{\rho_{\mathrm{bL}}}[u]\vee \sup \widehat\mcG(u)^{1/2}}^{-1} u$ satisfies~$u'\in \msC$,~$\widehat\mcG(u')\leq 1$, and~$u'(\eta_1)>0$ and~$u(\eta_2)=0$.
This shows that~$\rho$ separates points, and thus that the topology generated by~$\rho$ coincides with~$\T_\mrmn$.
This concludes the proof of~\cite[$(\mbfA)$, \S4, p.~742]{Kuw98b}, and thus the proof of~\ref{i:t:Girsanov:2}.

\medskip

A proof of~\ref{i:t:Girsanov:4} is straightforward.
\ref{i:t:Girsanov:5} is a consequence of~\ref{i:t:Girsanov:1}, \ref{i:t:Girsanov:2} and~\cite[Thm.~3.4(i)]{AlbRoe95}.

\paragraph{Proof of~\ref{i:t:Girsanov:6}-\ref{i:t:Girsanov:7}}
The conclusion follows from~\cite[Thm.~1.1 and Cor.~1.2]{Wu00} provided we verify the assumptions there and the standing assumptions~\cite[(A), (A(i))-(A(iii)), p.~270]{Wu00}.
The assumption that~$\phi\in\dom{\widehat\mcE_\pi}$ was verified in the proof of~\ref{i:t:Girsanov:1} above.
The assumption that~$\phi^2\in\dom{\widehat\mcE_\pi}$ is~\eqref{eq:t:Girsanov:000}.

Recall that~$(\msP^\pa_\iso,\T_\mrma)$ is a Polish space by Corollary~\ref{c:PpaisoPolish}.
(The first part of) Assumptions~(A) holds for the weak atomic topology~$\T_\mrma$ on~$\msP^\pa_\iso$ by Proposition~\ref{p:PropertiesCylinder}\ref{i:p:PropertiesCylinder:1} and~\ref{i:p:PropertiesCylinder:4};
(A(i)) on~$\msP^\pa_\iso$ follows from Proposition~\ref{p:PropertiesCylinder}\ref{i:p:PropertiesCylinder:3} since the Borel $\sigma$-algebra on~$\msP^\pa$ contains all points.
(A(ii)) is Theorem~\ref{t:TransferP}\ref{i:t:TransferP:1};
(A(iii)) is trivial since~$\widehat\mcL \hCylP{}{0}\subset \hCylP{}{0}\subset L^\infty(\mcQ_\pi)$ for every~$\pi\in\msP(\mbfT_\circ)$.
\end{proof}

Since~$\EM\colon (\widehat\mbfM,\widehat\T_\mrmp)\to (\msP^\pa,\T_\mrma)$ is a quasi-homeomorphism intertwining the forms~$\tparen{\widehat\bmssE_\pi,\dom{\widehat\bmssE_\pi}}$ and $\tparen{\widehat\mcE_\pi,\dom{\widehat\mcE_\pi}}$ (Thm.~\ref{t:TransferQuasiHomeo}), it also preserves quasi-notions, including broad local spaces and quasi-continuity.
In particular, a Borel function~$\phi\colon \msP^\pa \to [-\infty,\infty]$ satisfies~$\phi\in \dotloc{\dom{\widehat\mcE_\pi}}$ if and only if there exists~$\boldvarphi\in \dotloc{\dom{\widehat\bmssE_\pi}}$ with~$\phi\circ\EM=\boldvarphi$.
It is then a straightforward albeit tedious verification that the image process
\begin{gather*}
\widehat\bmssW_\pi^{\boldvarphi}\eqdef {\EM^{-1}}_* \widehat\mcW_\pi^\phi = \tparen{\widehat\Omega_\circ, \widehat\msF^\pi,
\widehat\mbfX^{\pi,\boldvarphi}_\bullet, \tseq{\widehat P_{\mbfs,\mbfx}^{\boldvarphi}}, \zeta^{\boldvarphi}}
\intertext{with}
\widehat\mbfX^{\pi,\boldvarphi}_\bullet=\EM^{-1}(\widehat\mcX^{\pi,\phi}_\bullet)\comma \qquad \widehat P_{\mbfs,\mbfx}^{\boldvarphi}= \widehat\mcP^{\pi,\phi}_{\EM(\mbfs,\mbfx)} \comma \ (\mbfs,\mbfx)\in \EM^{-1}(\msP_\phi)\comma 
\qquad \zeta^{\boldvarphi}= \zeta^\phi\comma
\end{gather*}
is properly associated with the transformed Dirichlet form~$\tparen{\widehat\bmssE_\pi^{\boldvarphi},\dom{\widehat\bmssE_\pi^{\boldvarphi}}}$ constructed in the usual way on the space~$L^2(\widehat\boldnu_\pi^{\boldvarphi})$ with~$\widehat\boldnu_\pi^{\boldvarphi}\eqdef \boldvarphi^2\widehat\boldnu_\pi$.
Then, \emph{mutatis mutandis}, we have the following result.

\begin{corollary}
Fix~$\pi\in\msP(\mbfT_\circ)$ and suppose that Assumption~\ref{ass:Girsanov} is satisfied.
Then, the quadratic form~$\tparen{\widehat\bmssE_\pi^{\boldvarphi},\dom{\widehat\bmssE_\pi^{\boldvarphi}}}$ defined by
\begin{equation}\label{eq:c:Girsanov:0}
\begin{gathered}
\dom{\widehat\bmssE_\pi^{\boldvarphi}}\eqdef \set{u\in \dotloc{\dom{\widehat\bmssE_\pi}} \cap L^2(\widehat\boldnu_\pi^{\boldvarphi}) : \int_{\widehat\mbfM} \widehat\bmssGamma_\pi(u)\, \diff\widehat\boldnu_\pi^{\boldvarphi}<\infty } \comma
\\
\widehat\bmssE_\pi^{\boldvarphi}(u,v)\eqdef \int_{\widehat\mbfM} \widehat\bmssGamma_\pi(u,v)\, \diff\widehat\boldnu_\pi^{\boldvarphi} \comma \qquad u,v\in \dom{\widehat\bmssE_\pi^{\boldvarphi}} \comma 
\end{gathered}
\end{equation}
is a $\widehat\T_\mrmp$-quasi-regular local Dirichlet form on~$L^2(\widehat\boldnu_\pi^{\boldvarphi})$.

If, additionally,
\begin{equation}\label{eq:c:Girsanov:00}
\boldvarphi\in L^2(\widehat\boldnu_\pi) \qquad \text{and} \qquad\int \widehat\bmssGamma_\pi(\boldvarphi) \diff \widehat\boldnu_\pi <\infty \fstop
\end{equation}
Then, 
\begin{enumerate}[$(i)$]
\item\label{i:c:Girsanov:0} the form~$\tparen{\widehat\bmssE_\pi^{\boldvarphi},\dom{\widehat\bmssE_\pi^{\boldvarphi}}}$ is additionally strongly local and recurrent;
\item\label{i:c:Girsanov:1} for the generator~$\tparen{\widehat\bmssL_\pi^{\boldvarphi},\dom{\widehat\bmssL_\pi^{\boldvarphi}}}$ of~$\tparen{\widehat\bmssE_\pi^{\boldvarphi},\dom{\widehat\bmssE_\pi^{\boldvarphi}}}$ we have~$\Cb(\mbfT)\otimes \Cyl{}{}\subset \dom{\widehat\bmssL_\pi^{\boldvarphi}}$ and
\[
\widehat\bmssL_\pi^{\boldvarphi} u= \widehat\bmssL^{\boldvarphi} u\eqdef \widehat\bmssL u+\boldvarphi^{-1}\,\widehat\bmssGamma_\pi(\boldvarphi, u) \quad \text{on}\quad \Cb(\mbfT)\otimes \Cyl{}{} \fstop
\]

\item\label{i:c:Girsanov:2} $\tparen{\widehat\bmssE_\pi^{\boldvarphi},\dom{\widehat\bmssE_\pi^{\boldvarphi}}}$ is properly associated with the $\widehat\boldnu_\pi^{\boldvarphi}$-symmetric process~$\widehat\bmssW_\pi^{\boldvarphi}$, which is thus recurrent (in particular: conservative);

\item\label{i:c:Girsanov:4} a $\widehat\boldnu_\pi^{\boldvarphi}$-measurable set~$\mbfA\subset \widehat\mbfM$ is $\widehat\bmssE_\pi^{\boldvarphi}$-invariant if and only if it is $\widehat\bmssE_\pi$-invariant;

\item\label{i:c:Girsanov:5} 
$\widehat\bmssW^{\boldvarphi}_\pi$ is a solution to the martingale problem for~$\tparen{\widehat\bmssL^{\boldvarphi},\Cb(\mbfT)\otimes\Cyl{}{}}$. In particular, for every~$\mbfu\in\Cb(\mbfT)\otimes\Cyl{}{}$ ($\subset \Cb(\widehat\T_\mrmp)$), the process
\begin{equation}\label{eq:c:Girsanov:1}
M^{\class{u}}_t\eqdef \mbfu(\widehat\mbfX^{\pi,\boldvarphi}_t)- \mbfu(\widehat\mbfX^{\pi,\boldvarphi}_0)-\int_0^t (\widehat\bmssL^{\boldvarphi} \mbfu)(\widehat\mbfX^{\pi,\boldvarphi}_s)\diff s
\end{equation}
is an adapted square-integrable martingale with quadratic variation
\begin{equation}\label{eq:c:Girsanov:1.1}
\tquadvar{M^{\class{\mbfu}}}_t = \int_0^t \widehat\bmssGamma(\mbfu)(\widehat\mbfX^{\pi,\boldvarphi}_s)\diff s \comma
\end{equation}
\end{enumerate}
\end{corollary}

\begin{proposition}\label{p:GirsanovDFSimple}
Fix~$\mbfs\in\mbfT_\circ$.
A stochastic process
\[
\bmssV^\mbfs= \tparen{\Omega,\msF,\msF_\bullet, \mbfY^\mbfs_\bullet, \seq{P_{\mbfx}}_{\mbfx\in \mbfM_{\boldvarphi}},\zeta}
\]
with state space~$\EM^{-1}(\msP^\pa_\iso\cap \msP_\phi)$ is a solution to the martingale problem for $\tparen{\widehat\bmssL^{\boldvarphi},\Cb(\mbfT)\otimes\Cyl{}{}}$ if and only if, for every~$i\in\N$,
\[
f(Y^i_t)-f(Y^i_0) = \int_0^t (\mssL f)(Y^i_r) \diff r - \int_0^t \boldvarphi(\mbfs,\mbfY_r)^{-1}\, \widehat\bmssGamma_\pi\tparen{\boldvarphi,\car_\mbfT\otimes\, \ev(f)}(\mbfs,\mbfY^i_r)\diff r \fstop
\]

\begin{proof}
The forward implication holds choosing~$\mbfu=\car_\mbfT\otimes \ev(f)$ in~\eqref{eq:c:Girsanov:1}.
The converse implication is a standard application of the diffusion property for~$\tparen{\widehat\mcL^\phi_\pi,\dom{\widehat\mcL^\phi_\pi}}$ together with the It\^o formula for continuous $\R^d$-valued semi-martingales, similarly to the proof of Theorem~\ref{t:DK}\ref{i:t:DK:3}$\implies$\ref{i:t:DK:2}.
\end{proof}
\end{proposition}

We now establish analogous statements to those in Theorem~\ref{t:DK} in the case of the Girsanov transform~$\widehat\mcW_\pi^\phi$.
Effectively, we transfer the results in Theorem~\ref{t:DK} to~$\widehat\mcW_\pi^\phi$.
In order to do so, we need for this procedure to be \emph{invertible}, which requires the following additional assumption
\begin{equation}\label{eq:PhiTotalPartSpace}
0<\phi<\infty  \qquad \widehat\mcE_\pi\text{-q.e.\ on } \msP\fstop
\end{equation}
Under this assumption, the part space~$\msP_\phi$ defined in the beginning of the section can be chosen as the whole of~$\msP$.
Thus, we can transfer the $L^2$-theory, since~$(1/\phi)^2\mcQ_\pi^\phi=\mcQ_\pi$ as measures on (the whole of)~$\msP$, as well as the $\widehat\mcE_\pi$-theory, since the function~$1/\phi$ satisfies~$1/\phi\in\dotloc{\dom{\widehat\mcE_\pi}}$ by~\cite[Cor.~6.2]{Kuw98} and since~$\dotloc{\dom{\widehat\mcE_\pi}}=\dotloc{\dom{\widehat\mcE_\pi^\phi}}$ by~\cite[Cor.~4.12]{Fit97} (also cf.~\cite[Thm.~3.1$(i)$]{Kuw98b}).

\begin{theorem}\label{i:t:GirsanovDK}
Fix~$\pi\in\msP(\mbfT_\circ)$, and suppose Assumption~\ref{ass:Girsanov} is satisfied.
Further let~$\phi\in\dotloc{\dom{\widehat\mcE_\pi}}$ 
be satisfying~\eqref{eq:t:Girsanov:00},~\eqref{eq:t:Girsanov:000}, and~\eqref{eq:PhiTotalPartSpace}.
Finally, let~$(\rho_\bullet,\mbbP)$ be any \emph{time-homogeneous} $\msP^\pa\cap \msP_\phi$-valued $\mcQ_\pi^\phi$-substationary Hunt process adapted to its natural filtration with initial distribution~$P\circ\rho_0^{-1}$ equivalent to~$\mcQ_\pi$. 

Then, the following are equivalent:
\begin{enumerate}[$(i)$]
\item\label{i:t:GirsanovDK:1} $\rho_\bullet$ is a solution to the martingale problem for~$\tparen{\widehat\mcL_\pi^\phi,\dom{\widehat\mcL_\pi^\phi}}$, cf.~\eqref{eq:t:Girsanov:1};
\item\label{i:t:GirsanovDK:2} $\rho_\bullet$ is a solution to the martingale problem for~$(\widehat\mcL^\phi,\hCylP{}{0})$, see~\eqref{eq:t:Girsanov:1};
\item\label{i:t:GirsanovDK:3} $\rho_\bullet$ is a distributional solution to the shadow martingale problem~\eqref{eq:d:ShadowMP:0} driven by~$\widehat\mcL_\pi^\phi$, that is
\begin{equation*}
\widehat M^{\class{\molli\otimes f}}_t\eqdef\ (\molli\otimes f)^\trid(\rho_t) - (\molli\otimes f)^\trid(\rho_0) -\int_0^t ({\mcL^\phi}'\rho_s)(\molli\otimes f)\diff s
\end{equation*}
is an $\msF_\bullet$-adapted $\mbbP$-martingale for each $\molli\in\msR_0$ and each~$f\in\msA$.
\end{enumerate}

\begin{proof}
As in the proof of Theorem~\ref{t:DK}, \ref{i:t:GirsanovDK:1}$\implies$\ref{i:t:GirsanovDK:2}$\implies$\ref{i:t:GirsanovDK:3} holds by definition.
Since~$\rho_\bullet$ is a $\mcQ_\pi^\phi$-substationary Hunt process, it is identified ---uniquely up to~$\mcQ_\pi^\phi$-equivalence--- by its \emph{Markovian} generator on~$L^2(\mcQ_\pi^\phi)$.
Thus, again in analogy with Theorem~\ref{t:DK}, \ref{i:t:GirsanovDK:2}$\implies$\ref{i:t:GirsanovDK:1} holds since~$(\widehat\mcL^\phi,\hCylP{}{0})$ is Markov unique on~$L^2(\mcQ_\pi^\phi)$, as a consequence of Theorem~\ref{t:Girsanov}\ref{i:t:Girsanov:7}.

Similarly to the proof of~\cite[Thm.~3]{KonLehvRe19b}, the rest of the proof will follow directly from Theorem~\ref{t:DK}, by inverting the Girsanov transform by~$\phi$, as we now show.
Since~$(\rho_\bullet,\mbbP)$ is a Hunt process by assumption, we may consider its Girsanov transform~$\rho_\bullet^{1/\phi}$ by weight~$1/\phi$, which is a Borel right process, see~\cite[\S3, p.~741]{Kuw98} or~\cite{Fit97}.
Since~$(\rho_\bullet,\mbbP)$ is $\msP^\pa\cap\msP_\phi$-valued (recall that we chose~$\msP_\phi=\msP$),~$\rho_\bullet^{1/\phi}$ is $\msP^\pa$-valued.

Since~$\rho_\bullet$ is identified ---uniquely up to~$\mcQ_\pi^\phi$-equivalence--- by~$(\widehat\mcL^\phi,\hCylP{}{0})$, we may compute the generator of~$\rho_\bullet^{1/\phi}$ on~$u\in\hCylP{}{0}$ as~$\widehat\mcL^{\rho^{1/\phi}} u = \widehat\mcL^{\phi} u + \phi \, \widehat\mcG_\pi(\phi^{-1},u)=\widehat\mcL u$.
Thus, since~$\tparen{\widehat\mcL,\hCylP{}{0}}$ is essentially self-adjoint by Theorem~\ref{t:TransferP}\ref{i:t:TransferP:1}, the generator~$\tparen{\widehat\mcL^{\rho^{1/\phi}_\pi},\dom{\widehat\mcL^{\rho^{1/\phi}}_\pi}}$ on~$L^2(\mcQ_\pi)$ is exactly~$\tparen{\widehat\mcL_\pi,\dom{\widehat\mcL_\pi}}$.

This concludes the proof in light of the analogous assertions in Theorem~\ref{t:DK} and of Proposition~\ref{p:GirsanovDFSimple}.
\end{proof}
\end{theorem}

\subsection{Regular interaction energies}\label{ss:RIE}
Let us now specialize Theorem~\ref{t:Girsanov} to \emph{mean-field pairwise interactions}.
For~$g\colon \R^+\to \R$, we define the following assumption:
\begin{equation}\label{eq:AssSmallScaleInt}
\begin{gathered}
\text{there exists~$\delta>0$ such that}
\\
g \quad \text{is (strictly) decreasing and strictly positive on~$(0,\delta]$}\comma
\\
\lim_{t\to 0^+}g(t)=\infty  \qquad \text{and} \qquad \sup_{t\in [\delta,\infty)} \abs{g(t)} < \infty \fstop
\end{gathered}
\end{equation}

\begin{definition}[Regular interactions]\label{d:RegularPotential}
We say that~$\sigma\colon \R^+\to \R$ is a \emph{regular pair potential} if~$\sigma$ is piecewise of class~$\mcC^1$ and~\eqref{eq:AssSmallScaleInt} simultaneously holds for~$g=\sigma$ \emph{and}~$g=-\sigma'$.
For any regular pair potential~$\sigma\colon \R^+\to \R$ we call (\emph{repulsive}) \emph{$\sigma$-interaction energy} the function
\begin{equation}\label{eq:InteractionEnergy}
\IE_\sigma\colon \mbfT\times\mbfM_\circ\longrar\overline\R\comma \qquad \IE_\sigma\colon (\mbfs,\mbfx)\longmapsto \tfrac{1}{2}\sum_{\substack{i,j\\i\neq j}} s_i\, s_j \, \sigma\tparen{\mssd(x_i,x_j)} \comma
\end{equation}
and \emph{attractive $\sigma$-interaction energy} the function~$-\IE_\sigma$.
\end{definition}

In the measure representation, the function~$\IE_\sigma$ corresponds to the natural interaction energy
\[
(\IE_\sigma\circ\EM^{-1})(\eta)=\tfrac{1}{2}\iint \car_{\set{x\neq y}} \sigma\tparen{\mssd(x,y)} \diff\eta^{\otym{2}}(x,y)\fstop
\]

\begin{example}\label{e:PhysicalPotentials}
Physically relevant examples of regular pair potentials include:
\begin{itemize}[leftmargin=2em]
\item the  (\emph{repulsive}) \emph{Riesz-type interactions}~$\sigma=\sigma_p\colon t \mapsto \tfrac{1}{p}t^{-p}$, $p\geq 0$, and in particular the  (\emph{repulsive}) \emph{Coulomb-type potentials} ($p\in\N$);
\item the (\emph{repulsive}) \emph{logarithmic interaction}~$\sigma\colon t \mapsto (-\log t)^+$;
\item the (\emph{repulsive}) \emph{Mie-type interactions}~$\sigma\colon t\mapsto at^{-\alpha}-bt^{-\beta}$, $a,b>0$,~$\alpha>\beta>0$, and in particular the (\emph{repulsive}) \emph{Lennard-Jones-type potentials} ($\alpha/2=\beta=6$).
\end{itemize}
\end{example}

In the next few statements we show that, under an extremely mild moment bound on~$\pi$, all the interaction energies~$\IE_\sigma$ induced by regular pair potentials~$\sigma$ are elements of the broad local space of~$\widehat\bmssE_\pi$.
Let us firstly introduce this moment bound.

\begin{definition}[]\label{d:OnePlusMoment}
Let~$\tau\colon \R^+\to\R^+$ be a Borel-measurable function satisfying
\begin{enumerate}[$(\tau_1)$]
\item $\tau$ is eventually non-decreasing;
\item $\tau$ is eventually submultiplicative;
\item $\lim_{t\to\infty} \tau(t)=\infty$.
\end{enumerate}
We denote by~$\msP_*(\mbfT)$ the space of all Borel probability measures on~$\mbfT$ satisfying
\begin{equation}\label{eq:PiTau}
C_\pi(\tau)\eqdef \int_\mbfT \sum_i s_i\,\tau(s_i^{-1})\diff \pi(\mbfs)<\infty \quad \text{form some $\tau$ as above} \fstop
\end{equation}
\end{definition}

\begin{example}
For~$p>1$, let~$p'\eqdef \tfrac{p}{p-1}$ be its H\"older conjugate exponent.
A prototypical choice for~$\tau$ is~$\tau_p\colon t\mapsto t^{1/p'}$ for some~$p>1$.
In this case, the moment bound~\eqref{eq:PiTau} is just the requirement of $\pi$-integrability for the function~$\mbfs\mapsto \norm{\mbfs}_{\ell^{1/p}}^{1/p}$.
This makes clear that the requirement that~$\pi\in\msP_*(\mbfT)$ is extremely mild, since~$\mbfs\mapsto\norm{\mbfs}_{\ell^1}\equiv \car$ is always in~$L^1(\pi)$ and we can choose~$p$ to be arbitrarily close to~$1$.
\end{example}

\begin{example}
For the Poisson--Dirichlet measure in Example~\ref{e:PoissonDirichlet} we have
\[
C_{\Pi_\beta}(\tau_p)=\beta\, \Beta(p^{-1},\beta) <\infty\comma \qquad \beta>0\comma \quad p>1\fstop
\]
In particular,~$\Pi_\beta\in\msP_*(\mbfT)$ for every~$\beta>0$.

\begin{proof} Denote by~$\Beta$ the Euler Beta function. We compute
\begin{align*}
\int_\mbfT \norm{\mbfs}_{\ell^{1/p}}^{1/p} \diff\Pi_\beta(\mbfs)=& \int_\mbfI \norm{(\boldUpsilon\circ \boldLambda)(\mbfr)}_{\ell^{1/p}}^{1/p} \diff\bmssBeta_\beta(\mbfr)
= \int_\mbfI \norm{\boldLambda(\mbfr)}_{\ell^{1/p}}^{1/p} \diff\bmssBeta_\beta(\mbfr)
\\
=& \sum_i \int_\mbfI r_i^{1/p} \prod_k^{i-1} (1-r_k)^{1/p} \diff\bmssBeta_\beta(\mbfr)
\\
=& \sum_i \beta^i \int_0^1 r_i^{1/p} (1-r_i)^{\beta-1}\diff r_i \prod_k^{i-1}\int_0^1 (1-r_k)^{1/p+\beta-1} \diff r_k
\\
=&\ \beta\, \Beta(p^{-1},\beta)<\infty \fstop \qedhere
\end{align*}
\end{proof}
\end{example}

We shall need the following technical lemma.

\begin{lemma}\label{l:NestPhi}
Let~$\sigma$ be a regular pair potential and~$\tau$ be satisfying the conditions in Definition~\ref{d:OnePlusMoment}.
Then, there exists~$\vareps\colon \R^+\to\R^+$ satisfying~\eqref{eq:NestPhi} and additionally such that
\[
\sigma\circ \vareps \leq \tau^{1/4} \quad \text{and} \quad -\sigma' \circ\vareps \leq \tau^{1/4} \quad \text{eventually on~$\R^+$} \fstop
\]
\begin{proof}
Let~$\bar\sigma\eqdef \sigma\vee -\sigma'$ and note that it too satisfies~\eqref{eq:AssSmallScaleInt}.
In particular, it is invertible on~$(0,\delta]$, and we denote by~$\bar\sigma^{-1}\colon (\bar\sigma(\delta),\infty)\to\R$ its inverse.
Define
\begin{equation}
\vareps\colon t \longmapsto \begin{cases} \delta & \text{if } 0<t<\bar\sigma(\delta) \\ (\bar\sigma^{-1}\circ \tau^{1/4})(t) &\text{if } t\geq \bar\sigma(\delta)\end{cases}
\end{equation}
and note that~$\vareps$ satisfies~\eqref{eq:NestPhi} since~$\tau$ is non-decreasing and~$\bar\sigma$ satisfies~\eqref{eq:AssSmallScaleInt}.
Furthermore, on~$[\bar\sigma(\delta),\infty)$,
\[
\sigma\circ \vareps\comma -\sigma'\circ\vareps \leq \bar\sigma\circ\vareps = \tau^{1/4}\comma
\]
which concludes the assertion.
\end{proof}
\end{lemma}

\begin{proposition}\label{p:BLocSpW}
Let~$\sigma$ be a regular pair potential.
Under Assumption~\ref{ass:Girsanov}, then
\begin{enumerate}[$(i)$]
\item\label{i:p:BLocSpW:1} $\IE_\sigma\in \dotloc{\dom{\widehat\bmssE_\pi}}$ for every~$\pi\in\msP_*(\mbfT)$.
\end{enumerate}
If furthermore, there exists~$d\geq 0$ such that
\begin{gather}
\label{eq:p:BLocSpW:00}
\sup_{x\in M}\nu\tparen{B_r(x)} \in O(r^d) \quad \text{as } r\downarrow 0\comma
\\
\label{eq:p:BLocSpW:0}
\int_0^1 \sigma(r)\, r^{d-1} \diff r <\infty \quad \text{and} \quad -\int_0^1 \sigma'(r)\, r^{d-1}\diff r<\infty\comma
\end{gather}
then
\begin{enumerate}[$(i)$, resume]
\item\label{i:p:BLocSpW:2} $\IE_\sigma\in\dom{\widehat\bmssE_\pi}$ for every~$\pi\in\msP(\mbfT)$.
\end{enumerate}

\begin{proof}
We divide the proof into several parts.

\paragraph{Proof of~\ref{i:p:BLocSpW:1}}
Let~$\tau$ as in Definition~\ref{d:OnePlusMoment} be witnessing that~$\pi\in \msP_*(\mbfT)$, and let~$\vareps$ be given by Lemma~\ref{l:NestPhi}.
Define functions~$f_{i,j}\eqdef f_{i,j,n}\colon \mbfx\mapsto \mssd(x_i,x_j)\wedge \vareps(ijn)$ and
\[
\mbfu_n\colon (\mbfs,\mbfx)\longmapsto \tfrac{1}{2}\sum_{\substack{i,j\\i\neq j}} s_i\, s_j \, \sigma\tparen{f_{i,j}(\mbfx)} \fstop
\]
Further let~$\widehat\mbfF_n\eqdef \widehat\mbfF_n$ be the $\widehat\bmssE_\pi$-nest constructed in Proposition~\ref{p:AndiHatPolar}\ref{i:p:AndiHatPolar:1} for~$\vareps$ chosen above, and note that~$\mbfu_n\equiv \IE_\sigma$ on~$\widehat\mbfF_n$.
In particular,~$\mbfu_n\equiv \IE_\sigma$ on the $\widehat\bmssE_\pi$-quasi-open set~$\widehat\mbfG_n\eqdef \inter_{\widehat\bmssE_\pi}\widehat\mbfF_n$.
Since~$\tseq{\widehat\mbfF_n}_n$ is an $\widehat\bmssE_\pi$-nest,~$\widehat\mbfG_n\nearrow_n \widehat\mbfM$ up to $\widehat\bmssE_\pi$-polar sets by~\cite[Lem.~3.1]{Kuw98}.
Thus, by definition of broad local space, it suffices to show that~$\mbfu_n\in\dom{\widehat\bmssE_\pi}$ for every~$n\in\N$.

\paragraph{$L^2$-bound} For the $L^2$-norm of~$\mbfu_n$, respectively:
by~\eqref{eq:AssSmallScaleInt} for~$\sigma$;
since~$\mbfs\in\mbfT$ implies~$i<s_i^{-1}$, and by~\eqref{eq:NestPhi} and~\eqref{eq:AssSmallScaleInt} for~$\sigma$;
by Lemma~\ref{l:NestPhi} and the properties of~$\tau$;
by Jensen's inequality and since~$\sum_i s_i=1$,
we have
\begin{align*}
\abs{\mbfu_n}^2 \lesssim& \abs{\sum_{i,j} s_i\, s_j\, \sigma\tparen{\vareps(ijn)}}^2 \leq \abs{\sum_{i,j} s_i\, s_j\, \sigma\tparen{\vareps(s_i^{-1}s_j^{-1}n)}}^2 
\\
\leq&\ \sqrt{\tau(n)} \sum_{i,j,h,k} s_i\, s_j\, s_h\, s_k \sqrt[4]{\tau(s_i^{-1})\, \tau(s_j^{-1})\, \tau(s_h^{-1}) \tau(s_k^{-1}) }
\\
=&\ \sqrt{\tau(n)} \paren{\sum_i s_i \sqrt[4]{\tau(s_i^{-1})}}^4 \leq \sqrt{\tau(n)} \sum_i s_i\, \tau(s_i^{-1}) \fstop
\end{align*}
Integrating the above inequality w.r.t.~$\widehat\boldnu_\pi$ shows that~$\mbfu_n\in L^2(\widehat\boldnu_\pi)$ by definition of~$\tau$.

\paragraph{Energy bound} As for the $\widehat\bmssE_\pi$-energy of~$\mbfu_n$, we have
\begin{align}
\nonumber
4\,\widehat\bmssGamma_\pi(\mbfu_n)_{\mbfs,\emparg}=&\ \widehat\bmssGamma_\pi\paren{\sum_{\substack{i,j\\i\neq j}} s_i\, s_j\, \sigma\circ f_{i,j}}
\\
\nonumber
=&\ \sum_{\substack{i,j\\i\neq j}} \sum_{\substack{h,k\\h\neq k}} s_i\, s_j\, s_h\, s_k\, \tparen{\sigma'\circ f_{i,j}} \tparen{\sigma'\circ f_{h,k}} \, \bmssGamma (f_{i,j},f_{h,k})
\\
\nonumber
=&\ \sum_{\substack{i,j\\i\neq j}} \sum_{\substack{h,k\\h\neq k}} s_i\, s_j\, s_h\, s_k\, \tparen{\sigma'\circ f_{i,j}} \tparen{\sigma'\circ f_{h,k}} \, \sum_{a\in\set{i,j,h,k}} s_a^{-1}\mssGamma(f_{i,j},f_{h,k})
\\
\label{eq:p:BLocSpW:0.5}
=&\ 8 \sum_k\sum_{\substack{i,j\\i,j \neq k}} s_i\, s_j\, s_k \tparen{\sigma'\circ f_{i,k}} \tparen{\sigma'\circ f_{j,k}} \, \mssGamma(f_{i,k},f_{j,k}) \comma
\end{align}
where we used that~$\mssGamma$ is symmetric and~$f_{i,j}=f_{j,i}$. 
Integrating the above equality w.r.t.~$\widehat\boldnu_\pi$, we conclude that
\begin{align}\label{eq:p:BLocSpW:1}
\widehat\bmssE_\pi(\mbfu_n)= 2 \int_\mbfT \sum_k\sum_{\substack{i,j\\i,j \neq k}} s_i\, s_j\, s_k \int_{M^\tym{3}}\tparen{\sigma'\circ f_{i,k}} \tparen{\sigma'\circ f_{j,k}}\mssGamma(f_{i,k},f_{j,k})\diff\nu^\otym{3} \diff\pi(\mbfs)\fstop
\end{align}
Respectively:
by~\eqref{eq:AssSmallScaleInt} for~$\sigma'$;
since~$\mbfs\in\mbfT$ implies~$i<s_i^{-1}$, and by~\eqref{eq:NestPhi} and~\eqref{eq:AssSmallScaleInt} for~$\sigma'$;
by Lemma~\ref{l:NestPhi} and the properties of~$\tau$,
we further have that
\begin{align*}
-\sigma'\circ f_{i,k}\lesssim -\sigma'\tparen{\vareps(ikn)}\leq -\sigma'\tparen{\vareps(s_i^{-1}s_k^{-1}n)}\leq \sqrt[4]{\tau(n)\tau(s_i^{-1})\tau(s_k^{-1}) }\comma
\end{align*}
and therefore
\begin{align}
\nonumber
\sum_k\sum_{\substack{i,j\\i,j \neq k}} s_i\, s_j\, s_k \tparen{\sigma'\circ f_{i,k}} \tparen{\sigma'\circ f_{j,k}} \lesssim&\ \sqrt{\tau(n)} \sum_{i,j,k} s_i\, s_j\, s_k\, \sqrt[4]{\tau(s_i^{-1})\tau(s_j^{-1})\tau(s_k^{-1})^2}
\\
\nonumber
=&\ \sqrt{\tau(n)} \paren{\sum_i s_i \sqrt[4]{\tau(s_i^{-1})}}^2 \sum_k s_k\sqrt{\tau(s_k^{-1})}
\\
\label{eq:p:BLocSpW:2}
\leq&\ \sqrt{\tau(n)} \sum_k s_k \tau(s_k^{-1})
\end{align}
by repeated applications of Jensen's inequality and since~$\sum_k s_k=1$.

Using the estimate~\eqref{eq:p:BLocSpW:2} in~\eqref{eq:p:BLocSpW:1} and the locality of~$\mssGamma$ then shows that
\[
\widehat\bmssE_\pi(\mbfu_n)\lesssim_{n,\pi} \int_{M^\tym{3}} \abs{\mssGamma\tparen{\mssd(\emparg, x), \mssd(\emparg,y)}(z)}\diff\nu^\otym{3} (x,y,z)<\infty
\]
by~\eqref{eq:RademacherBaseSp}.

This concludes the proof of~\ref{i:p:BLocSpW:1}.

\paragraph{Proof of~\ref{i:p:BLocSpW:2}}
We only show the energy bound. The $L^2$-bound is similar and simpler.
Similarly to the proof of~\eqref{eq:p:BLocSpW:1} above, we have
\begin{align*}
\widehat\bmssE_\pi(\IE_\sigma) =&\ 2 \int_\mbfT \sum_{\substack{i,j,k\\i,j \neq k}} s_i\, s_j\, s_k \diff\pi(\mbfs) \int_{M^\tym{3}}\sigma' \tparen{\mssd(x_i,x_k)}\, \sigma' \tparen{\mssd(x_j,x_k)}
\\
&\qquad\qquad\qquad\qquad\qquad\qquad \cdot \mssGamma\tparen{\mssd(x_i,\emparg),\mssd(x_j,\emparg)}(x_k)\diff\nu^\otym{3}(x_i,x_j,x_k)
\\
\lesssim& \int_{M^\tym{3}}\sigma' \tparen{\mssd(x,z)}\, \sigma' \tparen{\mssd(y,z)}\diff\nu^\otym{3}(x,y,z)
\end{align*}
where we used that~$\mssGamma\tparen{\mssd(x,\emparg)}\in L^\infty(\nu)$ by Assumption~\ref{ass:Girsanov}. Thus, by~\eqref{eq:p:BLocSpW:00},
\begin{align*}
\widehat\bmssE_\pi(\IE_\sigma) \lesssim& \int_M\braket{\int_M\sigma' \tparen{\mssd(x,z)}\diff\nu(x)}^2 \diff\nu(z) \lesssim \braket{\int_0^1 \sigma'(r)\, r^{d-1}\diff r}^2\comma
\end{align*}
which is finite by~\eqref{eq:p:BLocSpW:0}.
\end{proof}
\end{proposition}

As a consequence of the chain rule for~$\widehat\bmssGamma_\pi$ we have the following.

\begin{corollary}\label{c:RepulsivePhiDomLoc}
Let~$\sigma$ be a regular pair potential.
\begin{enumerate}[$(i)$]
\item\label{i:c:RepulsivePhiDomLoc:1} If~$\pi\in\msP_*(\mbfT)$, then
\[
\phi^\pm_{\beta,\sigma}\eqdef \exp\tbraket{-\beta \tparen{\pm\IE_\sigma}}\circ \EM \in \dotloc{\dom{\widehat\mcE_\pi}}\comma \qquad \beta>0 \fstop
\]

\item\label{i:c:RepulsivePhiDomLoc:2} If~$\pi\in\msP(\mbfT)$ and $\sigma$ additionally satisfies~\eqref{eq:p:BLocSpW:00} and~\eqref{eq:p:BLocSpW:0}, then $\log\phi^-_{\beta,\sigma}\in\dom{\widehat\mcE_\pi}$ and all the conclusions in Theorem~\ref{t:Girsanov} hold true.
\end{enumerate}

\begin{proof}
Since~$\sigma$ is bounded below, it is clear that~$\phi^-_{\beta,\sigma}\in L^\infty(\mcQ_\pi)$.
Furthermore, we have~$\log\phi^-_{\beta,\sigma}=-\beta \IE_\sigma\circ \EM\in \dom{\widehat\mcE_\pi}$ by Proposition~\ref{p:BLocSpW}, which verifies~\eqref{eq:t:Girsanov:000} and therefore~\eqref{eq:t:Girsanov:00} in Theorem~\ref{t:Girsanov}.
The conclusion follows by the theorem.
\end{proof}
\end{corollary}

In order to make the assertion of Proposition~\ref{p:BLocSpW} more concrete, let us discuss the following simple example, which makes apparent that the integrability conditions in~\eqref{eq:p:BLocSpW:0} coincide with the standard integrability conditions for interaction potentials.

\begin{example}[Riemannian manifolds]\label{e:RiemannianGirsanov}
Let~$(M,g,\nu)$ be a weighted Riemannian manifold as in Assumption~\ref{ass:WRM}.
Then, Assumption~\ref{ass:Girsanov} is satisfied with~$\mssd=\mssd_g$ the intrinsic distance on~$(M,g)$ in light of the standard Rademacher Theorem, and~\eqref{eq:p:BLocSpW:00} is satisfied with~$d$ the dimension of~$M$.

For the physical potentials in Example~\ref{e:PhysicalPotentials}, the integrability conditions in~\eqref{eq:p:BLocSpW:0} are readily verified when~$\sigma\colon t\mapsto (-\log t)^+$ is the logarithmic interaction potential, for every~$d\geq 2$; and for Riesz-type interaction potentials~$\sigma\colon t\mapsto \tfrac{1}{p}t^{-p}$ if and only if the interaction is integrable, i.e.\ when~$p<d-1$.
\end{example}

\begin{appendix}
\section{Dirichlet forms and Markov processes}\label{app:DirichletForms}
For the reader's convenience, let us collect here some definitions, notations, and facts on the theory of Dirichlet forms and Markov processes.
We refer the reader to the monographs~\cite{MaRoe92, FukOshTak11} for a systematic treatment of the subject.

\subsection{Dirichlet forms}
If not otherwise specified, in the following~$(M,\T)$ is a metrizable Luzin space, and~$\mssm$ is a $\sigma$-finite Borel measure with full $\T$-support.
By a \emph{Dirichlet form}~$\tparen{\mssE,\dom{\mssE}}$ on~$L^2(\mssm)$ we shall always mean a \emph{symmetric} Dirichlet form.
By definition,~$\dom{\mssE}$ is a dense linear subspace of~$L^2(\mssm)$, and~$\mssE\colon \dom{\mssE}^\tym{2}\to\R$ is a non-negative definite symmetric bilinear form which is both
\begin{itemize}
\item \emph{closed}, i.e., such that~$\dom{\mssE}$ is a Hilbert space with the inner product~$\mssE_1\eqdef \mssE + \scalar{\emparg}{\emparg}_{L^2(\mssm)}$, and 
\item \emph{Markovian}, i.e., such that~$f^+\wedge \car\in\dom{\mssE}$ and~$\mssE(f^+\wedge \car,f^+\wedge \car)\leq \mssE(f,f)$ for~$f\in\dom{\mssE}$.
\end{itemize}
As it is customary, we write~$\mssE(f)$ in place of~$\mssE(f,f)$ for every~$f\in L^2(\mssm)$.
We denote by~$\tparen{\mssL,\dom{\mssL}}$ the (non-positive definite) \emph{generator} of~$\tparen{\mssE,\dom{\mssE}}$, defined by~$\dom{\sqrt{-\mssL}}=\dom{\mssE}$ and~$\scalar{-\mssL f}{g}_{L^2(\mssm)}=\mssE(f,g)$ for every~$f\in \dom{\mssL}$ and~$g\in\dom{\mssE}$.
We further denote by~$\mssT_\bullet\eqdef \seq{\mssT_t}_{t\geq 0}$, with~$\mssT_t\eqdef e^{t\mssL}$, the strongly continuous contraction \emph{semigroup} associated to~$\tparen{\mssE,\dom{\mssE}}$ in the sense of, e.g.~\cite[p.~27]{MaRoe92}. We denote again by~$\mssT_\bullet$ all the semigroups of non-relabeled extensions~$\mssT_t\colon L^p(\mssm)\to L^p(\mssm)$ for~$p\in [1,\infty]$.

A linear subspace~$\msC$ of~$\dom{\mssE}$ is
\begin{itemize}
\item a \emph{form-core} if it is $\mssE_1^{1/2}$-dense in~$\dom{\mssE}$;
\item a \emph{core} if~$\msC\subset \Cc(\T)$ is a form-core and is~$\Cb$-dense in~$\Cz(\T)$;
\item a \emph{generator-core} if~$\msC\subset \dom{\mssL}$ and~$\tparen{\mssL,\dom{\mssL}}$ is the unique $L^2(\mssm)$-self-adjoint extension of~$(\mssL,\msC)$, in which case~$(\mssL,\msC)$ is called \emph{essentially self-adjoint}.
\end{itemize}

A Dirichlet form is called \emph{regular} if each of the following conditions holds:
\begin{itemize}
\item $(M,\T)$ is second-countable locally compact Hausdorff;
\item $\mssm$ is a Radon measure;
\item $\tparen{\mssE,\dom{\mssE}}$ has a core.
\end{itemize}

For an $\mssm$-measurable function~$f\colon M\to\R$ (equivalently, for its class up to $\mssm$-equivalence), we indicate by~$\supp[f]\eqdef \supp_\mssm[f]$ the support of the measure~$f\diff\mssm$, which is well-defined (i.e.\ independent of the choice of $\mssm$-representatives for~$f$) since~$(M,\T)$ is a metrizable Luzin space, cf.~\cite[p.~148]{MaRoe92}.
A regular Dirichlet form is called
\begin{itemize}
\item \emph{local} if~$\mssE(f,g)=0$ for every~$f,g\in\dom{\mssE}$ with~$\supp[f]$,~$\supp[g]$ compact and disjoint;
\item \emph{strongly local} if~$\mssE(f,g)=0$ for every~$f,g\in\dom{\mssE}$ with~$\supp[f]$, $\supp[g]$ compact and such that~$\supp[f-a\car]\cap \supp[g]=\emp$ for some~$a\in\R$;
\end{itemize}

For an $\mssm$-measurable~$A\subset M$ we write~$\car_A$ for the indicator function of~$A$, as well as for the multiplication operator~$\car_A\colon f\to \car_A \cdot f$.
An $\mssm$-measurable~$A\subset M$ is called \emph{$\mssT_\bullet$-invariant} (equivalently, \emph{$\mssE$-invariant}) if the commutation relation~$\mssT_t \car_A =\car_A\mssT_t$ holds on~$L^2(\mssm)$ for every~$t > 0$.
A semigroup~$\mssT_\bullet$ (equivalently, the corresponding Dirichlet form~$\tparen{\mssE,\dom{\mssE}}$) is
\begin{itemize}[wide]
\item \emph{irreducible} if every $\mssT_\bullet$ invariant~$A$ is either $\mssm$-negligible or $\mssm$-conegligible;
\item \emph{conservative} if~$\mssT_t\car=\car$ for every~$t\geq 0$.
\end{itemize}

\subsection{Transition kernels}
We denote by~$\mssh_\bullet\eqdef \seq{\mssh_t}_{t> 0}$ a \emph{semigroup of transition kernel measures} with~$\mssh_t=\mssh_t(x,\emparg)$ a Borel measure for every~$x$ and every~$t>0$.
If~$\mssh_t(x,\emparg)\ll \mssm$ for some~$x\in M$, we denote as well by~$\mssh_t(x,y)\eqdef\frac{\mssh_t(x,\diff\emparg)}{\diff\mssm}(y)$ its Radon--Nikod\'ym density w.r.t.~$\mssm$.
For a semigroup~$\mssh_\bullet$ of transition kernel measures we further let~$\mssH_\bullet\eqdef\seq{\mssH_t}_{t\geq 0}$ be the \emph{semigroup represented by~$\mssh_\bullet$}, defined by
\[
(\mssH_t f)(x)\eqdef \int f(y)\, \mssh_t(x,\diff y)\comma \qquad f\in \Bb \comma\quad t>0\comma x\in M\comma
\]
and again by the same symbol~$\mssH_\bullet$ its unique bounded extensions to~$L^2(\mssm)$.
We say that~$\mssh_\bullet$ is \emph{sub-Markovian} if~$\mssH_t \car\leq \car$ for every~$t\geq 0$, and \emph{Markovian} if~$\mssH_t \car= \car$ for every~$t\geq 0$.

\begin{definition}[$p$-sub-stationarity]\label{d:Substationary}
Fix~$p\geq 1$.
A semigroup of transition kernel measures~$\mssh_\bullet$ is \emph{$p$-sub-stationary w.r.t.~$\mssm$} if
\begin{equation}\label{eq:d:MartP:1}
\int (\mssH_t f)^p \diff\mssm \leq \int f^p \diff\mssm \comma \qquad f\in\Bb^+ \comma
\end{equation}
and it is simply called $\mssm$-\emph{sub-stationary} if it is $1$-sub-stationary w.r.t.~$\mssm$.
If~$\mssh_\bullet$ is $\mssm$-sub-stationary, the measure $\mssm$ is called \emph{$\mssh_\bullet$-super-median}.
\end{definition}

Let~$\mssh_{\bullet}$ be sub-Markovian.
It follows from H\"older's inequality that if~$\mssh_\bullet$ is sub-stationary, then it is~$p$-sub-stationary for every~$p\geq 1$.

In order for the semigroup~$\mssH_\bullet$ to coincide with the semigroup~$\mssT_\bullet$ of a Dirichlet form~$\tparen{\mssE,\dom{\mssE}}$ on~$L^2(\mssm)$, it is sufficient that the corresponding~$\mssh_\bullet$ be sub-Markovian, sub-stationary w.r.t.~$\mssm$, and such that~$\mssm$-$\lim_{t\downarrow 0}\mssH_t f=f$ for all~$f$ in a dense subset of~$L^2(\mssm)$; see~\cite[Prop.~II.4.3]{MaRoe92} (together with~\cite[Prop.~I.4.3 and Thm.~I.4.4]{MaRoe92}).
When~$\mssm$ is a finite measure, the converse implications holds too, since~$\mssT_{\bullet}$ extends to a non-relabeled contraction semigroup~$\mssT_{\bullet}\colon L^1(\mssm)\to L^1(\mssm)$.

\subsection{Markov processes}
Let~$M_\partial$ denote the space~$M$ with the addition of an isolated cemetery point~$\partial$.
For an $\mssm$-special standard process
\begin{equation}\label{eq:MarkovP}
\mssW\eqdef \seq{\Omega,\msF,\msF_\bullet,X_\bullet, \seq{P_x}_{x\in M_\partial},\zeta}
\end{equation}
with state space~$M_\partial$ and life-time~$\zeta$ set, for every Borel probability measure~$\mu$ on~$M$,
\begin{equation}\label{eq:MarkovProbab}
P_\mu\eqdef \int_{M_\partial} P_x\diff\mu(x) \fstop
\end{equation}
We denote the sub-Markovian semigroup transition kernel measures of~$\mssW$ by~$\mssh_\bullet$; see~e.g.~\cite[Thm.~IV.1.20, p.~96]{MaRoe92}.

\medskip

We refer the reader to~\cite[Dfn.~IV.6.3]{MaRoe92} for the notion of $\mssm$-equivalence of ($\mssm$-special standard) processes.

\subsection{Image objects}\label{sss:ImageObjects}
Let~$M,M^\sharp$ be measurable spaces, and~$j\colon M\to M^\sharp$ be an \emph{injective} measurable map.
For a measure~$\mssm$ on~$M$, we denote by~$\mssm^j\eqdef j_\pfwd\mssm$ its push-forward to~$M^\sharp$.
Since~$j$ is injective, the pullback~$j^*\colon L^2(\mssm^j)\to L^2(\mssm)$ defined by~$j^*f^\sharp=f^\sharp\circ j$ is an isomorphism of~$L^2$-spaces, with inverse denoted by~$j_*$.

For a bilinear form~$\tparen{\mssQ,\dom{\mssQ}}$ on~$L^2(\mssm)$, we denote by~$\tparen{\mssQ^j,\dom{\mssQ^j}}$ the image form of~$\tparen{\mssQ,\dom{\mssQ}}$ via~$j$, viz.
\begin{equation}\label{eq:ImageForm}
\dom{\mssQ^j}\eqdef j_*\dom{\mssQ} \comma \qquad \mssQ^j(f^\sharp,g^\sharp)\eqdef \mssQ(j^*f^\sharp,j^*g^\sharp) \fstop
\end{equation}
Images via~$j$ of operators, functionals, kernels, etc.\ are defined and denoted analogously.

Let~${M^\sharp}_{\partial^\sharp}$ be the augmentation of~$M^\sharp$ by an isolated point~$\partial^\sharp$.
We augment~$j$ to a map~$j^\partial\colon M_\partial\to {M^\sharp}_{\partial^\sharp}$ by setting~$j(\partial)\eqdef \partial^\sharp$.
For a stochastic process as in~\eqref{eq:MarkovP} with augmented state space~$M_\partial$, we denote by
\begin{equation}\label{eq:ImageMarkovP}
j_*\mssW\eqdef \seq{\Omega,\msF,\msF_\bullet, X^j_\bullet, \tseq{P^j_{x^\sharp}}_{x^\sharp\in {M^\sharp}_{\partial^\sharp}}}
\end{equation}
the image process of~$\mssW$ via~$j$ with state space~${M^\sharp}_{\partial^\sharp}$, defined on the same stochastic basis~$\seq{\Omega,\msF,\msF_\bullet}$ as~$\mssW$ with trajectories~$X^j_\bullet\eqdef j\circ X_\bullet$ and transition probabilities~$P^j_{x^\sharp}\eqdef j_\pfwd P_x$ if~$x^\sharp=j^\partial(x)$ (well-defined since~$j^\partial$ is injective) and~$P^j_{x^\sharp}\eqdef \delta_{\partial^\sharp}$ if~$x^\sharp$ is not in the image of~$j^\partial$.

\subsection{Quasi-regularity, quasi-homeomorphism, and transfer}
Let \linebreak $(M,\T,\mssm)$ be a Hausdorff topological $\sigma$-finite Borel space and~$\tparen{\mssE,\dom{\mssE}}$ be a Dirichlet form on $L^2(\mssm)$.

\subsubsection{Quasi-regularity}
For a Borel~$A\subset M$ set
\[
\dom{\mssE}_A\eqdef \set{f\in \dom{\mssE} : f = 0 \text{~$\mssm$-a.e.~on~} M\setminus A}\fstop
\]
A sequence $\seq{A_n}_n$ of Borel sets in~$M$ is a \emph{Borel $\mssE$-nest} if $\cup_n \dom{\mssE}_{A_n}$ is dense in~$\dom{\mssE}$.
If~$(p_A)$ is a property of Borel sets, a \emph{$(p)$-$\mssE$-nest} is a Borel nest~$\seq{A_n}$ so that~$(p_{A_n})$ holds for every~$n$. In particular, a \emph{closed $\mssE$-nest}, henceforth simply called an \emph{$\mssE$-nest}, is a Borel $\mssE$-nest consisting of closed sets.

A set~$N\subset M$ is \emph{$\mssE$-polar} if there exists an $\mssE$-nest~$\seq{F_n}_n$ so that~$N\subset M\setminus \cup_n F_n$.
A set~$G\subset M$ is \emph{$\mssE$-quasi-open} if there exists an $\mssE$-nest~$\seq{F_n}_n$ so that~$G\cap F_n$ is relatively $\T$-open in~$F_n$ for every~$n\in \N$.
A set~$F$ is \emph{$\mssE$-quasi-closed} if~$M \setminus F$ is $\mssE$-quasi-open.
Any countable union or finite intersection of $\mssE$-quasi-open sets is $\mssE$-quasi-open; analogously, any countable intersection or finite union of $\mssE$-quasi-closed sets is $\mssE$-quasi-closed;~\cite[Lem.~2.3]{Fug71}.

A property~$(p_x)$ depending on~$x\in M$ holds $\mssE$-\emph{quasi-everywhere} (in short:~$\mssE$-q.e.) if there exists an $\mssE$-polar set~$N$ so that~$(p_x)$ holds for every~$x\in M \setminus N$.
Given sets~$A_0,A_1\subset M$, we write~$A_0\subset A_1$ $\mssE$-q.e.\ if~$\car_{A_0}\leq \car_{A_1}$ $\mssE$-q.e.. Let the analogous definition of~$A_0=A_1$ $\mssE$-q.e.\ be given.

A Borel function~$f$ is \emph{$\mssE$-quasi-continuous} if there exists an $\mssE$-nest~$\seq{F_n}_n$ so that~$f\restr{F_n}$ is $\T$-continuous for every~$n\in \N$.
Equivalently,~$\reptwo f$ is $\mssE$-quasi-continuous if and only if it is $\mssE$-q.e.\ finite and $\reptwo f^{-1}(U)$ is $\mssE$-quasi-open for every open~$U\subset \R$, see e.g.~\cite[p.~70]{FukOshTak11}.
Whenever~$f\in L^0(\mssm)$ has an $\mssE$-quasi-continuous Borel $\mssm$-version, we denote it by~$\reptwo f$.

\begin{definition}[Quasi-regularity]\label{d:QuasiReg}
A Dirichlet form~$(\mssE,\dom{\mssE})$ is $\T$-\emph{quasi-regular} if
\begin{enumerate}[$(\mathsc{qr}_1)$]
\item\label{i:d:QuasiReg:1} there exists a $\T$-compact $\mssE$-nest;
\item\label{i:d:QuasiReg:2} there exists an $\mssE^{1/2}_1$-dense subset of~$\dom{\mssE}$ the elements of which all have $\mssE$-quasi-continuous $\mssm$-versions;
\item\label{i:d:QuasiReg:3} there exists an $\mssE$-polar set~$N$ and a countable family~$\seq{u_n}_n$ of functions~$u_n\in\dom{\mssE}$, all having $\mssE$-quasi-continuous $\mssm$-versions~$\reptwo u_n$ so that~$\seq{\reptwo u_n}_n$ separates points in~$M\setminus N$.
\end{enumerate}
\end{definition}

\subsubsection{Quasi-homeomorphism}
Let~$(M,\T)$ be a Hausdorff topological space,~$\mssm$ be a $\sigma$-finite Borel measure on~$(X,\T)$, and $\tparen{\mssE,\dom{\mssE}}$ be a Dirichlet form on~$L^2(\mssm)$. Further let the analogous definitions be given for~$\square^\sharp$ objects.

\begin{definition}[Quasi-homeomorphism]\label{d:QuasiHomeo}
The forms~$\tparen{\mssE,\dom{\mssE}}$ and \linebreak $\tparen{\mssE^\sharp,\dom{\mssE^\sharp}}$ are \emph{quasi-homeo\-morphic} if there exists an $\mssE$-nest~$\seq{F_k}_k$, an $\mssE^\sharp$-nest~$\tseq{F^\sharp_k}_k$ and an injective map~$j\colon \scup_k F_k\to\scup_k F^\sharp_k$ such that
\begin{enumerate}[$(\mathsc{qh}_1)$]
\item\label{i:d:QuasiHomeo:1} $j\colon (F_k,\T)\to(F^\sharp_k,\T^\sharp)$ is a homeomorphism;
\item\label{i:d:QuasiHomeo:2} $\mssm^\sharp=\mssm^j\eqdef j_\pfwd\mssm$ is the image measure of~$\mssm$ via~$j$;
\item\label{i:d:QuasiHomeo:3} $\tparen{\mssE^\sharp,\dom{\mssE^\sharp}}=\tparen{\mssE^j,\dom{\mssE^j}}$ is the image form~\eqref{eq:ImageForm} of~$\tparen{\mssE,\dom{\mssE}}$ via~$j$.
\end{enumerate}
\end{definition}

\subsection{Martingale problem}
Let~$\Omega$ be the Skorokhod space on~$[0,\infty)$ with values in~$M$ and possibly finite life-time~$\zeta\colon M\to [0,\infty]$.
That is,~$\Omega$ consists of all~$\omega\colon [0,\infty)\to M_\partial$ such that
\begin{enumerate}[$(a)$]
\item\label{i:d:Pathspace:1} $\omega(0)\in E$;
\item\label{i:d:Pathspace:2} $\omega$ is c\`adl\`ag on~$[0,\zeta)$ and~$\omega(t)=\partial$ for every~$t\geq \zeta$, where
\item\label{i:d:Pathspace:3} $\zeta(\omega)\eqdef \inf\set{t\geq 0: \omega(t)=\partial}$.
\end{enumerate}
Further let~$X_t(\omega)\eqdef \omega(t)$, with~$t\in [0,\infty)$, be the coordinate process on~$\Omega$, and set
\[
\msF^{0}_t\eqdef \sigma\tparen{X_s: s\in [0,t]}\comma \qquad \msF_t^+\eqdef \bigcap_{s>t} \msF^{0}_s \comma \qquad t\geq 0 \fstop
\]
This setting extends the one in~\cite[\S{I.1.a)}, p.10]{Ebe95} and coincides with it if we replace~\ref{i:d:Pathspace:2} with
\begin{enumerate}[$(a')$]\setcounter{enumi}{1}
\item $\omega$ is continuous on~$[0,\zeta)$ and~$\omega(t)=\partial$ for every~$t\geq \zeta$.
\end{enumerate}

\medskip

We shall make use of the following definition of solution to a martingale problem, which is well-suited to the generality of quasi-regular Dirichlet spaces.
We assume that the operator~$(\mssL,\msA)$ is defined on~$\msA\subset \Cb(\T)$ and that~$\mssm$ is an invariant measure for~$(\mssL,\msA)$, i.e.\ such that~$\mssL f\in L^1(\mssm)$ and~$\mssm(\mssL f)=0$ for every~$f\in\msA$.

\begin{definition}[Martingale problem]\label{d:MartP}
Fix a Borel probability measure~$\mu$ on~$M$, with~$\mu\ll\mssm$.
A (time-homogeneous) $\mssm$-special standard process~$\mssW$ as in~\eqref{eq:MarkovP} is called a \emph{sub-stationary solution to the martingale problem for~$(\mssL,\msA,\mu)$} if and only if, for every~$f\in\msA$ and every~$t\geq 0$,
\begin{enumerate}[$(a)$]
\item\label{i:d:MartP:1} the associated semigroup of transition kernel measures is $2$-sub-stationary w.r.t.~$\mssm$, see~\eqref{eq:d:MartP:1};
\item (\emph{martingale property}) the family
\begin{equation}\label{eq:d:MartP:2}
M^{\class{f}}_t\eqdef \tilde f(X_t)-\tilde f(X_0)-\int_0^t (\mssL f)(X_s)\diff s\comma \qquad t\geq 0\comma
\end{equation}
is an $\msF_\bullet$-martingale under~$P_\mu$ for any $\mssW$-quasi-continuous version~$\reptwo{f}$ of~$f$, where~$P_\mu$ is defined as in~\eqref{eq:MarkovProbab}.
\end{enumerate}
\end{definition}

\begin{remark}
\begin{enumerate*}[$(a)$]
\item The condition in~\eqref{eq:d:MartP:1} implies that~$\int_0^t (\mssL f)(X_s)\diff s$ is $P_\mssm$-a.s.\ (hence~$P_\mu$-a.s.) independent of the $\mssm$-version chosen for~$\mssL f$, so that the martingale in~\eqref{eq:d:MartP:2} is well-defined on $L^2(\mssm)$-classes~$f$; (Cf.~\cite[Dfn.~3.1 and Eqn.~(3.0)]{AlbRoe90}.)
\item In or general setting, uniqueness results ought not to be expected without the condition in Definition~\ref{d:MartP}\ref{i:d:MartP:1}. (Cf.~\cite[pp.~11f.]{Ebe95} for some negative statements and~\cite[Thm.~3.5]{AlbRoe95} for an affirmative one.)
\end{enumerate*}
\end{remark}

\section{Fine properties of cylinder functions}\label{app:FineCylinder}
In this section we discuss some fine properties of cylinder functions in relation to the Dirichlet form~$\tparen{\widehat\mcE_\pi,\dom{\widehat\mcE_\pi}}$ in Theorem~\ref{t:TransferP}.

\subsection{Capacity estimates}\label{app:Capacity}
Let~$(M,\T)$ be a metrizable Luzin space,~$\nu$ be a Borel \emph{probability} measure on~$(M,\T)$, and~$\tparen{\mssE,\dom{\mssE}}$ be a recurrent quasi-regular Dirichlet form on~$L^2(\nu)$ with properly associated Hunt process~$\mssW$.
For any $\mssE$-quasi-open~$E\subset M$, we denote by~$\tparen{\mssE^E,\dom{\mssE^E}}$ the \emph{part on~$E$ of~$\tparen{\mssE,\dom{\mssE}}$}, that is the Dirichlet form
\begin{gather*}
\dom{\mssE^E}\eqdef\set{f\in\dom{\mssE}: f\equiv 0 \ \mssE\text{-q.e.\ on } E^\complement}\comma
\\
\mssE^E(f,g)\eqdef \mssE(f,g)\comma \quad f,g\in\dom{\mssE^E} \fstop
\end{gather*}
%
For any~$A\subset E$ we further define the (\emph{first-order}) \emph{relative capacity}
\begin{align*}
\Cap(A,E)&=\Cap_\mssE(A,E)\eqdef \Cap_{\mssE^E}(A)\comma
\end{align*}
as the (first-order) 
$\mssE^E$-capacity of~$A$. Note that, by definition of capacity,
\begin{equation}\label{eq:CapRelCap}
\Cap(A)\leq \Cap(A,E) \comma \qquad E \subset X \ \text{$\mssE$-quasi-open} \fstop
\end{equation}

\medskip

For~$i=1,2$, let~$\tparen{\mssE^i,\dom{\mssE^i}}$ be defined as~$\tparen{\mssE,\dom{\mssE}}$ above, and~$\tparen{\mssE,\dom{\mssE}}$ be the corresponding product Dirichlet form on~$L^2(\nu_1\otimes \nu_2)$, see e.g.~\cite[Dfn.~V.2.1.1, p.~200]{BouHir91}.

\begin{lemma}\label{l:Capacity1}
For every~$\mssE^i$-capacitale~$A_i\subset M_i$ and every~$\mssE^i$-quasi-open~$U^i\subset M^i$,
\[
\Cap_{\mssE}(A_1\times A_2, U_1\times U_2)\leq \Cap_{\mssE^1}(A_1,U_1) \norm{e_2}_{L^2(\nu_2)}^2+  \Cap_{\mssE^2}(A_2,U_2) \norm{e_1}_{L^2(\nu_1)}^2 \comma
\]
where~$e_i$ is the equilibrium $\mssE^i$-potential for the pair~$(A_i, U_i)$.
\begin{proof} It suffices to estimate
\begin{align*}
\Cap_\mssE&(A_1\times A_2,U_1\times U_2) 
\\
=&\ \inf \set{\mssE_1(u) : u\in\dom{\mssE^{U_1\times U_2}} \comma \car_{A_1\times A_2}\leq u\leq \car \ \mssE\text{-q.e.}}
\\
\leq&\ \inf\set{ \mssE_1(u_1\otimes u_2) : u_i\in\dom{\mssE^{i,U^i}}\comma \car_{A_i} \leq u_i \leq \car \ \mssE^i\text{-q.e.}\comma i=1,2}
\\
\leq&\ \inf\set{\begin{gathered}\mssE^1 (u_1)\norm{u_2}_{L^2(\nu_2)}^2+\mssE^2 (u_2) \norm{u_1}_{L^2(\nu_1)}^2+ \norm{u_1}_{L^2(\nu_1)}^2\norm{u_2}_{L^2(\nu_2)}^2 :\\ u_i\in\dom{\mssE^{i,U^i}}\comma \car_{A_i}\leq u_i \leq \car\ \mssE^i\text{-q.e.}\comma i=1,2 \end{gathered}}
\\
\leq&\ \inf\set{\begin{gathered}\mssE^1_1(u_1)\norm{u_2}_{L^2(\nu_2)}^2+\mssE^2_1(u_2) \norm{u_1}_{L^2(\nu_1)}^2 :\\ u_i\in\dom{\mssE^{i,U^i}}\comma \car_{A_i}\leq u_i \leq \car\ \mssE^i\text{-q.e.}\comma i=1,2 \end{gathered}}
\\
\leq&\ \mssE^1_1(e_1)\norm{e_2}_{L^2(\nu_2)}^2+\mssE^2_1(e_2) \norm{e_1}_{L^2(\nu_1)}^2 \fstop \qedhere
\end{align*}
\end{proof}
\end{lemma}

\subsubsection{Capacity estimates for conformal rescalings}\label{sss:AppCapacityRescaling}
For~$a>0$ denote by $\tparen{\mssE^{a},\dom{\mssE^{a}}}$ the Dirichlet form~$\tparen{a\mssE,\dom{\mssE}}$ on~$L^2(\nu)$, by~$\Cap_a$ 
the associated 
capacity, and by~$\mssW^a$ the associated Markov process.
For a Borel~$A\subset M$, further let
\[
\tau^a_A\eqdef \inf\set{t>0: X^{a}_t\in A}
\]
be the hitting time of~$A$ for~$\mssW^a$.

Finally, for~$a,b>0$, let~$\tparen{\mssE^{a,b},\dom{\mssE^{a,b}}}$ be the product Dirichlet form on~$L^2(\nu^\otym{2})$ constructed from~$\tparen{\mssE^a,\dom{\mssE^a}}$ and~$\tparen{\mssE^b,\dom{\mssE^b}}$, with capacity~$\Cap_{a,b}$.

\begin{lemma}\label{l:Capacity00}
If~$a\geq 1$, then
\[
 \Cap(A)^{1/a} \leq \Cap_a(A)\leq a\, \Cap(A)\fstop
\]
\begin{proof}
Let~$a>0$.
Since~$(\mssE,\dom{\mssE})$ is recurrent, so is~$\tparen{\mssE^{a},\dom{\mssE^{a}}}$.
Thus, it follows from e.g.~\cite[Thm.~4.2.5, cf.\ Eqn.~(4.1.9)]{FukOshTak11} that, for every Borel~$A\subset M$, for every~$a>0$,
\[
\Cap_{a}(A) = E_\nu\tbraket{e^{-\tau^{a}_A}}
\]
where
\begin{align*}
 \tau^a_A =&\ \inf\set{t>0: X_{a t} \in A} = \inf\set{t>0: X_t \in A}/a = \tau_A/a\fstop
\end{align*}
Assume now that~$a\geq 1$.
By the reverse Jensen inequality,
\[
\Cap_a(A)=E_\nu\tbraket{e^{-\tau^a_A}} = E_\nu \tbraket{e^{-\tau_A/a}} \geq E_\nu [e^{-\tau_A}]^{1/a} = \Cap(A)^{1/a} \fstop
\]

Let us now show the second inequality. For every~$a\geq 1$, for every $\mssE$-capacitable~$A$,
\begin{align*}
\Cap_a(A) &= \inf\set{a\mssE(u)+\norm{u}_{L^2(\nu)}^2 : \car_A\leq u\leq \car \ \mssE\text{-q.e.}}
\\
&\leq a  \inf\set{\mssE(u)+\norm{u}_{L^2(\nu)}^2 : \car_A\leq u\leq \car \ \mssE\text{-q.e.}} = a\, \Cap(A) \fstop \qedhere
\end{align*}
\end{proof}
\end{lemma}

\begin{lemma}\label{l:ProductPointHitting}
Let~$\kappa_\mssE\geq 0$ be as in~\eqref{eq:NonPolarityConstant}.
Then,~$\Cap_{a,b}(\Delta M) \geq \kappa_\mssE^{1/(a\wedge b)}$ for every~$a,b\geq 1$.
\begin{proof}
For each fixed~$x\in M$ denote the sections of~$u\colon M^\tym{2}\to\R$ by $u_{1,x}\eqdef u(\emparg,x)$ and~$u_{2,x}\eqdef u(x,\emparg)$.
Since~$\dom{\mssE^{a,b}}=\dom{\mssE^{1,1}}$ for every~$a,b>0$ and
\[
\mssE^{a,b}(u)= \int \mssE^a(u_{1,\emparg})\diff\nu+\int \mssE^b(u_{2,\emparg})\diff\nu = a\int \mssE(u_{1,\emparg})\diff\nu+b\int \mssE(u_{2,\emparg})\diff\nu\comma
\]
we have
\[
 \mssE^{a\wedge b,a \wedge b}\leq \mssE^{a,b}\leq \mssE^{a\vee b,a \vee b}\comma
\]
and therefore
\[
\Cap_{a\wedge b, a\wedge b}\leq \Cap_{a,b} \leq \Cap_{a\vee b, a\vee b} \comma \qquad a,b>0\fstop
\]
Applying Lemma~\ref{l:Capacity00} to the product form~$\tparen{\mssE^{a\wedge b,a\wedge b},\dom{\mssE^{a \wedge b,a\wedge b}}}$ we thus have
\begin{align}
\label{eq:l:ProductPointHitting:1}
\Cap_{1,1}(\Delta M)^{1/(a\wedge b)} &\leq \Cap_{a\wedge b,a \wedge b}(\Delta M) \leq \Cap_{a,b}(\Delta M)
\intertext{and analogously, applying Lemma~\ref{l:Capacity00} to the product form~$\tparen{\mssE^{a\vee b,a\vee b},\dom{\mssE^{a \vee b,a\vee b}}}$}
\label{eq:l:ProductPointHitting:2}
\Cap_{a,b}(\Delta M) &\leq \Cap_{a\vee b,a \vee b}(\Delta M) \leq (a\vee b)\, \Cap_{1,1}(\Delta M)\fstop
\end{align}

Since~$\tparen{\mssE,\dom{\mssE}}$ is regular, so is~$\tparen{\mssE^{1,1},\dom{\mssE^{1,1}}}$.
Thus,~$\dom{\mssE^{1,1}}\cap \Cz(\T^\tym{2})$ is a special standard core for~$\tparen{\mssE^{1,1},\dom{\mssE^{1,1}}}$.
In particular, for every~$\eps>0$ there exists~$e^\eps_{\Delta M}\in \dom{\mssE^{1,1}}\cap \Cz(\T^\tym{2})$ satisfying
\[
e^\eps_{\Delta M}\geq 0 \comma \qquad e^\eps_{\Delta M}\equiv 1 \quad \text{on} \quad \Delta M\comma\qquad \mssE^{1,1}_1(e^\eps_{\Delta M}) \leq \Cap_{1,1}(\Delta M) + \eps \fstop
\]
Then, for every~$x\in M$, the section~$e^\eps_{\Delta M}(\emparg, x)$ satisfies
\[
e^\eps_{\Delta M}(\emparg, x)\geq 0 \comma \qquad e^\eps_{\Delta M}(\emparg, x)\equiv 1 \quad \text{on} \quad \set{x} \comma
\]
and thus
\[
\Cap(\set{x})\leq \mssE_1\tparen{e^\eps_{\Delta M}(\emparg, x)} \fstop
\]
Integrating the above inequality over~$\nu$, we have
\begin{align*}
\kappa_\mssE\eqdef \int \Cap (\set{x}) \diff\nu(x) &\leq  \int \mssE_1\tparen{e^\eps_{\Delta M}(\emparg, x)} \diff\nu(x) \leq \mssE^{1,1}_1(e^\eps_{\Delta M})
\\
& \leq \Cap_{1,1}(\Delta M)+\eps \fstop
\end{align*}
Letting~$\eps\to 0$ and applying~\eqref{eq:l:ProductPointHitting:1} concludes the assertion.
\end{proof}
\end{lemma}

\begin{lemma}\label{l:Capacity2}
Assume~$\tparen{\mssE,\dom{\mssE}}$ satisfies~\eqref{eq:QPP}. Then~$\Cap_{a,b}(\Delta M)=0$ for every~$a,b\geq 1$.
\begin{proof}
Consequence of~\eqref{eq:l:ProductPointHitting:1},~\eqref{eq:l:ProductPointHitting:2} and the assumption.
\end{proof}
\end{lemma}

For a function~$g\colon M^\tym{2}\to \R$ we let~$g^\diag\colon M\to\R$ be the function~$g^\diag \colon x\mapsto g(x,x)$.

\begin{lemma}\label{l:DiagonalEqPot}
Assume~$\tparen{\mssE,\dom{\mssE}}$ satisfies~\eqref{eq:QPP}. Then, for every~$\eps>0$ there exists a symmetric function~$g_\eps=g_\eps^{a,b}\in \dom{\mssE^{a,b}}$ with the following properties:
\begin{enumerate}[$(a)$]
\item\label{i:l:DiagonalEqPot:1} $0\leq g_\eps \leq 1$ on~$M^\tym{2}$;
\item\label{i:l:DiagonalEqPot:2} $g^\diag_\eps \equiv 1$ $\nu$-a.e.\ on~$M$;
\item\label{i:l:DiagonalEqPot:3} $\mssE^{a,b}_1(g_\eps)\leq 2^{-1/\eps}$.
\end{enumerate}
\begin{proof}
Note that, if~$g_\eps$ satisfies~\ref{i:l:DiagonalEqPot:1}--\ref{i:l:DiagonalEqPot:3}, then so does its symmetrization.
Therefore it suffices to show~\ref{i:l:DiagonalEqPot:1}--\ref{i:l:DiagonalEqPot:3}.
By Lemma~\ref{l:Capacity2} we have~$\Cap_{a,b}(\Delta M)=0$, and the existence of such~$g_\eps$ follows from the very definition of capacity. 
\end{proof}
\end{lemma}

\subsubsection{Quantitative polarity of points on manifolds}
Let~$(M,g)$ be a complete smooth Riemannian manifold of dimension~$d\geq 1$, possibly with (smooth) boundary.
Further let~$\phi\colon M\to \R$ be smooth.
The canonical form~$(\mssE_{g,\phi},\dom{\mssE_{g,\phi}})$ on the weighted manifold~$(M,g, e^{d\, \phi}\vol_{g})$ is the closure on~$L^2(e^{d\, \phi}\vol_g)$ of the pre-Dirichlet form
 \[
\mssE(u,v)\eqdef \int g(\nabla u, \nabla v)\, e^{d\, \phi}\, \diff\vol_g\comma \qquad u,v\in \Cc^1(M) \fstop
\]
 
\begin{proposition}\label{p:VerificationPolarityManifolds}
Assume~$d\geq 2$.
Then, the canonical form satisfies~\eqref{eq:QPP}.
\begin{proof}
It is clear from the definition of~\eqref{eq:QPP} that the assertion is local in nature.
Thus, it suffices to show the statement when~$M$ is compact, possibly with boundary.
In this case, the role of the weight~$\phi$ is irrelevant, since~$e^{\inf\phi}\mssE_{g,\car}\leq \mssE_{g,\phi}\leq e^{\sup\phi}\mssE_{g,\car}$ with~$0<\inf \phi\leq \sup \phi <\infty$.
The presence of a (smooth) boundary too does not play any role.
Thus, it suffices to show the assertion for the canonical form on closed manifolds.
This follows as in the proof of~\cite[Prop.~A.16]{LzDS17+}.
Alternatively, the assertion can be deduced from standard results about the relation between (Bessel) capacities and Hausdorff dimension.
Indeed, since~$d\geq 2$, the diagonal~$\Delta M$ has (Hausdorff) co-dimension greater than or equal to~$2$, and it is thus $\mssE$-polar, cf.\ e.g.~\cite[Thm.~2.6.16, p.~75]{Zie89} with~$\alpha= 1$,~$p= 2$ and~$n\eqdef d$.
\end{proof}
\end{proposition}

\subsection{Fine properties of cylinder functions}
This is an Appendix to~\S\ref{sss:StandardCylFormDom}.
We fix~$\pi\in\msP(\mbfT_\circ)$ and work under Assumption~\ref{ass:SettingLocal}.

\begin{lemma}
Let~$k\in \N$,~$F\in \Cb^\infty(\R^k)$, and~$\seq{f_i}_{i\leq k}\subset \dom{\mssE}$. 
For every choice of everywhere defined $\nu$-representatives~$\rep{f_i}$ of~$f_i$, the function
\begin{equation}\label{eq:l:FullDomain:0}
\rep{u}\colon \mu \longmapsto (F\circ\rep{\mbff}^\trid)(\mu)
\end{equation}
is well-defined on~$\msP$ and $\mcQ_\pi$-square-integrable. Its class~$u\eqdef \class[\mcQ_\pi]{\rep{u}}$ is independent of the choice of representatives for the~$f_i$'s and satisfies
\[
u\in \dom{\mcE_\pi}\qquad \text{and} \qquad \mcG_\pi(u)=\sum_{i,j}^{k,k} (\partial_i F\cdot\partial_j F) \circ \rep{\mbff}^\trid \cdot \mssGamma(\hat f_i, \hat g_j)^\trid \as{\mcQ_\pi}\fstop
\]

\begin{proof}
It suffices to show the statement for~$\reptwo{u} = \reptwo{f}^\trid$. The conclusion for~$\rep{u}$ as in~\eqref{eq:l:FullDomain:0} follows by the diffusion property for~$\mcG_\pi$.
Let~$\seq{f_n}_n\subset \msA$ be $\mssE^{1/2}_1$-convergent to~$f$ and consider the sequence~$\seq{f_n^\trid}_n$.
By Theorem~\ref{t:TransferP}\ref{i:t:TransferP:2}, by~\eqref{eq:l:DiffusionP:0.1}, and since~$f\mapsto f^\trid$ is linear on~$\msA$,
\begin{align}
\nonumber
\mcE(f^\trid_n-f^\trid_m)=&\ \widehat\bmssE_\pi\tparen{(f_n^\trid-f_m^\trid) \circ\EM}
\\
\nonumber
=& \int \bmssGamma^\mbfs\paren{\sum_i s_i (f_n-f_m)(\emparg_i)}_{(\mbfs,\mbfx)} \diff\widehat\boldnu_\pi(\mbfs,\mbfx)
\\
\nonumber
=& \iint \sum_i s_i \mssGamma(f_n-f_m)_{x_i} \diff\boldnu(\mbfx)\diff\pi(\mbfs)
\\
\nonumber
=& \int \sum_i s_i \int \mssGamma(f_n-f_m)_{x_i}\diff\nu(x_i)\diff\pi(\mbfs)
\\
\label{eq:l:DomainCylinderF:1}
=&\ \mssE(f_n-f_m) \fstop
\end{align}
Together with Lemma~\ref{l:SimpleTridLp}, the equality in~\eqref{eq:l:DomainCylinderF:1} shows that~$\seq{f_n^\trid}_n$ is $\mcE^{1/2}_1$-fund\-amental.
As a consequence of Lemma~\ref{l:SimpleTridLp}\ref{i:l:SimpleTridLp:3} there exists~$\nlim f_n^\trid= f^\trid$, thus~$f^\trid\in\dom{\mcE_\pi}$ by closability of~$(\mcE_\pi,\dom{\mcE_\pi})$.
The independence of~$\mcG_\pi(u)$ from the representatives~$\rep{f_i}$'s is straightforward.
\end{proof}
\end{lemma}

\begin{lemma}\label{l:CapacityConnector}
Let~$\tparen{\mcE_\pi,\dom{\mcE_\pi}}$ be the Dirichlet form defined in Proposition~\ref{p:StandardCyl}.
Further let~$\molli\in\mcC^\infty(\mbbI)$,~$f\in\msA$, and~$g\in\dom{\mssE^\otym{2}}$, and assume that~$g$ is $\nu^\otym{2}$-a.e.\ symmetric, and~$g^\diag$ is $\nu$-a.e.\ identically constant.
Then, for every Borel representative~$\reptwo g$ of~$g$, the class~$u_{f,\reptwo{g},\molli}\eqdef \class[\mcQ_\pi]{\rep{u}_{f,\reptwo{g},\molli}}$ of
\begin{equation}\label{eq:l:CapacityConnector:0}
\reptwo u_{f,\reptwo{g},\molli}\colon \mu \longmapsto \int f(x)\, \molli\! \paren{\int \reptwo{g}(x,y)\diff\mu(y)}\diff\mu(x)
\end{equation}
is an element of~$L^2(\mcQ_\pi)$, independent of the chosen representative for~$g$, and satisfying~$u_{f,g,\molli}\in\dom{\mcE_\pi}$ and
\begin{align}\label{eq:l:CapacityConnector:1}
\mcE_\pi(u_{f,g,\molli})\leq 2\, C_{f,\molli} \paren{\mssE(f)+\mssE^{1,1}(g)}
\end{align}
for some constant~$C_{f,\molli}>0$ independent of~$g$.
\begin{proof}

The independence from the chosen representative~$\reptwo{g}$ of~$g$ follows estimating in a standard way the difference~$\tabs{u_{f,\reptwo{g}^1,\molli}-u_{f,\reptwo{g}^2,\molli}}$ for two different representatives~$\reptwo{g}^1$,~$\reptwo{g}^2$ of~$g$ and applying Lemma~\ref{l:SimpleTridLp}.
For the sake of simplicity, in the following we omit the notation of such representatives, thus simply writing~$g$ in place of~$\reptwo{g}$.

\paragraph{Step 1} We first show that~$u_{f,g,\molli}\in\dom{\widehat\mcE_\pi}$ and~\eqref{eq:l:CapacityConnector:1} holds (with~$\widehat\mcG_\pi$ in place of~$\mcG_\pi$).
In light of the intertwining in Theorem~\ref{t:TransferP}\ref{i:t:TransferP:2},
\begin{align*}
\widehat\mcG_\pi(u_{f,g,\molli})= \widehat\bmssGamma (u_{f,g,\molli}\circ \EM)\circ \EM^{-1} \fstop
\end{align*}
Thus, it suffices to show that~$u_{f,g,\molli}\in \dotloc{\dom{\widehat\mcE_\pi}}\cap L^2(\mcQ_\pi)$.
The fact that~$u_{f,g,\molli}$ is straightforward since~$f$ and~$\molli$ are both uniformly bounded.
Thus, it suffices to show that~\eqref{eq:l:CapacityConnector:1} holds. (Recall that~$\widehat\mcE_\pi=\mcE_\pi$ on~$\dom{\mcE_\pi}$.)

For~$\mu\eqdef \sum_i^N s_i\delta_{x_i}$, we have
\begin{align*}
\big(\widehat\bmssGamma& (u_{f,g,\molli}\circ \EM)\circ \EM^{-1}\big)_\mu =
\\
=&\ \widehat\bmssGamma\paren{\sum_i^N *_i f(\emparg_i)\, \molli\! \paren{\sum_j^N *_j g(\emparg_i,\emparg_j)}}_{(\mbfs,\mbfx)}
\\
=&\ \bmssGamma^\mbfs\paren{\sum_i^N s_i f(\emparg_i)\, \molli\! \paren{\sum_j^N s_j g(\emparg_i,\emparg_j)}}_\mbfx
\\
=&\ \sum_{i_1}^N\sum_{i_2}^N s_{i_1} s_{i_2}\bmssGamma^\mbfs\paren{f(\emparg_{i_1})\, \molli\! \paren{\sum_j^N s_j g(\emparg_{i_1},\emparg_j)},f(\emparg_{i_2}) \, \molli\! \paren{\sum_j^N s_j g(\emparg_{i_2},\emparg_j)} }_\mbfx
\\
=&\ \sum_{i_1}^N\sum_{i_2}^N s_{i_1} s_{i_2} \molli \paren{\sum_{j_1}^N s_{j_1} g(x_{i_1},x_{j_1})}\, \molli \paren{\sum_{j_2}^N s_{j_2} g(x_{i_2},x_{j_2})}\, \bmssGamma^\mbfs\tparen{f(\emparg_{i_1}) , f(\emparg_{i_2}) }_\mbfx
\\
&+2\,\sum_{i_1}^N\sum_{i_2}^N s_{i_1} s_{i_2}\, \molli \paren{\sum_{j_1}^N s_{j_1} g(x_{i_1},x_{j_1})} f(x_{i_2})
\\
&\qquad\qquad\cdot \bmssGamma^\mbfs\paren{f(\emparg_{i_1}) , \molli \paren{\sum_{j_2}^N s_{j_2} g(\emparg_{i_2},\emparg_{j_2})}}_\mbfx
\\
&+\sum_{i_1}^N\sum_{i_2}^N s_{i_1} s_{i_2} f(x_{i_1}) f(x_{i_2})
\\
&\qquad\qquad\cdot \bmssGamma^\mbfs\paren{\molli \paren{\sum_{j_1}^N s_{j_1} g(\emparg_{i_1},\emparg_{j_1})} , \molli \paren{\sum_{j_2}^N s_{j_2} g(\emparg_{i_2},\emparg_{j_2})}}_\mbfx
\\
\defeq&\ I_1+2I_2+I_3 \fstop
\end{align*}

As for the first term, since
\begin{align*}
\bmssGamma^\mbfs\tparen{f(\emparg_{i_1}), f(\emparg_{i_2})}(\mbfx)= \car_{i_1=i_2} \tbraket{s_{i_1}^{-1} \mssGamma_{i_1}\tparen{f(\emparg_{i_1}),f(\emparg_{i_1})} + s_{i_2}^{-1} \mssGamma_{i_2}\tparen{f(\emparg_{i_2}),f(\emparg_{i_2})}}_\mbfx\comma
\end{align*}
we have
\begin{align*}
I_1 =&\ 2\sum_i s_i \molli \paren{\sum_{j_1}^N s_{j_1} g(x_i,x_{j_1})}\, \molli \paren{\sum_{j_2}^N s_{j_2} g(x_i,x_{j_2})} \mssGamma(f)(x_i)
\\
=&\ 2\int \molli\tparen{g(x,\emparg)^\trid\mu}^2 \, \mssGamma(f)_x \diff\mu(x) \fstop
\end{align*}

As for the second term, since
\begin{align*}
&\bmssGamma^\mbfs\paren{f(\emparg_{i_1}) , \molli \paren{\sum_{j_2}^N s_{j_2} g(\emparg_{i_2},\emparg_{j_2})}}_\mbfx =
\\
&= \molli' \paren{\sum_{j_3}^N s_{j_3} g(x_{i_2},x_{j_3})} \bmssGamma^\mbfs\paren{f(\emparg_{i_1}) , \sum_{j_2}^N s_{j_2} g(\emparg_{i_2},\emparg_{j_2})}_\mbfx
\\
&= \sum_{j_2}^N s_{j_2} \molli' \paren{\sum_{j_3}^N s_{j_3} g(x_{i_2},x_{j_3})} \bmssGamma^\mbfs\tparen{f(\emparg_{i_1}) , g(\emparg_{i_2},\emparg_{j_2})}_\mbfx
\\
&= \sum_{j_2}^N s_{j_2} \molli' \paren{\sum_{j_3}^N s_{j_3} g(x_{i_2},x_{j_3})} \car_{i_1=i_2} \mssGamma_{i_1}\tparen{f(\emparg_{i_1}) , g(\emparg_{i_1},x_{j_2})}_\mbfx
\\
&\qquad + \sum_{j_2}^N s_{j_2} \molli' \paren{\sum_{j_3}^N s_{j_3} g(x_{i_2},x_{j_3})} \car_{i_1=j_2} \mssGamma_{i_1}\tparen{f(\emparg_{i_1}) , g(x_{i_2},\emparg_{i_1})}_\mbfx
\\
&= \sum_{j_2}^N s_{j_2} \molli' \paren{\sum_{j_3}^N s_{j_3} g(x_{i_2},x_{j_3})} \car_{i_1=i_2} s_{i_1}^{-1}\mssGamma\tparen{f , g(\emparg,x_{j_2})}_{x_{i_1}}
\\
&\qquad + \sum_{j_2}^N s_{j_2} \molli' \paren{\sum_{j_3}^N s_{j_3} g(x_{i_2},x_{j_3})} \car_{i_1=j_2} s_{i_1}^{-1}\mssGamma\tparen{f , g(x_{i_2},\emparg)}_{x_{i_1}}\comma
\end{align*}
we have
\begin{align*}
I_2=& \sum_{i_1}^N\sum_{j_2}^N s_{i_1} s_{j_2} \molli \paren{\sum_{j_1}^N s_{j_1} g(x_{i_1},x_{j_1})} f(x_{i_1})
\\
&\qquad\qquad\cdot \molli'\paren{\sum_{j_3}^N s_{j_3} g(x_{i_1},x_{j_3})}\mssGamma\tparen{f , g(\emparg,x_{j_2})}_{x_{i_1}}
\\
&\ + \sum_{i_1}^N\sum_{i_2}^N s_{i_1} s_{i_2} \molli \paren{\sum_{j_1}^N s_{j_1} g(x_{i_1},x_{j_1})} f(x_{i_2})
\\
&\qquad\qquad\cdot \molli' \paren{\sum_{j_3}^N s_{j_3} g(x_{i_2},x_{j_3})} \mssGamma\tparen{f , g(x_{i_2},\emparg)}_{x_{i_1}}
\\
=& \iint \molli\tparen{g(x,\emparg)^\trid\mu} f(x)\, \molli'\tparen{g(x,\emparg)^\trid\mu}\, \mssGamma\tparen{f,g(\emparg, y)}_x \diff\mu^\otym{2}(x,y)
\\
&\ + \iint \molli\tparen{g(x,\emparg)^\trid\mu} f(y)\, \molli'\tparen{g(y,\emparg)^\trid\mu}\, \mssGamma\tparen{f,g(y,\emparg)}_x \diff\mu^\otym{2}(x,y) \fstop
\end{align*}

As for the third term, by chain rule and bilinearity
\begin{align*}
\bmssGamma^\mbfs&\paren{\molli \paren{\sum_{j_1}^N s_{j_1} g(\emparg_{i_1},\emparg_{j_1})} , \molli \paren{\sum_{j_2}^N s_{j_2} g(\emparg_{i_2},\emparg_{j_2})}}_\mbfx=
\\
&= \molli' \paren{\sum_{k_1}^N s_{k_1} g(x_{i_1},x_{k_1})} \molli' \paren{\sum_{k_2}^N s_{k_2} g(x_{i_2},x_{k_2})}
\\
&\qquad\qquad\cdot\bmssGamma^\mbfs\paren{\sum_{j_1}^N s_{j_1} g(\emparg_{i_1},\emparg_{j_1}) , \sum_{j_2}^N s_{j_2} g(\emparg_{i_2},\emparg_{j_2})}_\mbfx
\\
&= \sum_{j_1}^N \sum_{j_2}^N s_{j_1} s_{j_2} \molli' \paren{\sum_{k_1}^N s_{k_1} g(x_{i_1},x_{k_1})} \molli' \paren{\sum_{k_2}^N s_{k_2} g(x_{i_2},x_{k_2})}
\\
&\qquad\qquad\cdot \bmssGamma^\mbfs\tparen{g(\emparg_{i_1},\emparg_{j_1}) , g(\emparg_{i_2},\emparg_{j_2})}_\mbfx\comma
\end{align*}
and since
\begin{align*}
\bmssGamma^\mbfs\tparen{g(\emparg_{i_1},\emparg_{j_1}) , g(\emparg_{i_2},\emparg_{j_2})}_\mbfx =&\ 
s_{i_2}^{-1}\car_{i_1=i_2, i_1\neq j_2, j_1\neq i_2, j_1\neq j_2,} \mssGamma \tparen{g(\emparg,x_{j_1}) , g(\emparg,x_{j_2})}_{x_{i_1}}
\\
&+ s_{j_2}^{-1} \car_{i_1=j_2,  i_1\neq i_2, j_1\neq i_2, j_1\neq j_2} \mssGamma \tparen{g(\emparg,x_{j_1}) , g(x_{i_2},\emparg)}_{x_{i_1}}
\\
&+ s_{i_2}^{-1} \car_{i_1\neq i_2, i_1\neq j_2, j_1 = i_2, j_1 \neq j_2} \mssGamma \tparen{g(x_{i_1},\emparg) , g(\emparg,x_{j_2})}_{x_{j_1}}
\\
&+ s_{j_2}^{-1} \car_{i_1\neq i_2, i_1\neq j_2, j_1 \neq i_2, j_1= j_2} \mssGamma \tparen{g(x_{i_1},\emparg) , g(x_{i_2},\emparg)}_{x_{j_1}}
\\
&+ s_{i_2}^{-1}s_{j_1}^{-1} \car_{i_1= i_2= j_1 \neq j_2} \mssGamma^\otym{2} \tparen{g(\emparg,\emparg) , g(\emparg, x_{j_2})}_{x_{i_1}}
\\
&+ s_{i_2}^{-1}s_{j_2}^{-1} \car_{j_1\neq i_1= i_2= j_2} \mssGamma^\otym{2} \tparen{g(\emparg,\emparg) , g(x_{j_1}, \emparg)}_{x_{i_1}}
\\
&+ s_{j_1}^{-1}s_{j_2}^{-1} \car_{i_1= j_1= j_2 \neq i_2} \mssGamma^\otym{2} \tparen{g(\emparg,x_{i_2}) , g(\emparg, \emparg)}_{x_{i_1}}
\\
&+ s_{j_1}^{-1}s_{j_2}^{-1} \car_{i_1\neq  i_2= j_1 = j_2} \mssGamma^\otym{2} \tparen{g(x_{i_1},\emparg) , g(\emparg, \emparg)}_{x_{i_2}}
\\
&+ s_{j_1}^{-1} s_{j_2}^{-1} \car_{i_1=i_2=j_1=j_2} \mssGamma^\otym{2}\tparen{g(\emparg,\emparg),g(\emparg,\emparg)}_{x_{i_1}}\comma
\end{align*}
we have
\begin{align*}
I_3=&\sum_{i_1}^N\sum_{i_2}^N s_{i_1} s_{i_2} f(x_{i_1}) f(x_{i_2})
\\
&\qquad\qquad\cdot \bmssGamma^\mbfs\paren{\molli \paren{\sum_{j_1}^N s_{j_1} g(\emparg_{i_1},\emparg_{j_1})} , \molli \paren{\sum_{j_2}^N s_{j_2} g(\emparg_{i_2},\emparg_{j_2})}}_\mbfx
\\
=& \sum_{i_1}^N\sum_{i_2}^N \sum_{j_1}^N \sum_{j_2}^N f(x_{i_1}) f(x_{i_2})\, \molli' \paren{\sum_{k_1}^N s_{k_1} g(x_{i_1},x_{k_1})} \molli' \paren{\sum_{k_2}^N s_{k_2} g(x_{i_2},x_{k_2})} 
\\
\cdot&
s_{i_1} s_{j_1} s_{j_2} \car_{i_1=i_2, i_1\neq j_2, j_1\neq i_2, j_1\neq j_2,} \mssGamma \tparen{g(\emparg,x_{j_1}) , g(\emparg,x_{j_2})}_{x_{i_1}}
\\
&+ s_{i_1} s_{i_2} s_{j_1} \car_{i_1=j_2,  i_1\neq i_2, j_1\neq i_2, j_1\neq j_2} \mssGamma \tparen{g(\emparg,x_{j_1}) , g(x_{i_2},\emparg)}_{x_{i_1}}
\\
&+ s_{i_1} s_{j_1} s_{j_2} \car_{i_1\neq i_2, i_1\neq j_2, j_1 = i_2, j_1 \neq j_2} \mssGamma \tparen{g(x_{i_1},\emparg) , g(\emparg,x_{j_2})}_{x_{j_1}}
\\
&+ s_{i_1} s_{i_2} s_{j_1} \car_{i_1\neq i_2, i_1\neq j_2, j_1 \neq i_2, j_1= j_2} \mssGamma \tparen{g(x_{i_1},\emparg) , g(x_{i_2},\emparg)}_{x_{j_1}}
\\
&+ s_{i_1} s_{j_2} \car_{i_1= i_2= j_1 \neq j_2} \mssGamma^\otym{2} \tparen{g(\emparg,\emparg) , g(\emparg, x_{j_2})}_{x_{i_1}}
\\
&+ s_{i_1} s_{j_1} \car_{j_1\neq i_1= i_2= j_2} \mssGamma^\otym{2} \tparen{g(\emparg,\emparg) , g(x_{j_1}, \emparg)}_{x_{i_1}}
\\
&+ s_{i_1} s_{i_2} \car_{i_1= j_1= j_2 \neq i_2} \mssGamma^\otym{2} \tparen{g(\emparg,x_{i_2}) , g(\emparg, \emparg)}_{x_{i_1}}
\\
&+ s_{i_1} s_{i_2} \car_{i_1\neq  i_2= j_1 = j_2} \mssGamma^\otym{2} \tparen{g(x_{i_1},\emparg) , g(\emparg, \emparg)}_{x_{i_2}}
\\
&+ s_{i_1} s_{i_2} \car_{i_1=i_2=j_1=j_2} \mssGamma^\otym{2}\tparen{g(\emparg,\emparg),g(\emparg,\emparg)}_{x_{i_1}}
\\
=& \iiint f(x)^2\, \molli'\tparen{g(x,\emparg)^\trid\mu}^2\, \mssGamma\tparen{g(\emparg, y),g(\emparg,z)}_x \diff\mu^\otym{3}(x,y,z)
\\
&+ \iiint f(x) f(y)\, \molli'\tparen{g(x,\emparg)^\trid\mu} \molli'\tparen{g(y,\emparg)^\trid\mu}\, \mssGamma\tparen{g(\emparg, z),g(y,\emparg)}_x \diff\mu^\otym{3}(x,y,z)
\\
&+ \iiint f(x) f(y) \molli'\tparen{g(x,\emparg)^\trid\mu} \molli'\tparen{g(y,\emparg)^\trid\mu}\, \mssGamma\tparen{g(x,\emparg),g(\emparg,z)}_y \diff\mu^\otym{3}(x,y,z)
\\
&+ \iiint f(x) f(y)\, \molli'\tparen{g(x,\emparg)^\trid\mu} \molli'\tparen{g(y,\emparg)^\trid\mu}\, \mssGamma\tparen{g(x,\emparg),g(y,\emparg)}_z \diff\mu^\otym{3}(x,y,z)
\\
&+ \iint f(x)^2\, \molli'\tparen{g(x,\emparg)^\trid\mu}^2\, \mssGamma^\otym{2}\tparen{g(\emparg,\emparg), g(\emparg,y)}_x \diff\mu^\otym{2}(x,y)
\\
&+ \iint f(x)^2\, \molli'\tparen{g(x,\emparg)^\trid\mu}^2\, \mssGamma^\otym{2}\tparen{g(\emparg,\emparg), g(y,\emparg)}_x \diff\mu^\otym{2}(x,y)
\\
&+ \iint f(x) f(y)\, \molli'\tparen{g(x,\emparg)^\trid\mu} \molli'\tparen{g(y,\emparg)^\trid\mu} \, \mssGamma^\otym{2}\tparen{g(\emparg,y),g(\emparg,\emparg)}_x \diff\mu^\otym{2}(x,y)
\\
&+ \iint f(x) f(y)\, \molli'\tparen{g(x,\emparg)^\trid\mu} \molli'\tparen{g(y,\emparg)^\trid\mu}\, \mssGamma^\otym{2}\tparen{g(x,\emparg), g(\emparg,\emparg)}_y \diff\mu^\otym{2}(x,y)
\\
&+ \iint f(x)^2 \, \molli'\tparen{g(x,\emparg)^\trid\mu}^2 \, \mssGamma^\otym{2}\tparen{g(\emparg,\emparg), g(\emparg,\emparg)}_x \car_{\set{x=y}} \diff\mu^\otym{2}(x,y) \fstop
\end{align*}

We may now choose a representative of~$g$ so that~$g^\diag$ is identically constant and~$g$ is everywhere symmetric.
In this way, the last five summands in~$I_3$ just above identically vanish by locality of~$\mssGamma$.
By symmetry of~$\mssGamma$ we further have
\begin{align*}
I_3=& \iiint f(x)^2\, \molli'\tparen{g(x,\emparg)^\trid\mu}^2\, \mssGamma\tparen{g(\emparg, y),g(\emparg,z)}_x \diff\mu^\otym{3}(x,y,z)
\\
&+ 2\iiint f(x) f(y)\, \molli'\tparen{g(x,\emparg)^\trid\mu} \molli'\tparen{g(y,\emparg)^\trid\mu}\, \mssGamma\tparen{g(\emparg, z),g(y,\emparg)}_x \diff\mu^\otym{3}(x,y,z)
\\
&+ \iiint f(x) f(y)\, \molli'\tparen{g(x,\emparg)^\trid\mu} \molli'\tparen{g(y,\emparg)^\trid\mu}\, \mssGamma\tparen{g(x,\emparg),g(y,\emparg)}_z \diff\mu^\otym{3}(x,y,z) \fstop
\end{align*}

Estimating~$f$ by~$\norm{f}_{\Cb}$ and~$\molli,\molli'$ by~$\norm{\molli}_{\mcC^1}$, by repeated applications of Cauchy--Schwarz inequality for~$\mssGamma$ and for the~$L^2(\mu)$-scalar product, we have
\begin{align*}
I_1 \leq&\ 2 \norm{\molli}_{\mcC^1}^2 \int \mssGamma(f)\diff\mu\comma
\\
I_2 \leq&\ 2 \norm{f}_{\Cb}\norm{\molli}_{\mcC^1}^2 \iint \mssGamma\tparen{f,g(\emparg, y)}_x\diff\mu^\otym{2}(x,y) 
\\
\leq&\ 2 \norm{f}_{\Cb}\norm{\molli}_{\mcC^1}^2 \braket{\int \mssGamma(f)\diff\mu}^{1/2} \braket{\iint \mssGamma\tparen{g(\emparg, y)}_x\diff\mu^\otym{2}(x,y)}^{1/2}\comma
\\
I_3 \leq&\ 4 \norm{f}_{\Cb}^2\norm{\molli}_{\mcC^1}^2 \iiint \mssGamma\tparen{g(\emparg,x),g(\emparg,y)}_z \diff\mu^\otym{3}(x,y,z)
\\
\leq&\ 4 \norm{f}_{\Cb}^2\norm{\molli}_{\mcC^1}^2 {\iint \mssGamma\tparen{g(\emparg,y)}_x \diff\mu^\otym{2}(x,y)} \fstop
\end{align*}

Integrating the above inequalities w.r.t.~$\mcQ_\pi$ and by Cauchy--Schwarz inequality for the $L^2(\mcQ_\pi)$-scalar product,
\begin{align*}
\widehat\mcE_\pi(u_{f,g,\molli}) \leq 2\, C_{f,\molli} \paren{\int \mssGamma(f)^\trid\mu \diff\mcQ_\pi(\mu)+2\int{ \!\!\!\int\tparen{\mssGamma\tparen{g(\emparg,y)}^\trid\mu} \diff\mu(y)} \diff\mcQ_\pi(\mu)}
\end{align*}
and, by Lemma~\ref{l:SimpleTridLp}, and symmetry of~$g$,
\begin{align*}
\widehat\mcE_\pi(u_{f,g,\molli}) \leq&\ 2\, C_{f,\molli} \paren{\mssE(f) +2 \int\mssGamma\tparen{g(\emparg,y)}_x\diff\nu^\otym{2}(x,y)}
\\
=&\ 2\, C_{f,\molli} \paren{\mssE(f)+\mssE^{1,1}(g)} \comma
\end{align*}
which concludes~\eqref{eq:l:CapacityConnector:1}. (Recall that~$\widehat\mcE_\pi=\mcE_\pi$ on~$\dom{\mcE_\pi}$.)

\paragraph{Step 2} The fact that~$u_{f,g,\molli}\in\dom{\mcE_\pi}$ follows now by a standard argument, approximating~$\molli$ in the~$\mcC^1$-topology with a polynomial, and approximating~$g\in\dom{\mssE^\otym{2}}$ in the $\mssE^\otym{2}_1$-norm with a function in tensor-product form.
We refer the to the proof of~\cite[Lem.~A.19]{LzDS17+} for details on this very same argument.
\end{proof}
\end{lemma}

\begin{lemma}\label{l:CoincidenceDomains}
Assume~$\tparen{\mssE,\dom{\mssE}}$ satisfies~\eqref{eq:QPP}, and let~$\tparen{\mcE_\pi,\dom{\mcE_\pi}}$ be the Dir\-ichlet form defined in Proposition~\ref{p:StandardCyl}. Then,~$\hCylP{}{0}\subset\dom{\mcE_\pi}$.

\begin{proof}
In light of the diffusion property~\eqref{eq:l:DiffusionP:0.1}, it suffices to show that~$\hat f^\trid\in\dom{\mcE_\pi}$ for every~$\hat f=\molli\otimes f$ with~$\molli\in\msR_0$ and~$f\in\msA$.
Let~$\seq{g_\eps}_{\eps}$ be the family of functions constructed in Lemma~\ref{l:DiagonalEqPot}, and choose representatives~$\rep{g_\eps}$ with~$\rep{g_\eps}^\diag= 1$ everywhere on~$M$ and~$\rep{g_\eps}\geq 0$ everywhere on~$M^\tym{2}$.

Now, consider the family of functions~$\seq{u_{f,\rep{g_\eps},\molli}}_\eps$  in~\eqref{eq:l:CapacityConnector:0}.
It suffices to show that~$\eps\mapsto \mcE_\pi(u_{f,\rep{g_\eps},\molli})$ is uniformly bounded and that~$u_{f,\rep{g_\eps},\molli}$ converges to~$\hat f^\trid$ in~$L^2(\mcQ_\pi)$, in which case the conclusion follows from~\cite[Lem.~I.2.12, p.~21]{MaRoe92}.

The first bound follows from~\eqref{eq:l:CapacityConnector:1} and Lemma~\ref{l:DiagonalEqPot}\ref{i:l:DiagonalEqPot:3}.

As for the $L^2(\mcQ_\pi)$ convergence, respectively by: Jensen's inequality and boundedness of~$f\in\msA\subset\Cz(\T)$, by smoothness of~$\molli$, and again by Jensen's inequality,
\begin{align*}
\int& \abs{\int\braket{\molli(\mu_x) - \molli\paren{\int \rep{g}_\eps(x,y) \diff\mu(y)}f(x)}\diff\mu(x)}^2\,\diff\mcQ_\pi(\mu)
\\
&\leq \norm{f}_0^2\int \abs{\molli(\mu_x)-\molli\paren{\int \int \rep{g_\eps}(x,y)\diff\mu(y)}}^2 \diff\mu(x) \diff\mcQ_\pi(\mu)
\\
&\leq \norm{f}_0^2 \norm{\molli'}_0^2 \int \int\abs{\mu_x-\int \rep{g_\eps}(x,y)\diff\mu(y)}^2 \diff\mu(x) \diff\mcQ_\pi(\mu)
 \\
&=\norm{f}_0^2 \norm{\molli'}_0^2 \iint \sum_i^N s_i \abs{s_i-\sum_j^N s_j\, \rep{g_\eps}(x_i,x_j)}^2 \diff\boldnu(\mbfx)\diff\pi(\mbfs)
\\
&=\norm{f}_0^2 \norm{\molli'}_0^2 \iint \sum_i^N s_i \abs{\sum_{j:j\neq i}^N s_j\, \rep{g_\eps}(x_i,x_j)}^2 \diff\boldnu(\mbfx)\diff\pi(\mbfs)
\\
&\leq \norm{f}_0^2 \norm{\molli'}_0^2 \int \sum_i^N s_i(1-s_i) \sum_{j:j\neq i}^N s_j\, \int\abs{\rep{g_\eps}(x_i,x_j)}^2 \diff\boldnu(\mbfx)
\\
&\leq \norm{f}_0^2 \norm{\molli'}_0^2 \norm{\rep{g_\eps}}_{L^2(\nu^\otym{2})}^2 \fstop
\end{align*}
Since the latter term vanishes as~$\eps\to 0$ by Lemma~\ref{l:DiagonalEqPot}\ref{i:l:DiagonalEqPot:3}, we have~$L^2(\mcQ_\pi)$-$\lim_{\eps\downarrow 0} u_{f,\rep{g_\eps},\molli}=\hat f^\trid$.
\end{proof}
\end{lemma}

\section{Direct integrals}\label{app:DirInt} 
We collect here an auxiliary lemma used in proof of Proposition~\ref{p:PDirIntESA}.
We recall only the main definitions about direct integrals of separable Hilbert spaces, referring to the monograph~\cite[\S\S{II.1}, II.2]{Dix81} for a systematic treatment, and to~\cite[\S\S2.2, 2.3]{LzDS20} for a brief and self-contained account of direct integrals of quadratic and Dirichlet forms.

\begin{definition}[Measurable fields,~{\cite[\S{II.1.3}, Dfn.~1,~p.~164]{Dix81}}]\label{d:DirInt}
\
Let~$(Z,\mcZ,\nu)$ be a $\sigma$-finite measure space,~$\seq{H_\zeta}_{\zeta\in Z}$ be a family of separable Hilbert spaces, and~$F$ be the linear space~$F\eqdef \prod_{\zeta\in Z} H_\zeta$. We say that~$\zeta\mapsto H_\zeta$ is a \emph{$\nu$-measurable field of Hilbert spaces} (\emph{with underlying space~$S$}) if there exists a linear subspace~$S$ of~$F$ with
\begin{enumerate}[$(a)$]
\item\label{i:d:DirInt1} for every~$u\in S$, the function~$\zeta\mapsto \norm{u_\zeta}_{\zeta}$ is $\nu$-measurable;
\item\label{i:d:DirInt2} if~$v\in F$ is such that $\zeta\mapsto \scalar{u_\zeta}{v_\zeta}_{\zeta}$ is $\nu$-measurable for every~$u\in S$, then~$v\in S$;
\item\label{i:d:DirInt3} there exists a sequence~$\seq{u_n}_n\subset S$ such that~$\seq{u_{n,\zeta}}_n$ is a total sequence\footnote{A sequence in a Banach space~$B$ is called \emph{total} if the strong closure of its linear span coincides with~$B$.} in~$H_\zeta$ for every~$\zeta\in Z$.
\end{enumerate}

Any such $S$ is called a \emph{space of $\nu$-measurable vector fields}. Any sequence in~$S$ possessing property~\iref{i:d:DirInt3} is called a \emph{fundamental} sequence.
\end{definition}

\begin{definition}[Direct integrals, {\cite[\S{II.1.5}, Prop.~5,~p.~169]{Dix81}}]
A $\nu$-measur\-able vector field~$u$ is called ($\nu$-)\emph{square-integrable} if
\begin{align}\label{eq:NormH}
\norm{u}\eqdef \braket{\int_Z \norm{u_\zeta}_{\zeta}^2 \, \diff\nu(\zeta)}^{1/2}<\infty\fstop
\end{align}

Two square-integrable vector fields~$u$,~$v$ are called ($\nu$-)\emph{equivalent} if~$\norm{u-v}=0$. The space~$H$ of equivalence classes of square-integrable vector fields, endowed with the non-relabeled quotient norm~$\norm{\emparg}$, is a Hilbert space~\cite[\S{II.1.5}, Prop.~5(i), p.~169]{Dix81}, called the \emph{direct integral of~$\zeta\mapsto H_\zeta$} (\emph{with underlying space~$S$}) and denoted by
\begin{align}\label{eq:DirInt}
H=\dint[]{Z} H_\zeta \diff\nu(\zeta) \fstop
\end{align}
\end{definition}

\begin{definition}[Measurable fields of closable operators]
Let~$H$ be as in~\eqref{eq:DirInt}. For each~$\zeta\in Z$ further let~$(A_\zeta,D_\zeta)$ be a closable operator on~$H_\zeta$.
A field of closable operators~$\zeta\mapsto (A_\zeta, D_\zeta)$ is called \emph{$\nu$-measurable} (\emph{with underlying space~$S$}) if
\begin{enumerate*}[$(a)$]
\item $\zeta\mapsto A_\zeta u_\zeta\in H_\zeta$ is a $\nu$-measurable vector field for every $\nu$-measurable vector field~$u$ with~$u_\zeta\in D_\zeta$ for every~$\zeta\in Z$;
\item there exists a sequence~$\seq{u_n}_n$ of $\nu$-measurable vector fields such that~$u_{n,\zeta}\in D_\zeta$ and~$\tseq{(u_{n,\zeta},A_\zeta u_{n,\zeta})}_n$ is a total sequence in~$\Graph(\overline{A_\zeta})$ for all~$\zeta\in Z$.
\end{enumerate*}
For a $\nu$-measurable vector field of closable operators~$\zeta\mapsto (A_\zeta, D_\zeta)$, its direct-integral operator~$(A,D)$ is the operator on~$H$ defined as
\begin{gather*}
D\eqdef \set{u\in H : \begin{gathered} u_\zeta\in D_\zeta \forallae{\nu} \zeta\in Z\comma \\ \zeta\mapsto A_\zeta u_\zeta \text{~is square-integrable}\end{gathered}}\comma
\\
A u \eqdef (\zeta\mapsto A_\zeta u_\zeta)\comma \qquad u\in D\fstop
\end{gather*}
\end{definition}

\begin{lemma}\label{l:DirIntESA}
Let~$\zeta\mapsto \tparen{A_\zeta,D_\zeta}$ be a measurable field of closable operators.
Then, $(A,D)$~is essentially self-adjoint on~$H$ if and only if~$(A_\zeta, D_\zeta)$ is essentially self-adjoint on~$H_\zeta$ for $\nu$-a.e.~$\zeta\in Z$.
\begin{proof}
Assume~$(A_\zeta, D_\zeta)$ is essentially self-adjoint for $\nu$-a.e.~$\zeta\in Z$.
Then, respectively by~\cite[Thm.~1.4]{Lan75} for the closure, by the assumption, and again by~\cite[Thm.~1.4]{Lan75} for the adjunction
\[
\overline{A}=\dint{Z} \overline{A_\zeta} \diff\nu(\zeta) = \dint{Z} A_\zeta^* \diff\nu(\zeta) = A^* \fstop
\]
The converse implication follows in a similar way.
\end{proof}
\end{lemma}

\end{appendix}

{\small
\bibliographystyle{abbrvurl}
\bibliography{/Users/lorenzodelloschiavo/Documents/MasterBib.bib}

\begin{thebibliography}{100}

\bibitem{AkeBaiDiF15}
G.~Akemann, J.~Baik, and P.~Di~Francesco, editors.
\newblock {\em {The Oxford Handbook of Random Matrix Theory}}.
\newblock Oxford University Press, September 2015.
\newblock \href {https://doi.org/10.1093/oxfordhb/9780198744191.001.0001}
  {\path{doi:10.1093/oxfordhb/9780198744191.001.0001}}.

\bibitem{AlbDalKon97}
S.~A. Albeverio, A.~{\relax Yu}. Daletskii, and {\relax Yu}.~G. Kondrat'ev.
\newblock {Infinite Systems of Stochastic Differential Equations and Some
  Lattice Models on Compact Riemannian Manifolds}.
\newblock {\em {Ukr.\ Math.\ J.}}, 49(3):360--372, 1997.
\newblock \href {https://doi.org/10.1007/BF02487239}
  {\path{doi:10.1007/BF02487239}}.

\bibitem{AlbDalKon00}
S.~A. Albeverio, A.~{\relax Yu}. Daletskii, and {\relax Yu}.~G. Kondrat'ev.
\newblock {Stochastic Analysis on Product Manifolds: Dirichlet Operators on
  Differential Forms}.
\newblock {\em {J.\ Funct.\ Anal.}}, 176:280--316, 2000.
\newblock \href {https://doi.org/10.1006jfan.2000.3629}
  {\path{doi:10.1006jfan.2000.3629}}.

\bibitem{AlbKonRoe98}
S.~A. Albeverio, {\relax Yu}.~G. Kondratiev, and M.~R{\"{o}}ckner.
\newblock {Analysis and Geometry on Configuration Spaces}.
\newblock {\em {J.\ Funct.\ Anal.}}, 154(2):444--500, 1998.
\newblock \href {https://doi.org/10.1006/jfan.1997.3183}
  {\path{doi:10.1006/jfan.1997.3183}}.

\bibitem{AlbKonRoe98b}
S.~A. Albeverio, {\relax Yu}.~G. Kondratiev, and M.~R{\"{o}}ckner.
\newblock {Analysis and Geometry on Configuration Spaces: The Gibbsian Case}.
\newblock {\em {J.\ Funct.\ Anal.}}, 157:242--291, 1998.
\newblock \href {https://doi.org/10.1006/jfan.1997.3215}
  {\path{doi:10.1006/jfan.1997.3215}}.

\bibitem{AlbRoe90}
S.~A. Albeverio and M.~R{\"{o}}ckner.
\newblock {Classical Dirichlet Forms on Topological Vector Spaces ---
  Closability and a Cameron--Martin Formula}.
\newblock {\em {J.\ Funct.\ Anal.}}, 88:395--436, 1990.
\newblock \href {https://doi.org/10.1016/0022-1236(90)90113-Y}
  {\path{doi:10.1016/0022-1236(90)90113-Y}}.

\bibitem{AlbRoe95}
S.~A. Albeverio and M.~R{\"{o}}ckner.
\newblock {Dirichlet form methods for uniqueness of martingale problems and
  applications}.
\newblock In M.~Cranston and M.~Pinsky, editors, {\em {Stochastic Analysis --
  Proceedings of the Summer Research Institute on Stochastic Analysis, Held at
  Cornell University, Ithaca, New York, July 11--30, 1993}}, volume~57 of {\em
  {Proceedings of Symposia in Pure Mathematics}}, pages 513--528. American
  Mathematical Society, 1994.
\newblock \href {https://doi.org/10.1090/pspum/057}
  {\path{doi:10.1090/pspum/057}}.

\bibitem{AltSim10}
A.~Altland and B.~Simons.
\newblock {\em {Condensed Matter Field Theory}}.
\newblock {Cambridge University Press}, {Second} edition, 2010.

\bibitem{AriHin05}
T.~Ariyoshi and M.~Hino.
\newblock {Small-time Asymptotic Estimates in Local Dirichlet Spaces}.
\newblock {\em {Electron.\ J.\ Probab.}}, 10(37):1236--1259, 2005.

\bibitem{BakGenLed14}
D.~Bakry, I.~Gentil, and M.~Ledoux.
\newblock {\em {Analysis and Geometry of Markov Diffusion Operators}}, volume
  348 of {\em {Grundlehren der mathematischen Wissenschaften}}.
\newblock {Springer}, 2014.
\newblock \href {https://doi.org/10.1007/978-3-319-00227-9}
  {\path{doi:10.1007/978-3-319-00227-9}}.

\bibitem{BauKel18}
F.~Baudoin and D.~Kelleher.
\newblock {Differential one-forms on Dirichlet spaces and Bakry-{\'{E}}mery
  estimates on metric graphs}.
\newblock {\em {Trans.\ Amer.\ Math.\ Soc.}}, 371(5):3145--3178, December 2018.
\newblock \href {https://doi.org/10.1090/tran/7362}
  {\path{doi:10.1090/tran/7362}}.

\bibitem{BenBre00}
J.-D. Benamou and Y.~Brenier.
\newblock {A computational fluid mechanics solution to the Monge-Kantorovich
  mass transfer problem}.
\newblock {\em {Num.\ Math.}}, 84:375--393, 2000.
\newblock \href {https://doi.org/10.1007/s002119900117}
  {\path{doi:10.1007/s002119900117}}.

\bibitem{BenSaC97}
A.~Bendikov and L.~Saloff-Coste.
\newblock {Elliptic diffusions on infinite products}.
\newblock {\em {J. reine angew. Math.}}, 493:171--220, 1997.

\bibitem{BenKipLan95}
O.~Benois, C.~Kipnis, and C.~Landim.
\newblock {Large deviations from the hydrodynamical limit of mean zero
  asymmetric zero range processes}.
\newblock {\em {Stoch.\ Proc.\ Appl.}}, 55(1):65--89, January 1995.
\newblock \href {https://doi.org/10.1016/0304-4149(95)91543-a}
  {\path{doi:10.1016/0304-4149(95)91543-a}}.

\bibitem{Ber76}
C.~Berg.
\newblock {Potential Theory on the Infinite Dimensional Torus}.
\newblock {\em {Invent.\ Math.}}, 32:49--100, 1976.

\bibitem{BerDeSGabJon15}
L.~Bertini, A.~De~Sole, D.~Gabrielli, G.~Jona-Lasinio, and C.~Landim.
\newblock Macroscopic fluctuation theory.
\newblock {\em {Rev.\ Mod.\ Phys.}}, 87(2):593--636, June 2015.
\newblock \href {https://doi.org/10.1103/revmodphys.87.593}
  {\path{doi:10.1103/revmodphys.87.593}}.

\bibitem{BezCimRoe18}
L.~Beznea, I.~C{\^\i}mpean, and M.~R{\"o}ckner.
\newblock Irreducible recurrence, ergodicity, and extremality of invariant
  measures for resolvents.
\newblock {\em Stochastic Processes and their Applications}, 128(4):1405--1437,
  April 2018.
\newblock \href {https://doi.org/10.1016/j.spa.2017.07.009}
  {\path{doi:10.1016/j.spa.2017.07.009}}.

\bibitem{Bil99}
P.~Billingsley.
\newblock {\em {Convergence of Probability Measures}}.
\newblock {Wiley Series in Probability and Mathematical Statistics}. {Wiley},
  {Second} edition, 1999.

\bibitem{BouHir91}
N.~Bouleau and F.~Hirsch.
\newblock {\em {Dirichlet forms and analysis on Wiener space}}.
\newblock {De Gruyter}, 1991.

\bibitem{Bre91}
Y.~Brenier.
\newblock {Polar Factorization and Monotone Rearrangement of Vector-Valued
  Functions}.
\newblock {\em {Comm.\ Pure Appl.\ Math.}}, 44:375--417, 1991.

\bibitem{Bre23}
P.~C. Bressloff.
\newblock {A generalized Dean--Kawasaki equation for an interacting Brownian
  gas in a partially absorbing medium}.
\newblock {\em {Proc.\ R.\ Soc.\ A}}, 480(2296), August 2024.
\newblock \href {https://doi.org/10.1098/rspa.2023.0915}
  {\path{doi:10.1098/rspa.2023.0915}}.

\bibitem{CarDelLasLio19}
P.~Cardaliaguet, F.~Delarue, J.-M. Lasry, and P.-L. Lions.
\newblock {\em {The Master Equation and the Convergence Problem in Mean Field
  Games}}, volume 201 of {\em {Annals of Mathematics Studies}}.
\newblock {Princeton University Press}, 2019.

\bibitem{CarDel18a}
R.~Carmona and F.~Delarue.
\newblock {\em {Probabilistic Theory of Mean Field Games with Applications~I}},
  volume~83 of {\em {Probability Theory and Stochastic Modelling}}.
\newblock Springer International Publishing, 2018.
\newblock \href {https://doi.org/10.1007/978-3-319-58920-6}
  {\path{doi:10.1007/978-3-319-58920-6}}.

\bibitem{CarDel18b}
R.~Carmona and F.~Delarue.
\newblock {\em {Probabilistic Theory of Mean Field Games with
  Applications~II}}.
\newblock {Probability Theory and Stochastic Modelling}. {Springer
  International Publishing}, 2018.
\newblock \href {https://doi.org/10.1007/978-3-319-56436-4}
  {\path{doi:10.1007/978-3-319-56436-4}}.

\bibitem{CecDel22}
A.~Cecchin and F.~Delarue.
\newblock {Weak solutions to the master equation of potential mean field
  games}.
\newblock {\em {arXiv:2204.04315}}, 2022.
\newblock \href {https://doi.org/10.48550/ARXIV.2204.04315}
  {\path{doi:10.48550/ARXIV.2204.04315}}.

\bibitem{CheSun06}
C.-Z. Chen and W.~Sun.
\newblock {Strong continuity of generalized Feynman--Kac semigroups: Necessary
  and sufficient conditions}.
\newblock {\em {J.\ Funct.\ Anal.}}, 237(2):446--465, aug 2006.
\newblock \href {https://doi.org/10.1016/j.jfa.2006.04.013}
  {\path{doi:10.1016/j.jfa.2006.04.013}}.

\bibitem{CheMaRoe94}
Z.-Q. Chen, Z.-M. Ma, and M.~R{\"{o}}ckner.
\newblock {Quasi-homeomorphisms of Dirichlet forms}.
\newblock {\em {Nagoya Math. J.}}, 136:1--15, 1994.

\bibitem{CheZha02}
Z.-Q. Chen and T.-S. Zhang.
\newblock {Girsanov And Feynman--Kac Type Transformations For Symmetric Markov
  Processes}.
\newblock {\em {Ann.~I.~H.~Poincar{\'{e}}~B}}, 38(4):475--505, 2002.

\bibitem{ChiPeySchVia16}
L.~Chizat, G.~Peyr{\'{e}}, B.~Schmitzer, and F.-X. Vialard.
\newblock {An Interpolating Distance Between Optimal Transport and Fisher--Rao
  Metrics}.
\newblock {\em {Found.\ Computat.\ Math.}}, 18(1):1--44, October 2016.
\newblock \href {https://doi.org/10.1007/s10208-016-9331-y}
  {\path{doi:10.1007/s10208-016-9331-y}}.

\bibitem{ChoGan17}
Y.~T. Chow and W.~Gangbo.
\newblock {A partial Laplacian as an infinitesimal generator on the Wasserstein
  space}.
\newblock {\em {J.\ Differ.\ Equat.}}, 267(10):6065--6117, 2019.

\bibitem{CipSav03}
F.~Cipriani and J.-L. Sauvageot.
\newblock {Derivations as square roots of Dirichlet forms}.
\newblock {\em {J.\ Funct.\ Anal.}}, 201(1):78--120, June 2003.
\newblock \href {https://doi.org/10.1016/s0022-1236(03)00085-5}
  {\path{doi:10.1016/s0022-1236(03)00085-5}}.

\bibitem{ConKraTamTon24}
G.~Conforti, R.~C. Kraaij, L.~Tamanini, and D.~Tonon.
\newblock {Hamilton--Jacobi equations for Wasserstein controlled gradient
  flows: existence of viscosity solutions}.
\newblock {\em {arXiv:2401.02240}}, 2024.
\newblock \href {https://doi.org/10.48550/ARXIV.2401.02240}
  {\path{doi:10.48550/ARXIV.2401.02240}}.

\bibitem{ConKraTon23b}
G.~Conforti, R.~C. Kraaij, and D.~Tonon.
\newblock {Hamilton--Jacobi equations for controlled gradient flows:
  cylindrical test functions}.
\newblock {\em {arXiv:2302.06571}}, 2023.
\newblock \href {https://doi.org/10.48550/ARXIV.2302.06571}
  {\path{doi:10.48550/ARXIV.2302.06571}}.

\bibitem{ConKraTon23}
G.~Conforti, R.~C. Kraaij, and D.~Tonon.
\newblock {Hamilton--Jacobi equations for controlled gradient flows: The
  comparison principle}.
\newblock {\em {J.~Funct.\ Anal.}}, 284(9):109853, May 2023.
\newblock \href {https://doi.org/10.1016/j.jfa.2023.109853}
  {\path{doi:10.1016/j.jfa.2023.109853}}.

\bibitem{CorFis23b}
F.~Cornalba and J.~Fischer.
\newblock {Multilevel Monte Carlo methods for the Dean--Kawasaki equation from
  Fluctuating Hydrodynamics}.
\newblock {\em {arXiv:2311.08872}}, 2023.
\newblock \href {https://doi.org/10.48550/ARXIV.2311.08872}
  {\path{doi:10.48550/ARXIV.2311.08872}}.

\bibitem{CorFis23}
F.~Cornalba and J.~Fischer.
\newblock {The Dean--Kawasaki Equation and the Structure of Density
  Fluctuations in Systems of Diffusing Particles}.
\newblock {\em {Arch.\ Rational Mech.\ Anal.}}, 247(5), August 2023.
\newblock \href {https://doi.org/10.1007/s00205-023-01903-7}
  {\path{doi:10.1007/s00205-023-01903-7}}.

\bibitem{CorFisIngRai23}
F.~Cornalba, J.~Fischer, J.~Ingmanns, and C.~Raithel.
\newblock {Density fluctuations in weakly interacting particle systems via the
  Dean--Kawasaki equation}.
\newblock {\em {arXiv:2303.00429}}, 2023.
\newblock \href {https://doi.org/10.48550/ARXIV.2303.00429}
  {\path{doi:10.48550/ARXIV.2303.00429}}.

\bibitem{Dav89}
{Davies, E.\ B.}
\newblock {\em {Heat kernels and spectral theory}}.
\newblock {Cambridge University Press}, 1989.

\bibitem{DePSodTam25}
N.~De~Ponti, G.~E. Sodini, and L.~Tamanini.
\newblock {The infimal convolution structure of the Hellinger-Kantorovich
  distance}.
\newblock {\em {arxiv:2503.12939}}, 2025.
\newblock \href {https://doi.org/10.48550/ARXIV.2503.12939}
  {\path{doi:10.48550/ARXIV.2503.12939}}.

\bibitem{Dea96}
D.~S. Dean.
\newblock {Langevin equation for the density of a system of interacting
  Langevin processes}.
\newblock {\em {J.~Phys.~A: Math.\ Gen.}}, 29:L613--L617, 1996.

\bibitem{DelHam22}
F.~Delarue and W.~R.~P. Hammersley.
\newblock {Rearranged Stochastic Heat Equation}.
\newblock {\em {Probab.\ Theory Relat.\ Fields}}, October 2024.
\newblock \href {https://doi.org/10.1007/s00440-024-01335-8}
  {\path{doi:10.1007/s00440-024-01335-8}}.

\bibitem{DelHam24}
F.~Delarue and W.~R.~P. Hammersley.
\newblock {Rearranged Stochastic Heat Equation: Ergodicity and Related Gradient
  Descent on the Space of Probability Measures}.
\newblock {\em {arXiv:2403.16140}}, 2024.
\newblock \href {https://doi.org/10.48550/arXiv.2403.16140}
  {\path{doi:10.48550/arXiv.2403.16140}}.

\bibitem{DelOllLopBlaHer16}
J.-B. Delfau, H.~Ollivier, C.~L{\'{o}}pez, B.~Blasius, and
  E.~Hern{\'{a}}ndez-Garc{\'{i}}a.
\newblock {Pattern formation with repulsive soft-core interactions: Discrete
  particle dynamics and Dean-Kawasaki equation}.
\newblock {\em {Phys.\ Rev.~E}}, 94(4), October 2016.
\newblock \href {https://doi.org/10.1103/physreve.94.042120}
  {\path{doi:10.1103/physreve.94.042120}}.

\bibitem{LzDS19a}
L.~Dello~Schiavo.
\newblock {Characteristic functionals of Dirichlet measures}.
\newblock {\em {Electron.\ J.\ Probab.}}, 24(115):1--38, 2019.
\newblock \href {https://doi.org/10.1214/19-EJP371}
  {\path{doi:10.1214/19-EJP371}}.

\bibitem{LzDS19b}
L.~Dello~Schiavo.
\newblock {A Rademacher-type theorem on $L^2$-Wasserstein spaces over closed
  Riemannian manifolds}.
\newblock {\em {J.\ Funct.\ Anal.}}, 278(6):108397, 2020.
\newblock {51~pp.}
\newblock \href {https://doi.org/10.1016/j.jfa.2019.108397}
  {\path{doi:10.1016/j.jfa.2019.108397}}.

\bibitem{LzDS17+}
L.~Dello~Schiavo.
\newblock {The Dirichlet--Ferguson Diffusion on the Space of Probability
  Measures over a Closed Riemannian Manifold}.
\newblock {\em {Ann.\ Probab.}}, 50(2):591--648, 2022.
\newblock \href {https://doi.org/10.1214/21-AOP1541}
  {\path{doi:10.1214/21-AOP1541}}.

\bibitem{LzDS20}
L.~Dello~Schiavo.
\newblock {Ergodic Decomposition of Dirichlet Forms via Direct Integrals and
  Applications}.
\newblock {\em {Potential Anal.}}, 58:573--615, 2023.
\newblock {43 pp.}
\newblock \href {https://doi.org/10.1007/s11118-021-09951-y}
  {\path{doi:10.1007/s11118-021-09951-y}}.

\bibitem{LzDSKon25}
L.~Dello~Schiavo and V.~Konarovskyi.
\newblock {Ill-posedness of the pure-noise Dean--Kawasaki equation}.
\newblock {\em arXiv:2501.09677}, 2025.
\newblock \href {https://doi.org/10.48550/arXiv.2501.09677}
  {\path{doi:10.48550/arXiv.2501.09677}}.

\bibitem{LzDSLyt17}
L.~Dello~Schiavo and E.~W. Lytvynov.
\newblock {A Mecke-type Characterization of the Dirichlet--Ferguson Measure}.
\newblock {\em {Electron.\ Commun.\ Probab.}}, 28:1--12, 2023.
\newblock \href {https://doi.org/10.1214/23-ECP528}
  {\path{doi:10.1214/23-ECP528}}.

\bibitem{LzDSSod24}
L.~Dello~Schiavo and G.~E. Sodini.
\newblock {The Hellinger--Kantorovich metric measure geometry on spaces of
  measures}.
\newblock {\em {arXiv:2503.07802}}, 2025.

\bibitem{LzDSSuz21}
L.~Dello~Schiavo and K.~Suzuki.
\newblock {Configuration Spaces over Singular Spaces I -- Dirichlet-Form and
  Metric Measure Geometry}.
\newblock {\em {arXiv:2109.03192}}, 2021.
\newblock {85~pp.}

\bibitem{LzDSSuz20}
L.~Dello~Schiavo and K.~Suzuki.
\newblock {Rademacher-type theorems and Sobolev-to-Lipschitz properties for
  strongly local Dirichlet spaces}.
\newblock {\em {J.~Funct.\ Anal.}}, 281(11):109234, dec 2021.
\newblock {63~pp.}
\newblock \href {https://doi.org/10.1016/j.jfa.2021.109234}
  {\path{doi:10.1016/j.jfa.2021.109234}}.

\bibitem{LzDSSuz22a}
L.~Dello~Schiavo and K.~Suzuki.
\newblock {Configuration Spaces over Singular Spaces II -- Curvature}.
\newblock {\em {arXiv: 2205.01379}}, 2022.

\bibitem{LzDSWir21}
L.~Dello~Schiavo and M.~Wirth.
\newblock {Ergodic Decompositions of Dirichlet Forms under Order Isomorphisms}.
\newblock {\em {J.~Evol.\ Equ.}}, 23(9), 2023.
\newblock {22~pp.}
\newblock \href {https://doi.org/10.1007/s00028-022-00859-7}
  {\path{doi:10.1007/s00028-022-00859-7}}.

\bibitem{DirStaZim16}
N.~Dirr, M.~Stamatakis, and J.~Zimmer.
\newblock {Entropic and gradient flow formulations for nonlinear diffusion}.
\newblock {\em {J.~Math.\ Phys.}}, 57(8), August 2016.
\newblock \href {https://doi.org/10.1063/1.4960748}
  {\path{doi:10.1063/1.4960748}}.

\bibitem{Dix81}
J.~Dixmier.
\newblock {\em {Von Neumann Algebras}}.
\newblock {North-Holland}, 1981.

\bibitem{DjuKrePer24}
A.~Djurdjevac, H.~Kremp, and N.~Perkowski.
\newblock {Weak error analysis for a nonlinear SPDE approximation of the
  Dean--Kawasaki equation}.
\newblock {\em {Stoch.\ PDE: Anal.\ Comp.}}, March 2024.
\newblock \href {https://doi.org/10.1007/s40072-024-00324-1}
  {\path{doi:10.1007/s40072-024-00324-1}}.

\bibitem{DonGri93}
P.~Donnelly and G.~Grimmett.
\newblock {On the asymptotic distribution of large prime factors}.
\newblock {\em {J.~London Math.\ Soc.}}, 47(3):395--404, 1993.

\bibitem{Ebe95}
A.~Eberle.
\newblock {Weak Sobolev spaces and Markov uniqueness of operators}.
\newblock {\em {C.R.\ Acad.\ Sci.\ Paris S{\'{e}}r.~I - Math.}},
  320:1249--1254, 1995.

\bibitem{Ebe96}
A.~Eberle.
\newblock {Girsanov-type transformations of local Dirichlet forms: An analytic
  approach}.
\newblock {\em {Osaka J.~Math.}}, 33(2):497--531, 1996.

\bibitem{Ebe99}
A.~Eberle.
\newblock {\em {Uniqueness and Non-Uniqueness of Semigroups Generated by
  Singular Diffusion Operators}}, volume 1718 of {\em {Lecture Notes in
  Mathematics}}.
\newblock {Springer}, 1999.

\bibitem{EmbDirZimRei18}
P.~Embacher, N.~Dirr, J.~Zimmer, and C.~Reina.
\newblock {Computing diffusivities from particle models out of equilibrium}.
\newblock {\em {Proc.\ R.~Soc.~A}}, 474(2212):20170694, April 2018.
\newblock \href {https://doi.org/10.1098/rspa.2017.0694}
  {\path{doi:10.1098/rspa.2017.0694}}.

\bibitem{Eng89}
R.~Engelking.
\newblock {\em {General Topology}}, volume~6 of {\em {Sigma series in pure
  mathematics}}.
\newblock {Heldermann}, {Berlin}, {1989}.

\bibitem{ErbHue15}
M.~Erbar and M.~Huesmann.
\newblock {Curvature bounds for configuration spaces}.
\newblock {\em {Calc.\ Var.}}, 54:307--430, 2015.
\newblock \href {https://doi.org/10.1007/s00526-014-0790-1}
  {\path{doi:10.1007/s00526-014-0790-1}}.

\bibitem{EthKur94}
S.~N. Ethier and T.~G. Kurtz.
\newblock {Convergence to Fleming--Viot process in the weak atomic topology}.
\newblock {\em {Stoch.\ Proc.\ Appl.}}, 54(1):1--27, {1994}.
\newblock \href {https://doi.org/10.1016/0304-4149(94)00006-9}
  {\path{doi:10.1016/0304-4149(94)00006-9}}.

\bibitem{FehGes23}
B.~Fehrman and B.~Gess.
\newblock {Non-equilibrium large deviations and parabolic-hyperbolic PDE with
  irregular drift}.
\newblock {\em {Invent.\ math.}}, 234(2):573--636, July 2023.
\newblock \href {https://doi.org/10.1007/s00222-023-01207-3}
  {\path{doi:10.1007/s00222-023-01207-3}}.

\bibitem{FehGes24}
B.~Fehrman and B.~Gess.
\newblock {Well-Posedness of the Dean--Kawasaki and the Nonlinear
  Dawson--Watanabe Equation with Correlated Noise}.
\newblock {\em {Arch.\ Rational Mech.\ Anal.}}, 248(2), March 2024.
\newblock \href {https://doi.org/10.1007/s00205-024-01963-3}
  {\path{doi:10.1007/s00205-024-01963-3}}.

\bibitem{Fer73}
T.~S. Ferguson.
\newblock {A Bayesian analysis of some nonparametric problems}.
\newblock {\em {Ann.\ Statist.}}, 1:{209--230}, 1973.
\newblock \href {https://doi.org/10.1214/aos/1176342360}
  {\path{doi:10.1214/aos/1176342360}}.

\bibitem{Fit97}
P.~J. Fitzsimmons.
\newblock {Absolute continuity of symmetric diffusions}.
\newblock {\em {Ann.\ Probab.}}, 25(1), January 1997.
\newblock \href {https://doi.org/10.1214/aop/1024404287}
  {\path{doi:10.1214/aop/1024404287}}.

\bibitem{FleVio79}
W.~H. Fleming and M.~Viot.
\newblock {Some Measure-Valued Markov Processes in Population Genetics Theory}.
\newblock {\em {Indiana J.\ Math.}}, 28(5):817--843, 1979.

\bibitem{ForSavSod22}
M.~Fornasier, G.~Savar{\'{e}}, and G.~E. Sodini.
\newblock {Density of subalgebras of Lipschitz functions in metric Sobolev
  spaces and applications to Wasserstein Sobolev spaces}.
\newblock {\em {J.~Funct.\ Anal.}}, 285(11):110153, dec 2023.
\newblock \href {https://doi.org/10.1016/j.jfa.2023.110153}
  {\path{doi:10.1016/j.jfa.2023.110153}}.

\bibitem{FruHay00}
H.~Frusawa and R.~Hayakawa.
\newblock {On the controversy over the stochastic density functional equation}.
\newblock {\em {J.~Phys.~A: Math.\ Gen.}}, 33:L155--L160, 2000.

\bibitem{Fug71}
B.~Fuglede.
\newblock {The quasi topology associated with a countably subadditive set
  function}.
\newblock {\em {Ann.\ I.\ Fourier}}, 21(1):123--169, 1971.

\bibitem{FukOshTak11}
M.~Fukushima, Y.~Oshima, and M.~Takeda.
\newblock {\em {Dirichlet forms and symmetric Markov processes}}, volume~19 of
  {\em {De Gruyter Studies in Mathematics}}.
\newblock {de Gruyter}, 2011.

\bibitem{GanMaySwi21}
W.~Gangbo, S.~Mayorga, and A.~{\'{S}}wi{\k{e}}ch.
\newblock {Finite Dimensional Approximations of Hamilton--Jacobi--Bellman
  Equations in Spaces of Probability Measures}.
\newblock {\em {SIAM J.~Math.\ Anal.}}, 53(2):1320--1356, January 2021.
\newblock \href {https://doi.org/10.1137/20m1331135}
  {\path{doi:10.1137/20m1331135}}.

\bibitem{GiaLebPre99}
G.~Giacomin, J.~Lebowitz, and E.~Presutti.
\newblock {\em {Deterministic and stochastic hydrodynamic equations arising
  from simple microscopic model systems}}, pages 107--152.
\newblock American Mathematical Society, November 1999.
\newblock \href {https://doi.org/10.1090/surv/064/03}
  {\path{doi:10.1090/surv/064/03}}.

\bibitem{Gig11}
N.~Gigli.
\newblock {On the inverse implication of Brenier--McCann theorems and the
  structure of $\big(\msP_2(M),W_2\big)$}.
\newblock {\em {Methods and Applications of Analysis}}, 18(2):127--158, 2011.

\bibitem{Gig12}
N.~Gigli.
\newblock {Second Order Analysis on~$\big(\msP_2(M),W_2\big)$}.
\newblock {\em {Mem.\ Am.\ Math.\ Soc.}}, 216(1018), 2012.

\bibitem{Gig13}
N.~Gigli.
\newblock {The splitting theorem in non-smooth context}.
\newblock {\em {arXiv:1302.5555}}, 2013.

\bibitem{Gig18}
N.~Gigli.
\newblock {Nonsmooth Differential Geometry -- An Approach Tailored for Spaces
  with Ricci Curvature Bounded from Below}.
\newblock {\em {Mem.\ Am.\ Math.\ Soc.}}, 251(1196), 2018.
\newblock \href {https://doi.org/10.1090/memo/1196}
  {\path{doi:10.1090/memo/1196}}.

\bibitem{Gil74}
T.~L. Gill.
\newblock {\em {Tensor products of contraction semigroups on Hilbert spaces}}.
\newblock PhD thesis, {Wayne State University}, 1974.

\bibitem{Gil78}
T.~L. Gill.
\newblock {Infinite tensor products of Banach spaces I}.
\newblock {\em {J.~Funct.\ Anal.}}, 30(1):17--35, October 1978.
\newblock \href {https://doi.org/10.1016/0022-1236(78)90052-6}
  {\path{doi:10.1016/0022-1236(78)90052-6}}.

\bibitem{Gri09}
{Grigor'yan, A.}
\newblock {\em {Heat Kernel and Analysis on Manifolds}}, volume~47 of {\em
  {Advanced Studies in Mathematics}}.
\newblock {AMS--IP}, 2009.

\bibitem{GubImkPer15}
M.~Gubinelli, P.~Imkeller, and N.~Perkowski.
\newblock Paracontrolled distributions and singular pdes.
\newblock {\em {Forum of Mathematics, Pi}}, 3, August 2015.
\newblock \href {https://doi.org/10.1017/fmp.2015.2}
  {\path{doi:10.1017/fmp.2015.2}}.

\bibitem{Gui69a}
A.~Guichardet.
\newblock {\em {Tensor Products of $C^*$-Algebras -- Part I. Finite Tensor
  Products}}.
\newblock Number~12 in {Lecture Notes Series}. {Aarhus Universitet}, 1969.

\bibitem{Gui69b}
A.~Guichardet.
\newblock {\em {Tensor Products of $C^*$-Algebras -- Part II. Infinite Tensor
  Products}}.
\newblock Number~13 in {Lecture Notes Series}. {Aarhus Universitet}, 1969.

\bibitem{GuiZeg02}
A.~Guionnet and B.~Zegarlinksi.
\newblock {Lectures on Logarithmic Sobolev Inequalities}.
\newblock In J.~Az{\'{e}}ma, M.~{\'{E}}mery, M.~Ledoux, and M.~Yor, editors,
  {\em {S{\'{e}}minaire de Probabilit{\'{e}}s XXXVI}}, volume 1801 of {\em
  {Lecture Notes in Mathematics}}, pages 1--134. Springer Berlin Heidelberg,
  2003.
\newblock \href {https://doi.org/10.1007/978-3-540-36107-7_1}
  {\path{doi:10.1007/978-3-540-36107-7_1}}.

\bibitem{Hai14}
M.~Hairer.
\newblock A theory of regularity structures.
\newblock {\em {Invent.\ math.}}, 198(2):269--504, March 2014.
\newblock \href {https://doi.org/10.1007/s00222-014-0505-4}
  {\path{doi:10.1007/s00222-014-0505-4}}.

\bibitem{Hin13}
M.~Hino.
\newblock {Measurable Riemannian structures associated with strong local
  Dirichlet forms}.
\newblock {\em {Math.\ Nachr.}}, 286(14--15):1466--1478, February 2013.
\newblock \href {https://doi.org/10.1002/mana.201200061}
  {\path{doi:10.1002/mana.201200061}}.

\bibitem{HinRam03}
M.~Hino and J.~A. Ram{\'{i}}rez.
\newblock {Small-Time Gaussian Behavior of Symmetric Diffusion Semigroups}.
\newblock {\em {Ann.\ Probab.}}, 31(3):1254--1295, 2003.

\bibitem{HinRoeTep13}
M.~Hinz, M.~R{\"{o}}ckner, and A.~Teplyaev.
\newblock Vector analysis for dirichlet forms and quasilinear pde and spde on
  metric measure spaces.
\newblock {\em {Stoch.\ Proc.\ Appl.}}, 123(12):4373--4406, December 2013.
\newblock \href {https://doi.org/10.1016/j.spa.2013.06.009}
  {\path{doi:10.1016/j.spa.2013.06.009}}.

\bibitem{Ill24}
P.~Illien.
\newblock {The Dean-Kawasaki equation and stochastic density functional
  theory}.
\newblock 2024.
\newblock \href {https://doi.org/10.48550/ARXIV.2411.13467}
  {\path{doi:10.48550/ARXIV.2411.13467}}.

\bibitem{IonRogTep12}
M.~Ionescu, L.~G. Rogers, and A.~Teplyaev.
\newblock {Derivations and Dirichlet forms on fractals}.
\newblock {\em {J.\ Funct.\ Anal.}}, 263(8):2141--2169, October 2012.
\newblock \href {https://doi.org/10.1016/j.jfa.2012.05.021}
  {\path{doi:10.1016/j.jfa.2012.05.021}}.

\bibitem{JacZim14}
R.~L. Jack and J.~Zimmer.
\newblock {Geometrical interpretation of fluctuating hydrodynamics in diffusive
  systems}.
\newblock {\em {J.~Phys.~A: Math.\ Theor.}}, 47(48):485001, November 2014.
\newblock \href {https://doi.org/10.1088/1751-8113/47/48/485001}
  {\path{doi:10.1088/1751-8113/47/48/485001}}.

\bibitem{Jan97}
S.~Janson.
\newblock {\em {Gaussian Hilbert Spaces}}.
\newblock Cambridge University Press, June 1997.
\newblock \href {https://doi.org/10.1017/cbo9780511526169}
  {\path{doi:10.1017/cbo9780511526169}}.

\bibitem{Kak48}
S.~Kakutani.
\newblock {On Equivalence of Infinite Product Measures}.
\newblock {\em {Ann.\ Math.}}, 49(1):214--224, Jan 1948.

\bibitem{Kaw94}
K.~Kawasaki.
\newblock {Stochastic model of slow dynamics in supercooled liquids and dense
  colloidal suspensions}.
\newblock {\em {Physica A: Statist.\ Mech.\ Appl.}}, 208(1):35--64, July 1994.
\newblock \href {https://doi.org/10.1016/0378-4371(94)90533-9}
  {\path{doi:10.1016/0378-4371(94)90533-9}}.

\bibitem{Kec95}
A.~S. Kechris.
\newblock {\em {Classical Descriptive Set Theory}}, volume 156 of {\em
  {Graduate Texts in Mathematics}}.
\newblock {Springer-Verlag}, {New York}, {1995}.

\bibitem{KimKawJacvWi14}
B.~Kim, K.~Kawasaki, H.~Jacquin, and F.~van Wijland.
\newblock Equilibrium dynamics of the dean-kawasaki equation: Mode-coupling
  theory and its extension.
\newblock {\em {Phys.\ Rev.~E}}, 89(1), January 2014.
\newblock \href {https://doi.org/10.1103/physreve.89.012150}
  {\path{doi:10.1103/physreve.89.012150}}.

\bibitem{Kin75}
J.~F.~C. Kingman.
\newblock {Random Discrete Distributions}.
\newblock {\em {J.~Roy.\ Stat.\ Soc.~B Met.}}, 37(1):1--22, 1975.

\bibitem{Kon17}
V.~V. Konarovskyi.
\newblock {A System of Coalescing Heavy Diffusion Particles on the Real Line}.
\newblock {\em {Ann.\ Probab.}}, 45(5):3293--3335, 2017.
\newblock \href {https://doi.org/10.1214/16-AOP1137}
  {\path{doi:10.1214/16-AOP1137}}.

\bibitem{KonLehvRe19}
V.~V. Konarovskyi, T.~Lehmann, and M.-K. von Renesse.
\newblock {Dean-Kawasaki dynamics: ill-posedness vs.~triviality}.
\newblock {\em {Electron.\ Commun.\ Probab.}}, 24, jan 2019.
\newblock \href {https://doi.org/10.1214/19-ecp208}
  {\path{doi:10.1214/19-ecp208}}.

\bibitem{KonLehvRe19b}
V.~V. Konarovskyi, T.~Lehmann, and M.-K. von Renesse.
\newblock {On Dean-Kawasaki Dynamics with Smooth Drift Potential}.
\newblock {\em {J.~Statist.\ Phys.}}, 178(3):666--681, nov 2019.
\newblock \href {https://doi.org/10.1007/s10955-019-02449-3}
  {\path{doi:10.1007/s10955-019-02449-3}}.

\bibitem{KonMue23}
V.~V. Konarovskyi and F.~M{\"{u}}ller.
\newblock {Dean--Kawasaki equation with initial condition in the space of
  positive distributions}.
\newblock {\em {J.~Evol.\ Equ.}}, 24(4), October 2024.
\newblock \href {https://doi.org/10.1007/s00028-024-01018-w}
  {\path{doi:10.1007/s00028-024-01018-w}}.

\bibitem{KonvRe18}
V.~V. Konarovskyi and M.-K. von Renesse.
\newblock {Modified Massive Arratia flow and Wasserstein diffusion}.
\newblock {\em {Comm.\ Pure Appl.\ Math.}}, 72(4):764--800, 2019.
\newblock \href {https://doi.org/10.1002/cpa.21758}
  {\path{doi:10.1002/cpa.21758}}.

\bibitem{KonvRe17}
V.~V. Konarovskyi and M.-K. von Renesse.
\newblock {Reversible coalescing-fragmentating Wasserstein dynamics on the real
  line}.
\newblock {\em {J.~Funct.\ Anal.}}, 286(8):110342, April 2024.
\newblock \href {https://doi.org/10.1016/j.jfa.2024.110342}
  {\path{doi:10.1016/j.jfa.2024.110342}}.

\bibitem{KonLytVer15}
{\relax Yu}.~G. Kondratiev, E.~W. Lytvynov, and A.~M. Vershik.
\newblock {Laplace operators on the cone of Radon measures}.
\newblock {\em {J.\ Funct.\ Anal.}}, 269(9):2947--2976, 2015.
\newblock \href {https://doi.org/10.1016/j.jfa.2015.06.007}
  {\path{doi:10.1016/j.jfa.2015.06.007}}.

\bibitem{KonMonVor16}
S.~Kondratyev, L.~Monsaingeon, and D.~Vorotnikov.
\newblock {A new optimal transport distance on the space of finite Radon
  measures}.
\newblock {\em {Adv.\ Diff. Equ.}}, 21(11/12), November 2016.
\newblock \href {https://doi.org/10.57262/ade/1476369298}
  {\path{doi:10.57262/ade/1476369298}}.

\bibitem{Kuw98b}
K.~Kuwae.
\newblock {Recurrence and conservativeness of symmetric diffusion processes by
  Girsanov transformations}.
\newblock {\em {J.~Math.\ Kyoto Univ.}}, 37(4), January 1997.
\newblock \href {https://doi.org/10.1215/kjm/1250518212}
  {\path{doi:10.1215/kjm/1250518212}}.

\bibitem{Kuw98}
K.~Kuwae.
\newblock {Functional Calculus for Dirichlet Forms}.
\newblock {\em {Osaka J.~Math.}}, 35:683--715, 1998.

\bibitem{Lan75}
C.~Lance.
\newblock {Direct Integrals of Left Hilbert Algebras}.
\newblock {\em {Math.\ Ann.}}, 216:11--28, 1975.

\bibitem{LeJRai04}
Y.~Le~Jan and O.~Raimond.
\newblock {Flows, Coalescence and Noise}.
\newblock {\em {Ann.\ Probab.}}, 32(2):1247--1315, 2004.
\newblock \href {https://doi.org/10.1214/009117904000000207}
  {\path{doi:10.1214/009117904000000207}}.

\bibitem{LieMieSav17}
M.~Liero, A.~Mielke, and G.~Savar{\'{e}}.
\newblock {Optimal Entropy-Transport problems and a new Hellinger--Kantorovich
  distance between positive measures}.
\newblock {\em {Invent.\ math.}}, 211(3):969--1117, dec 2017.
\newblock \href {https://doi.org/10.1007/s00222-017-0759-8}
  {\path{doi:10.1007/s00222-017-0759-8}}.

\bibitem{LieMieSav22}
M.~Liero, A.~Mielke, and G.~Savar{\'{e}}.
\newblock {Fine Properties of Geodesics and Geodesic $\lambda$-Convexity for
  the Hellinger--Kantorovich Distance}.
\newblock {\em {Arch.\ Rational Mech.\ Anal.}}, 247(6), November 2023.
\newblock \href {https://doi.org/10.1007/s00205-023-01941-1}
  {\path{doi:10.1007/s00205-023-01941-1}}.

\bibitem{Lot07}
J.~Lott.
\newblock {Some Geometric Calculations on Wasserstein Space}.
\newblock {\em {Comm.\ Math.\ Phys.}}, 277(2):423--437, 2007.

\bibitem{LunMetPal20}
A.~Lunardi, G.~Metafune, and D.~Pallara.
\newblock {The Ornstein--Uhlenbeck semigroup in finite dimension}.
\newblock {\em {Phil.\ Trans.\ Royal Soc.~A}}, 378(2185):20200217, October
  2020.
\newblock \href {https://doi.org/10.1098/rsta.2020.0217}
  {\path{doi:10.1098/rsta.2020.0217}}.

\bibitem{MaRoe92}
Z.-M. Ma and M.~R{\"{o}}ckner.
\newblock {\em Introduction to the Theory of (Non-Symmetric) Dirichlet Forms}.
\newblock {Graduate Studies in Mathematics}. Springer, 1992.

\bibitem{MarTar99}
U.~Marini Bettolo~Marconi and P.~Tarazona.
\newblock {Dynamic density functional theory of fluids}.
\newblock {\em {J.~Chem.\ Phys.}}, 110(16):8032--8044, April 1999.
\newblock \href {https://doi.org/10.1063/1.478705}
  {\path{doi:10.1063/1.478705}}.

\bibitem{MarMay22}
A.~Martini and A.~Mayorcas.
\newblock {An Additive Noise Approximation to Keller--Segel--Dean--Kawasaki
  Dynamics Part I: Local Well-Posedness of Paracontrolled Solutions}.
\newblock {\em {arXiv:2207.10711v2}}, 2022.

\bibitem{McC97}
R.~J. McCann.
\newblock {A Convexity Principle for Interacting Gases}.
\newblock {\em {Adv.\ Math.}}, 128, 1997.

\bibitem{McC01}
R.~J. McCann.
\newblock {Polar factorization of maps on Riemannian manifolds}.
\newblock {\em {Geom.\ Funct.\ Anal.}}, 11(3):589--608, 2001.

\bibitem{MuevReZim24}
F.~M{\"{u}}ller, M.-K. von Renesse, and J.~Zimmer.
\newblock {Well-Posedness for Dean--Kawasaki Models of Vlasov-Fokker-Planck
  Type}.
\newblock {\em {arXiv:2411.14334}}, nov 2024.
\newblock \href {https://doi.org/10.48550/arXiv.2411.14334}
  {\path{doi:10.48550/arXiv.2411.14334}}.

\bibitem{Nel73}
E.~Nelson.
\newblock {The free Markoff field}.
\newblock {\em {J.~Funct.\ Anal.}}, 12(2):211--227, February 1973.
\newblock \href {https://doi.org/10.1016/0022-1236(73)90025-6}
  {\path{doi:10.1016/0022-1236(73)90025-6}}.

\bibitem{Osa96}
H.~Osada.
\newblock {Dirichlet Form Approach to Infinite-Dimensional Wiener Processes
  with Singular Interactions}.
\newblock {\em {Comm.\ Math.\ Phys.}}, 176:117--131, 1996.

\bibitem{OsaTan20}
H.~Osada and H.~Tanemura.
\newblock {Infinite-dimensional stochastic differential equations and tail
  $\sigma$-fields}.
\newblock {\em {Probab.\ Theory Relat.\ Fields}}, 177:1137--1242, 2020.
\newblock \href {https://doi.org/10.1007/s00440-020-00981-y}
  {\path{doi:10.1007/s00440-020-00981-y}}.

\bibitem{OshPopDie17}
G.~Oshanin, M.~N. Popescu, and S.~Dietrich.
\newblock {Active colloids in the context of chemical kinetics}.
\newblock {\em {J.\ Phys.~A: Math.\ Theor.}}, 50(13):134001, March 2017.
\newblock \href {https://doi.org/10.1088/1751-8121/aa5e91}
  {\path{doi:10.1088/1751-8121/aa5e91}}.

\bibitem{Ott01}
F.~Otto.
\newblock {The Geometry of Dissipative Evolution Equations: The Porous Medium
  Equation}.
\newblock {\em {Comm. Part. Diff. Eq.}}, 26(1-2):101--174, 2001.

\bibitem{OveRoeSch95}
L.~Overbeck, M.~R{\"{o}}ckner, and B.~Schmuland.
\newblock {An analytic approach to Fleming--Viot processes with interactive
  selection}.
\newblock {\em {Ann.\ Probab.}}, 23(1):1--36, 1995.

\bibitem{Per02}
E.~Perkins.
\newblock {Dawson-Watanabe Superprocesses and Measure-Valued Diffusions}.
\newblock In P.~Bernard, editor, {\em {Lectures on Probability Theory and
  Statistics: Ecole d'Et{\'{e}} de Probabilit{\'{e}}s de Saint-Flour XXIX -
  1999}}. Springer Berlin Heidelberg, 2002.
\newblock \href {https://doi.org/10.1007/b93152} {\path{doi:10.1007/b93152}}.

\bibitem{QuaRezVar99}
J.~Quastel, F.~Rezakhanlou, and S.~R.~S. Varadhan.
\newblock {Large deviations for the symmetric simple exclusion process in
  dimensions $d\geq3$}.
\newblock {\em {Probab.\ Theory Relat.\ Fields}}, 113(1):1--84, February 1999.
\newblock \href {https://doi.org/10.1007/s004400050202}
  {\path{doi:10.1007/s004400050202}}.

\bibitem{ReeSim75}
M.~C. Reed and B.~Simon.
\newblock {\em {Methods of Modern Mathematical Physics II -- Fourier Analysis,
  Self-Adjointness}}.
\newblock Academic Press, New York, London, 1975.

\bibitem{ReeSim80a}
M.~C. Reed and B.~Simon.
\newblock {\em {Methods of Modern Mathematical Physics I -- Functional
  Analysis}}.
\newblock {Academic Press}, 1980.

\bibitem{RenWan21}
P.~Ren and F.-Y. Wang.
\newblock {Derivative formulas in measure on Riemannian manifolds}.
\newblock {\em {Bull.\ London Math.\ Soc.}}, 53(6):1786--1800, October 2021.
\newblock \href {https://doi.org/10.1112/blms.12542}
  {\path{doi:10.1112/blms.12542}}.

\bibitem{RenWan24}
P.~Ren and F.-Y. Wang.
\newblock {Ornstein--Uhlenbeck Type Processes on Wasserstein Space}.
\newblock {\em {Stoch.\ Proc.\ Appl.}}, 172:104339, June 2024.
\newblock \href {https://doi.org/10.1016/j.spa.2024.104339}
  {\path{doi:10.1016/j.spa.2024.104339}}.

\bibitem{RoeSch95}
M.~R{\"{o}}ckner and B.~Schmuland.
\newblock {Quasi-regular Dirichlet forms: Examples and Counterexamples}.
\newblock {\em {Can.\ J.\ Math.}}, 47(1):165--200, 1995.

\bibitem{Sav14}
G.~Savar{\'{e}}.
\newblock {Self-Improvement of the Bakry--{\'{E}}mery Condition and Wasserstein
  Contraction of the Heat Flow in $\mathrm{RCD}(K,\infty)$ Metric Measure
  Spaces}.
\newblock {\em {Discr.\ Cont.\ Dyn.\ Syst.}}, 34(4):1641--1661, 2014.
\newblock \href {https://doi.org/10.3934/dcds.2014.34.1641}
  {\path{doi:10.3934/dcds.2014.34.1641}}.

\bibitem{Sch97}
A.~Schied.
\newblock {Geometric Aspects of Fleming--Viot and Dawson--Watanabe Processes}.
\newblock {\em {Ann.\ Probab.}}, 25(3):1160--1179, 1997.

\bibitem{ShuEijvdB07}
L.~Shui, J.~C. Eijkel, and A.~van~den Berg.
\newblock {Multiphase flow in microfluidic systems -- Control and applications
  of droplets and interfaces}.
\newblock {\em {Adv.\ Colloid Interface Sci.}}, 133(1):35--49, May 2007.
\newblock \href {https://doi.org/10.1016/j.cis.2007.03.001}
  {\path{doi:10.1016/j.cis.2007.03.001}}.

\bibitem{Stu11}
K.-T. Sturm.
\newblock {Entropic Measure on Multidimensional Spaces}.
\newblock In {Dalang, R.}, {Dozzi, M.}, and {Russo, F.}, editors, {\em {Seminar
  on Stochastic Analysis, Random Fields and Applications VI}}, volume~63 of
  {\em {Progr.\ Probab.}}, pages 261--277. {Birkh{\"{a}}user/Springer Basel
  AG}, Basel, 2011.

\bibitem{Stu24}
K.-T. Sturm.
\newblock {Wasserstein Diffusion on Multidimensional Spaces}.
\newblock {\em {arxiv:2401.12721}}, 2024.
\newblock \href {https://doi.org/10.48550/ARXIV.2401.12721}
  {\path{doi:10.48550/ARXIV.2401.12721}}.

\bibitem{VruLoeWitt20}
M.~te~Vrugt, H.~L{\"{o}}wen, and R.~Wittkowski.
\newblock Classical dynamical density functional theory: from fundamentals to
  applications.
\newblock {\em {Adv.\ Phys.}}, 69(2):121--247, April 2020.
\newblock \href {https://doi.org/10.1080/00018732.2020.1854965}
  {\path{doi:10.1080/00018732.2020.1854965}}.

\bibitem{tElRobSikZhu06}
A.~F.~M. ter Elst, D.~W. Robinson, A.~Sikora, and Y.~Zhu.
\newblock {Dirichlet Forms and Degenerate Elliptic Operators}.
\newblock In E.~Koelink, J.~van Neerven, B.~de~Pagter, G.~H. Sweers, A.~Luger,
  and H.~Woracek, editors, {\em {Partial Differential Equations and Functional
  Analysis: The Philippe Cl{\'e}ment Festschrift}}, pages 73--95.
  {Birkh{\"a}user}, 2006.

\bibitem{VenIllGraBen25}
D.~Venturelli, P.~Illien, A.~Grabsch, and O.~B{\'e}nichou.
\newblock {Dynamics of soft interacting particles on a comb}.
\newblock {\em {arXiv:2502.16951}}, 2025.
\newblock \href {https://doi.org/10.48550/ARXIV.2502.16951}
  {\path{doi:10.48550/ARXIV.2502.16951}}.

\bibitem{Neu39}
J.~von Neumann.
\newblock {On infinite direct products}.
\newblock {\em {Compositio Math.}}, 6:1--77, 1939.

\bibitem{vReStu09}
M.-K. von Renesse and K.-T. Sturm.
\newblock {Entropic measure and Wasserstein diffusion}.
\newblock {\em {Ann.\ Probab.}}, 37(3):1114--1191, 2009.

\bibitem{Wu00}
L.~Wu.
\newblock {Uniqueness of Nelson's Diffusions II: Infinite Dimensional Setting
  and Applications}.
\newblock {\em {Potential Anal.}}, 13(3):269--301, 2000.
\newblock \href {https://doi.org/10.1023/a:1008707703191}
  {\path{doi:10.1023/a:1008707703191}}.

\bibitem{Zie89}
W.~P. Ziemer.
\newblock {\em {Weakly Differentiable Functions: Sobolev Spaces and Functions
  of Bounded Variation}}, volume 120 of {\em {Graduate Texts in Mathematics}}.
\newblock {Springer}, 1989.

\end{thebibliography}
}

\end{document}